\documentclass[reqno,11pt]{amsart}
\usepackage{amssymb,amsmath,amsthm,}
  \usepackage{a4wide} 
\usepackage{amscd}
\usepackage{amsfonts}
\usepackage{amssymb}
\usepackage{latexsym}
\usepackage{color}
\usepackage{esint}

\usepackage{graphicx}
\graphicspath{ {images/} }

\usepackage[makeroom]{cancel}

\newtheorem{theorem}{Theorem}

\newtheorem{corollary}[theorem]{Corollary}
\newtheorem{definition}{Definition}
\newtheorem{lemma}{Lemma}
\newtheorem{proposition}[theorem]{Proposition}
\newtheorem{remark}{Remark}
 \newtheorem*{theorem*}{Rough version of the Main theorem}

\let\e=\varepsilon

\let\p=\partial

\let\O=\Omega

\let\o=\omega

\numberwithin{equation}{section}

\let\hide\iffalse
\let\unhide\fi

\DeclareMathAlphabet{\mathpzc}{OT1}{pzc}{m}{it}

\newcommand{\R}{\mathbb{R}}

\renewcommand{\S}{\mathbb{S}}

\newcommand{\be}{\begin{equation}}
\newcommand{\bm}{\begin{multline}}
\newcommand{\ee}{\end{equation}}
\newcommand{\dd}{\mathrm{d}}

\newcommand{\xb}{x_{\mathbf{b}}}
\newcommand{\tb}{t_{\mathbf{b}}}

\newcommand{\f}{\frac}

\newcommand{\Bes}{\begin{eqnarray*}}
\newcommand{\Ees}{\end{eqnarray*}}
\newcommand{\Be}{\begin{equation} }
\newcommand{\Ee}{\end{equation}}
\newcommand{\Bs}{\begin{split}}   

\newcommand{\vertiii}[1]{{\left\vert\kern-0.25ex\left\vert\kern-0.25ex\left\vert #1 
    \right\vert\kern-0.25ex\right\vert\kern-0.25ex\right\vert}}
    
    \newcommand{\vertip}[1]{{\left\{\kern-0.7ex\left\{\kern-0.25ex
    #1 
     \right\}\kern
    -0.7ex\right\}}}

 \makeatletter
\def\munderbar#1{\underline{\sbox\tw@{$#1$}\dp\tw@\z@\box\tw@}}
\makeatother

\def\p{\partial}

\def\O{\Omega}
\def\R{\mathbb{R}}

\def\B{\begin{equation}}
\def\E{\end{equation}}
\def\BN{\begin{eqnarray*}}
\def\EN{\end{eqnarray*}}

\def\bcb{\begin{color}{blue}}
\def\ec{\end{color}}

\def\bcr{\begin{color}{red}}
\def\ec{\end{color}}

\begin{document}

\title{
Incompressible Euler limit from Boltzmann equation with Diffuse Boundary Condition for Analytic data
}

 \author{Juhi Jang}
 \address{Department of Mathematics, University of Southern California, Los Angeles, CA, 90089 USA, email: juhijang@usc.edu}
 \author{Chanwoo Kim}
 \address{Department of Mathematics, University of Wisconsin-Madison, Madison, WI, 53706, USA, email: ckim.pde@gmail.com; chanwoo.kim@wisc.edu
 }

\maketitle


\begin{abstract}  
A rigorous derivation of the incompressible Euler equations with the no-penetration boundary condition from the Boltzmann equation with the diffuse reflection boundary condition has been a challenging open problem. We settle this open question in the affirmative when the initial data of fluid are well-prepared in a real analytic space, in 3D half space. As a key of this advance we capture the Navier-Stokes equations of  
\Be
\textit{viscosity} 
 \sim \frac{\textit{Knudsen number}}{\textit{Mach number}} \notag
\Ee
 satisfying the no-slip boundary condition, as an \textit{intermediary} approximation of the Euler equations through a new Hilbert-type expansion of the Boltzmann equation with the diffuse reflection boundary condition. Aiming to justify the approximation we establish a novel quantitative $L^p$-$L^\infty$ estimate of the Boltzmann perturbation around a local Maxwellian of such viscous approximation, along with the commutator estimates and the integrability gain of the hydrodynamic part in various spaces; we also establish direct estimates of the Navier-Stokes equations in higher regularity with the aid of the initial-boundary and boundary layer weights using a recent Green's function approach. The incompressible Euler limit follows as a byproduct of our framework.

 \end{abstract}


\tableofcontents

\section{Introduction}

An important and active research direction in mathematical physics/PDE is on the so-called Hilbert's sixth problem \cite{Hilbert} seeking a unified theory of the gas dynamics including different levels of descriptions from a mathematical standpoint by connecting the behavior of solutions to equations from kinetic theory to solutions of other systems that arise in formal limits, such as the N-body problem, the Euler equations, the Navier-Stokes equations, etc. In particular, the hydrodynamic limit of the Boltzmann equation has received a great deal of attention and enthusiasm in the mathematics and physics communities
since the pioneering work \cite{Hilbert1} by Hilbert, which was the first example of his sixth problem. Remarkably, all the basic fluid equations of  compressible, incompressible,  inviscid or viscous fluid dynamics can be derived from the Boltzmann equation of a rarefied gas dynamics upon the choice of appropriate scalings in small mean free path limit. Though formal derivations are rather well-understood, as far as mathematical justifications go, despite great progress over the decades 
(for example see \cite{BGL91, BGL93, EGKM2, GSR, Guo2006, Laure} and the references therein), full understanding of the hydrodynamic limit incorporating important physical applications such as boundary effects or physical phenomena is still far from being complete. The goal of this paper is to make a rigorous connection between the Boltzmann equation and the incompressible Euler equations in the presence of the boundary by bypassing the {\it inviscid limit} of the incompressible Navier-Stokes equations.


The dimensionless Boltzmann equation with the \textit{Strouhal number} $\mathpzc{St}$ and the \textit{Knudsen number} $\mathpzc{Kn}$ takes the form of 
\Be\label{Boltzmann}
\mathpzc{St} \p_t F  +  v\cdot \nabla_x F  = \f{1}{\mathpzc{Kn}} Q(F ,F ) .
\Ee 
Here the distribution function of the gas is denoted by $F(t,x,v) \geq 0$ with the time variable $t \in \R_+: = \{ t \geq 0\}$, the space variable $x = (x_1,x_2,x_3)\in \O \subset \R^3$, and the velocity variable $v  = (v_1,v_2,v_3) \in \R^3$.  The Boltzmann collision operator $Q(\cdot, \cdot)$ of the hard sphere takes the form of 
\Be\label{Q}
\begin{split}
Q(F,G)  = \frac{1}{2} \int_{\R^3} \int_{\S^2}  |(v-v_*) \cdot  \mathfrak{u} | \{ &
F(v^\prime) G(v^\prime_*) + G(v^\prime) F(v^\prime_*)\\
 &- F(v ) G(v_* ) - G(v ) F(v_* ) 
\}\dd \mathfrak{u} \dd v_*, \end{split}\Ee
where $v^\prime := v- ((v-v_*) \cdot  \mathfrak{u} )  \mathfrak{u} $ and $v_*^\prime := v_*+ ((v-v_*) \cdot  \mathfrak{u} )  \mathfrak{u} $. 
This operator satisfies  the so-called collision invariance property: for any $F(v)\text{ and } G(v) $ decaying sufficiently fast as $|v|\rightarrow 0$, 
\Be\label{collision_inv}
\int_{\R^3} Q(F,G) (v)  \Big( 1 , v  , \frac{| {v } |^2-3}{\sqrt{6}}
 \Big) \dd v=(0,0,0), 
\Ee
which represents the local conservation laws of mass, momentum and energy. The celebrated Boltzmann's H-theorem reveals the entropy dissipation: 
\Be\label{entropy}
\int_{\R^3} Q(F,F) (v)  \ln F(v) \dd v \leq 0, 
\Ee
for any $F(v)>0$ decaying sufficiently fast as $|v|\rightarrow 0$. An intrinsic equilibrium, satisfying $Q(\cdot, \cdot )=0$, is given by a local Maxwellian associated with the density $R>0$, the macroscopic velocity $U\in \R^3$ and the temperature $T>0$  
\Be\label{Maxwellian}
M_{R, U , T} (v): = \f{R}{(2\pi T)^{ \f{3}{2}}}\exp \left\{- \frac{|v-U|^2}{2T}\right\}, 
\Ee
which is known as the only configuration attaining the equality in \eqref{entropy}. 

In addition to the Strouhal number and Knudsen number we introduce the \textit{Mach number} $\mathpzc{Ma}$. By passing $\mathpzc{St}, \mathpzc{Kn},$ and $\mathpzc{Ma}$ to zero, one can formally derive PDEs of hydrodynamic variables for the fluctuations around the reference state $(1,0,1)$, which are determined as 
 \Be
 \Big( \rho (t,x), 
 {u} (t,x),
 \theta  (t,x)
 \Big)
 = \lim_{\mathpzc{Ma} \downarrow 0}
\f{1}{\mathpzc{Ma}}\int_{\R^3}
\{F  (t,x,v) -M_{1,0,1} (v)  \}
 \Big( 
1 ,
 v, 
 \frac{|v|^2-3}{\sqrt{6}}
 \Big) 
 \dd v  
 .\label{hydro_limit}
\Ee
The famous Reynolds number appears as a ratio between the Knudsen number and Mach number through  the von Karman relation: 
\Be\label{Re}
\frac{1}{\mathpzc{Re}} =   \frac{\mathpzc{Kn}}{\mathpzc{Ma}}. 
\Ee
For instance, the incompressible Navier-Stokes equations with $\mathpzc{Re}=1$, namely the viscosity of order one, can be derived by setting $\mathpzc{St}= \mathpzc{Ma}=\mathpzc{Kn}= \e$ as $ \e \downarrow 0$.  
 In this paper we are particularly interested in a scale of {\it large Reynolds number} as follows:
\Be\label{scale}
\mathpzc{St}=\e = \mathpzc{Ma}  \  \ \text{and} \ \ 
\mathpzc{Kn} =\kappa \e  \ \ \text{with} \ 
\kappa= \kappa(\e) \downarrow 0 \  \text{as} \   \e \downarrow 0, 
\Ee
through which we will derive the incompressible Euler equations with the no-penetration boundary condition in the limit 
\begin{align}
\p_t   {u} _E +   {u}_E  \cdot \nabla_x  {u}_E  +\nabla_x p _E  =0, \ 
\nabla_x \cdot   {u} _E  =0
 \ \  &\text{in} \ \ \O,\label{Euler} 
 \\
 u _E\cdot n  =0  \ \ &\text{on} \ \ \p\O,\label{no-pen}
\end{align}   
with
$\p_t \theta  + u _E \cdot \nabla_x \theta  =0$ and $
\nabla_x \theta  (t,x) +\nabla_x  \rho  (t,x)= 0$. Here $n=n(x)$ denotes an outward normal at $x$ on the boundary $\p\O$. 
For the sake of simplicity we set an initial datum $\theta_{0}(x)= 0=\rho_{0}(x)$ so that  
\Be\label{theta=zero}
\theta(t,x) = 0=\rho(t,x) \text{ for all } t \geq 0.
\Ee

In many important physical applications such as a turbulence theory, it would be relevant to take into account the physical boundary in the hydrodynamic limit. A boundary condition of the Boltzmann equation is determined by the interaction law of the gas with the boundary surface. One of the physical conditions is the so-called \textit{diffuse reflection boundary condition}, which takes into account an instantaneous thermal equilibration of reflecting 
gas particle (see \cite{DV,EGKM}): for $ (x,v) \in 
\{ \p\O \times \R^3: n(x) \cdot v<0\},$
\Be\label{diffuse_BC}
F (t,x,v)  = c_\mu M_{1,0,1}(v) \int_{n(x) \cdot \mathfrak{v}>0} F  (t,x,\mathfrak{v}) (n(x) \cdot \mathfrak{v}) \dd\mathfrak{v} ,
\Ee
where we have taken an isothermal boundary with a rescaled temperature $1$ for the sake of simplicity. Here, the normalization constant $c_\mu:= 1/\big(\int_{n(x) \cdot \mathfrak{v}>0} M_{1,0,1}  (\mathfrak{v}) (n(x) \cdot \mathfrak{v}) \dd\mathfrak{v}\big)$ leads to 
the null flux condition 
$
\int_{\R^3} F(t,x,v) (n(x) \cdot v) \dd v=0 \text{ on }x \in \p\O$.
\hide
In many important physical applications such as a turbulence theory, it would be relevant to take into account the physical boundary in the hydrodynamic limit. A boundary condition of the Boltzmann equation is determined by the interaction law of the gas with the boundary surface. One of the most physical conditions is the so-called \textit{diffuse reflection boundary condition}, which takes into account an instantaneous thermal equilibration of bounce-back gas particle (see \cite{DV,EGKM}). For the sake of simplicity we consider an isothermal boundary with a rescaled temperature $1$: 
\Be\label{diffuse_BC}
F (t,x,v)  = c_\mu M_{1,0,1}(v) \int_{n(x) \cdot \mathfrak{v}>0} F  (t,x,\mathfrak{v}) (n(x) \cdot \mathfrak{v}) \dd\mathfrak{v} \ \text{ for }  (x,v) \in 
\{ \p\O \times \R^3: n(x) \cdot v<0\}.
\Ee
Here $n(x)$ denotes an outward normal at $x$ on the boundary $\p\O$, while $c_\mu:= 1/\big(\int_{n(x) \cdot \mathfrak{v}>0} M_{1,0,1}  (\mathfrak{v}) (n(x) \cdot \mathfrak{v}) \dd\mathfrak{v}\big)$ implies the null flux condition $\int_{\R^3} F(t,x,v) (n(x) \cdot v) \dd v=0$ on $x \in \p\O$.

Perhaps inspired by (\ref{hydro_limit}) and (\ref{scale}) one might attempt to expand $F(t,x,v)$ around a local Maxwellian $M_{1,\e  u_E, 1}(v)$ associated with the Euler flow (\ref{Euler})-(\ref{no-pen}). Unfortunately this local Maxwellian does not honor the diffuse reflection boundary condition when the Euler flow satisfies a no-penetration boundary condition (\ref{no-pen}). %
In fact a size of the boundary mismatch could be an order of tangential component of the Euler flow $u_E$ at the boundary.

Therefore a uniform bound to verify the limit (\ref{hydro_limit}) is not expected even in the formal level. It is worth to note that such a mismatch does not appear at least in a formal level when a specular reflection boundary condition, $F(t,x,v)= F(t,x,R_x v)$ on $x \in \p\O$ where $R_x v= v- 2n(x) (n(x) \cdot v)$, is imposed (see Chapter 5 of \cite{Laure}).
\unhide
In particular, it is well-known that 
the diffuse boundary condition (\ref{diffuse_BC}) is a kinetic boundary condition featuring  a mismatch with the no-penetration boundary condition (\ref{no-pen}) of the the Euler flow under \eqref{scale}, without any small parameter with respect to $1/\mathpzc{Re}$ or $\mathpzc{Ma}$. One can readily see this by  
\hide
 Aiming for the limit (\ref{hydro_limit}) with the scale (\ref{scale}), 
 as a first attempt, one might try to study $F(t,x,v)$ through the  expansion of 
 \unhide
expanding $F$ around a local Maxwellian $M_{1,\e  u_E, 1}(v)$ associated with a flow of the no-penetration boundary condition (\ref{no-pen}) \textit{directly}. 
Unfortunately, this local Maxwellian \textit{does not honor}  
the diffuse reflection boundary condition when a flow satisfies \textit{the no-penetration boundary condition} (\ref{no-pen}).  %
In fact a size of the boundary mismatch could be an order of the tangential component of the Euler flow $u_E$ at the boundary. Therefore  a uniform bound to verify the limit (\ref{hydro_limit}) in a scale of large Reynolds number (\ref{scale}) is not expected even at the formal level.  This poses  a major obstacle in the Euler limit from the Boltzmann equation with the diffuse reflection boundary.  It is worth 
 noting that such a mismatch does not appear at least at the formal level when the specular reflection boundary condition is imposed: $F(t,x,v)= F(t,x,R_x v)$ on $x \in \p\O$ where $R_x v= v- 2n(x) (n(x) \cdot v)$; while the mismatch can possess a small factor for  the so-called Maxwell boundary condition, a convex combination of the specular reflection and the diffuse reflection boundary conditions, by choosing the coefficient for diffuse reflection known as \textit{the accommodation constant} to vanish as $\mathpzc{Re}\rightarrow \infty$.

 Remarkably, an analogous, better-known boundary mismatch phenomenon exists in the realm of mathematical fluid dynamics, specifically in the inviscid limit problem of the Navier-Stokes equations that addresses the validity of the Euler solutions as the leading order approximation of the Navier-Stokes solutions in the vanishing viscosity limit. The inviscid limit for the no-slip boundary condition features a boundary mismatch between two different boundary conditions for the Navier-Stokes and Euler flows. 
In fact, whether the solution to the Navier-Stokes equations with a $ \kappa\eta_0 $-viscosity (a physical constant $\eta_0$ can be computed explicitly from the Boltzmann theory as in \eqref{eta_0})  satisfying the no-slip boundary condition    
\begin{align}
\p_t u+ u \cdot \nabla_x u - \kappa \eta_0 \Delta u + \nabla_x p  &=0 \ \ \text{in} \ \O, \label{NS_k}\\
 \nabla_x \cdot u &=0 \ \ \text{in} \ \O,
 \label{incomp}
 \\
 \ \ u  &=0  \ \ \text{on} \ \p\O, \label{noslip}
\end{align} 
converges to the solution of the Euler equations satisfying the no-penetration boundary condition \eqref{Euler}-\eqref{no-pen} in $\kappa =  {1} / {\mathpzc{Re}} \downarrow 0$ is an outstanding problem, which is arguably the most relevant and challenging because of the mismatch of two boundary conditions between \eqref{noslip} and \eqref{no-pen} resulting in the formation of boundary layers such as Prandtl  layer and unbounded vorticity near the boundary. While the verification of the inviscid limit is still largely open, it holds under certain symmetry assumption on the domain and data or under the flat boundary and strong regularity such as analyticity at least near the boundary \cite{MM16}.  A classical way to tackle the inviscid limit problem is to study the Prandtl expansion \cite{SC1, SC2, Mae14}: $u(t,x_1,x_2,x_3)= u_E(t,x_1,x_2,x_3) + u_P (t,x_1,x_2,\frac{x_3}{\sqrt\kappa})+ O(\sqrt\kappa)$. Recently, different frameworks that avoid the boundary layer expansion have become available \cite{NN2018,KVW,FW}. 

The incompressible Euler limit from the Boltzmann equation turns out to be intimately tied to the inviscid limit of the incompressible Navier-Stokes equations, which accounts for the similarity of two boundary mismatches. A beautiful connection stems from the Navier-Stokes solutions of \eqref{NS_k}-\eqref{noslip} in large Reynolds numbers: at least formally, not only they are approximated by the Euler equations \eqref{Euler}-\eqref{no-pen} but also they approximate 
the Boltzmann equation \eqref{Boltzmann} under  \eqref{scale} with \eqref{diffuse_BC}, in fact better than the Euler equations  \eqref{Euler}-\eqref{no-pen} at each Mach number  $\e>0$, because the Navier-Stokes equations contain a high order correction term $\kappa \eta_0 \Delta u$ that captures the dissipative nature of the Boltzmann collision operator (as we will see in Section \ref{sec:1.1}). And importantly, a local Maxwellian $M_{1,\e  u, 1}(v)$ associated with $u$ satisfying the no-slip boundary condition (\ref{noslip}), satisfies the diffuse reflection boundary condition (\ref{diffuse_BC}) without singular terms. In other words, the Navier-Stokes solutions  are compatible with the diffuse reflection boundary condition. Therefore, under the scale \eqref{scale} the Navier-Stokes solution of \eqref{NS_k}-\eqref{noslip} stands in between the Boltzmann solution of  \eqref{Boltzmann}, \eqref{diffuse_BC} and the Euler solution \eqref{Euler}-\eqref{no-pen}. 

 In this paper, inspired by these observations, we propose to study the Euler limit from the Boltzmann equation through the Navier-Stokes solutions that hold both features of the Euler and the Boltzmann under \eqref{scale} at each Mach number $\e>0$. To this end, we expand the Boltzmann solution $F$ around a local Maxwellian associated with a Navier-Stokes flow  $u$ to (\ref{NS_k})-(\ref{noslip}): 
\Be\label{mu_e}
\mu(v) : = M_{1 , \e u , 1 }(v),
\Ee
as 
 \Be\label{F_e1}
F = \mu+ \e^2 f_2 \sqrt{\mu }  + \e^{3/2} f_R\sqrt{\mu}   ,
\Ee 
and analyze \eqref{F_e1} via a new Hilbert expansion presented in Section \ref{sec:1.1}.  
Although the notations $F^\e$ and $f^\e$ may be more precise for the equation depending on $\e$, 
 we will abuse the notations by dropping the superscript $\e$ for the sake of simplicity.  
\hide
Here, we introduce an auxiliary parameter $\delta=\delta(\e)\downarrow 0$ as $\e \downarrow 0$, which indicates a size of the fluctuation $(F-\mu)/ {\e}$    
 and of course one of the main goals of this article is to prove that $\delta$ can be chosen such that $\delta=\delta(\e)\downarrow 0$ as $\e \downarrow 0$, which will ensure that the first order leading approximation of $F$ at each fixed $\e>0$ is given by a Navier-Stokes flow.   
\unhide
The next order correction $f_2$ can be entirely determined by the Navier-Stokes flow and it turns out that its contribution is always smaller than $f_R$'s one in our choice of $\e$ and $\kappa$. 
 A choice of the range of the Mach number with respect to the Reynolds number: $\e \ll \kappa={1} / {\mathpzc{Re}}$ in $\e\downarrow 0$ plays an important role in our analysis and the formal expansion. We will discuss the relation and its role in Section \ref{sec:B_int}. With such a choice of the scale, uniform-in-$\e$ estimates of the Boltzmann remainder $f_R$ are achieved by a novel quantitative $L^p$-$L^\infty$ estimate in a setting of the local Maxwellian of the Navier-Stokes approximation \eqref{mu_e}, along with the commutator estimates and the integrability gain of the hydrodynamic part in various spaces. 
 \hide
Another natural choice might be $\e^{\mathfrak{q}}=\kappa$ with an integer $\mathfrak{q}\geq 1$. Then the second correction $\frac{1}{\kappa}Lf_2$ is shifted to the next hierarchy (see (\ref{eqtn_f_2})) and as a consequence the fluid equations become the Euler equations by losing $\kappa \eta_ 0 \Delta u$. Without the boundary, a higher order expansion $F=\mu_E+  [\e f_1   + \e^2 f_2  + \e^3 f_3 + \cdots + \e^r f_R] \sqrt{\mu_E}$ where $\mu_E=M_{1 , \e u_E , 1 }$ has been established in \cite{DEL,WZL}. In the presence of the boundary, on the other hand, such an expansion features a boundary mismatch. The usual approach is then based on a boundary layer expansion, correcting an interior Hilbert-like expansion at the boundary to satisfy the boundary conditions (\cite{GW,Wu}). Our approach is based on an interior expansion up to the second correction $f_2$ that {\it avoids} the boundary layer expansion, while uniform-in-$\e$ estimates of the Boltzmann remainder $f_R$ can be achieved by a novel quantitative $L^p$-$L^\infty$ estimate in a setting of the local Maxwellian of the Navier-Stokes approximation \eqref{mu_e}, along with commutator estimates and the integrability gain of the hydrodynamic part in various spaces. 
 \unhide

In order to establish the Euler limit by using the Navier-Stokes solutions of \eqref{NS_k}-\eqref{noslip} as a reference state as $ \e \downarrow 0$ in a scale of large Reynolds number (\ref{scale}), it is imperative to show the uniform-in-$\kappa$ convergence of the Navier-Stokes solutions to the Euler solutions of \eqref{Euler}-\eqref{no-pen}, where the inviscid limit comes into play. 
In this paper, we take the spatial domain to be the upper-half space with periodic boundary conditions in the horizontal components and analytic data for the Navier-Stokes solutions of \eqref{NS_k}-\eqref{noslip} and obtain uniform-in-$\kappa$ estimates built upon a recent development on the inviscid limit problem in the half-space based on the Green's function approach using the boundary vorticity formulation \cite{Mae14,NN2018,KVW,FW}.


\hide
This is possible since the associated $\theta\equiv0$ and $\rho\equiv0$ solve $\p_t \theta  + u  \cdot \nabla_x \theta -C_2\kappa \Delta \theta   =0, \ 
\nabla_x \theta  (t,x) + \nabla_x \rho  (t,x)= 0$.

 \unhide

\hide

To overcome such mismatch we use solutions of incompressible Navier-Stokes equations (with $\kappa \eta_0$-viscosity and a no-slip boundary condition) as a reference state
\begin{align}
\p_t u+ u \cdot \nabla_x u - \kappa \eta_0 \Delta u + \nabla_x p  &=0 \ \ \text{in} \ \O, \label{NS_k}\\
 \nabla_x \cdot u &=0 \ \ \text{in} \ \O,
 \label{incomp}
 \\
 \ \ u  &=0  \ \ \text{on} \ \p\O, \label{noslip}
\end{align}
where a physical constant $\eta_0$ will be determined later in (\ref{NS_k}). 
In other words we expand the Boltzmann solution $F$ around a local Maxwellian associated with a Navier-Stokes' solution $u$ to (\ref{NS_k})-(\ref{noslip})
\Be\label{mu_e}
\mu(v) : = M_{1 , \e u , 1 }(v),
\Ee
as 
\Be\label{F_e}
F = \mu + \e^2 f_2 \sqrt{\mu }  + \delta\e f_R\sqrt{\mu}.
\Ee 
Here, we are introducing an additional parameter $\delta>0$ satisfying $\delta(\e)\downarrow 0$ as $\e \downarrow 0$. We will discuss a desired relation between $\delta, \kappa,$ and $\e$ later in (\ref{delta}), which plays a role in our analysis. It is easy to check that $\mu$ satisfies the diffuse reflection BC (\ref{diffuse_BC}) from the no-slip BC (\ref{noslip}). As the Euler case we let the fluctuation of density $\rho$ and temperature $\theta$ vanish for all time. This is possible since the associate $\theta\equiv0$ and $\rho\equiv0$ solve $\p_t \theta  + u  \cdot \nabla_x \theta -C_2\kappa \Delta \theta   =0, \ 
\nabla_x \theta  (t,x) + \nabla_x \rho  (t,x)= 0$.

   \unhide



\hide
Inspired by 
 (\ref{hydro_limit}), we try to construct a family of solutions $F$ as an expansion around $\mu$ of (\ref{mu_e})
\Be\label{F_e}
F = \mu + \e^2 f_2 \sqrt{\mu }  + \delta\e f_R\sqrt{\mu}.
\Ee 

\bcr [Need a paragraph?: Do we need to mention viscous limit of fluid here? Then it would be natural to assume the half space and analytic data on the hydrodynamic part in the main Theorem. For example]

\ec

\unhide


Our main result concerns a rigorous justification of the passage from the solutions to the dimensionless Boltzmann equation (\ref{Boltzmann}) of the scale (\ref{scale}) with the diffuse reflection boundary condition (\ref{diffuse_BC}) to the solution of the incompressible Euler equation (\ref{Euler}) with the no-penetration boundary condition (\ref{no-pen}), without introducing any boundary expansion of the Boltzmann equation: 


\begin{theorem}[Informal statement] We consider a half space in 3D \label{thm_informal}
\Be\label{domain}
\O := \mathbb{T}^2 \times \R_+ \ni (x_1, x_2, x_3),
 \ \ \text{where } \mathbb{T} \text{ is a periodic interval of } (-\pi, \pi).
\Ee
For some choice of $\e$ and $\kappa(\e)$,
\hide
 \Be 
\delta= \sqrt{\e} \ \ \text{and} \ \ 
\delta \lesssim \exp (   -\kappa^{ -\mathfrak{P}}   ) 
 \ \ \text{for some   } \mathfrak{P}>0.  
\Ee
\unhide
 there exists a large set of initial data $u_{in}$, $f_{2,in}$ and $f_{R,in}$ 
such that a unique solution $F(t,x,v)$ of the form  \eqref{F_e1}   
to (\ref{Boltzmann}) and (\ref{diffuse_BC}) with (\ref{scale}) exists on $ [0,T]$ for some $T>0$ 
and satisfies 
\Be\notag
\sup_{0 \leq t \leq T}\left\|\frac{F (t,x,v)- M_{1, \e u , 1}  }{\e \sqrt{M_{1,\e u,1}}}  \right\|_{L^2(\O \times \R^3)}    \longrightarrow 0   \ \ \text{as} \ \ \e \downarrow 0,
\Ee 
and  
\Be\notag
\sup_{0 \leq t \leq T}\left\|\frac{F (t,x,v)- M_{1, \e u_E , 1}  }{\e
(1+ |v|)^2\sqrt{M_{1,0,1} }
} \right\|_{L^2(\O\times\R^3)}
     \longrightarrow 0   \ \ \text{as} \ \ \e \downarrow 0,
\Ee 
while $u$ and $u_E$ denote solutions of the Navier-Stokes (\ref{NS_k})-(\ref{noslip}) and Euler equations (\ref{Euler})-(\ref{no-pen}), respectively. 
\end{theorem}

The precise statement of Theorem \ref{thm_informal} is given in Theorem \ref{main_theorem} and Corollary \ref{Cor_EL} in Section \ref{sec:MT}. 
\hide
We refer to Theorem \ref{main_theorem} and Corollary \ref{Cor_EL} in Section \ref{sec:MT} for the full statement of the main theorem. 
\unhide

\begin{remark}
To the best of our knowledge our result of this paper appears to be the first rigorous incompressible Euler limit result from the Boltzmann solutions with the sole diffuse reflection (therefore the accommodation constant $\sim$1) in the boundary condition! Moreover, our framework captures the inviscid limit of mathematical fluid dynamics from the Boltzmann theory. 
\end{remark}

 \begin{remark}
 Another natural choice of the scale in the study of the Euler limit might be $\e^{\mathfrak{q}}=\kappa$ with an integer $\mathfrak{q}\geq 1$. Then the second correction $\frac{1}{\kappa}Lf_2$ is shifted to the next hierarchy (see (\ref{eqtn_f_2})) and as a consequence the Euler equations become the leading approximation with loss of $\kappa \eta_ 0 \Delta u$. Without the boundary, a higher order expansion $F=\mu_E+  [\e f_1   + \e^2 f_2  + \e^3 f_3 + \cdots + \e^r f_R] \sqrt{\mu_E}$ for $\mu_E=M_{1 , \e u_E , 1 }$ has been established in \cite{DEL,WZL}. In the presence of the boundary, on the other hand, such an expansion features a boundary mismatch. The usual approach is then drawn on a boundary layer expansion, correcting an interior Hilbert-like expansion at the boundary to satisfy the boundary conditions (for example, see \cite{GW,Wu}). Our approach is based on an interior expansion up to the second correction $f_2$ that {\it avoids} the boundary layer expansion under our choice of scale $\e \ll \kappa$ (see (\ref{choice:delta})).  \end{remark}

\hide
 \bcb

\begin{theorem*} 
\label{main_rough} 
We consider a half space in 3D
\Be\label{domain}
\O := \mathbb{T}^2 \times \R_+ \ni (x_1, x_2, x_3),
 \ \ \text{where } \mathbb{T} \text{ is a periodic interval of } (-\pi, \pi).
\Ee 
Let \Be \label{choice:delta:intro}
\delta= \sqrt{\e} \ \ \text{and} \ \ 
\delta \lesssim \exp (   -\kappa^{ -\mathfrak{P}}   ) 
 \ \ \text{for some   } \mathfrak{P}>0.  
\Ee

Suppose that an initial datum $u_{ \text{in}}$ is a real analytic divergence-free vector field in $\O$ and satisfying suitable compatibility conditions, and that  an initial datum for the Boltzmann equation is given as
\Be\label{F_initial}
F  |_{t=0}= M_{1, \e u_{ \text{in}} ,1 }+ \e^2 f_{2, in} \sqrt{\mu}   + \e \delta f_{R,  in}  \sqrt{\mu},
\Ee
while $f_{2,in}$ is determined by $u_{in}$, and $f_{R,in}$ is bounded in some functional space and satisfies suitable compatibility conditions.

Then there exists a unique real analytic solution $u$ in a time interval $[0,T]$ to the incompressible Navier-Stokes equation with the no-slip boundary condition (\ref{NS_k})- (\ref{noslip}) with a viscosity proportional to some $\kappa$ in (\ref{choice:delta:intro}). Moreover $u$ has a strong bound and convergence toward $u_E$ solving the Euler equation with the no-penetration boundary condition \eqref{Euler}-\eqref{no-pen} in $[0,T]$. 
And we construct a unique solution of the initial datum (\ref{F_initial}) solving the Boltzmann equation (\ref{Boltzmann}) with (\ref{scale}) satisfying the diffuse reflection boundary condition (\ref{diffuse_BC}) for all $\e$'s: 
$$F(t,x,v) =M_{1, \e u  ,1 }(v) + \e^2 f_2(t,x,v) \sqrt{\mu}  + \e \delta f_R (t,x,v) \sqrt{\mu}, \ \ \text{for}  \ t \in [0,T]. $$ 
Moreover 
\Be
 \frac{F  (t,x,v)- M_{1, 0, 1} (v)}{\e}
  \rightarrow u_E(t,x) \ \ \text{strongly in } L^\infty ((0,T);L^2(\O \times \R^3)) \ \ \text{as} \ \ \e \rightarrow 0.\notag
\Ee

\hide

Then for each $\e>0$ there exists a unique solution $F_\e$ to (\ref{Boltzmann}) and 
(\ref{diffuse_BC})  with (\ref{scale})  for $t \in [0, T]$ such that 
\Be
\sup_{0 \leq t \leq T}\left\|\frac{F_\e(t,x,v)- M_{1, 0, 1} (v)}{\e} -u_E(t,x)\right\|_{L^2(\O)}
\leq C \kappa (\e) \ \ \text{as} \ \ \e \rightarrow 0.
\Ee\unhide
\end{theorem*}
\ec
\unhide


 Before discussing the essence of the methodology and novelty of our result, we shall briefly overview some relevant literatures on the hydrodynamic limit of the Boltzmann equation. One of the first mathematical studies of the limits at the formal level may go back to a work \cite{Hilbert1} of Hilbert, in which he introduced so-called the Hilbert expansion. Based on the truncated Hilbert expansion rigorous justifications of fluid limits have been shown as long as the solutions of corresponding fluids are bounded in some suitable spaces, for example, in the compressible fluid limits in \cite{Caflisch, UA}, incompressible fluid limits in \cite{DEL, Guo2006, Mouhot}, diffusive limits from the Vlasov-Maxwell-Boltzmann system in \cite{Jang}, and relativistic fluid limits in \cite{SS}. All the derivations mentioned above did not take into account the boundary, while one of the main obstacles to study the Boltzmann solutions with the boundary is its boundary singularity (see \cite{Kim,GKTT1,GKTT2}). In \cite{Guo10}, an $L^p$-$L^\infty$ framework has been developed to construct a unique global solution of the Boltzmann equation with physical boundary conditions. Such a framework has been developed successfully in various problems of the Boltzmann theory (for example \cite{GJJ, GJ,GJ1,Guo_M, EGKM3, Chen, GW,KL,KL1, Wu}). In particular, in \cite{EGKM, EGKM2}, the authors have constructed a solution of the Boltzmann equation satisfying the diffuse reflection boundary condition and proved the validity of the hydrodynamic limit toward the incompressible Navier-Stokes-Fourier system in both steady and unsteady settings, based on a novel $L^6$-bound of the hydrodynamic part.

 Rigorous passage from the renormalized solutions of \cite{DP} (\cite{Mischler} with the physical boundary) of the Boltzmann equation  toward (weak) solutions of fluid equations has been also extensively explored (see \cite{Golse, Laure,JM} for the references in this direction). In particular, the program of the incompressible Navier-Stokes limit to the Leray-Hopf weak solutions has been  developed successfully in \cite{BGL91,BGL93, LM1,LM2,GSR} without the physical boundary and with the boundary in \cite{MS,JM}.
 As for the incompressible Euler limit, based on the relative entropy method, a dissipative solution of the incompressible Euler equations  in \cite{Lions} has been studied in \cite{LM1,LM2,SR} without the boundary. Notably the results have been extended to the domain with the boundary for the specular reflection boundary condition in \cite{Laure}, 
 and for the Maxwell boundary condition in \cite{BGP}, assuming to set that the accommodation constant (a factor of diffuse reflection) vanishes as $\e \downarrow 0$.
 
  \hide

 {\color{green}
[CK note: It is little bit funny that Bardos-Golse-P did not recognize \cite{Laure} but saying all incompressible Euler results had been established for periodic box or the whole space.I don't know why.. BGP was announced later I think....   It could be tricky later since the referee could be Saint-Raymond but also could be Toan then we will ask Bardos' opinion. Considering the situation I wrote  intentionally the sentence blandly for both not making anyone angry (but maybe too low key).] \bcr JJ: It looks good to me. \ec}
 
 \unhide
 \hide
 To the best of our knowledge our result of this paper appears to be the first rigorous incompressible Euler limit result from the Boltzmann solutions with the sole diffuse reflection (therefore the accommodation constant $\sim$1) in the boundary condition! Moreover, our framework captures the inviscid limit of mathematical fluid dynamics from the Boltzmann theory. 
\unhide  
 For the rest of this section, we present the strategy and key ideas developed in the proof of our result starting with a new (formal) Hilbert expansion followed by the control of the Boltzmann remainder $f_R$ and higher regularity of Navier-Stokes flows, for the rigorous justification of the formal expansion.

\hide

  pioneered in [38] and adapted to the case of the Boltzmann Equation in [10, 25], L. Saint-Raymond [28, 29] succeeded in deriving dissipative solutions of the Euler Equations in arbitrary space dimension (or classical solutions whenever they exist) from weak solutions of the BGK model [28] or from renormalized solutions of the Boltzmann Equation [29]. Later it is extened to the specular reflection boundary condition in \cite{Laure} and Bardos establish the incompressible Euler limit of the Boltzmann Equation set in a domain with boundary with Maxwell?s accommodation condition assuming that the accommodation parameter is small enough in terms of the Knudsen number.

 In [23], P.-L. Lions proposed a notion of dissipative solution of the incompressible Euler Equations ? in the same spirit of his defini- tion of the notion of viscosity solutions of Hamilton-Jacobi Equations, but using the conservation of energy instead of the maximum principle as in the Hamilton-Jacobi case. The weak-strong uniqueness property is satisfied by dissipative solutions of the incompressible Euler Equations (essentially by definition): if there exists a classical (C1) solution of the incompressible Euler Equations, all dissipative solutions with the same initial data must coincide with this classical solution on its maximal time in- terval of existence. Unfortunately, dissipative solutions are not known to satisfy the incompressible Euler Equations in the sense of distributions.

There has been a lot of mathematical study of the singular limit ? ? 0oftherescaled kinetic equation (2.2) to the Navier-Stokes-Fourier equations for n = 1. Anatural approach is to show that the Diperna-Lions renormalized solutions [12] tothe Boltzmann equation converge to the Leray-Hopf weak solutions to the incom-pressible Navier-Stokes-Fourier equations. Such a program was initiated in [1, 2].Unfortunately, certain a priori assumptions made were without justi?cation. Toovercome such a mathematical dif?culty, there have been many important contribu-tions over the years [3, 14, 27, 28, 30, 31], among others. Finally, in [15], a completeproof for such a limit was established for certain classes of collision kernels. Sucha breakthrough soon stimulated new progress in the study of more general collisionkernels [16, 26]. See review papers [29, 38] for more complete references in thisdirectio

On the other hand, in the context of the classical solutions for the Navier-Stokes-Fourier equations, whose existence is open for large-amplitude initial data, it wasproven in [11] that for any given smooth solution of the Navier-Stokes-Fourier sys-tem, one can ?nd nearby Boltzmann solutions F?such that f?2is of order ?2.Itisnot clear how arbitrary such a remainder f?2can initially be; see also [7, 25]. In[4], the Navier-Stokes-Fourier limit was justi?ed for general initial f?(0, x, v) withsmall amplitude, in which f?(t, x, v) decays to 0 in the whole space.
\unhide

\subsection{Hilbert expansion in a scale of large Reynolds number}\label{sec:1.1}
Through a new formal Hilbert-type expansion of Boltzmann equation with the diffuse reflection boundary condition we aim to capture the Navier-Stokes equations of vanishing viscosity proportional to $\mathpzc{Kn}/\mathpzc{Ma}$ and satisfying the no-slip boundary condition.

It is worth pointing out that although 
more convenient choice of an expansion of $F$ is seemingly the one around the global Maxwellian $\mu_0:= M_{1,0,1}$ such as $F=\mu_0 + \e (u \cdot v) \mu_0 +  \e^2 \tilde{f}_2 \sqrt{\mu_0} + \e \delta \tilde{f}_R \sqrt{\mu_0}$, unfortunately this choice will produce, in the Hilbert expansion (\ref{eqtn_f_1})-(\ref{eqtn_f_5}), an unbounded term $\frac{2}{\kappa \e } \frac{1}{\sqrt{\mu_0  }}Q(u \cdot v \mu_0  ,f_R \sqrt{\mu_0  })$ even compared to the strongest control in hand, namely a dissipation term (see (\ref{dissipation}))! To achieve a sharper estimate, which provides weaker restriction on $\kappa$ and $\e$, and hence weaker restriction on the initial data, we work on an expansion around the local Maxwellian $\mu$.

It is conceptually 
convenient in our analysis to introduce an auxiliary parameter $\delta=\delta(\e)\downarrow 0$ as $\e \downarrow 0$, 
which indicates a size of the fluctuation $(F-\mu)/ {\e}$:    
 \Be\label{F_e}
F = \mu+ \e^2 f_2 \sqrt{\mu }  + \e \delta f_R\sqrt{\mu}   .
\Ee 
In (\ref{F_e1}) we have chosen $\delta=\sqrt{\e}$ and in Section \ref{sec:MT} we will have the same choice such as (\ref{choice:delta}),
\hide
\Be \label{choice:delta}
\delta= \sqrt{\e} \ \ \text{and} \ \ 
\delta=\exp\Big(  \frac{  - \mathfrak{C}  T}{ \kappa^{ \mathfrak{P}} }\Big), 
\Ee\unhide
however in Section \ref{sec:B1}, Section \ref{sec:3} and Section \ref{sec:B}, 
$\delta$ will be regarded as a free parameter 
and 
will be chosen at the last step of closing our argument (as (\ref{choice:delta})!). 


\

{\bf Interior Expansion.} We investigate an expansion (\ref{F_e}) of the Boltzmann equation (\ref{Boltzmann}) at the local Maxwellian $\mu$ in (\ref{mu_e}). Let
\Be\label{L_Gamma}
L  f = \frac{-2}{\sqrt{\mu }} Q(\mu , \sqrt{\mu }f)   ,\ \
\Gamma (f,g)= \frac{1}{ \sqrt{\mu}}  Q(\sqrt{\mu}f, \sqrt{\mu}g) .
\Ee

  The operators $L$ and $\Gamma$ can be read as 
 \begin{align}
Lf(v)&= \nu  f(v) -Kf(v)= \nu (v) f(v)- \int_{\R^3} \mathbf{k} (v , v_* ) f(v_*) \dd v_*,\label{nu_K}\\
\Gamma(f,g) (t,v)  &=  \Gamma_+ (f,g) (t,v) - \Gamma_- (f,g) (t,v)
\notag
\\
&= \iint_{\R^3 \times \S^2} 
|(v-v_*) \cdot \mathfrak{u}| \sqrt{ \mu(v_* )} 
\big(
f(t,v ^\prime) g(t,v_*^\prime)
+ g(t,v ^\prime) f(t,v_*^\prime)\big)\dd \mathfrak{u} \dd v_*\label{Gamma}
\\
& \ \ -   \iint_{\R^3 \times \S^2} 
|(v-v_*) \cdot \mathfrak{u}| \sqrt{ \mu(v_* )} 
\big(f(t,v  ) g(t,v_* )+g(t,v  ) f(t,v_* )\big)
\dd \mathfrak{u} \dd v_*,\notag
\end{align} 
where the precise form of $\mathbf{k}$ is delayed to be presented in (\ref{decom_L0}). We will demonstrate basic properties of operators $L$ and $\Gamma$ in Section \ref{sec:2}. From (\ref{collision_inv}) the null space of $L$, denoted by $\mathcal{N}$, is 
a subspace of $L^2(\R^3)$ spanned by orthonormal bases $\{\varphi_i \sqrt{\mu }\}_{i=0}^4$ with
\hide  \Be
\begin{split}
\label{basis}
\varphi_0 := \frac{1}{\sqrt{1+ \delta \sigma}}
,   \  \ \ \varphi_i: = \frac{1}{\sqrt{1+ \delta \sigma}} \frac{v_i -\delta u_i }{\sqrt{1+ \delta \theta}}
 \ \ \text{for} \ i=1,2,3 
,   \  \ \ \varphi_4: =  \frac{1}{ \sqrt{1+ \delta \sigma}}
( \big|\f{v-\delta u}{\sqrt{1+ \delta \theta}}\big|^2-3)/\sqrt{6},\\
\varphi_0 := \frac{1}{\sqrt{1+ \delta \sigma}}
,   \  \ \ \varphi_i: = \frac{1}{\sqrt{1+ \delta \sigma}} \frac{v_i -\delta u_i }{\sqrt{1+ \delta \theta}}
 \ \ \text{for} \ i=1,2,3 
,   \  \ \ \varphi_4: =  \frac{1}{ \sqrt{1+ \delta \sigma}}
( \big|\f{v-\delta u}{\sqrt{1+ \delta \theta}}\big|^2-3)/\sqrt{6}
\end{split}
\Ee\unhide
\Be\label{basis}
\begin{split}
 &\varphi_0 := 1
,   \  \ \ \varphi_i: =   {v_i -\e u_i } 
 \ \ \text{for} \ i=1,2,3 
,   \  \ \ \varphi_4: =   
( | {v-\e u} |^2-3 )/{\sqrt{6}}.
%
%
\end{split}\Ee
We define a hydrodynamic projection $\mathbf{P}$ as an $L_v^2$-projection on $\mathcal{N}$ such as 
\Be\begin{split}
\label{P}
\mathbf{P} g:= \sum
  ( {P}_j g) \varphi_j \sqrt{\mu }, \ \ 
 {P}_j g:= \langle g ,\varphi_j \sqrt{\mu } \rangle  , \  \text{and}  \ \  
P g:= (P_0 g, P_1 g, P_2 g, P_3 g, P_4 g ),
%
%
\end{split}
\Ee
where $\langle \cdot, \cdot \rangle $ stands for an $L^2_v$-inner product. It is well-known that the operators enjoy $\mathbf{P}L=L \mathbf{P}=\mathbf{P} \Gamma=0$. Importantly the linear operator $L$ enjoys a coercivity away from the kernel $\mathcal{N}$: 
for $\nu(v)\geq0$ defined in (\ref{nu_K}) 
\Be\label{s_gap}
\langle Lf, f\rangle \geq \sigma_0 \|  \sqrt{\nu}(\mathbf{I} - \mathbf{P}) f \|_{L^2 (\R^3)}^2 \ \ \text{for some }   \sigma_0>0.
\Ee

Now we 
plug the expansion (\ref{F_e}) into the rescaled equation (\ref{Boltzmann}) with the scale (\ref{scale}). It turns out that by relating $f_2$ with the flow and locating it carefully in the hierarchy we can exhibit the dissipative nature of the Boltzmann collision operator at the leading order of the fluid approximation. In particular we locate $(v-\e u) \cdot \nabla_x (\mathbf{I} - \mathbf{P})f_2$ in $\frac{1}{\delta}$-order hierarchy to capture $\kappa$-order viscosity in the fluid equation (\ref{NS_k}):   
 \begin{align}
& 
 \p_t f_R 
 +  \frac{1}{\e} v\cdot \nabla_x f_R+    \frac{1}{ \e^2\kappa} Lf_R
 + \frac{(\p_t + \e^{-1} v\cdot \nabla_x) \sqrt{\mu}}{\sqrt{\mu}} f_R \label{eqtn_f_1}
\\
&=-
\frac{1}{\e \delta} 
 \Big\{
\frac{ \e^{-1} (v- \e u )\cdot \nabla_x \mu}{   \sqrt{\mu}} 
+ \frac{ 1}{    \kappa} L f_2\Big\}
\label{eqtn_f_2}
\\
& \ \    -
\frac{1}{\delta}  \Big\{ \frac{\e^{-1} \p_t \mu}{\sqrt{\mu}}
+  \frac{ \e^{-1}  u \cdot \nabla_x \mu}{   \sqrt{\mu}}  +   ( v- \e u)\cdot \nabla_x f_2
\Big\}
\label{eqtn_f_3}
\\
& \ \ -   \frac{\e}{\delta}   \Big\{  \p_t f_2
 +  u \cdot \nabla_x f_2
 + \frac{(\p_t + \e^{-1} v\cdot \nabla_x ) \sqrt{\mu}}{\sqrt{\mu}}f_2
 \Big\}\label{eqtn_f_4}
  \\
&  \ \ +  
    \frac{2}{\kappa}  \Gamma({f_2}, f_R) +   \frac{\e }{\delta\kappa} \Gamma({f_2}, f_2)
+    \frac{  \delta }{\e\kappa}\Gamma(f_R,f_R)   .\label{eqtn_f_5}
\end{align} 
\hide where we have used the following identities
\begin{align}
\e^{-1} v \cdot \nabla_x \mu = \sum_{\ell, m=1}^3 \varphi_\ell \p_\ell u_m \varphi_m  \mu 
+ \e  \sum_{\ell, m =1}^3 u_\ell \p_\ell u_m \varphi_m \mu
\end{align} 

\hide We note that from (\ref{f_1})
\Be\label{H1}
  \frac{1}{ \delta\e^2 \kappa} L f_1=0.
\Ee  \unhide \unhide
We can readily see an $L^2$-energy structure of $f_R$ with a strong dissipation 
\Be\label{dissipation}
\iint_{\O \times \R^3} \frac{1}{\e^{2} \kappa } L f_R f_R \dd v \dd x \gtrsim \| \e^{-1} \kappa^{-1/2} \sqrt{\nu} (\mathbf{I} - \mathbf{P}) f_R  \|_{L^2 (\O \times \R^3)}^2,
\Ee
 which inherits its lower bound from the coercivity 
 (\ref{s_gap}).

Let us first consider an $\frac{1}{\e \delta}$-hierarchy (\ref{eqtn_f_2}). For any non-vanishing term of $(\ref{eqtn_f_2})$ would cause unpleasant unboundedness, we make the term vanish entirely by solving an equation $(\ref{eqtn_f_2})=0$. By the Fredholm alternative, an inverse map 
\Be\label{L_inverse}
L^{-1}: \mathcal{N}^{\perp} \rightarrow \mathcal{N}^{\perp},  \ \text{ where }   \  \mathcal{N}^\perp  \   \text{stands an $L^2_v$-orthogonal complement of } \mathcal{N},
\Ee
 is well-defined and hence 
the solvability condition is given by
\Be\label{Comp1}
\frac{ \e^{-1} (v- \e u )\cdot \nabla_x \mu}{   \sqrt{\mu}}
=\sum_{\ell, m=1}^3 \p_\ell u_m\varphi_\ell  \varphi_m  \sqrt{\mu}
  \in \mathcal{N}^\perp.
\Ee
 This condition indeed implies the incompressible condition (\ref{incomp}). 
 
 Once (\ref{incomp}) holds,  we have $\sum_{\ell, m=1}^3 \p_\ell u_m\varphi_\ell  \varphi_m  \sqrt{\mu}= \sum_{\ell, m=1}^3 \p_\ell u_m
 (\varphi_\ell  \varphi_m - \frac{|v-\e u|^2}{3} \delta_{\ell m}
 )
  \sqrt{\mu}$. Now we solve $(\ref{eqtn_f_2})=0$ by 
setting
 \Be\label{f_2} 
 (\mathbf{I}-\mathbf{P}) f_2 \hide = \kappa  L^{-1} (\mathbf{I} - \mathbf{P})
   \bigg(\frac{\e^{-1} (v- \e u )\cdot \nabla_x \mu }{\sqrt{\mu}}
 \bigg) 

  \sum_{\ell, m=1}^3 \varphi_\ell \p_\ell u_m \varphi_m  \sqrt{\mu} 
\unhide
=   -\kappa \sum_{\ell,m=1}^3  A_{\ell m} \p_\ell u_m
 \ \ \text{with} \ \ A_{\ell m}:=L^{-1} \Big( \varphi_\ell \varphi_m \sqrt{\mu} - \frac{|v- \e u |^2}{3} \delta_{\ell m}\sqrt{\mu} \Big).
 \Ee
  \hide where we define  such that 
 \Be\label{AB}
 A_{ij}= L^{-1} \big( \varphi_i \varphi_j \sqrt{\mu} - \frac{|v- \e u |^2}{3} \delta_{ij}\sqrt{\mu} \big)  
  \Ee 
 while coefficients of $\mathbf{P}f_2 =\{ \tilde{\rho}  \varphi_0+ \sum_{\ell=1}^3 \tilde{u}  \cdot  \varphi_\ell +\tilde{\theta}  \varphi_4  \}\sqrt{\mu}$ haven't been determined yet. \unhide

Then we move to an $\frac{1}{\delta}$-hierarchy (\ref{eqtn_f_3}). The hydrodynamic part of (\ref{eqtn_f_3}), unless it vanishes, would induce an unbounded term again. We expand $\delta \times (\ref{eqtn_f_3})$, using (\ref{mu_e}) and (\ref{f_2}), as
\Be
\begin{split}
 - (v- \e u) \cdot   (\p_t u + u \cdot \nabla_x u ) \sqrt{\mu}  
 +  (v- \e u) \cdot \nabla_x   \mathbf{P} f_2\\
+ \kappa  (v- \e u) \cdot  \nabla_x
\Big( \sum_{\ell,m=1}^3 A_{\ell m} \p_\ell u_m\Big)
 . 
 \label{H_NS}
\end{split}\Ee
The leading order term of the last term in (\ref{H_NS}) contributes the following to the hydrodynamic part of (\ref{H_NS}) as
\Be
\begin{split}
\kappa  \sum_{\ell,m,k=1}^3 \big\langle \varphi_i \varphi_k  \sqrt{\mu},   
A_{\ell m}
\big\rangle  \p_k \p_\ell u_{  m}
&= \kappa  \sum_{\ell,m,k=1}^3 \Big\langle
\big( \varphi_i  \varphi_k  -\frac{|v-\e u|^2}{3} \delta_{i k}
\big)\sqrt{\mu}
,   
A_{\ell m}
\Big\rangle  \p_k \p_\ell u_{  m}\\
&
= \kappa    \sum_{\ell ,  m, k=1}^3 \left\langle 
L A_{i k}
,  A_{\ell m}
\right\rangle  \p_{k} \p_\ell u_{  m}
,
 \label{f_2:main}
\end{split}
\Ee
where we have used the fact $A_{\ell m } \in \mathcal{N}^\perp$ and $\frac{|v-\e u|^2}{3} \sqrt{\mu} \in \mathcal{N}$  at  the first step and the definition of $A_{ik}$ at the last step.  It is well-known (e.g. Lemma 4.4 in \cite{BGL93}) that for some constant $\eta_0 >0$
\hide
\Be
\langle LA_{\ell k }, A_{ij}\rangle = \eta_0   ( \delta_{ik} \delta_{j\ell} + \delta_{i\ell} \delta_{jk} ) - \frac{2}{3} \eta_0   \delta_{ij} \delta_{k\ell}.\label{eta_0}
\Ee
Therefore we deduce (\ref{f_2:main}) vanishes for $\ell=0,4$, while we can observe the $\kappa \eta_0$-viscosity term in (\ref{NS_k}) as 
 \begin{align}
(\ref{f_2:main}) = \kappa \eta_0 \sum_{i,j,k}   \{  ( \delta_{ik} \delta_{j\ell} + \delta_{i\ell} \delta_{jk} ) - \frac{2}{3}  \delta_{ij} \delta_{k\ell}\} \p_{k} \p_i u_{  j} 
 =\kappa\eta_0 \{\Delta u_{\ell}  -   \p_\ell    \nabla \cdot  u  -  \frac{2}{3}   \p_\ell  \nabla \cdot  u   \}= \kappa\eta_0\Delta u_{ \ell}
 \ \ \text{for} \ \ell=1,2,3. 
 \unhide
 \Be
\langle LA_{i k }, A_{\ell m}\rangle = \eta_0   ( \delta_{\ell k} \delta_{mi} + \delta_{\ell i} \delta_{mk} ) - \frac{2}{3} \eta_0   \delta_{\ell m} \delta_{ik}.\label{eta_0}
\Ee
Therefore we deduce that (\ref{f_2:main}) vanishes for $i=0,4,$ and the $\kappa \eta_0$-viscosity term in (\ref{NS_k}) can be captured:
 \Be\begin{split}
(\ref{f_2:main}) &= \kappa \eta_0 \sum_{\ell,m,k}   \{  ( \delta_{\ell k} \delta_{mi} + \delta_{\ell i} \delta_{mk} ) - \frac{2}{3}  \delta_{\ell m} \delta_{ik}\} \p_{k} \p_\ell u_{  m} \\
&
 =\kappa\eta_0 \{\Delta u_{i}  -   \p_i    \nabla \cdot  u  -  \frac{2}{3}   \p_i  \nabla \cdot  u   \}= \kappa\eta_0\Delta u_{i}
 \ \ \text{for} \ i =1,2,3.
\label{Delta_u}
 \end{split}\Ee
Here we have used the incompressible condition (\ref{incomp}) at the last step. 
 On the other hand, a leading order term of the hydrodynamic part of $(v- \e u) \cdot \nabla_x   \mathbf{P} f_2$ contributes to the pressure term of (\ref{NS_k}) by choosing a special form of $\mathbf{P}f_2$ as in (\ref{Pf_2}). Therefore the whole leading order terms of the hydrodynamic part in (\ref{eqtn_f_3}) do vanish by solving the Navier-Stokes equations (\ref{NS_k}) and (\ref{incomp})! For the sake of brevity we refer to Section \ref{sec:3} for the full expansion of (\ref{eqtn_f_1})-(\ref{eqtn_f_5}). 

\
 
{\bf Boundary Conditions.} Now we consider a {boundary condition of $f_R$}. Noticeably the local Maxwellian $\mu$ becomes $M_{1,0,1}$ on the boundary from the no-slip boundary condition (\ref{noslip}), and hence $\mu$ satisfies the diffuse reflection boundary condition (\ref{diffuse_BC}). For the detailed study of the boundary condition of $f_R$ we introduce the incoming and outgoing boundaries 
\[
\gamma_\pm := \{(x,v)\in \p\O \times \R^3: n(x) \cdot v \gtrless0 \}. 
\]
Since $\mu$ satisfies the diffuse reflection boundary condition (\ref{diffuse_BC}) with a constant wall temperature $ =1$, by plugging (\ref{F_e}) into the boundary condition, we arrive at 
\Be\notag
 (\e^2 f_2  + \delta \e f_R )|_{\gamma_-}=
 c_\mu  \sqrt{\mu(v) } \int_{n(x) \cdot \mathfrak{v}>0}  (  \e^2 f_2 + \delta \e f_R )  \sqrt{\mu(\mathfrak{v})}(n(x) \cdot \mathfrak{v}) \dd \mathfrak{v} .
\Ee 
Letting $P_{\gamma_+}$ be an $L^2 (\{ v : n(x) \cdot v>0 \})$-projection of $\sqrt{c_\mu \mu}$, we derive that 
\Be \label{bdry_fR_general}
\begin{split}
&f_R(t,x,v)|_{\gamma_-}=    P_{\gamma_+}   f_R(t,x,v)- \frac{\e}{\delta} (1- P_{\gamma_+}) f_2(t,x,v)\\
&:=      \sqrt{c_\mu \mu(v)} \int_{n(x) \cdot \mathfrak{v}>0} f_R  (t,x,\mathfrak{v})\sqrt{c_\mu \mu(\mathfrak{v})} (n(x) \cdot \mathfrak{v}) \dd \mathfrak{v} - \frac{\e}{\delta} (1- P_{\gamma_+}) f_2(t,x,v).
 \end{split} \Ee
 Note that $\int_{n(x) \cdot v>0} c_\mu\mu(v) (n(x) \cdot v) \dd v=1$.   

On the other hand, we emphasize that, with the no-penetrate boundary condition of $(\ref{no-pen})$, the associated local Maxwellian $M_{1,\e u_E,1}$ does not satisfy the diffuse reflection boundary condition in general. Therefore the Boltzmann remainder $f_R$ would have a singularity of an order of $1/\sqrt{\e}$ in (\ref{F_e1}).
 
\smallskip

\subsection{Uniform controls of the Boltzmann remainder $f_R$}\label{sec:B_int}
For 
 a rigorous justification of the Hilbert expansion (\ref{F_e}), the major task is to establish uniform-in-$\e$ estimates of the Boltzmann remainder $f_R$ in $L^2$. The equation of the Boltzmann remainder $f_R$ in (\ref{eqtn_f_1})-(\ref{eqtn_f_5}) with the boundary condition (\ref{bdry_fR_general}) features a discrepancy between the behavior of 
the hydrodynamic part $\mathbf{P}f_R$ and pure kinetic part $(\mathbf{I}-\mathbf{P})f_R$: schematically an $L^2$-energy estimate reads 
\Be\notag
\frac{d}{dt} \| f_R(t) \|_{L^2}^2  + \| \e^{-1}  \kappa^{-1/2}   (\mathbf{I}-\mathbf{P})f_R\|_{L^2}^2 \sim 
\|\nabla_x u\|_{L^\infty} \| \mathbf{P} f_R \|_{L^2}^2+ \iint_{\O \times \R^3} \frac{\delta}{\e \kappa} \Gamma ( \mathbf{P}f_R, \mathbf{P}f_R)
  (\mathbf{I}-\mathbf{P})f_R
. 
\Ee
A key difficulty arises from a growth of the hydrodynamic part at least as $
e^{\|\nabla_x u\|_{L^\infty} } $ which might behave as an exponential of the reciprocal of some power of the viscosity $\kappa$ due to the unbounded vorticity formed near the boundary, while such strong singularity of the hydrodynamic part enters the nonlinear estimate in turn. In fact such trilinear estimate can be effectively handled only by a point-wise bound of the solutions. Unfortunately as the physical boundary conditions create singularities in general (\cite{Kim}), the high Sobolev estimates would not be possible. In this paper we develop a quantitative $L^p$-$L^\infty$ estimate \textit{solely} in the setting of the local Maxwellian associated with the Navier-Stokes flow, in the presence of the diffuse reflection boundary.


Thanks to a strong control of the dissipation from the spectral gap of (\ref{s_gap}), the nonlinear term can be bounded as 
\Be
\delta \kappa^{-\frac{1}{2}}  
  \|   {P} f_R  \|_{L^\infty_tL^6_{x }}  \|   {P} f_R  \|_{L^2_tL^3_{x }} 
 \| \e^{-1} \kappa^{-\frac{1}{2}}\sqrt{\nu} (\mathbf{I} - \mathbf{P}) f_R \|_{L^2_{t,x,v}} 
 .
\label{nonlinear_0}
\Ee 
Notably an \textit{integrability gain} of the hydrodynamic part $Pf_R$ should play a role; however a classical velocity average lemma $Pf_R \in  H^{1/2}_x\subset  L^3_x$ fails to fulfill the need in $3D$. We achieve such a higher integrability by developing a recent $L^6$-bound of hydrodynamic part of \cite{EGKM2} in the setting of the local Maxwellian on the scale of large Reynolds number. We utilize the micro-macro decomposition and the equation to control $\kappa^{1/2}v\cdot \nabla_x \mathbf{P}  f_R$ mainly by $\frac{1}{\e \kappa^{1/2}} L(\mathbf{I} - \mathbf{P} )f_R$ and $ \e \kappa^{1/2}\p_t f_R$. 
We invert the operator $v\cdot \nabla_x \mathbf{P}$, employing a recent test function method of \cite{EGKM} in the local Maxwellian setting, to establish a crucial $L^6$-bound of the hydrodynamic part, which is controlled by the dissipation plus  the a priori $L^2$-bound of $\p_t f_R$:
  \Be\label{L6_intro}
  \| \kappa^{1/2} Pf_R(t) \|_{L^6_x} \lesssim    \| \e^{-1} \kappa^{-1/2}    (\mathbf{I} - \mathbf{P}) f_R (t)  \|_{L^{2}_{x,v}}  + \e \kappa^{1/2} \| \p_t f_R(t) \|_{L^2_{x,v}}
   + l.o.t.
  \Ee
In other words we can achieve the $L^6$-estimate of (\ref{L6_intro}) as ``one spatial derivative gain''  through the dissipation provided a temporal derivative being controlled, while the temporal derivative preserves the boundary conditions. It is a critical point in which a temporal derivative gets involved in our analysis of Boltzmann and fluids as well!

\hide

(\ref{NS_k})-(\ref{noslip}): 
the corresponding estimate to (\ref{nonlinear_0}) can be controlled,  
without $ P\p_t f_R$ in $L^6_x$ (hence a $\p_t^2 f_R$-estimate is not necessary), as  
\Be
\delta \kappa^{-\frac{3}{2}} 
  \|\kappa^{ \frac{1}{2}}  {P} f_R \|_{L^\infty_tL^6_{x }} 
  \|\kappa^{ \frac{1}{2}}  {P}\p_t  f_R \|_{L^2_tL^3_{x }} 
 \| \e^{-1} \kappa^{-\frac{1}{2}}\sqrt{\nu} (\mathbf{I} - \mathbf{P})\p_t  f_R \|_{L^2_{t,x,v}}  .
\label{nonlinear_1}
\Ee \unhide

New difficulties arise as commutator estimates of $\frac{1}{\e^2 \kappa} \{\p_t Lf_R- L \p_t f_R\}$ induce singularities even at the linear level, as well as $\p_t(\p_t + \e^{-1}  v\cdot \nabla_x) \sqrt{\mu}  f_R/\sqrt{\mu}$ and the source terms in the equation of $\p_t f_R$ possess higher temporal derivatives of the fluid with an initial layer. In fact after a careful analysis we realize such singular terms amount to
\Be\notag
 \frac{1}{\kappa^{\mathfrak{P}}}  \int^t_0\| Pf_R(s)\|_{ L^2_x}^2 \dd s  ,
\Ee
while $\mathfrak{P}$ depends on the singularity of derivatives of the Navier-Stokes flow in large Reynolds numbers. 

We establish a unified $L^\infty$-estimate in the local Maxwellian setting, devising a special weight function $\mathfrak{w}_{\varrho, \ss}(x,v)$ in order to control an extra growth in $|v|$ from $(\p_t + \e^{-1}  v\cdot \nabla_x) \sqrt{\mu}  f_R/\sqrt{\mu}$ and its temporal derivative. We control $f_R$ in $L^\infty_t L^\infty_x$ by the hydrodynamic part $Pf_R$ in $L^6$ and the dissipation, studying the particle-trajectory bouncing against the diffuse reflection boundary and geometric change of variables related to the bouncing trajectories. The temporal derivative $\p_t f_R$ needs some special attention since the source term of the equation of $\p_t f_R$ possesses $\nabla_x \p_t^2 u$, which turns out to have an initial-boundary layer. For that we measure $\p_t f_R$ using a different time-space norm, namely a weighted $L^2_tL^\infty_x$, and control it by the hydrodynamic part of $\p_t f_R$ in $L^2_tL^3_x$ (with more singular factor-in-$\e$ than the counterpart for $f_R$) and the dissipation. Although our estimate of $\p_tf_R$ is singular than $f_R$ due to our choice of different spaces, we are able to balance such extra singularity by the strong dissipation and careful trilinear estimates.

 We establish $L^2_tL^3_x-$controls for $Pf_R$ and $P\p_t f_R$ via the trajectory rather than the classical average lemma. In fact a direct application of such average lemma has some subtle issue since  the source terms of $f_R$ and $\p_t f_R$ equations are known to be bounded within a finite time interval only, while the $L^2_tL^3_x$-control enters the nonlinear estimates. In fact it is not clear whether our iteration of estimates would guarantee a nonempty finite time interval of validity. Instead we utilize the Duhamel formula along the trajectories and an extension of solutions in specially designed domains, and employ the $TT^*$-method developed in \cite{JV,GV,EGKM3}. As a result we achieve $L^2_tL^P_x$ estimates for $f_R$ and $\p_t f_R$ uniformly for all $p<3$, which gives us a sufficient bound in $L^2_tL^3_x$ by interpolating with our $L^\infty$-estimates.


 \hide
Moreover a commutator $\frac{1}{\e^2 \kappa} \{\p_t Lf_R- L \p_t f_R\}$ induces a singular term, even in a linear level of the energy estimate for $\p_t f_R$, as 
\Be\label{sing_Pf}
\frac{1}{\e^2 \kappa} \e \|\p_t u\|_{L^\infty_{t,x}}
\| Pf_R\|_{L^2_tL^2_x} 
\| \sqrt{\nu} (\mathbf{I} -\mathbf{P} )f_R \|_{L^2 _{t,x,v}} 
\lesssim 
\Big(\frac{1}{  \sqrt{\kappa}} \|\p_t u\|_{L^\infty_{t,x}}\Big)^2 \int^t_0\| Pf_R(s)\|_{ L^2_x}^2 \dd s + 
\text{Dissipation}. 
\Ee
Such singular linear terms also appear in the energy estimates for both $f_R$ and $\p_t f_R$ inevitably for the local Maxwellian being used, from $(\p_t + \e^{-1}  v\cdot \nabla_x) \sqrt{\mu}  f_R/\sqrt{\mu}$ and its temporal derivative, while its quadratic growth in $|v|$ is controlled by utilizing an $L^\infty_{x,v}$-bound with a strong (exponential) weight in $|v|$.


as
\Be\label{sing_Pf_2}
 \Big( \|\nabla_x u\|_{L^\infty_{t,x}}
 +  \|\nabla_x \p_t  u\|_{L^\infty_{t,x}}
 \Big)^2  \int^t_0\iint_{\O\times \R^3} |v|^2\big(|f_R|^2+ |\p_t f_R|^2\big)
 .
\Ee
The hydrodynamic parts $Pf_R$ and $P\p_t f_R$ in (\ref{sing_Pf_2}) contribute to $ \big( \|\nabla_x u\|_{L^\infty_{t,x}}
 +  \|\nabla_x \p_t  u\|_{L^\infty_{t,x}}
 \big)^2  \int^t_0   \big(\|Pf_R\|_{L^2_x}^2+ |P\p_t f_R|_{L^2_x}^2\big)
 $, while a strong (exponential) weight in $|v|$ for both $ f_R$ and $ \p_t f_R$ in $L^\infty_{x,v}$ is utilized to control a quadratic growth term in $|v|$. 
 \unhide

\hide we end up with 
\Be
\text{Energy} (t) \lesssim  1+
\Big(
\frac{1}{  \sqrt{\kappa}} \|\p_t u\|_{L^\infty_{t,x}}
+  \|\nabla_x u\|_{L^\infty_{t,x}}
 +  \|\nabla_x \p_t  u\|_{L^\infty_{t,x}}
\Big)^2
 \int_0^t \text{Energy} (s) \dd s .
\Ee\unhide

Finally upon combining all the estimates above together we are able to bound an energy by the Gronwall's inequality. The resulting bound is not uniform but growing exponentially as $e^{1/\kappa^{\mathfrak{P}}}$, in which the power depends on the higher regularity of the fluid.  Luckily we are able to find a range of $\e$ with respect to $\kappa$ in a scale of large Reynolds number to absorb the Gronwall growth, and achieve a uniform bound of the Boltzmann remainder, which ensures the rigorous justification of the Hilbert expansion in Section \ref{sec:1.1}. The main theorem of the uniform controls of the Boltzmann remainder $f_R$ is given in Theorem \ref{main_theorem:conditional}.

\hide
  Finally for a uniform bound of (\ref{nonlinear_1}) we balance the parameters as \bcb in \ec (\ref{scale}) with
\Be\label{delta}
\delta=\delta  (\e) \ll  \kappa(\e)^{3/2} \ \ \text{for all } \ \e>0.
\Ee

It turns out that $\p_t^2 \nabla_x u$ and $\p_t \nabla_x^2 u$ get involved in the energy estimate for $\p_t f_R$.  

Most singular contribution can be bounded using an integrability gain of the classical average lemma $L^2_t H^{1/2}_x \subset L^2_t L^3_x$ , as  

{\color{red}\texttt{ ...Need to incorporate $L^p, p<3$ bound in this subsection and mention the subtlety the velocity averaging lemma...}}

We obtain such a higher integrability in $L^6_x$ through \bcb the \ec a priori $L^2$-bound of $\p_t f_R$ and the dissipation by viewing the equation of $f_R$ as 
\Be\label{vf_x}
\kappa^{1/2}v\cdot \nabla_x \mathbf{P}  f_R= - \e \kappa^{1/2}\p_t f_R - \frac{1}{\e \kappa^{1/2}} L(\mathbf{I} - \mathbf{P} )f_R
  + g,
\Ee
with some prescribed $g$. 

This estimate \bcb (\ref{L6_intro}) \ec will provide the desired bound for (\ref{nonlinear_1}) only by taking a supremum in $t$ to the dissipation. We use a version of Sobolev embedding $L^\infty_t \subset H^1_t$ to control $L^\infty_t L^2_{x,v}$-norm of the dissipation provided $\e^{-1} \kappa^{-1/2}  \p_t  [(\mathbf{I} - \mathbf{P})  f_R  ]  \in L^2_t L^2_{x,v}$ holds. \unhide


\subsection{Higher Regularity of Navier-Stokes equations in the Inviscid Limit}\label{sec:NSlimit}

The inviscid limit of the Navier-Stokes equations \eqref{NS_k}-\eqref{noslip} is at the heart of our approach. Furthermore, in order to control $f_R$, as explained in the above, we need to derive quantitative higher regularity estimates of the Navier-Stokes solutions which are not directly available in the usual inviscid limit results.  
Before discussing new features of our analysis, we briefly discuss some prior works on the inviscid limit most relevant to our result. 
Due to the formation of boundary layers in the limit caused by the mismatch of boundary conditions \eqref{noslip} and \eqref{no-pen},  
a classical way to tackle the inviscid limit problem is via 
the Prandtl expansion,  
of which  rigorous justification was shown in  \cite{SC1, SC2} for well-prepared data with analytic regularity and in \cite{Mae14} for the initial datum with Sobolev regularity when the initial vorticity is bounded away from the boundary. In particular, the author of \cite{Mae14} introduced the boundary vorticity formulation of  \eqref{NS_k}-\eqref{noslip} (see  \eqref{NS}-\eqref{NSB})
which prompted subsequent interesting works in the field. Among others, in a recent work \cite{NN2018}, the authors proved the inviscid limit in 2D based on the Green's function approach based on Maekawa's vorticity formulation without having to construct  Prandtl boundary layer corrections but by utilizing the boundary layer weights in the norm. In \cite{KVW, FW}, the inviscid limit was shown for initial data that is analytic only near the boundary and has finite Sobolev regularity in the complement in 2D and 3D respectively. 

Our analysis of the Navier-Stokes solutions in the limit is based on the Green's function approach for the Stokes problem using the vorticity formulation \eqref{NS}-\eqref{NSB} in the same spirit of \cite{NN2018}. However, the existing methods \cite{NN2018,KVW,FW} do not immediately fulfill the goal of our hydrodynamic limit  because the analysis of our remainder $f_R$ requires higher regularity of Navier-Stokes solutions, more specifically $L^2$ and $L^\infty$ bounds for higher order derivatives up to two temporal derivatives of $\nabla_x u$ and $p$ and two spatial derivatives of $\p_tu$, while the existing methods do not decipher any bounds for temporal derivatives and the boundedness of the conormal derivatives in their analytic norms does not rule out $\frac{1}{x_3}$ singularity of the normal derivative of the vorticity in the boundary layer, which may cause the loss of $L^2$ integrability. To get around these issues, we pursue new estimates of temporal derivatives of the vorticity $\o$ by demanding the compatibility conditions for the initial data. With such conditions, the initial layer is absent for $\o$ and $\p_t\o$; we can derive an analogous integral representation formula for $\p_t \o$ so that we may run the same fixed point argument for  $\p_t\o$ as in \cite{NN2018} without the initial layer. For the second temporal derivative, we handle the initial-boundary layer for the horizontal part with the initial-boundary weight function. 
These new features allow us to attain the derivative estimates of the vorticity in the normal direction without {$\frac{1}{x_3}$} singularity near the boundary at the expense of losing a power of $\sqrt\kappa$, which is crucial for the control of $f_R$. The velocity and pressure estimates are then recovered by utilizing elliptic regularity results  and the Biot-Savart law in the analytic setting. 
The main results of Navier-Stokes solutions to \eqref{NS_k}-\eqref{noslip} are given in Theorem \ref{thm_bound}. 
\hide
\bcr

The rest of the paper proceeds as follows....

\ec
\unhide

\section{Main Results}\label{sec:MR}


For the sake of the readers we present the precise statement of main theorems and their notations in this section. We first present the uniform controls of the Boltzmann remainder $f_R$ of Theorem \ref{main_theorem:conditional}, and the higher regularity of the Navier-Stokes equations in the inviscid limit of Theorem \ref{thm_bound}. As a consequence of those two theorems we will show a rigorous justification of kinetic approximation of Navier-Stokes in high Reynolds numbers of Theorem \ref{main_theorem}. Then using the vorticity estimates in Theorem \ref{thm_bound} and the famous Kato's condition in the inviscid limit, we prove a hydrodynamic limit toward the incompressible Euler equations in Corollary \ref{Cor_EL}.

\subsection{Uniform controls of the Boltzmann remainder $f_R$ (Theorem \ref{main_theorem:conditional})}\label{sec:B1} We recall the  expansion of Boltzmann solution $F = \mu+ \e^2 f_2 \sqrt{\mu }  + \delta \e  f_R\sqrt{\mu} $ in (\ref{F_e}) around the local Maxwellian $\mu(v) : = M_{1 , \e u , 1 }(v)$ for any given flow $(u,p)$ solving the incompressible Navier-Stokes equation with the no-slip boundary condition (\ref{NS_k})-(\ref{noslip}).

 Inspired by the energy structure of the PDE and the coercivity of the linear operator $L$ in (\ref{s_gap}), we define an energy and a dissipation as 
\Be\label{ED}
\begin{split}
\mathcal{E} (t):= & \ \| f_R (t) \|_{L^2 (\O \times \R^3)}^2 + \| \p_t  f_R (t) \|_{L^2 (\O \times \R^3)}^2 ,\\
 \mathcal{D} (t) 
:=&  \ \int^t_0 \| \kappa^{-\frac{1}{2}} \e^{-1}
\sqrt{\nu}
 (\mathbf{I} - \mathbf{P}) f_R (s) \|_{L^2 (\O \times \R^3)}^2\dd s \\
 &  +\int^t_0  \|
 \kappa^{-\frac{1}{2}} \e^{-1} \sqrt{\nu}  (\mathbf{I} - \mathbf{P})\p_t  f_R (s) \|_{L^2 (\O \times \R^3)}^2
  \dd s \\
 &
 + \int^t_0 \Big(   | \e^{-\frac{1}{2}}
f_R(s)
|_{L^2_\gamma  }^2 +  | \e^{-\frac{1}{2}}
\p_t f_R(s)
|_{   L^2_\gamma}^2\Big) \dd s
.
\end{split}
\Ee
As explained in Section 1.2, the temporal derivative gets involved mainly in order to access the $L^6$-bound of the hydrodynamic part $\mathbf{P}f_R$, while we will control the following auxiliary norm to be used in order to handle the nonlinearity: for $p<3$ and $t>0$ 
\Be
\begin{split}
\mathcal{F}_p(t):= \sup_{0 \leq s \leq t} \Big\{&  \| \kappa^{1/2} Pf_R(s)\|_{L^6 (\O)} ^2
+ \| \kappa^{1/2} P   f_R\|_{L^2((0,s); L^p(\O))}^2\\
&+ \|  \kappa^{ \mathfrak{P}+1/2}   P \p_t f_R\|_{L^2((0,s); L^p(\O))}^2
+ \| \e^{1/2}\kappa \mathfrak{w}_{\varrho,\ss} f_R(s) \|_{L^{\infty}(\O \times \R^3) }^2
\\
&
+ \|   (\e  \kappa)^{3/p}  \kappa^{ \frac{1}{2}+\mathfrak{P}} \mathfrak{w}_{\varrho^\prime,\ss} f_R(s) \|_{L^2((0,s);L^{\infty}(\O \times \R^3)) }^2
\Big\}.
\end{split}
\Ee
Here we have introduced weight functions, in order to control an extra quadratic growth in $|v|$ from $(\p_t + \e^{-1}  v\cdot \nabla_x) \sqrt{\mu}  f_R/\sqrt{\mu}$
\Be\label{weight}
\mathfrak{w}_{\varrho, \ss}(x,v)=\mathfrak{w} := \exp\{\varrho |v|^2 - 
\mathfrak{z}_{\ss}(x_3) (x \cdot v) 
\} \ \ \text{for} \ \ 0< \ss \ll \frac{\varrho}{2\pi} \ \text{and} \ 0 < \varrho < \frac{1}{4},
\Ee 
where $\mathfrak{z}_{\ss}: \R_+ \rightarrow \R_+$ is defined as, for $\ss>0$ 
\Be\label{z_ss}
\begin{split}
\mathfrak{z}_{\ss}(x_3)= \ss   \ \ \text{for} \ \  x_3 \in \big[0, \frac{1}{\ss}-1 \big],  \ \ \text{and} \ \ 
\mathfrak{z}_{\ss}(x_3)=\frac{1}{1+ x_3}  \  \ \text{for} \ \  x_3\in \big[\frac{1}{\ss}-1 , \infty\big).
\end{split}
\Ee 
 We have denoted $\mathfrak{w}_{\varrho', \ss}(x,v) =\mathfrak{w}'$ for $\varrho'<\varrho$. Also we have denoted the boundary norms and integral as 
\Be\label{bdry_int}
\begin{split}
 |g|_{L^p_{\gamma }}  := \left( \int_{\gamma_+}|g|^p+\int_{\gamma_-}|g|^p \right)^{1/p},    \ |g|_{L^p_{\gamma_\pm}}  := \left(  \int_{\gamma_\pm}|g|^p\right)^{1/p},    \\ 
\int_{\gamma_\pm}  f := \int_{\p\O} \int_{n(x) \cdot v \gtrless0}f (x, v) |n(x) \cdot v| \dd v \dd S_x.
\end{split}\Ee

Next we discuss the initial data of the Boltzmann equation. We note that an initial datum of $f_2$ is already determined by given flow $(u,p)$. For given initial data $f_{R,0}:=f_{R,in}$, inspired by the PDE, we define  
\Be\label{initial_f}
\begin{split}
\p_t f_{R,0}:= &
 - \frac{1}{\e} v\cdot \nabla_x  f_{R,in}
    -   \frac{1}{ \e^2\kappa} L_{in} f_{R,in}
+  \frac{2}{\kappa} \Gamma_{in}({f_2}, f_{R,in})
+    \frac{  \sqrt{\e}}{\e\kappa}\Gamma_{in}(f_{R,in}, f_{R,in})
\\
&
-  \frac{( \p_t + 
\e^{-1} v\cdot \nabla_x) \sqrt{\mu_{in}}}{\sqrt{\mu_{in}}} f_{R,in}  + 
(\mathbf{I}- \mathbf{P})\mathfrak{R}_1 (u,p)|_{t=0} + \mathfrak{R}_2(u,p)|_{t=0},
\end{split}
\Ee
where   $(\mathbf{I}- \mathbf{P})\mathfrak{R}_1$ and $\mathfrak{R}_2$ are defined in (\ref{eqtn_fR}) with $\delta=\sqrt{\e}$ and $\mu_{in}$, $L_{in}$, $\Gamma_{in}$ are induced by the initial Naiver-Stokes velocity $u_{in}$. For the remainder $f_R$ in (\ref{F_e1}), we will use the norms of the initial data:
\Be\label{E} 
\mathcal{E}(0):=\mathcal{E} (f_{R,0}):=   \ \| f_{R,0}\|_{L^2 (\O \times \R^3)}^2 + \| \p_t  f_{R,0}  \|_{L^2(\O \times \R^3)}^2 ,\\
\Ee
\Be \label{initial_F}
\begin{split}
\mathcal{F}_p(0):=
 &\big\{\kappa^{\frac{1}{2}} |  f_{R,0}  |_{L^2_\gamma}
+ \kappa^{\mathfrak{P}+ \frac{1}{2}}   | \p_t f_{R,0}   |_{L^2_\gamma}\\
&+
\e^{\frac{1}{2}} \kappa \| \mathfrak{w} f_
{R,0}\|_{L^\infty (\bar{\O} \times \R^3)}
+ (\e \kappa)^{1+ \frac{3}{p}}  \kappa^{\mathfrak{P}} \| \mathfrak{w}^\prime \p_t f_{R,0}\|_{L^\infty(\bar{\O} \times \R^3)}\big\}^2.
\end{split}\Ee


\ 

\begin{theorem}[Uniform controls of the Boltzmann remainder $f_R$]
\label{main_theorem:conditional}
Suppose for $T>0$ and $\mathfrak{P}\geq 1/2$ 
\Be
  \sum_{\ell=0,1} \|\nabla_x\p_t^\ell  u\|_{L^\infty ([0,T]\times \bar{\O})}+
 \frac{1}{\kappa^{1/2}} \sum_{\ell=0,1,2}\|  \p_t^\ell u\|_{L^\infty ([0,T]\times \bar{\O})}
+ \frac{1}{\kappa^{1/2}}  \| p \|_{L^\infty ([0,T]\times \bar{\O})}
\lesssim 
\frac{1}{ \kappa^{  \mathfrak{P}}}.\label{mathfrak_C'}
\Ee  
We further assume that, for $0 \leq \mathfrak{P}^\prime < \mathfrak{P}$, 
\Be\label{condition:theorem}
\begin{split}
& 
\sum_{\ell=1,2}\| \p_t ^\ell u \|_{L^\infty ([0,T]; L^\infty(\bar{\O})\cap L^2(  {\O})  )}
+\sum_{
 \substack{0 \leq \ell \leq 1 \\
  1 \leq  |\beta| \leq 2
 } } \|\nabla_x^\beta \p_t^\ell  u\|_{L^\infty ([0,T]; L^\infty( \bar{\O})
 \cap L^2( {\O}) )}\\
 &+ \sum_{|\beta|=1} \|\nabla_x^\beta \p_t^2 u\|_{L^2([0,T]; L^\infty( \bar{\O})
 \cap L^2( {\O}) )}
 \\
 &+  \| \p_t^2 p \|_{L^2 ([0,T]; L^\infty(\bar{\O}) \cap L^2(\O)   )}
  +
  \sum_{|\beta|=0,1} \| \nabla_x ^\beta \p_t  p \|_{L^\infty ([0,T]; L^\infty(\bar{\O}) \cap L^2(\O))}
 \lesssim  \exp \Big({\frac{1}{\kappa^{\mathfrak{P}^\prime }}}\Big)
 . 
\end{split}\Ee
For given such $T>0$, let us choose $\e, \delta$ and $\kappa$ as, for some $\mathfrak{C}\gg 1$, 
\Be\label{choice:delta}
\delta= \sqrt{\e} \ \ \text{and} \ \ 
\delta=\exp\Big(  \frac{  - \mathfrak{C}  T}{ \kappa^{ \mathfrak{P}} }\Big). 
\Ee
Assume that an initial datum for the remainder $f_{R,in }$ satisfies,  for some $p<3$ and $|p-3|\ll 1$,  
\Be\label{initial_EF}
 \sqrt{\mathcal{E}(0) }+ \sqrt{\mathcal{F} _p (0) } \lesssim1.
\Ee
  Then we construct a unique solution $f_R(t,x,v)$ of the form of 
 \Be\notag
 F  = M_{1, \e u  ,1 } + \e^2 f_{2} 
  + \delta \e   f_{R}  \ \ \text{in} \ \ [0,T] \times \O \times \R^3,
 \Ee
which solves the Boltzmann equation (\ref{Boltzmann}) and the diffuse reflection boundary condition (\ref{diffuse_BC}) with the scale of (\ref{scale}) and (\ref{choice:delta}), and satisfies the initial condition $F|_{t=0}= M_{1, \e u _{in} ,1 } +  \e^2  f_{2} |_{t=0}
  +  \delta \e f_{R,in}$, in a time interval $t \in [0,T]$. Moreover, we have 
\Be\label{est:E}
\delta^{ \frac{1}{2}- \frac{3}{p} (1- \frac{p}{3})} \sup_{0 \leq t \leq T}   \big\{ \sqrt{\mathcal{E}(t)} +  \sqrt{\mathcal{D}(t)}+  \sqrt{\mathcal{F}_{ p } (t)}\big\} 
 \lesssim 1.
\Ee

  \hide

For $\e,\kappa$ in (\ref{choice:delta:intro}) (\ref{F_e}):
 \Be\notag
F = \mu+ \e^2 f_2 \sqrt{\mu }  + \e ^{3/2}f_R\sqrt{\mu} ,
\Ee 
as , for some $\mathfrak{C}\gg1$,

\hide
\Be\label{condition:delta_1}
\begin{split}
\delta &\leq \left[  \frac{
\kappa^{  \frac{3}{2} + \mathfrak{P} + \frac{3}{  p} (1- \frac{p}{3}) }
}{\mathfrak{C}_1  \big(
\mathcal{E}(0) + \mathcal{F}_p(0)\bcb+ 1  \ec\big) }  
\exp\Big(  \frac{  -2\mathfrak{C}_2 T}{ \kappa^{ \mathfrak{P}} }\Big) 
    \right]^{\frac{1}{1- \frac{6}{p} (1- \frac{p}{3})}},
    \end{split}
\Ee  \unhide

\hide
Further we assume that 
\Be\label{condition:theorem}
\begin{split}
&\e \kappa^{\frac{1}{2}}
\{\|(\ref{est:R1})\|_{L^2_{t,x}}+ \|(\ref{est:R2})\|_{L^2_{t,x}}+\|(\ref{est:R3})\|_{L^2_{t,x}}+ \|(\ref{est:R4})\|_{L^2_{t,x}}\}\\
&+ \e^{\frac{5}{2}} \kappa^{2} \{\|(\ref{est:R1})\|_{L^\infty_{t,x}}
+\|(\ref{est:R2})\|_{L^\infty_{t,x}}+\|(\ref{est:R3})\|_{L^\infty_{t,x}}
+\|(\ref{est:R4})\|_{L^\infty_{t,x}} \}\\
&
 + \e^{1/2} \| (\ref{est:f2})\|_{L^\infty_t L^6_x}
+ (  \frac{\e^{\frac{3}{2}}\kappa  }{\delta}+ \e^2
+ \e^{2- 3/p} \kappa^{1-\mathfrak{P}} \delta^{-1}
)
 \| (\ref{est:f2}) \|_{L^\infty_{t,x}} +  
 \frac{\e}{\kappa} \| 
 (\ref{est:f2})
  \|_{L^\infty_t   L_x^{\frac{2p}{p-2}}  }  +  \frac{\e \kappa^{\frac{1}{2}}}{\delta}| (\ref{est:f2})|_{L^\infty_tL^4(\p\O)}  \\
  &+ 
\end{split}
\Ee 
\Be\label{L6Linfty:force}
\begin{split}
  \frac{\e^{\frac{3}{2}}\kappa  }{\delta}
 \| (\ref{est:f2}) \|_{L^\infty_{t,x}} 
 +  \e^{\frac{3}{2}}   \kappa^{3} 
  \| (\ref{est:R2})\|_{L^\infty_{t,x}} 
 +  \frac{\e \kappa^{\frac{1}{2}}}{\delta}| (\ref{est:f2})|_{L^4(\p\O)}   + \e
  \kappa^{\frac{1}{2}}
\{  \| 
 (\ref{est:R1})  \|_{L^2_{x,v}} + \|(\ref{est:R2})
 \|_{L^2_{x,v}} \}
 <1 .\\
 \e \kappa^{1/2} \| (\ref{transp:mu}) \|_{L^\infty_{t,x}} + \e \kappa^{-1/2} \| (\ref{est:R2})\|_{L^\infty_{t,x}} <1\\
 \delta \e \| w f(0)\|_\infty+ \e^2 \| (\ref{est:f2})\|_\infty+ \delta \kappa^2 \e^2\| (\ref{est:R2}) \|_\infty<1\\
 \kappa^{1/2}\e \| (\ref{transp:mu})\|_{L^2_t L^\infty_x} < 1\\
\end{split}
\Ee
\Be\label{average_3D_ED:force}
\begin{split}
\kappa^{\frac{1}{2}}\e \| (\ref{transp:mu})\|_{L^2_t L^\infty_x}<1
\\
 \e^{\frac{1}{2}} \kappa^{-\frac{1}{2}} \delta (\ref{L6Linfty:force})
+\kappa^{ \frac{1}{2}}  \| f_R (0) \|_{L^2_\gamma} +
 \e \kappa^{ \frac{1}{2}}\{  \|
 (\ref{est:R3}) \|_{L^2_{t,x}   } +\| (\ref{est:R4})
  \||_{L^2_{t,x}   } \}<1,
  \end{split}
\Ee
\[
O(\e) \|u\|_{L^\infty_{t,x}} 
+
 \frac{\e}{\kappa} \| 
 (\ref{est:f2})
  \|_{L^\infty_t   L_x^{\frac{2p}{p-2}}  }+\e
\| (\ref{transp:mu})\|_{L^\infty_{t } L^{\frac{2p}{p-2}}_x }<1
\]
with 
\Be\label{average_3Dt_ED:force}
\begin{split}
& \kappa^{\frac{1}{2}}  \|\p_t  f_R (0) \|_{L^2_\gamma} 
+ 
(\e  \kappa)^{\frac{3}{p}+1} 
\| \mathfrak{w} ^\prime  \p_t  f_R(0 ) \|_{{ L^{\infty}_{x,v}}}
+ \e \kappa^{\frac{1}{2}}\Big[
 (
 \| (\ref{est:R3})\|_{L^2_{t,x}}
+    \| (\ref{est:R4})\|_{L^2_{t,x}})
  + (\e  \kappa)^{\frac{3}{p}+1} ( \|(\ref{est:R3}) \|_{L^\infty_{t,x}}+   \|(\ref{est:R4}) \|_{L^\infty_{t,x}})
  \Big]
\\
&
+(\e  \kappa)^{3/p} \frac{\e \kappa^{\frac{1}{2}}  }{\delta}
\Big[
 \|(\ref{est:f2_t})\|_{L^\infty_{t,x}} 
 + \e \|\p_t u\|_{L^\infty_{t,x}}  \|(\ref{est:f2}) \|_{L^\infty_{t,x}}
 \Big] +(\e  \kappa)^{3/p} \e^{\frac{1}{2}}  \big[  \|\p_t u \|_{L^\infty_{t,x}}+ \e 
\|(\ref{est:f2_t})\|_{L^\infty_{t,x}}+
  \e  \kappa \|(\ref{transp:mu_t}) \|_{L^\infty_{t,x}}
 \big]  (\ref{L6Linfty:force}),
\end{split}
\Ee
where we have assumed that 
\Be\label{average_3Dt_ED:assump}
\begin{split}
 \kappa^{ \frac{1}{2}}  + 
 \kappa^{ \frac{1}{2}} \e  (
1
 + \e^{\frac{3}{p}+ \frac{1}{2}} 
 )
  \| (\ref{transp:mu_t})\|_{L^\infty_{t,x}}
+ \e(\e  \kappa)^{\frac{3}{p}- \frac{1}{2}}  
   \big(  \|\p_t u \|_{L^\infty_{t,x}}+ \e 
\|(\ref{est:f2_t})\|_{L^\infty_{t,x}}
 \big)   
  + \e^2 \|   (\ref{est:f2})\|_{L^\infty_{x,v}}  
+ \e^2 \kappa  \| 
(\ref{transp:mu})
 \|_{L^\infty_{t,x}  }
  \lesssim 1.
\end{split}
\Ee

\hide
 \Be 
 \begin{split}
  \label{Linfty_3D_ED}
\sup_{0 \leq s \leq t} \| \mathfrak{w}_{\varrho, \ss} f _R(s) \|_\infty   
 \lesssim  & \ 
 \frac{1}{\e^{1/2} \kappa }
 \Big[\sqrt{ \mathcal{E}(t) }+ \sqrt{ \mathcal{D}(t) } \Big]  +\frac{\e}{\delta}
 \| (\ref{est:f2}) \|_{L^\infty_{t,x}} 
 +  \frac{\e^{\frac{1}{2}}}{ \kappa^{\frac{1}{2}}\delta}| (\ref{est:f2})|_{L^4(\p\O)} 
 + 
\e \kappa^2 \| (\ref{est:R2})\|_{L^\infty_{t,x}} 
 + \e^{\frac{1}{2}} \kappa^{-\frac{1}{2}} \| 
 (\ref{est:R1}) + (\ref{est:R2})
 \|_{L^2_{x,v}}
\\
&
 + \e^{-\frac{1}{2}} \kappa^{-1}   
   \|\e^{-1}\kappa^{-\frac{1}{2}}  (\mathbf{I} - \mathbf{P}) f_R (0) \|_{ {L^2 (\O \times \R^3)}} 
   +     \|  \mathfrak{w}_{\varrho, \ss} f (0)\|_\infty
   + 
 (\e \kappa )^{-1}   |(1- P_{\gamma_+}) f_R(0)|_{L^2({\gamma_+})}.
 \end{split}
 \Ee
 \unhide

\Be\label{relat:kappa_delta}
 \delta  \lesssim \kappa^{ \mathfrak{q}} e^{-  \mathfrak{C}^\prime \kappa^{- \mathfrak{P}} T } ,
 %
 %
\Ee
\unhide
If (\ref{initial_EF}) hold with $|p-3|\ll1$  for $p<3$ 
then a unique solution $f_R(t,x,v)$ of (\ref{F_e1}) to (\ref{Boltzmann}) and (\ref{diffuse_BC}) with (\ref{choice:delta}), satisfying the initial condition $F|_{t=0}= M_{1, \e u _{in} ,1 } + \e^2 f_{2} |_{t=0}
  + \delta \e   f_{R,in}$, exists\unhide \end{theorem}
 

\begin{remark}The condition (\ref{choice:delta}) in the theorem is indeed the largest $\sqrt{\e}$ can be allowed. Any smaller $\sqrt{\e}$ than $\exp\Big(  \frac{  - \mathfrak{C}  T}{ \kappa^{1/2} }\Big)$ (which means $\sqrt{\e}$ decaying faster than $\exp\Big(  \frac{  - \mathfrak{C}  T}{ \kappa^{1/2} }\Big)$ as $\kappa \downarrow 0$) will produce the same result. 
In terms of (\ref{scale}) the relation (\ref{choice:delta}) implies that the Knudsen number $\mathpzc{Kn}$ has to vanish only slightly faster than the Mach number $\mathpzc{Ma}$:
\Be\label{scale_1}
\mathpzc{St}=\e = \mathpzc{Ma}  \  \ \text{and} \ \ 
 \sqrt{ \frac{T }{ \ln \e^{-1}   }} \lesssim \frac{\mathpzc{Kn}}{\mathpzc{Ma}} \downarrow 0  \ \ \text{as} \ \  \e \downarrow 0.
\Ee 
\hide
On the other hand, the factor $\frac{1}{\kappa^{1/2}}$ in (\ref{choice:delta:intro}) inherits from the asymptotic bound of  Navier-Stokes flow: 
\Be
  \sum_{\ell=0,1} \|\nabla_x\p_t^\ell  u\|_{L^\infty ([0,T]\times \bar{\O})}+
 \frac{1}{\kappa^{1/2}} \sum_{\ell=0,1,2}\|  \p_t^\ell u\|_{L^\infty ([0,T]\times \bar{\O})}
+ \frac{1}{\kappa^{1/2}}  \| p \|_{L^\infty ([0,T]\times \bar{\O})}
\lesssim 
\frac{1}{ \kappa^{  1/2}}.\label{mathfrak_C:1/2}
\Ee 
See Theorem \ref{main_theorem:conditional} for the relation.\unhide

 \hide
 We remark that the choice (\ref{choice:delta:intro}) implies that, in terms of (\ref{scale}),
\Be\notag
\mathpzc{St}=\e = \mathpzc{Ma}  \  \ \text{and} \ \ 
\kappa(\e) \downarrow 0 \ \ \text{as} \ \ \e \downarrow 0
 \ \  \text{while} \ \  
\kappa (\e)  \gtrsim     \Big(\ln \Big(\frac{1}{\sqrt{\e}}\Big)\Big)^{-\frac{1}{\mathfrak{P}}}  .
\Ee 
\unhide

 \hide

Otherwise the nonlinearity of (\ref{Boltzmann}), as a singular perturbation, turns out too severe in the scope of our mathematical analysis. 
\unhide
\end{remark}
 The proof of Theorem \ref{main_theorem:conditional} will be given in Section \ref{sec:B}.

\subsection{Higher regularity of Navier-Stokes equations in the inviscid limit (Theorem \ref{thm_bound})}  For the Navier-Stokes solutions to (\ref{NS_k})-(\ref{noslip}), we introduce real analytic norms and function spaces, adopted from \cite{NN2018} and \cite{FW} for the 3D counter part with slight modifications. 

In this subsection and Section \ref{sec:5}, we will use the following notations: $x=(x_h,x_3)=(x_1,x_2,x_3)\in \mathbb T^2\times \mathbb R_+=\Omega$, $\nabla_x=\nabla=(\nabla_h,\p_3)=(\p_{x_1},\p_{x_2},\p_{x_3})$;  for a vector valued function $g\in \mathbb R^3$, $g=(g_h,g_3)=(g_1,g_2,g_3)$.

We denote the vorticity by 
\Be
\o=\nabla \times   u, \ \ \ u =\nabla \times (- \Delta)^{-1} {\o},\label{vorticity}
\Ee
while the second identity  is the famous Biot-Savart law. Here $(- \Delta)^{-1}$ denotes the inverse of $-\Delta$ with the zero Dirichlet boundary condition on $\p\O$.

Our analysis of the Navier-Stokes solutions is based on the vorticity formulation in 3D (\cite{Mae13,Mae14}):
  \begin{align}
\p_t \o - \kappa \eta_0 \Delta\o = - u \cdot \nabla \o
+ \o \cdot \nabla u \ \ &\text{in} \ \ \O
,  \quad \label{NS} \\
\o \,|_{t=0}   = \o_{in}  \ \ &\text{in} \ \ \O, \label{NSI} \\
\kappa \eta_0 (\p_{x_3} + \sqrt{- \Delta_h})\o_h \,    = [\p_{x_3}(- \Delta)^{-1} (-u \cdot \nabla \o_h 
+ \o \cdot \nabla  u_h
) ] \,  , \ \ \o_3 =0 \ \ &\text{on} \ \ \p\O, \label{NSB}
\end{align} 
where $\sqrt{- \Delta_h}=|\nabla_h|$ is defined as 
\Be\label{sqrt_D}
 \sqrt{- \Delta_h}g(x_h,x_3)= \sum_{\xi \in \mathbb{Z}^2} |\xi| g_\xi (x_3)  e^{i x_h \cdot \xi}. 
\Ee Here, $g_{\xi } (x_3)= \frac{1}{(2\pi)^2} \iint_{\mathbb{T}^2} e^{-i x_h \cdot \xi} g(x_h,x_3) \dd x_h \in \mathbb C \text{ with } \xi= (\xi_1, \xi_2) \in \mathbb{Z}^2$ denotes the Fourier transform in the horizontal variables, which 
 satisfies $g(x_1,x_2,x_3)
 = \sum_{\xi \in \mathbb{Z}^2} g_{\xi} (x_3) e^{i x_h \cdot \xi }.$ The Fourier transform can be regarded as a function $g_{\xi} (z)$ where $z$ is sitting in a pencil-like complex domain: for any $\lambda> 0$, 
\Be\label{complex_domain}
\mathcal{H}_\lambda:=\Big\{ z \in \mathbb C : \text{Re}\,z\geq 0, \; | \text{Im}\, z| < \lambda \min \{ \text{Re}\,z, 1\} \Big\}.
\Ee

We define analytic function spaces without the boundary layer, $\mathfrak{L}^{p,\lambda}$, for holomorphic functions with a finite norm,  for $p\geq 1$,  
\Be\label{norm_L1}
\| g \|_{p,\lambda } := \sum_{\xi\in \mathbb Z^2} e^{\lambda |\xi|} \| g_{\xi } \|_{\mathcal{L}^p_\lambda}   \ \ \text{where} \ \ 
\|g_{\xi }\|_{\mathcal{L}^p_\lambda} := \sup_{0\leq \sigma\leq \lambda} 
\left(
\int_{\p \mathcal{H}_\sigma} | g_{\xi } (z)  |^p  |\dd z|
\right)^{1/p}.
\Ee

Next we introduce an $L^\infty$-based analytic boundary layer function space, for $\lambda>0$ and $\kappa \geq 0$, that consists of holomorphic functions in $\mathcal{H}_\lambda$ with a finite norm 
\Be\label{norm_BL}
\begin{split}
\|g\|_{\infty,\lambda,\kappa}  = \sum_{\xi\in {\mathbb Z^2}} e^{\lambda |\xi|} \| g_{\xi } \|_{\mathcal{L}^\infty_{\lambda, \kappa}}  ,
\end{split}
\Ee
where $\| g_{\xi } \|_{\mathcal{L}^\infty_{\lambda,0}} : = 
\| e^{\bar{\alpha} \text{Re}\,z}  g_{\xi } (z) \|_{\mathcal{L}^\infty_\lambda}
:= \sup_{z \in \mathcal{H}_\lambda}e^{\bar{\alpha} \text{Re}\,z}  g_{\xi } (z) 
$ and 
\Be\notag
 \| g_{\xi} \|_{\mathcal{L}^\infty_{\lambda, \kappa}}  := 
  \bigg\| \frac{e^{\bar{\alpha} \text{Re}\,z}}{ 1+ \phi_\kappa(z)} g_{\xi } (z) \bigg\|_{\mathcal{L}^\infty_\lambda}
  := \sup_{z \in \mathcal{H}_\lambda}\frac{e^{\bar{\alpha} \text{Re}\,z}}{ 1+ \phi_\kappa(z)} |g_{\xi } (z)| 
  .
\Ee
Here, a boundary layer weight function is defined  as 
\Be\label{BL}
\phi_\kappa (z):= \frac{1}{\sqrt\kappa} \phi ( \frac{z}{\sqrt\kappa}) 
 \ \ \text{with} \ \ \phi (z) = \frac{1}{ 1+|\text{Re}\,z|^\mathfrak{r}} \ \text{for some  } \mathfrak{r}>1.
\Ee
We define $\mathfrak B^{\lambda, \kappa}$ for holomorphic functions $g= (g_1,g_2,g_3)$ with a finite norm 
 \Be\label{[]}
[[g ]]_{ \infty, \lambda,\kappa}=\sum_{i=1,2} \| g_i \|_{\infty, \lambda, \kappa} + \| g_3 \|_{\infty, \lambda, 0}.
\Ee
We note that $\mathfrak B^{\lambda, \kappa } \subset \mathfrak{L}^{1,\lambda}$, but $\mathfrak{B}^{\lambda, 0} \subsetneqq\mathfrak{L}^{\infty,\lambda}$ if $\bar{\alpha}>0$.

\smallskip

Due to its singular nature of the Navier-Stokes flow in the inviscid limit, we introduce the conormal derivatives
\Be
 D= (D_h, D_3)= (\nabla_h, \zeta(x_3) \p_3) \ \ \text{where}  \ \ \zeta(z) = \frac{z}{1+z}. 
\Ee
With the multi-indices $\beta=(\beta_h,\beta_3):=(\beta_1,\beta_2,\beta_3)\in \mathbb N_0^3$, the higher derivatives are denoted by
 $D^\beta = \p_1^{\beta_1} \p_2 ^{\beta_2} D_3^{\beta_3}$ and 
 $D^\beta_\xi = (i \xi_1)^{\beta_1} (i \xi_2) ^{\beta_2} D_3^{\beta_3}$. 

Now we define, for $\lambda_0>0$, $\gamma_0>0$, $\alpha>0$, $\kappa \geq 0$, and $t \in (0, \frac{\lambda_0}{2 \gamma_0})$
 \Be\label{norm_BLT}
 \vertiii{g}_{\infty,\kappa}= \sup_{\lambda<\lambda_0-\gamma_0 t} \bigg\{ \sum_{ 0 \leq |\beta| \leq 1}  
[[
D^\beta g ]]_{
\infty, \lambda, \kappa } +  \sum_{ 
|\beta| =2}(\lambda_0-\lambda-\gamma_0 t)^\alpha [[D^\beta g
]]_{
 {\infty, \lambda, \kappa}
}  \bigg\},
\Ee
\Be\label{norm_L1T}
\begin{split}
\vertiii{g}_1 
= \sup_{\lambda<\lambda_0-\gamma_0 t} \bigg\{& \sum_{0\leq |\beta|\leq 1} \| 
D^\beta (1+|\nabla_h|) g \|_{1,\lambda} \\
& \ \ \ \ \  \ \ \ \ \ +  (\lambda_0-\lambda-\gamma_0 t)^\alpha\sum_{ |\beta|=2} \| D^\beta (1+|\nabla_h|) g \|_{1,\lambda}  \bigg\}.
\end{split}\Ee

With an initial-boundary layer weight function as in \cite{NN2018}
\Be
\phi_{\kappa t} (z)= \frac{1}{\sqrt{\kappa t}} \phi ( \frac{z}{\sqrt{\kappa t}}) 
,\label{IB_layer}
 \Ee
we define an initial-boundary layer function space $\mathfrak B^{\lambda, \kappa t}$ for holomorphic functions $g= (g_1,g_2,g_3)$ with a finite norm 
\Be\label{[[]]}
[[g ]]_{ \infty, \lambda,\kappa t}=\sum_{i=1,2} \| g_i \|_{\infty, \lambda, \kappa t} + \| g_3 \|_{\infty, \lambda, 0} ,
\Ee 
where an $L^\infty$-based analytic norm with the initial-boundary layer is defined as
\Be\label{norm_IBL}
\|g\|_{\infty,\lambda,\kappa t} = \sum_{\xi\in \mathbb Z^2} e^{\lambda |\xi|} \| g_{\xi } \|_{\mathcal{L}^\infty_{\lambda, \kappa t}}, \quad 
\| g_{\xi} \|_{\mathcal{L}^\infty_{\lambda, \kappa t}} = 
\bigg\| \frac{e^{\bar{\alpha} \text{Re}\,z}}{ 1+ \phi_\kappa(z)+ \phi_{\kappa t} (z)}   g_{\xi } (z)\bigg\|_{\mathcal{L}^\infty_\lambda} . 
\Ee
We finally define, for $t \in (0, \frac{\lambda_0}{2\gamma_0})$, 
 \Be\label{norm_IBLT}
 \vertiii{g}_{\infty, \kappa t}= \sup_{\lambda<\lambda_0-\gamma_0 t} \bigg\{
  \sum_{
 0 \leq |\beta| \leq 1
  }
 [[
D^\beta g]]_{\infty,\lambda, \kappa t}+ \sum_{|\beta|=2} 
 (\lambda_0-\lambda-\gamma_0 t)^\alpha
  [[
D^\beta g]]_{\infty,\lambda, \kappa t}
 \bigg\}.
\Ee
In this subsection and Section \ref{sec:5}, $\alpha$, $\bar{\alpha}$ are  given  positive small constants, $\lambda_0$ is a given positive constant, and $\gamma_0$ is a sufficiently large constant to be determined in Theorem \ref{thm_bound}.

Next we discuss the initial data of the velocity $u_{in}$ and the corresponding vorticity $\o_{in}= \nabla_x \times u_{in}$. Inspired by the PDEs, let 
\Be\begin{split}\label{idata}
 \o_0:=\o_{in},\quad \p_t\o_0:= \kappa\eta_0\Delta \o_0 - u_0\cdot\nabla \o_0+ \o_0 \cdot \nabla u_0,\\
 \quad
 u_0 := \nabla \times (-\Delta)^{-1} \o_0 ,
  \quad
  \p_t u_0 := \nabla \times (-\Delta)^{-1} \p_t \o_0 ,
 \\
\p_t^2\o_0:= \kappa\eta_0 \Delta \p_t\o_0 - u_{0} \cdot \nabla \p_t\o_0  - \p_t u_0  \cdot \nabla \o_0 + \o_0 \cdot \nabla \p_t u_0 + \p_t\o_0 \cdot \nabla u_0. 
\end{split}\Ee

\hide




\begin{definition}[Definition of $\mathfrak B^{\lambda, \kappa}$]
Define for

For  we define 

Let $\mathfrak B^{\lambda, \kappa}$ denote the space of analytic functions $g$ with $[[g]]_{\infty,\lambda,\kappa} <\infty$.

We also define, for $t \in (0, \frac{\lambda_0}{2\gamma})$,

\end{definition}

\begin{definition}

We define, for $t \in (0, \frac{\lambda_0}{2\gamma})$, 
\Be\label{norm_L1T}
\vertiii{g}_1= \sup_{\lambda<\lambda_0-\gamma t} \bigg\{ \sum_{0\leq |\beta|\leq 1} \| 
D^\beta (1+|\nabla_h|) g \|_{1,\lambda} +  (\lambda_0-\lambda-\gamma t)^\alpha\sum_{ |\beta|=2} \| D^\beta (1+|\nabla_h|) g \|_{1,\lambda}  \bigg\}
\Ee
where $ = \sqrt{-\Delta_h}$. 
\end{definition}

 \begin{remark} 
 \end{remark}

Consider analytic functions $g=g(x_1,x_2,x_3)$. 
The following symbols and notations will be used throughout the rest of the section. 

\begin{itemize}
\item 
\item 
\item 
\item 
\end{itemize}

\

\begin{definition}[$L^\infty$ based analytic norms including initial layers] 
\end{definition}


\

\begin{theorem}\label{thm_NS} Let $\lambda_0>0$ and $\o  \in \mathfrak B^{\lambda_0,\kappa}$ with
\Be
 \sum_{0\leq |\beta|\leq 2}  \|D^\beta \p_t^\ell\o_0  \|_{1,\lambda_0}+\sum_{0\leq |\beta|\leq 2} \|D^\beta \p_t^\ell\o_0 \|_{\infty,\lambda_0, \kappa}  <\infty \  \text{ for } \ \ell=0,1,2. \label{initial_norm}
\Ee
Further assume that $\o_0$ satisfies the compatibility conditions 
\Be\label{CC}
\begin{split}
\kappa \eta_0 (\p_{x_3} + \sqrt{- \Delta_h})\o_{0,h} \, |_{x_3=0} &= [\p_{x_3} (-\Delta)^{-1} (-u_0 \cdot \nabla \o_{0, h} + \o_0 \cdot \nabla u_{0,h}) ] \, |_{x_3=0}\\
\o_{0,3} |_{x_3=0} &=0, \quad \p_t\o_{0,3} |_{x_3=0} = 0 . 
\end{split}
\Ee

Then there exists a $\gamma>0$ and a time $T>0$ depending only on $\lambda_0$ and the size of the initial data such that the solution $\o(t)$ to the Navier-Stokes equations \eqref{NS}-\eqref{NSB} exists in $C^1([0,T]; \mathfrak B^{\lambda, \kappa})$ with $\p_t^2\o $ in $C(0,T; \mathfrak B^{\lambda, \kappa t})$ for $0<\lambda<\lambda_0$ satisfying 
\Be\label{norm_bound}
\sup_{t\in [0,T]} \left[\sum_{\ell=0}^2 \vertiii{\p_t^\ell \o(t)}_1 + \sum_{\ell=0}^1 \vertiii{\p_t^\ell \o(t)}_{\infty,\kappa} + \vertiii{\p_t^2\o(t)}_{\infty,\kappa t}   \right]<\infty 
\Ee
\end{theorem}

The following point-wise bounds for $\o$, $u$, and $p$ are the consequences of Theorem \ref{thm_bound}. 
\unhide

\begin{theorem}\label{thm_bound} Let $\lambda_0>0$ and $\o_{in}  \in \mathfrak B^{\lambda_0,\kappa}$ with (\ref{idata}) satisfy 
\Be
 \sum_{0\leq |\beta|\leq 2}  \|D^\beta \p_t^\ell\o_0  \|_{1,\lambda_0}+\sum_{0\leq |\beta|\leq 2} \|D^\beta \p_t^\ell\o_0 \|_{\infty,\lambda_0, \kappa}  <\infty \  \text{ for } \ \ell=0,1,2. \label{initial_norm}
\Ee
Further assume that $\o_{in}=\o_0$ and (\ref{idata}) satisfies the compatibility conditions on $\p\O$  
\Be \label{CC}
\begin{split}
\kappa \eta_0 (\p_{x_3} + \sqrt{- \Delta_h})\o_{0,h}   &= [\p_{x_3} (-\Delta)^{-1} (-u_0 \cdot \nabla \o_{0, h} + \o_0 \cdot \nabla u_{0,h}) ]  , \\ 
\o_{0,3}  =0,  \  \  \p_t\o_{0,3}   = 0 . 
\end{split}
\Ee
Then there exists a constant $\gamma_0>0$ and a time $T>0$ depending only on $\lambda_0$ and the size of the initial data such that the solution $\o(t)$ to the vorticity formulation of the Navier-Stokes equations \eqref{NS}-\eqref{NSB} exists in $C^1([0,T]; \mathfrak B^{\lambda, \kappa})$ with $\p_t^2\o $ in $C(0,T; \mathfrak B^{\lambda, \kappa t})$ for $0<\lambda<\lambda_0$ satisfying  
\Be\label{norm_bound}
\sup_{t\in [0,T]} \left[\sum_{\ell=0}^2 \vertiii{\p_t^\ell \o(t)}_1 + \sum_{\ell=0}^1 \vertiii{\p_t^\ell \o(t)}_{\infty,\kappa} + \vertiii{\p_t^2\o(t)}_{\infty,\kappa t}   \right]<\infty .
\Ee

Furthermore, for each $(t,x)\in [0,T]\times\Omega $, 
\begin{enumerate}
\item (Bounds on the vorticity and its derivatives) $\o(t,x)$ enjoys the following bounds: 
\begin{align}
| \nabla_{h}^i \p_t^\ell \o_h (t,x) |& \lesssim e^{-\bar{\alpha} x_3} \left( 1 + \phi_\kappa (x_3) \right), \ \  | \nabla_{h}^i \p_t^\ell \o_3 (t,x) | \lesssim e^{-\bar{\alpha} x_3} \text{ for } i,\ell =0,1, \label{b1} \\
| \p_t^2 \o_h (t,x) |& \lesssim e^{-\bar{\alpha} x_3} \left( 1 + \phi_\kappa (x_3) + \phi_{\kappa t} (x_3) \right), \ \ | \p_t^2 \o_3 (t,x) | \lesssim e^{-\bar{\alpha} x_3}, \label{b2}\\
| \p_{x_3} \p_t^\ell \o_h (t,x) |& \lesssim {\kappa}^{-1} e^{-\bar{\alpha} x_3}, \ \  | \p_{x_3} \p_t^\ell \o_3 (t,x) | \lesssim  e^{-\bar{\alpha} x_3} \left( 1 + \phi_\kappa (x_3) \right) \text{ for } \ell =0,1. \label{b3}
\end{align}

 \item (Bounds on the velocity and its derivatives) The corresponding velocity field $u(t,x)$ satisfies the following: 
 \begin{align}
  \label{est:u_t}
|\p_t^\ell u (t,x)| &\lesssim 1 \ \ \text{for} \ \ell=0,1,2, \\ 
 \label{est:u1}
\sum_{1 \leq |\beta | \leq 2}  |\nabla ^{\beta } \p_t^\ell u(t,x )|
& \lesssim \big(1+ \phi_\kappa (x_3) + (|\beta|-1) {\kappa}^{-1}  \big)
 e^{-\min (1, \frac{\bar{\alpha}}{2} )x_3}  \ \ \text{for} \ \ell=0,1, \\
 \label{est:u2}
\sum_{ |\beta | =1}  |\nabla ^{\beta } \p_t^2 u(t,x )|
& \lesssim  \big(1+ \phi_\kappa (x_3) + \phi_{\kappa t} (x_3)\big)e^{-\min (1, \frac{\bar{\alpha}}{2} )x_3} . 
\end{align}
Moreover, we have the decay estimate for $\p_t^\ell u$: 
\Be\label{ut}
|\p_t^\ell u|\lesssim \kappa^{-\frac{1}{2}} e^{-\min(1,\frac{\bar\alpha}{2}) x_3}  \ \ \text{for} \ \ell=1,2. 
\Ee

\item (Bounds on the pressure and its derivatives) 
The pressure defined in \eqref{pressure} satisfies the following: 
\begin{align}
|\p_t^\ell p (t,x)|& \lesssim 1 \ \ \text{for} \ \ell=0,1,2,  \label{est:p}\\
\label{est:pdecay}
\sum_{0 \leq |\beta | \leq 1}  |\nabla ^{\beta } \p_t^\ell p(t,x )|
& \lesssim \kappa^{-\frac{1}{2}} 
 e^{-\min (1, \frac{\bar{\alpha}}{2} )x_3}   \ \ \text{for} \ \ell=0,1, \\
\label{est:pt2}
|\p_t^2 p|   &\lesssim ( \kappa^{-\frac12}  +\phi_{\kappa t}(x_3) )e^{-\min (1, \frac{\bar{\alpha}}{2} )x_3}  . 
\end{align}
\end{enumerate}
\end{theorem}

\begin{remark} For simplicity of the presentation, we have taken the analytic data with the same analyticity radius in $x_1$, $x_2$ and  $x_3$ with the exponential decay for large $x_3$. As shown in \cite{KVW, FW}, more general initial data requiring the analyticity only near the boundary can be taken. 
\end{remark}

\begin{remark}
The horizontal vorticity $\o_h$ and the vertical vorticity $\o_3$ obey different boundary conditions \eqref{NSB} which enforce different behaviors near the boundary. This is well-reflected in our $L^\infty$ based norms in \eqref{[]} and \eqref{[[]]}. 
As noted in \cite{FW},  such incompatible behaviors of $\o_h$ and $\o_3$ in 3D are dealt with the $L^1$ based norm \eqref{norm_L1T} which contains one more tangential derivative $(1+|\nabla_h|)$, which is different from   2D analysis \cite{NN2018,KVW}. 
\end{remark}

\begin{remark} We demand the compatibility conditions in \eqref{CC} in order to avoid singular initial-boundary layers for the temporal derivatives of the vorticity. 
If the first two conditions in \eqref{CC} were not satisfied, the initial-boundary layers would occur for the first temporal derivative of the vorticity. For the second temporal derivative, we handle the initial-boundary layer for the horizontal part 
with the initial-boundary layer weight, while for the vertical part we further demand $\p_t\o_{0,3}|_{x_3=0}=0$ in order to rule out a singular initial-boundary layer caused by the Dirichlet boundary condition. This amounts to requiring the second order vanishing condition at the boundary for $\o_{0,3}$, which is satisfied by a large class of $\o_0$. We remark that the first condition of \eqref{CC} is also satisfied by a large class of $\o_0$. In fact, if not, by the result of 
\cite{NN2018}, we can obtain a short time solution $\tilde\o(t)$ to \eqref{NS}-\eqref{NSB} and may reset the initial data by $\o_0 = \tilde\o(t=t_0)$ for sufficiently small $t_0>0$.  
\end{remark}
The proof of Theorem \ref{thm_bound} will be given in Section \ref{sec:5}.


 \hide

{\color{green}[CK note: I don't have a good idea where we should put the main theorem and how. Actually I don't like that main theorem appears in the end of the paper. But since there are many technical notions and conditions in the theorem I tentatively locate it here. Where we should put it?] \bcr JJ: I moved the main theorem to the intro. See how you feel. Another possibility that just occurred to me  is to give an informal statement of the theorem there, something like there exists a large set of initial data of $F=...$ leading to the solvability and convergence (so that we don't have to be specific about the regularity for the data):

\begin{theorem}[Informal statement] We consider a half space in 3D
$ 
\O := \mathbb{T}^2 \times \R_+ \ni (x_1, x_2, x_3)$ 
 where $ \mathbb{T} $ is a periodic interval of $ (-\pi, \pi)$.
Let \Be 
\delta= \sqrt{\e} \ \ \text{and} \ \ 
\delta \lesssim \exp (   -\kappa^{ -\mathfrak{P}}   ) 
 \ \ \text{for some   } \mathfrak{P}>0.  
\Ee
Then there exists a large set of initial data $u_{in}$, $f_{2,in}$ and $f_{R,in}$ 
such that a unique solution $F(t,x,v)$ of the form  \eqref{F_e}   
to (\ref{Boltzmann}) and (\ref{diffuse_BC}) with (\ref{scale}) exists on $ [0,T]$ for some $T>0$ 
and satisfies 
\Be\notag
\sup_{0 \leq t \leq T}\left\|\frac{F (t,x,v)- M_{1, 0, 1} (v)}{\e} -u_E(t,x)\right\|_{L^2(\O)}
\lesssim   \kappa \rightarrow 0   \ \ \text{as} \ \ \e \rightarrow 0.
\Ee
\end{theorem}

\

and perhaps put a remark after the statement about the assumptions and conditions about the data (analyticity, compatibility conditions  etc), and function spaces for $f_R$ etc by referring to later section, and have the precise theorem here as before...
 \ec}

 \unhide 
 
\hide

{\color{red}
\begin{remark}
The condition of $\delta, \e ,\kappa$ in the theorem is indeed the largest $\delta$ (and $\sqrt{\e}$) we can choose. Any smaller $\delta$ than $\exp\Big(  \frac{  - \mathfrak{C}  T}{ \kappa^{ \mathfrak{P}} }\Big)$ (which means $\delta$ decaying faster than $\exp\Big(  \frac{  - \mathfrak{C}  T}{ \kappa^{ \mathfrak{P}} }\Big)$ as $\kappa \downarrow 0$) will produce the same result. 
In terms of (\ref{scale}) the relation (\ref{choice:delta}) implies that  
\Be\label{scale_1}
\mathpzc{St}=\e = \mathpzc{Ma}  \  \ \text{and} \ \ 
\mathpzc{Kn} \gtrsim    \frac{T^{\mathfrak{P}} }{(\ln \e^{-1})^{\mathfrak{P}} }.
\Ee 
\end{remark} }

\unhide

\hide
Now we look hierarchies closely. We refer to Section 3 for the full details. Let us first consider an $\frac{1}{\e \delta}$-hierarchy (\ref{eqtn_f_2}). In the energy estimate any nonvanishing term of $(\ref{eqtn_f_2})$ would cause an unbounded term. We remove the term entirely by solving an equation $(\ref{eqtn_f_2})=0$. By the Fredholm alternative, an inverse map 
\Be\label{L_inverse}
L^{-1}: \mathcal{N}^{\perp} \rightarrow \mathcal{N}^{\perp},  \ \text{ where }   \  \mathcal{N}^\perp  \   \text{stands an $L^2$-orthogonal complement of } \mathcal{N},
\Ee
 is well-defined and hence 
the solvability condition is given by
\Be\label{Comp1}
\frac{ \e^{-1} (v- \e u )\cdot \nabla_x \mu}{   \sqrt{\mu}}
=\sum_{\ell, m=1}^3 \p_\ell u_m\varphi_\ell  \varphi_m  \sqrt{\mu}
  \in \mathcal{N}^\perp.
\Ee
 This condition indeed implies the incompressible condition (\ref{incomp}). Once (\ref{incomp}) hold, we have $\sum_{\ell, m=1}^3 \p_\ell u_m\varphi_\ell  \varphi_m  \sqrt{\mu}= \sum_{\ell, m=1}^3 \p_\ell u_m
 (\varphi_\ell  \varphi_m - \frac{|v-\e u|^2}{3} \delta_{\ell m}
 )
  \sqrt{\mu}$. Now we solve $(\ref{eqtn_f_2})=0$ by 
setting
 \Be\label{f_2} 
 (\mathbf{I}-\mathbf{P}) f_2 \hide = \kappa  L^{-1} (\mathbf{I} - \mathbf{P})
   \bigg(\frac{\e^{-1} (v- \e u )\cdot \nabla_x \mu }{\sqrt{\mu}}
 \bigg) 

  \sum_{\ell, m=1}^3 \varphi_\ell \p_\ell u_m \varphi_m  \sqrt{\mu} 
\unhide
=  \kappa \sum_{\ell,m=1}^3  A_{\ell m} \p_\ell u_m
 \ \ \text{with} \ \ A_{\ell m}:=L^{-1} \Big( \varphi_\ell \varphi_m \sqrt{\mu} - \frac{|v- \e u |^2}{3} \delta_{\ell m}\sqrt{\mu} \Big).
 \Ee
  \hide where we define  such that 
 \Be\label{AB}
 A_{ij}= L^{-1} \big( \varphi_i \varphi_j \sqrt{\mu} - \frac{|v- \e u |^2}{3} \delta_{ij}\sqrt{\mu} \big)  
  \Ee 
 while coefficients of $\mathbf{P}f_2 =\{ \tilde{\rho}  \varphi_0+ \sum_{\ell=1}^3 \tilde{u}  \cdot  \varphi_\ell +\tilde{\theta}  \varphi_4  \}\sqrt{\mu}$ haven't been determined yet. \unhide

Now we move to an $\frac{1}{\delta}$-hierarchy (\ref{eqtn_f_3}). Hydrodynamic part of (\ref{eqtn_f_3}), unless it vanishes, would induce an unbounded term in the energy estimate. We expand $\delta \times (\ref{eqtn_f_3})$, using (\ref{mu_e}) and (\ref{f_2}), 
\begin{align}
 (v- \e u) \cdot   (\p_t u + u \cdot \nabla_x u ) \sqrt{\mu}  
 +  (v- \e u) \cdot \nabla_x   \mathbf{P} f_2
+ \kappa  (v- \e u) \cdot  \nabla_x
\Big( \sum_{\ell,m=1}^3 A_{\ell m} \p_\ell u_m\Big)
 . 
 \label{H_NS}
\end{align}
The leading order term of the last term of (\ref{H_NS}) contributes the following to the hydrodynamic part of (\ref{H_NS}) as
\begin{align}
\kappa  \sum_{\ell,m,k=1}^3 \big\langle \varphi_i \varphi_k  \sqrt{\mu},   
A_{\ell m}
\big\rangle  \p_k \p_\ell u_{  m}
= \kappa  \sum_{\ell,m,k=1}^3 \Big\langle
\big( \varphi_i  \varphi_k  -\frac{|v-\e u|^2}{3} \delta_{i k}
\big)\sqrt{\mu}
,   
A_{\ell m}
\Big\rangle  \p_k \p_\ell u_{  m}
= \kappa    \sum_{\ell ,  m, k=1}^3 \left\langle 
L A_{i k}
,  A_{\ell m}
\right\rangle  \p_{k} \p_\ell u_{  m}
,
 \label{f_2:main}
\end{align}
where we have used the fact $A_{\ell m } \in \mathcal{N}^\perp$ and $\frac{|v-\e u|^2}{3} \sqrt{\mu} \in \mathcal{N}$ as the first step and the definition of $A_{ik}$ at the last step.  It is well-known (e.g. Lemma 4.4 in \cite{BGL93}) that for some constant $\eta_0 >0$
\Be
\langle LA_{\ell k }, A_{ij}\rangle = \eta_0   ( \delta_{ik} \delta_{j\ell} + \delta_{i\ell} \delta_{jk} ) - \frac{2}{3} \eta_0   \delta_{ij} \delta_{k\ell}.\label{eta_0}
\Ee
Therefore we deduce (\ref{f_2:main}) vanishes for $\ell=0,4$, while we can observe the $\kappa \eta_0$-viscosity term in (\ref{NS_k}) as 
 \begin{align}
(\ref{f_2:main}) = \kappa \eta_0 \sum_{i,j,k}   \{  ( \delta_{ik} \delta_{j\ell} + \delta_{i\ell} \delta_{jk} ) - \frac{2}{3}  \delta_{ij} \delta_{k\ell}\} \p_{k} \p_i u_{  j} 
 =\kappa\eta_0 \{\Delta u_{\ell}  -   \p_\ell    \nabla \cdot  u  -  \frac{2}{3}   \p_\ell  \nabla \cdot  u   \}= \kappa\eta_0\Delta u_{ \ell}
 \ \ \text{for} \ \ell=1,2,3
. \label{Delta_u}
 \end{align}
Here we have used the incompressible condition (\ref{incomp}) at the last step. 
 On the other hand, a leading order term of hydrodynamic part of $(v- \e u) \cdot \nabla_x   \mathbf{P} f_2$ contributes the pressure term of (\ref{NS_k}) by choosing a special form of $\mathbf{P}f_2$ as in (\ref{Pf_2}). Therefore whole leading order terms of the hydrodynamic part in (\ref{eqtn_f_3}) do vanish by solving the Navier-Stokes equation (\ref{NS_k}), and (\ref{incomp})!

\subsection{Temporal derivatives of $f_R$ 
}  A major analytic challenge is estimating a contribution of the nonlinear term uniformly. Most singular contribution can be bounded using an integrability gain of the classical average lemma $L^2_t H^{1/2}_x \subset L^2_t L^3_x$ in $3D$, as 
\Be
\left|\int_0^t\iint_{\O \times \R^3}\frac{\delta}{\e\kappa} \Gamma( \mathbf{P}f_R, \mathbf{P}f_R) (\mathbf{I} - \mathbf{P}) f_R\right| \lesssim 
\delta \kappa^{-\frac{3}{2}} 
  \|\kappa^{ \frac{1}{2}}  {P} f_R \|_{L^\infty_tL^6_{x }} 
  \|\kappa^{ \frac{1}{2}}  {P} f_R \|_{L^2_tL^3_{x }} 
 \| \e^{-1} \kappa^{-\frac{1}{2}}\sqrt{\nu} (\mathbf{I} - \mathbf{P}) f_R \|_{L^2_{t,x,v}}  
.
\label{nonlinear_1}
\Ee 
We obtain such a higher integrability in $L^6_x$ through a priori $L^2$-bound of $\p_t f_R$ and the dissipation by viewing the equation of $f_R$ as 
\Be\label{vf_x}
\kappa^{1/2}v\cdot \nabla_x \mathbf{P}  f_R= - \e \kappa^{1/2}\p_t f_R - \frac{1}{\e \kappa^{1/2}} L(\mathbf{I} - \mathbf{P} )f_R
  + g,
\Ee
with some prescribed $g$. A simple derivatives count and the Gagliardo-Nirenberg-Sobolev inequality $\dot{W}^{1,2} (\R^3)\subset L^6  (\R^3)$ suggests an $L_x^6$-bound of $\kappa^{1/2} P f_R$. Using a recent test function method of \cite{EGKM2}, indeed we can ``invert'' an operator $v\cdot \nabla_x \mathbf{P}$ and deduce a crucial estimate 
  \Be\label{L6}
  \| \kappa^{1/2} Pf_R(t) \|_{L^6_x} \lesssim \e \kappa^{1/2} \| \p_t f_R(t) \|_{L^2_{x,v}}
  + \| \e^{-1} \kappa^{-1/2}    (\mathbf{I} - \mathbf{P}) f_R (t)  \|_{L^{2}_{x,v}} 
  + \| g(t)  \|_{L^2_{x,v}} + l.o.t.
  \Ee
This estimate (\ref{L6}) will provide the desired bound for (\ref{nonlinear_1}) only by taking a supremum in $t$ to the dissipation. We use a version of Sobolev embedding $L^\infty_t \subset H^1_t$ to control $L^\infty_t L^2_{x,v}$-norm of the dissipation provided $\e^{-1} \kappa^{-1/2}  \p_t  [(\mathbf{I} - \mathbf{P})  f_R  ]  \in L^2_t L^2_{x,v}$ holds. This is a critical point in which a temporal derivatives get involved in our analysis!

An energy estimate for $\p_t f_R$ requires higher regularity for solutions of the fluid equation (\ref{NS_k})-(\ref{noslip}). Naturally this energy estimate for $\p_t f_R$ works similarly as one for $f_R$ provided higher regularity of the fluid since the equation for $\p_t f_R$ is basically quasi-linear. More precisely the corresponding estimate to (\ref{nonlinear_1}) only requires $\kappa^{1/2}\p_t Pf_R \in L^2_t L^3_x$ and $\e^{-1} \kappa^{-\frac{1}{2}}\sqrt{\nu}\p_t (\mathbf{I} - \mathbf{P}) f_R \in L^2_{t,x,v}$ but not necessarily $\kappa^{1/2}\p_t Pf_R \in L^\infty_t L^6_x$. Finally for a uniform bound of (\ref{nonlinear_1}) we balance the parameters as (\ref{scale}) with
\Be\label{delta}
\delta=\delta  (\e) \ll  \kappa(\e)^{3/2} \ \ \text{for all } \ \e>0.
\Ee

It turns out that $\p_t^2 \nabla_x u$ and $\p_t \nabla_x^2 u$ get involved in the energy estimate for $\p_t f_R$.

\subsection{Higher Regularity of Navier-Stokes equation in the Inviscid Limit}\label{sec:NSlimit}

\subsection{Main Theorem}

\unhide

 \subsection{Main Theorem}\label{sec:MT}

Now we present the full statement of the main theorems of this paper:
 \begin{theorem}[\textbf{Kinetic approximation of Navier-Stokes in large Reynolds numbers}]\label{main_theorem}
 We consider a half space $\O$ in 3D as in (\ref{domain}). Suppose an initial datum of the Navier-Stokes flow $u_{in}$ is divergence-free $\nabla_x \cdot u_{in}=0$ in $\O$ and the corresponding initial vorticity $\o_{in}= \nabla_x \times u_{in}$ belongs to the real analytic space $\mathfrak B^{\lambda_0,\kappa}$ of (\ref{[]}) for some $\lambda_0>0$ such that (\ref{initial_norm}) holds.
Further we assume that $\o_{in}$ satisfies the compatibility conditions (\ref{CC}) on $\p\O$.
 Then there exists a unique real analytic solution $(u(t,x), \nabla_x p(t,x))$ to \eqref{NS_k}-\eqref{noslip} in 
$[0,T] \times \O$, while $T>0$ only depends on $\lambda_0$ and the size of the initial data as in (\ref{initial_norm}).

\hide

 for $\lambda_0>0$ and satisfies (\ref{initial_norm}), while 
where $(\o_0, \p_t \o_0, \p_t^2 \o_0)$ and $(u_0,\p_t u_0)$ defined through the equation as in (\ref{idata}). Further assume that $\o_0$ satisfies the compatibility conditions (\ref{CC}).   
\unhide
Choosing a pressure $p(t,x)$ such that $p(t,x)\rightarrow 0$ as $x_3 \uparrow \infty$, we set the local Maxwellian and the second order correction $f_2$ as
\Be  \notag
\mu := M_{1,\e u, 1} =\f{1}{(2\pi  )^{ \f{3}{2}}}\exp \left\{- \frac{|v-\e u|^2}{2}\right\},\ \
f_2:=  \mathbf{P}f_2 + ( \mathbf{I}-\mathbf{P})f_2  =  p \varphi_0 \sqrt{\mu}+ (\ref{f_2}).
 \Ee
For given such $T>0$, let us choose $\e$ and $\kappa$ in the relation of (\ref{choice:delta}).
\hide
, for some $\mathfrak{C}\gg 1$, 
 \Be\label{choice:delta:intro}
\sqrt{\e}=\exp\Big(  \frac{  - \mathfrak{C}  T}{ \kappa^{1/2} }\Big) 
 . 
\Ee
\unhide


Assume that an initial datum for the remainder $f_{R,in }$ satisfies (\ref{initial_EF}) for some $p<3$ and $|p-3|\ll 1$.
  Then we construct a unique solution $f_R(t,x,v)$ of the form of 
 \Be\notag
 F  = M_{1, \e u  ,1 } + \e^2 f_{2} 
  + \e ^{3/2} f_{R}  \ \ \text{in} \ \ [0,T] \times \O \times \R^3,
 \Ee
which solves the Boltzmann equation (\ref{Boltzmann}) and the diffuse reflection boundary condition (\ref{diffuse_BC}) with the scale of (\ref{scale}) and (\ref{choice:delta}), and satisfies the initial condition $F|_{t=0}= M_{1, \e u _{in} ,1 } + \e^2 f_{2} |_{t=0}
  + \e ^{3/2} f_{R,in}$. 


Moreover we derive that, for each $\e$ and $\kappa$ of (\ref{choice:delta}), 
\Be\label{approx}
\sup_{0 \leq t \leq T}\left\|  \frac{F (t,x,v)- M_{1, \e u(t,x), 1} (v)}{\e  \sqrt{M_{1, \e u(t,x), 1} (v)}
}  \right\|_{L^2(\O \times \R^3)}  \lesssim 
\exp\Big(  \frac{  - \mathfrak{C}  T}{ 2\kappa^{1/2} }\Big)   \ \ \text{for} \ \ \kappa \ll1. 
\Ee

  \end{theorem}

  \hide

We give a proof of Theorem \ref{main_theorem}, as a direct consequence of theorems following in this section,
while they are deliberately delayed to be stated since some of notions are complex:
\unhide
\begin{proof}
The existence of the Navier-Stokes solutions follows from Theorem \ref{thm_bound}. For the remaining assertions, we note that  
all  the estimates (\ref{est:u_t})-(\ref{ut}) of Theorem \ref{thm_bound} ensure the conditions of Theorem \ref{main_theorem:conditional} with $\mathfrak{P}=\frac{1}{2}$. Therefore the conclusion follows  directly as a consequence of Theorem \ref{main_theorem:conditional} and Theorem \ref{thm_bound}. 
\end{proof}

 The incompressible Euler limit follows as a byproduct of the main theorem:
\hide  
  As a direct consequence of the main theorems, 
we prove the incompressible Euler limit in the vanishing viscosity limit:\unhide
  
  \begin{corollary}[\textbf{Hydrodynamic limit toward the incompressible Euler equation}]\label{Cor_EL}
  Let $u_E(t,x)$ be a (unique) solution of the incompressible Euler equations (\ref{Euler})-(\ref{no-pen}) with the initial condition $u_E|_{t=0}= u_{in}$ in $\O$. Then 
   \Be\notag
\sup_{0 \leq t \leq T}\left\|
\frac{F (t,x,v)- M_{1, \e u_E (t,x), 1} (v)}{\e 
 (1+|v|)^2\sqrt{M_{1,0,1}(v) }
}  \right\|_{L^2(\O \times \R^3 )}   \longrightarrow 0   \ \ \text{as} \ \ \e \downarrow 0.
\Ee
  \end{corollary}
\begin{proof}Note that 
\[
  F (t,x,v)- M_{1, \e u_E(t,x), 1} (v)    =   \big[F (t,x,v)- M_{1, \e u(t,x), 1} (v)\big] +\big[M_{1, \e u(t,x), 1} (v) - M_{1, \e u_E(t,x), 1} (v)\big].
\]
The first term can be bounded as in (\ref{approx}). We bound the second term by an expansion: 
\Be\begin{split}\notag
|u(t,x)-u_E(t,x)|
\int^\e_0 |(v-\e u_E) + a(u_E-u)| e^{- \frac{|(v-\e u_E) + a(u_E-u)|^2}{2}} \dd a .
\end{split}\Ee
Note that $\|\e  u \|_{L^\infty} \ll1 $ and $\|\e u_E\|_{L^\infty}\ll 1$ from Theorem \ref{thm_bound}. Then we conclude that the second term converges to $0$ as $\kappa \downarrow 0$ from Theorem \ref{thm_bound} and the famous Kato's condition for vanishing viscosity limit in \cite{kato}.\end{proof}

\section{Hilbert expansion around a local Maxwellian and Source terms}\label{sec:3}

 In this section we complete the Hilbert expansion along with the outline of the introduction. As a result we prove 
 \begin{proposition}\label{prop:Hilbert}
 Suppose that $F$ of (\ref{F_e}), with a free parameter $\delta$, solve (\ref{Boltzmann}) and (\ref{diffuse_BC}) with (\ref{scale}) and that $(u,p)$ solves (\ref{NS_k})-(\ref{noslip}).  We choose a hydrodynamic part $f_2$ as 
  \Be\label{Pf_2}
   \mathbf{P}f_2 =  p \varphi_0 \sqrt{\mu}, 
\Ee
with the pressure $p$ of the Navier-Stokes flow in (\ref{NS_k}), and $(\mathbf{I}- \mathbf{P})f_2$ has been given in (\ref{f_2}). Then $f_{R}$ in (\ref{F_e}) satisfies that  
 \Be
\Big[    \p_t 
  + \frac{1}{\e} v\cdot \nabla_x 
    +    \frac{1}{ \e^2\kappa} L
  \Big] f_R
=    \frac{2}{\kappa} \Gamma({f_2}, f_R)
+    \frac{   \delta }{\e\kappa}\Gamma(f_R, f_R)
-  \frac{( \p_t + 
\e^{-1} v\cdot \nabla_x) \sqrt{\mu}}{\sqrt{\mu}} f_{R}
  + 
(\mathbf{I}- \mathbf{P})\mathfrak{R}_1 + \mathfrak{R}_2,   \label{eqtn_fR}
\Ee
\Be
\begin{split}
&  \Big[    \p_t 
  + \frac{1}{\e} v\cdot \nabla_x 
    +    \frac{1}{ \e^2\kappa} L
  \Big] \p_t f_R \\
=&   - \frac{1}{\e^2 \kappa } L_t (\mathbf{I} - \mathbf{P}) f_R +\frac{1}{\e^2 \kappa } L(\mathbf{P}_t f_R)+    \frac{ 2 \delta }{\e\kappa}\Gamma(f_R,\p_t  f_R)
+ \frac{2}{\kappa} \Gamma({f_2},\p_t  f_R) 
 +  \frac{2}{\kappa} \Gamma(\p_t {f_2}, f_R) \\
&
+  \frac{2}{\kappa} \Gamma_t(  {f_2}, f_R)+ \frac{\delta}{\e \kappa } \Gamma_t (f_R,f_R)
\\
&-
 \frac{( \p_t +
  \e^{-1} v\cdot \nabla_x) \sqrt{\mu}}{\sqrt{\mu}} \p_{t} f_{R}-
\p_t \Big( \frac{( \p_t +
  \e^{-1} v\cdot \nabla_x) \sqrt{\mu}}{\sqrt{\mu}} \Big) f_{R} 
  \\
  & 
  +(\mathbf{I} -  \mathbf{P})\mathfrak{R}_3 + \mathfrak{R}_4 ,
\label{eqtn_fR_t}
 \end{split}\Ee 
where the commutators $L_t$, $\mathbf{P}_t$ and $\Gamma_t$ are given in (\ref{def:L_t}), while 
 \begin{align}
 e^{ \varrho |v-\e u|^{2}} |(\mathbf{I} - \mathbf{P}) \mathfrak{R}_{1} (t,x,v)|
&\lesssim 
  \frac{1}{\delta}
 \kappa |\nabla _x^2 u|
%
,\label{est:R1}\\
 e^{ \varrho |v-\e u|^{2}} |\mathfrak{R}_{2} (t,x,v)| 
 &\lesssim   
 \frac{\e}{\delta } (|p| + \kappa |\nabla_x u|)|\nabla_x u|
 + \frac{\e}{\delta} (|\p_t p| + \kappa |\nabla_x u|)
 \notag
 \\
 &
 + \frac{\e \kappa}{\delta} (|\nabla_x \p_t u| + |u| |\nabla_x^2 u|) \label{est:R2} \\
 &
 + \frac{\e^2}{\delta} 
 (|p| + \kappa | \nabla_x u|)
 (|\p_t u| + |u| |\nabla_x u|)
 %
,\notag\\
 e^{ \varrho |v-\e u|^{2}} |(\mathbf{I} - \mathbf{P}) \mathfrak{R}_{3} (t,x,v)|
&\lesssim 
\frac{\kappa}{\delta}  |\nabla_x^2 \p_t u| 
,\label{est:R3} 
\end{align}
 
\hide
 \Be\begin{split}
 e^{ \varrho |v-\e u|^{2}}|\mathfrak{R}_4 (t,x,v)| \lesssim& \  \frac{\e}{\delta \kappa}
 \{1+ 
 |\p_t^2 p| + |\nabla_x \p_t^2 u| 
 \}\notag
\\
&  \times 
 \mathfrak{q} (
|p|, |\nabla_x p|, |\p_t p|,|\nabla_x \p_t p | , |u| , |\nabla_x u|, |\p_t u|, |\nabla_x \p_t u|, |\nabla_x^2 u|, |\nabla_x^2 \p_t u|, 
|\tilde{u}|, |\nabla_x \tilde{u}|,|\p_t \tilde{u}| ,|\nabla_x \p_t \tilde{u}|, |\p_t ^2 \tilde{u}|
) 
,
\end{split}
\Ee\unhide
\Be\begin{split}\label{est:R4}
& e^{ \varrho |v-\e u|^{2}}|\mathfrak{R}_4 (t,x,v)| \\
&\lesssim
 \frac{\e}{\delta}  
|\p_t ^2 p|  +
\frac{\e \kappa}{\delta}   |\nabla_x \p_t^2 u| + \frac{\e}{\delta}|\nabla_x \p_t p||u| +
 \frac{\e\kappa }{\delta}|u| |\nabla_x^2 \p_t u|
+
\frac{\e \kappa}{\delta} (1+ \e \kappa |u|)|\p_t u ||\nabla_x^2 u|
\\
&
+\frac{\e}{\delta}
\{
(1+  |u|)(|p|+ \kappa |\nabla_x u|)
 +\kappa \e |\p_t u|
\}  
|\nabla_x \p_t u|+ \frac{\e^2}{\delta} \{|p|+\kappa |\nabla_x u|\} |\p_t^2 u| \\
&
+ 
\frac{\e}{\delta}
\{
(|u| + \e |p| + \e^2 |p||u|)|\p_t u|
+ (1+ \e |u|) |\p_t p|
\}
|\nabla_x u| 
+\frac{\e^2\kappa}{\delta} (1+ \e |u|) |\p_t u |
|\nabla_x u|^2
\\
&+ \frac{\e}{\delta}
\{
|\p_t u| + |\nabla_x p| + \e|\p_t p|
+ \frac{\e}{\kappa} (|p|^2 + \kappa |u| |\nabla_x p| + \e \kappa |\p_t u| |p|)
\}
|\p_t u| 
 .
\end{split}\Ee


\hide
\Be\begin{split}
\frac{\e\kappa }{\delta}|u| |\nabla_x^2 \p_t u|
+
\frac{\e \kappa}{\delta} (1+ \e \kappa |u|)|\p_t u ||\nabla_x^2 u|
+
\frac{\e \kappa}{\delta}   |\nabla_x \p_t^2 u|  
+\frac{\e}{\delta}
\{
(1+  |u|)(|p|+ \kappa |\nabla_x u|)
 +\kappa \e |\p_t u|
\}  
|\nabla_x \p_t u|\\
+ \frac{\e^2}{\delta} \{|p|+\kappa |\nabla_x u|\} |\p_t^2 u| 
+ 
\frac{\e}{\delta}
\{
(|u| + \e |p| + \e^2 |p||u|)|\p_t u|
+ (1+ \e |u|) |\p_t p|
\}
|\nabla_x u| 
+\frac{\e^2\kappa}{\delta} (1+ \e |u|) |\p_t u |
|\nabla_x u|^2
\\
+ \frac{\e}{\delta}
\{
|\p_t u| + |\nabla_x p| + \e|\p_t p|
+ \frac{\e}{\kappa} (|p|^2 + \kappa |u| |\nabla_x p| + \e \kappa |\p_t u| |p|)
\}
|\p_t u| 
+ \frac{\e}{\delta}|\nabla_x \p_t p||u| +\frac{\e}{\delta}  
|\p_t ^2 p|  
 + \frac{\e}{\delta \kappa} |p||\p_t p| 
\end{split}\Ee

\Be 
\begin{split}
&   \frac{\e}{\delta \kappa}\{|p|   + \kappa |\nabla_x u| \} 
 \{
 |\p_t p|  + \e |\p_t u|   |p|  
 + \kappa (|\nabla_x \p_t  u| + \e |\p_t u| |\nabla_x u|)
 \}
  \nu(v) e^{-\varrho |v-\e u|^2}\\
&+\frac{\e}{\delta} \{
|\p_t ^2 p| + \kappa |\nabla_x \p_t^2 u|
\}\\
&+ \frac{\e}{\delta}\{ |p|+ 
\kappa |\nabla_x u|  
\} \{
|\nabla_x \p_t  u|
 +\e 
(
 |\p_t^2 u| + |u|| \nabla_x \p_t u|+| \p_t u || \nabla_x u|
)
+ \e^2|\p_t u | (|\p_t u| + |u| | \nabla_x u|)
 \}
\\
&+ \frac{\e}{\delta}\{
|\p_t p| + \kappa |\p_t \nabla_x u| + \e |\p_t u | ( |p|  +   \kappa   |\nabla_x u|)
\} \{ |\nabla_x u| + \e |\p_t u| + \e |u| |\nabla_x u|\}
\\ 
 & + \frac{\e^2}{\delta \kappa}|\p_t u|\{|p|   + \kappa |\nabla_x u| \}^2   
\\
 &+  \frac{\e}{\delta}|\p_t u| \big\{
|\p_t u| + |u| |\nabla_x u|  + |\nabla_x p| 
+ \kappa |\nabla_x ^2 u|+  \e |\nabla_x u|
|p| 
+ \kappa \e |\nabla_x u| ^2 
\big\}  
\\
&+  \frac{\kappa \e}{\delta} 
\{
 (|\nabla_x \p_t u| + \e |\p_t u||\nabla_x u| )
|\nabla_x u| + \e | \p_t u||\nabla_x^2 u| 
\}
e^{- \varrho |v-\e u|^2}\\
&+ \frac{\e}{\delta} \bigg\{
 \Big\{   
  |\p_t u| |\nabla_x p|+ | u| |\nabla_x \p_t p| 
+ \e 
| \p_t p|
  \{ |\p_t u| + |u| |\nabla_x u|\} \\
& \  \ + \e   |  p|
\{ |\p_t^2 u| + |\p_t u| |\nabla_x u|+  |u| |\nabla_x \p_t  u|\}+ \e |\p_t u| 
 \Big\{
 |\p_t p| + |u| |\nabla_x p| 
+ \e 
|p|
\{ |\p_t u| + |u| |\nabla_x u|\}
\Big\}
\Big\}   \langle v-\e u\rangle ^{ 2}  
\bigg\}\\
&+ \frac{\e\kappa}{\delta}  
 \bigg\{   |\p_t u| |\nabla_x^2 u|+ |u| |\nabla_x^2 \p_t  u| 
+ \e (|\p_t u| + |u| |\nabla_x u|) |\nabla_x \p_t u|\\
& \ \ \ \ + \e (|\p_t^2 u| + |\p_t u| |\nabla_x u|+ |u| |\nabla_x \p_t u|) |\nabla_x u|
+ \e\kappa |\p_t u| 
\Big(
 | \nabla_x \p_t u| + |u| |\nabla_x^2 u| 
+ \e(|\p_t u| + |u| |\nabla_x u|) |\nabla_x u|
\Big)
 \bigg\}    
  \end{split}
\Ee

\unhide

 At the boundary $f_R$ and $\p_t f_R$ satisfy 
 \Be \label{bdry_fR}
\begin{split}
f_R(t,x,v)|_{\gamma_-} & =    P_{\gamma_+}   f_R(t,x,v)- \frac{\e}{\delta} (1- P_{\gamma_+}) (\mathbf{I}- \mathbf{P})f_2(t,x,v), 
 \end{split} \Ee 
 \Be\label{bdry_fR_t}
 \begin{split}
 \p_t f_R |_{\gamma_-} &= P_{\gamma_+} \p_t f_R- \frac{\e}{\delta}  (1-P_{\gamma_+})  \p_t (\mathbf{I} - \mathbf{P})f_2 
+r_{\gamma_+} (f_R)- \frac{\e}{\delta}r_{\gamma_+}  ((\mathbf{I} - \mathbf{P}) f_2),\\
r _{\gamma_+}(g) &:= \p_t \sqrt{c_\mu \mu(v)} \int_{n(x) \cdot \mathfrak{v}>0} g \sqrt{c_\mu \mu(\mathfrak{v})} n(x) \cdot \mathfrak{v} \dd \mathfrak{v} \\
& \ \ \ 
+ \sqrt{c_\mu \mu(v)} \int_{n(x) \cdot \mathfrak{v}>0}g  \p_t \sqrt{c_\mu \mu(\mathfrak{v})} n(x) \cdot \mathfrak{v} \dd \mathfrak{v}.
 \end{split}
 \Ee
 In addition, 
 \begin{align}
e^{  \varrho |v-\e u|^2}|f_2 (t,x,v)| \lesssim    & \ 
|p(t,x)|
+ \kappa |\nabla_x u(t,x)|,
\label{est:f2}\\ 
e^{  \varrho |v-\e u|^2}|\p_t f_2 (t,x,v)| \lesssim & \ 
  |\p_t p|
  + \kappa ( |\nabla_x\p _t u| + 
 \e |\p_t u| |\nabla_x u|)   
 + \e |\p_t u|
 |p| 
 ,
\label{est:f2_t}\\
\langle  v-\e u \rangle^{-2}  \Big| \frac{(\p_t + \e^{-1} v\cdot \nabla_x) \sqrt{\mu}}{\sqrt{\mu}} 
\Big| 
  \lesssim & \  |\nabla_x u| + 
  \underbrace{\e |\p_t u| + \e |u| |\nabla_x u|}_{(\ref{transp:mu})_*}
  ,\label{transp:mu}
  \end{align}
  \Be
  \begin{split}\label{transp:mu_t}
&\langle  v-\e u \rangle^{-2}\Big| 
\p_t \left(   \frac{(\p_t + \e^{-1} v\cdot \nabla_x) \sqrt{\mu}}{\sqrt{\mu}} \right)\Big|\\
&   \lesssim   \ 
 |\nabla_x \p_t  u|
 +
\underbrace{ \e 
 \{
  |\p_t^2 u| + |u|| \nabla_x \p_t u|+| \p_t u || \nabla_x u|
 \}
+ \e^2|\p_t u | (|\p_t u| + |u| | \nabla_x u|)}_{(\ref{transp:mu_t})_*}
 . 
\end{split}\Ee

 \hide
Here, from (\ref{est:f_2}), (\ref{r2}), (\ref{r3}), (\ref{r_4}), (\ref{r_5}), 
\Be \begin{split}\label{PR_1}
 \mathbf{P}\mathfrak{R}_1&:= \frac{\e}{2 \delta} \sum_{\ell=1}^3 \{- \tilde{u} \cdot \p_\ell u + \p_\ell u_\ell \tilde{u}_\ell \} \varphi_\ell \sqrt{\mu}  
 =O( \frac{\e}{\delta \kappa^{\frac{1}{2}}}) \langle v-\e u \rangle \sqrt{\mu} |\tilde{u}| | \kappa^{\frac{1}{2}} \nabla u|,
 \end{split}\Ee
 \Be\label{I-PR_1}
  \begin{split}
(\mathbf{I}- \mathbf{P})\mathfrak{R}_1&:=    \frac{\e }{\delta\kappa} \Gamma({f_2}, f_2)
 + \frac{\kappa}{\delta} (\mathbf{I}- \mathbf{P}) r_1 + \frac{1}{\delta} (\ref{r2}) + \frac{1}{\delta} (\ref{r3})  \\
 &= O(\frac{1}{\delta \kappa}) \langle v-\e u \rangle ^4 \mu^{\frac{1}{2}}
 \Big\{
 |\kappa \nabla^2 u|  + \kappa^{\frac{1}{2}} |\kappa^{\frac{1}{2}} \nabla u| 
 \big[ \e |(\tilde{\rho}, \tilde{u}, \tilde{\theta})| + \e \kappa^{-\frac{1}{2}}|\kappa^{\frac{1}{2}} \nabla_x u |\big]
 + \e |(\tilde{\rho}, \tilde{u}, \tilde{\theta})|^2 + \kappa^{\frac{1}{2}} 
 |\kappa^{\frac{1}{2}}  \nabla(\tilde{\rho}, \tilde{u}, \tilde{\theta}) |
 \Big\} ,
 \end{split}\Ee
  \Be\label{R2}
 \begin{split}
 \mathfrak{R}_2:= &\frac{\kappa}{\delta} r_1 + (\ref{eqtn_f_4}) + (\ref{eqtn_f_5})\\
 = & O(\frac{\e}{\delta \kappa^{\frac{1}{2}}})
 \langle v-\e u\rangle^6 \sqrt{\mu} \bigg\{
  \kappa^{\frac{1}{2}}(1+ |u|)\Big[ | \kappa \p_t \nabla u| + | \kappa \nabla^2 u| \Big]
  +  (1+ |u|) \Big[ |\kappa^{\frac{1}{2}} \p_t( \tilde{\rho}, \tilde{u}, \tilde{\theta})| + |\kappa^{\frac{1}{2}} \nabla( \tilde{\rho}, \tilde{u}, \tilde{\theta})|   \Big]
  \\
  & \ \ \ \ \ \ \ \ \ \ \ \ \ \ \ \ \ \ \ \ \ \ \ \ \ \ \  
 + 
 \Big[
  (1+ |u|) |(\tilde{\rho}, \tilde{u}, \tilde{\theta})|
 + (\frac{\e}{\kappa^{\frac{1}{2}}} + \kappa^{\frac{1}{2}}) |\kappa^{\frac{1}{2}} \nabla u |
 \Big]
  \Big[ |\kappa^{\frac{1}{2}}\p_t u| + |\kappa^{\frac{1}{2}} \nabla u| \Big]\bigg\}.
 \end{split}
 \Ee\unhide
\end{proposition}


\begin{remark} 
We note that due to the choice of (\ref{Pf_2}) we remove a contribution of $p^2$ in $\frac{\e}{\delta \kappa} \Gamma(f_2,f_2)$.
And also we remark that $\mathfrak{R}_4$ is quasi-linear for $\p_t^2p$ and $\nabla_x \p_t^2 u$. 
\end{remark}

\subsection{Derivatives of $A_{ij}$ and Commutators in the local Maxwellian setting  
}\label{sec:2}  First we check properties of $L$ and $\Gamma$ defined in (\ref{L_Gamma}). Recall the notation of the global Maxwellian $\mu_0:= M_{1,0,1}$. It is convenient to define  
\Be\label{L_0}
L_0  f (v) := \frac{-2}{\sqrt{\mu_0 }} Q( \mu_0, \sqrt{\mu_0}  {f}) (v), \ \ 
\Gamma_0 ( {f}, {g}) (v) := \frac{1}{\sqrt{\mu_0 }} Q(\sqrt{\mu_0}  {f}, \sqrt{\mu_0} {g}) (v).
\Ee
For a given $\e u$, we define $\tilde{f}(\cdot) := f(\cdot + \e u)$. Then we have 
\Be\label{L_0=L}
Lf (v+\e u)=L_0 \tilde{f}(v), \ \ \Gamma(f,g) (v+\e u)=\Gamma_0 (\tilde{f}, \tilde{g} )(v).
\Ee
 As in (\ref{basis}) a null space of $L_0$, denoted by $\mathcal{N}_0$, is a subspace of $L^2(\R^3)$ spanned by orthonormal bases $\{ \tilde{\varphi}_i \sqrt{\mu_0 }\}_{i=0}^4$ with
\Be\label{basis_0}
\begin{split}
 \tilde{\varphi}_0 := 1
,   \  \ \  \tilde{\varphi}_i: =   {v_i   } 
 \ \ \text{for} \ i=1,2,3 
,   \  \ \  \tilde{\varphi}_4: =   
( | {v } |^2-3 )/{\sqrt{6}}.
\end{split}\Ee
\hide
We define a hydrodynamic projection $\mathbf{P}$ as an $L_v^2$-projection on $\mathcal{N}$ such as 
\Be\begin{split}
\label{P_0}
\mathbf{\tilde{P}} g:= \sum
( \tilde{P}_j g) \tilde{\varphi}_j \sqrt{\mu_0}, \ \ 
 \tilde{P}_j g:= \langle g ,\tilde{\varphi}_j \sqrt{\mu } \rangle  , \  \text{and}  \ \  
\tilde{P} g:= (\tilde{P_0} g, \tilde{P_1} g, \tilde{P_2} g, \tilde{P_3} g, \tilde{P_4} g ),
\end{split}
\Ee
where $\langle \cdot, \cdot \rangle $ stands an $L^2_v$-inner product.
\unhide 
We denote a projection $\tilde{\mathbf{P}}$ on $\mathcal{N}_0$ as in (\ref{P}). From standard properties of $L_0$ and (\ref{L_0=L}), we can easily deduce the corresponding properties of $L$, namely the null space in (\ref{basis}), the spectral gap estimate in (\ref{s_gap}), and the existence of a unique inverse $L^{-1}: \mathcal{N}^\perp \rightarrow \mathcal{N}^\perp$ in (\ref{L_inverse}) which is defined via $L^{-1}_0: \mathcal{N}_0^\perp \rightarrow  \mathcal{N}_0^\perp$ with the identity
\Be
(L^{-1} f) (v) =( L^{-1}_0 \tilde{f}) (v-\e u).\label{L_inverse}
\Ee
 The inverse enjoys the following bound which turns out useful to prove Lemma \ref{lemma:A}. 
\begin{lemma}\label{lemma:L-1}For $0 < \varrho < \frac{1}{4}$ and $g \in \mathcal{N}_0^\perp$
 \Be\label{map:L-1}
 \|  \nu_0(v)e^{\varrho |v|^2}L_0^{-1} g (v)\|_{L^\infty_v} \lesssim 
 \| e^{\varrho |v|^2} g(v) \|_{L^\infty_v}
+   \| \nu_0(v)^{-1} e^{\varrho |v|^2} g(v) \|_{L^2_v}.
  \Ee 
\end{lemma}
 The proof is based on the well-known decomposition of $L_0= \nu_0 -K_0$ and the compactness of $K_0$: We first recall a standard decomposition 
\Be\label{decom_L0}
\begin{split}
L_0 g(v) =&
 \ \nu_0(v) g (v) - K_0 g(v)\\
:=&
\iint_{\R^3 \times \mathbb{S}^2} |(v-v_*) \cdot \mathfrak{u}| \mu_0(v_*) \dd \mathfrak{u} \dd v_*
g(v)\\
&
-  
\frac{1}{\sqrt{\mu_0 (v)}}
\iint_{\R^3 \times \mathbb{S}^2} |(v-v_*) \cdot \mathfrak{u}|\big\{
\mu_0 (v) \sqrt{\mu_0 (v_*)}  g(v_*)
\\
& \ \ \ \ \  \ \ \ \ \  \ \ \ \ \  \ \ \ \ \  \ \ \ \ \  \ \ \ \ \  
-\mu_0 (v^\prime) \sqrt{\mu_0 (v_*^\prime) } g(v_*^\prime)-\mu_0 (v^\prime_*) \sqrt{\mu_0 (v^\prime) } g(v^\prime)
\big\}
\dd v_*
,
\end{split}
\Ee
where $ \langle v\rangle\lesssim \nu_0 (v) \lesssim  \langle v\rangle$. For (\ref{nu_K}) we have $\nu(v)= \nu_0 (v-\e u)$ and $\mathbf{k}(v,v_*) = \mathbf{k}_0 (v-\e u, v_*-\e u)$.  It is well-known (see (3.50) and (3.52) in \cite{gl}) that one can write $K_0 g (v)= \int_{\R^3} \mathbf{k}_0 (v,v_*) g(v_*) \dd v_*$ such that for some constants $C_1, C_2>0$ 
\Be\label{est:k}
 \mathbf{k}_0(v,v_*) = C_1 |v-v_*| e^{- \frac{|v|^2 + |v_*|^2}{4}}-
 \frac{C_{2}}{|v-v_*|}  e^{- \frac{|v-v_*|^2}{8}
- \frac{1}{8} \frac{(|v|^2 - |v_*|^2)^2}{|v-v_*|^2}}. 
\Ee
 It is convenient to introduce a new notation, for $\vartheta>0$,
\Be\label{def:k}
k_{\vartheta}(v,v_*):= \frac{1}{|v-v_*|}
e^{- \vartheta{|v-v_*|^2} 
-  \vartheta\frac{(|v|^2 - |v_*|^2)^2}{|v-v_*|^2}}.
\Ee
Clearly $|\mathbf{k}_0 (v,v_*)| \lesssim k_{\vartheta} (v,v_*)$ for $0<\vartheta\leq 1/8$.

Standard compactness estimates read as follows: 
\begin{lemma}
 For $0<\varrho<2 \vartheta$ and $C \in \R^3$, there exists $C_{\varrho,\vartheta}>0$ such that 
\Be\label{kw/w}
\Big| {k}_\vartheta (v,v_*) \frac{e^{\varrho |v|^2+ C \cdot v}}{e^{\varrho |v_*|^2 + C \cdot v_*}} \Big| \lesssim \frac{1}{|v-v_*|} e^{-C_{\varrho} \frac{|v-v_*|^2}{2} } 
\ \ \text{for} \ 
0 < \varrho < 2 \vartheta.
\Ee
 Moreover
\begin{equation}
\begin{split}
 \label{est:int_k}
 \int_{\R^3} 
 (1+ |v-v_*|)
k_{\vartheta}(v,v_*)
 \frac{e^{\varrho|v|^2
 + C \cdot v 
 }}{e^{\varrho |v_*|^2  +   C \cdot v_*  
 }}  
 \dd v_*\lesssim_{ \vartheta,\varrho
 } \frac{1}{1+|v|} ,\\ 
 \int_{\R^3} \frac{1}{|v-v_*|
}
k_{\vartheta}(v,v_*)
 \frac{e^{\varrho|v|^2   +  C \cdot v
 }}{e^{\varrho |v_*|^2 +  C \cdot v_* 
 }} 
 \dd v_*\lesssim_{ \vartheta,\varrho
} 1,
\end{split}
\end{equation}  
 while the same bounds replacing $|v|$ with $|v_*|$ hold for integrations over $v$.
\end{lemma} 
The proof of (\ref{kw/w}) relies on a fact that the exponent has a majorant  $- \vartheta {|v-v_*|^2} 
-  \vartheta   \frac{(|v|^2 - |v_*|^2)^2}{|v-v_*|^2}\leq -2   \vartheta (|v|+|v_*|) ||v|- |v_*||$ which is a negative definite. Note that an exponent of $\frac{e^{\varrho|v|^2
}}{e^{\varrho |v_*|^2 
 }}$ equals
${\varrho (|v|+ |v_*|) ||v| - |v_*||  } 
 $ which can be absorbed as long as $0 < \varrho < 2 \vartheta$. This yields 
(\ref{kw/w}). 
 We refer to a proof of Lemma 5 in \cite{GKTT1} for details to show (\ref{est:int_k}). 
  
 %

\hide
\begin{align}
 \nu  (v)  :=   
 \iint_{\R^3 \times \S^2} |(v-\e u-v_*) \cdot  \mathfrak{u} |   \mu_0 (v_*) \dd  \mathfrak{u}  \dd v_* \sim \langle v-\e u \rangle 
  ,\label{nu}\\
  Kf(v)  := K_1f - K_2f :=
\frac{1}{\sqrt{\mu}}  \iint_{\R^3 \times \S^2} |(v-v_*) \cdot \mathfrak{u}|
\mu \sqrt{\mu_*}   f(v_*)   
-\frac{1}{\sqrt{\mu}}  \iint_{\R^3 \times \S^2} |(v-v_*) \cdot \mathfrak{u}|
\{   \mu^\prime \sqrt{\mu_*^\prime} f(v_*^\prime) +\mu_*^\prime
\sqrt{\mu^\prime} f(v^\prime)
\} 
   , 
   \label{K}
\end{align}
where we have used a change of variables $v_* \mapsto v_* - \e u$ in (\ref{nu}) and a standard estimate (e.g. (3.54) of \cite{gl}) at the last step of (\ref{nu}). In (\ref{K}) we have used notations $\mu= \mu(v), \mu_* = \mu(v_*), \mu^\prime = \mu(v^\prime)$, and $\mu_*^\prime= \mu(v_*^\prime)$ for the sake of simplicity. A null space of $L$ inherits (\ref{basis}) from bases of $\frac{1}{\sqrt{\mu_0 }} Q( \mu_0, \sqrt{\mu_0} \tilde{f})$. A spectral gap estimate of $\frac{1}{\sqrt{\mu_0 }} Q( \mu_0, \sqrt{\mu_0} \tilde{f})$ (e.g. Lemma 3.4.1 in \cite{gl}) implies (\ref{s_gap}). A compactness of $K_0$ in $L^2$ implies $K$ is also compact.  \unhide

\begin{proof}[Proof of Lemma \ref{lemma:L-1}]
We consider an operator $g(v) \mapsto \nu_0^{-1} L_0 g (v):= \frac{1}{\nu_0(v)}L_0 g(v)$ on a restricted space of $\{ g \in L^2(\R^3): e^{\varrho |v|^2} g(v) \in L^2(\R^3)\}$. First we claim that 
\Be\label{L0:restric}
\nu_0^{-1} L_0: \{ g \in L^2(\R^3): e^{\varrho |v|^2} g(v) \in L^2(\R^3)\} \rightarrow \{ g \in L^2(\R^3): e^{\varrho |v|^2} g(v) \in L^2(\R^3)\}.
\Ee
 From (\ref{decom_L0}) we have $\nu_0^{-1} L_0 g (v)= g(v) - \nu_0^{-1} 
e^{-\varrho |v|^2}\int_{\R^3} \mathbf{k}_0 (v,v_*) \frac{e^{\varrho |v|^2}}{e^{\varrho |v_*|^2}}  e^{\varrho |v_*|^2}g(v_*) \dd v_*$, \hide  Setting $\eta=v-v_*$ we compute an exponent of $e^{- \frac{|v-v_*|^2}{8}
- \frac{1}{8} \frac{(|v|^2 - |v_*|^2)^2}{|v_*|^2}}  {e^{\varrho (|v|^2- |v_*|^2)}}$
\Be
- \frac{1}{8} |\eta|^2 - \frac{1}{8} \frac{||\eta|^2- 2 v\cdot \eta|^2}{|\eta|^2}
- \varrho \{|v-\eta|^2 - |v|^2 \}
= \Big(- \frac{1}{4} - \varrho\Big)|\eta|^2 + \Big(\frac{1}{2}+ 2 \varrho\Big) v\cdot \eta - \frac{1}{2} \frac{|v \cdot \eta|^2}{|\eta|^2}.\label{exponent}
\Ee 
For $|\varrho|< 1/4$ the discriminant of the above quadratic form of $\eta$ and $\frac{v\cdot \eta}{\eta}$ is negative: $(\frac{1}{2}+ 2 \varrho)^2+ 2\big(-\frac{1}{4} - \varrho\big)=4( \varrho^2-\frac{1}{16})<0$. Hence the exponent can be bounded above by $- C_{\varrho} \frac{|v-v_*|^2}{2} - C_\varrho |v \cdot (v-v_*)|$ for some $0<C_\varrho \ll_\varrho 1$. This yields that 
\Be\label{kw/w}
\Big|\mathbf{k}_0 (v,v_*) \frac{e^{\varrho |v|^2}}{e^{\varrho |v_*|^2}} \Big| \lesssim \frac{1}{|v-v_*|} e^{-C_{\varrho} \frac{|v-v_*|^2}{2} } \in L^\infty_{v_*}L^1_{v} \cap L^\infty_{v}L^1_{v_*}.
\Ee
\unhide
and, using (\ref{est:int_k}), for $\varrho < 2 \vartheta \leq 1/4$,
\begin{align*}
&|e^{\varrho |v|^2} \nu_0^{-1} L_0 g (v)|\\
 &\leq |e^{\varrho |v|^2} g(v)|
+ \nu_0(v)^{-1} \sup_v \sqrt{ \int_{\R^3}  k_{\vartheta}  (v,v_*) \frac{e^{\varrho |v|^2}}{e^{\varrho |v_*|^2}}\dd v_*}
\sqrt{ \int_{\R^3}   k_{\vartheta}  (v,v_*) \frac{e^{\varrho |v|^2}}{e^{\varrho |v_*|^2}}
| e^{\varrho |v_*|^2}  g(v_*)|^2
\dd v_*}\\
&\lesssim |e^{\varrho |v|^2} g(v)|
+ 
\sqrt{ \int_{\R^3}  k_{\vartheta}   (v,v_*) \frac{e^{\varrho |v|^2}}{e^{\varrho |v_*|^2}}
| e^{\varrho |v_*|^2}  g(v_*)|^2
\dd v_*}.
\end{align*}
Therefore we prove (\ref{L0:restric}) from 
\Be\label{cont_L0}
\begin{split}
\|e^{\varrho |v|^2} \nu_0^{-1} L_0 g (v)\|_{L^2_v}
&\lesssim \|e^{\varrho |v|^2}  g (v)\|_{L^2_v} + \sqrt{
\sup_{v_*} \int_{\R^3} k_\vartheta (v,v_*) \frac{e^{\varrho |v|^2}}{e^{\varrho |v_*|^2}} \dd v \int_{\R^3}  |e^{\varrho |v_*|^2} g(v_*)|^2
}\\
& \lesssim  \|e^{\varrho |v|^2}  g (v)\|_{L^2_v} . 
\end{split}\Ee

Now we view $\{ g \in L^2(\R^3): e^{\varrho |v|^2} g(v) \in L^2(\R^3)\}$ as the Hilbert space with an inner product $\langle e^{\varrho |v|^2} \cdot, e^{\varrho |v|^2} \cdot \rangle $. Then the compactness of $\nu_0^{-1}K_0$ in this space is equivalent to the compactness of $g \mapsto \int_{\R^3} \mathbf{k}_0 (v,v_*) \frac{e^{\varrho |v|^2}}{e^{\varrho |v_*|^2}} g(v_*) \dd v_*$ in a usual $L^2_v$. From Lemma 3.5.1 of \cite{gl}, it suffices to prove that (i) $\int_{\R^3} \mathbf{k}_0 (v,v_*) \frac{e^{\varrho |v|^2}}{e^{\varrho |v_*|^2}} \dd v$ is bounded in $v_*$, (ii) $\mathbf{k}_0 (v,v_*) \frac{e^{\varrho |v|^2}}{e^{\varrho |v_*|^2}} \in L^2 ( \{|v-v_*| \geq \frac{1}{n} \ \text{and}\  |v| \leq n\})$ for all $n\in \mathbb{N}$, and (iii) $\sup_{v}\int_{\R^3}  \mathbf{k}_0 (v,v_*) \frac{e^{\varrho |v|^2}}{e^{\varrho |v_*|^2}} \{ \mathbf{1}_{|v-v_*| \leq \frac{1}{n}} + \mathbf{1}_{|v|\geq n} \}\dd u\rightarrow 0$ as $n\rightarrow \infty$. Both conditions (i) and (ii) come from the first bound of (\ref{est:int_k}) directly. We prove (iii) from (\ref{kw/w}) and the first bound of (\ref{est:int_k}). Now applying the Fredholm alternative to $\nu_0^{-1} L_0=id -\nu_0^{-1}K_0 $ in the Hilbert space, we obtain an inverse map $(\nu_0^{-1} L_0)^{-1}$ which is a bounded operator of the Hilbert space. Note that $L_0^{-1} (g)= (\nu_0^{-1} L_0)^{-1} (\nu_0^{-1} g)$. Hence we derive that 
\Be\label{cont:L-1}
 \|  e^{\varrho |v|^2}L_0^{-1} g \|_{L^2_v}
 = \|  e^{\varrho |v|^2} (\nu_0^{-1} L_0)^{-1} (\nu_0^{-1} g)\|_{L^2_v}
  \lesssim 
 \|e^{\varrho |v|^2} \nu_0^{-1}g \|_{L^2_v} . 
\Ee

From the decomposition of $L_0$, we have $L^{-1}_0 g(v)= \nu_0(v)^{-1} g(v) + \nu_0(v)^{-1} KL^{-1}_0 g(v)$ for $g \in \mathcal{N}_0^\perp$. Then we have 
\begin{align*}
&|e^{\varrho |v|^2} L_0^{-1} g(v)| \\
& \leq | \nu_0(v)^{-1}e^{\varrho |v|^2}  g (v)|
+ \Big|
\nu_0(v)^{-1}  
\int_{\R^3} 
\mathbf{k}_0 (v,v_*) \frac{e^{\varrho |v|^2}}{e^{\varrho |v_*|^2}}
e^{\varrho |v_*|^2} L^{-1}_0 g(v_*)  
\dd v_*
\Big|\\
 &\leq \nu_0(v)^{-1}  \bigg\{
 |  e^{\varrho |v|^2}  g (v)|
 + \sqrt{\int_{\R^3}   \Big| \mathbf{k}_0 (v,v_*) \frac{e^{\varrho |v|^2}}{e^{\varrho |v_*|^2}}\Big|^2   \dd v_*}
 \sqrt{\int_{\R^3} 
|  e^{\varrho |v_*|^2} L^{-1}_0 g (v_*)|^2 \dd v_*
 }
 \bigg\},
\end{align*}
while $\big|\mathbf{k}_0 (v,v_*) \frac{e^{\varrho |v|^2}}{e^{\varrho |v_*|^2}} \big|^2 \lesssim \frac{1}{|v-v_*|^2} e^{-2C_{\varrho} \frac{|v-v_*|^2}{2} } \in L^{\infty}_{v} L^{1}_{v_*}$ from (\ref{kw/w}). Hence we prove (\ref{map:L-1}).  \end{proof}

  \hide
we have $|\nu_0(v)^{-1}e^{\varrho |v|^2} L_0g(v) |
  \lesssim \| e^{\varrho |\cdot |^2} g(\cdot )\|_\infty$ for $0 < \varrho < \frac{1}{4}$. This implies that, for $0 < \varrho < \frac{1}{4}$,

 From the Weyl's theorem we deduce a spectral gap estimate (\ref{s_gap}). \unhide

 Equipped with Lemma \ref{lemma:L-1} we provide bounds of $A_{ij}$ in (\ref{f_2}) and its derivatives: 
\begin{lemma}\label{lemma:A}
 For $0 < \varrho < \frac{1}{4}$
\Be\label{est:A&A_x}
\begin{split}
 |A_{ij}(v)|\lesssim e^{- \varrho |v-\e u|^2},
 \ 
  |\nabla_x A_{ij}(v)|\lesssim \e |\nabla_x u|e^{- \varrho |v-\e u|^2} 
, \    |\p_t A_{ij}(v)|\lesssim \e |\p_t u|e^{- \varrho |v-\e u|^2}, 
  \\
  |\nabla_x \p_t A_{ij}(v)|\lesssim \e\{ |\nabla_x \p_t u| + \e |\nabla_x u| |\p_t u|
 \}
 e^{- \varrho |v-\e u|^2}.
  \end{split}
 \Ee 
\end{lemma}
\begin{proof}
  It is convenient to introduce a notation, with $L_0$ in (\ref{L_0}),
 \Be\label{A0_ij}
 {A}_{0,ij}(v):= L^{-1}_0\Big((v_i v_j   - \frac{|v|^2}{3} \delta_{ij} )\sqrt{\mu_0}\Big)(v).
 \Ee
 Then from (\ref{map:L-1}) and (\ref{L_inverse}) we can immediately prove the first bound in (\ref{est:A&A_x}). 

 Recall the notations in (\ref{L_0}) and (\ref{L_0=L}).  By taking a derivative to $L_0 (\ref{A0_ij})$, it follows that, from the decomposition of $L_0A_{0,ij} (v)
= \nu_0(v) A_{0,ij} (v)- \int_{\R^3}
\mathbf{k}_0 (v, v-v_*) A_{0,ij} (v-v_*)
  \dd v_* 
$ and (\ref{est:k}),
\Be
 \begin{split}
 L_0 \p_{v_k} A_{0,ij} =& 
 \p_{v_k} (v_i v_j - \frac{|v|^2}{3}) \sqrt{\mu_0} + (v_i v_j - \frac{|v|^2}{3})  \p_{v_k}  \sqrt{\mu_0}\\
 & - 
\Big\{ \p_{v_k} \nu_0(v)
 A_{0,ij}(v)
 - \int_{\R^3} \p_{v_k} [ \mathbf{k}_0 (v,v-v_*)] A_{0,ij} (v-v_*)\dd v_*
 \Big\}.\label{nabla_vL}
 \end{split}
 \Ee
From (\ref{est:k}) and $\nabla_v(|v|^2- |v-v_*|^2)^2= 4 v_{*} (|v| + |v-v_*|) (|v- |v-v_*|)$, it follows
 $|\nabla_v [ \mathbf{k}_0 (v,v-v_*)]| \lesssim 
|v_*| \exp\{- \frac{ |v-v_*|^2+ |v*|^2}{8}\}
+ \frac{1}{|v_*|}  \exp\{- \frac{|v_*|^2}{8}
- \frac{1}{8} \frac{(|v  |^2 - |v-v_*|^2)^2}{| v_*|^2}\}$. From the first bound of (\ref{est:int_k}), it follows $|\int_{\R^3} \p_{v_k} [ \mathbf{k}_0 (v,v-v_*)] A_{0,ij} (v-v_*)\dd v_*| \lesssim e^{-\varrho |v|^2}$ for any $0 < \varrho< 1/4$.  Recall a projection $\tilde{\mathbf{P}}$ on $\mathcal{N}_0$.  Then $|(\mathbf{I}- \tilde{\mathbf{P}}) \text{ r.h.s. of }(\ref{nabla_vL})| \lesssim e^{-\varrho|v|^2}$. Now applying (\ref{map:L-1}) to $ \p_{v_k} A_{0,ij}= L_0^{-1} ((\mathbf{I}- \tilde{\mathbf{P}}) \text{ r.h.s. of }(\ref{nabla_vL}))$ we derive 
\Be\label{est:A_v}
| \nabla_v A_{0,ij} (v)| \lesssim e^{-\varrho |v|^2} \text{ for any } 0 < \varrho< 1/4.
\Ee
 From (\ref{f_2}) and (\ref{L_inverse}), and the fact $\tilde{\varphi}_i=v_i$ for $i=1,2,3$, and $\tilde{\varphi}_4=\frac{|v|^2}{3}$ (the notation $\tilde{f}$ is defined in (\ref{L_0=L})), we have  
  \Be\label{A_ij} 
 {A}_{ij}(v)
 = L^{-1}_0\Big((v_i v_j   - \frac{|v|^2}{3} \delta_{ij} )\sqrt{\mu_0}\Big)(v-\e u)=
 {A}_{0,ij}(v-\e u).
 \Ee
 Therefore we prove the second and third bounds  in (\ref{est:A&A_x}) using the fact that $\nabla_{x,t} A_{ij}(v)= -\e \nabla_{x,t} u \nabla_v A_{0,ij}(v-\e u)$. 
 
 Now we prove 
 \Be\label{est:A_vv}
 |\nabla_v^2 A_{0,ij} (v)| \lesssim e^{- \varrho  |v|^2} .
 \Ee
 By taking one more derivative to (\ref{nabla_vL}), we derive that  
 \begin{align*}
 L_0 \p_{v_k}\p_{v_\ell} A_{0,ij} &=  \ 
 \p_{v_k} \p_{v_\ell}  (v_i v_j - \frac{|v|^2}{3}) \sqrt{\mu_0} + \p_{v_\ell} (v_i v_j - \frac{|v|^2}{3})  \p_{v_k}  \sqrt{\mu_0}
 +  (v_i v_j - \frac{|v|^2}{3}) \p_{v_\ell} \p_{v_k}  \sqrt{\mu_0}
 \\
 & - 
 \p_{v_k} \p_{v_\ell}  \nu_0(v)
 A_{0,ij}(v)- 
  \p_{v_k}   \nu_0(v)
 \p_{v_\ell}A_{0,ij}(v)\\
 &
+\int_{\R^3} \p_{v_k}  \p_{v_\ell}[ \mathbf{k}_0 (v,v-v_*)] A_{0,ij} (v-v_*)
 + \p_{v_k} [ \mathbf{k}_0 (v,v-v_*)]  \p_{v_\ell} [A_{0,ij} (v-v_*)] \dd v_*
 .\label{nabla_vL}
 \end{align*}
The terms in the first two lines in r.h.s are easily bounded above as $e^{-\varrho |v|^2}$, recalling the fact $|\nabla_v\nu_0(v)| + |\nabla^2_v\nu_0(v)| \lesssim 1$. We only focus on the terms in the last line. From $|\p_{v_\ell} \nabla_v (|v|^2  - |v-v_*|^2)^2| \leq  \big|4 v_* \Big( \frac{v_\ell}{|v|} + \frac{(v-v_*)_\ell}{|v-v_*|}
 \Big)(|v| - |v-v_*|) \big| + 4 \big| v_* (|v| + |v-v_*|)\Big( \frac{v_\ell}{|v|} - \frac{(v-v_*)_\ell}{|v-v_*|} \Big)\big|\lesssim 
 |v_*|^2 + |v_*||v|,
 $ we have $| \p_{v_\ell }\nabla_v [ \mathbf{k}_0 (v,v-v_*)]| \lesssim 
|v_*| \exp\{- \frac{ |v-v_*|^2+ |v*|^2}{8}\}
+ \frac{1+|v|}{|v_*|^2}  \exp\{- \frac{|v_*|^2}{8}
- \frac{1}{8} \frac{(|v  |^2 - |v-v_*|^2)^2}{| v_*|^2}\}$. Using the second estimate of (\ref{est:int_k}) with the first bound of (\ref{est:A&A_x}), we have $\big|\int_{\R^3} \p_{v_k}  \p_{v_\ell}[ \mathbf{k}_0 (v,v-v_*)] A_{0,ij} (v-v_*) \dd v_*\big|\lesssim e^{- \varrho |v |^2}$. From (\ref{est:A_v}) and the first bound of (\ref{est:A&A_x}), it follows that $\big|\int_{\R^3}   \p_{v_k} [ \mathbf{k}_0 (v,v-v_*)]  \p_{v_\ell} [A_{0,ij} (v-v_*)] \dd v_*\big| \lesssim e^{- \varrho |v |^2}$. Now we invert the operator $L_0$ and use (\ref{map:L-1}) to conclude (\ref{est:A_vv}). 

Finally from $\p_t \nabla_x A_{ij} (v) = -\e \p_{t} \nabla_x u \nabla_v A_{0,ij} (v-\e u) + \e^2 \nabla_x u \p_t u \nabla_v^2 A_{0,ij} (v-\e u)$, (\ref{est:A_v}), and (\ref{est:A_vv}), we conclude the last estimate of (\ref{est:A&A_x}). 
\end{proof}

For the estimates of $\p_t f_R$ we derive the commutator estimate of $\p_t L- L\p_t $ and the corresponding one for $\Gamma$ as follows. 
\begin{lemma}Suppose $\e |u|\lesssim 1$ in the definition of $\mu$ in (\ref{mu_e}).  For $L$ and $\Gamma$ in (\ref{L_Gamma}) and (\ref{nu_K}), 
\Be\label{dec:L_t}
\begin{split}
\p_t (Lf)      &=L\p_tf +
 L_t( \mathbf{I}-\mathbf{P})f - L (\mathbf{I} - \mathbf{P})( \mathbf{P}_t f),\\
 \p_t (\Gamma (f,g))&=
 \{\Gamma(\p_t f, g) + \Gamma(f,\p_t g)\} +  \Gamma_t (f,g),
\end{split}
\Ee
where 
\Be
\begin{split}\label{def:L_t}
L_t g (t,v)&:= - \e \p_t u \cdot \nabla_v \nu_0 (v-\e u ) g(t, v)\\
& \ \ \ \ 
+ \e \p_t u \cdot  \int_{\R^3} (\nabla_v \mathbf{k}_0 + \nabla_{v_*} \mathbf{k}_0  )  (v-\e u, v_* - \e u)g(t,v_*)\dd v_*,
\\
(\mathbf{I} -\mathbf{P})\mathbf{P}_t g&:=- \e \sum _{j=0}^4
  ( {P}_j g) (\mathbf{I} -\mathbf{P})\big( \p_tu \cdot \nabla_v( \varphi_j \sqrt{\mu })\big)
  ,\\
  \Gamma_t(f,g) (t,v) &:=\frac{\e}{2} \iint_{\R^3 \times \S^2} 
|(v-v_*) \cdot \mathfrak{u}| 
 \p_t u  \cdot (v_*- \e u) 
\sqrt{ \mu(v_* )} \big\{
f(t,v ^\prime) g(t,v_*^\prime)\\
& \ \ \ \ \ \  \ \ \ \ \ \    \times 
+ g(t,v ^\prime) f(t,v_*^\prime)
- f(t,v  ) g(t,v_* )- g(t,v  ) f(t,v_* )
\big\}
\dd \mathfrak{u} \dd v_*.
\end{split}
\Ee
We have 
\Be
\begin{split}\label{est:L_t}
&\bigg|\int_{\R^3} L_t (\mathbf{I} - \mathbf{P})f(v)  g(v)  \dd v\bigg|
 \lesssim \e |\p_t u|
\| \nu^{1/2} (\mathbf{I} -\mathbf{P}) f \|_{L^2_v} \| \nu^{1/2} g \|_{L^2_v}
,\\
&\bigg|\int_{\R^3} L(\mathbf{P}_t f) (v) g(v)\dd v\bigg|
 \lesssim \e |\p_t u| 
|Pf| 
\| \nu^{1/2} (\mathbf{I} -\mathbf{P}) g \|_{L^2_v},\\ 
&\bigg|\int_{\R^3}\Gamma_t (f,g)(v) h(v) \dd v \bigg| \\
&\lesssim 
\e |\p_t u| 
\Big(
\| e^{\varrho |v|^2 + C \cdot v} g\|_{L^\infty_v}
\| \nu^{1/2}  (\mathbf{I}- \mathbf{P}) f \|_{L^2_v}\\
& \ \ \ \ \ \ \ \   \ \ \ + \| e^{\varrho |v|^2 + C \cdot v} f\|_{L^\infty_v}
\| \nu^{1/2}  (\mathbf{I}- \mathbf{P}) g \|_{L^2_v}
+ |Pf||Pg|    \Big)
\| \nu^{1/2} h \|_{L^2_v}
.
\end{split}
\Ee
Pointwise estimates are given as follows: for $0 < \varrho< 1/4$ and $C \in \R^3$ 
\Be
\begin{split}
\label{est_infty:L_t}
|
 L_t( \mathbf{I}-\mathbf{P})f(t, v) - L( \mathbf{P}_t f)(t, v)
|  \lesssim  \e |\p_t u|  
 \| e^{ \varrho |v|^2 + C \cdot v} f(t,v) \|_{L^\infty_v} \nu(v)^2 
 e^{- \varrho |v-\e u |^2},\\
|  \Gamma _t(f,g) (t,v)  
|  \lesssim
\e |\p_t u|   
 \| e^{\varrho |v|^2+ C \cdot v} f(t,v) \|_{L^\infty_v}  \| e^{\varrho |v|^2
 + C \cdot v
 }g(t,v) \|_{L^\infty_v}
  \frac{\nu(v)}{e^{    \varrho |v |^2+ C \cdot v } },\\
\left|    \Gamma (f,g)(v)\right| \lesssim \|e^{\varrho |v|^{2}+ C \cdot v}  f(v) \|_{L^{\infty}_{v}}\|e^{\varrho |v|^{2}+ C \cdot v}  g(v) \|_{L^{\infty}_{v}} \frac{\nu(v)}{e^{    \varrho |v |^2+ C \cdot v } } ,
\end{split}
\Ee
and 
\Be
|\Gamma(f, g)(v)| \lesssim \| e^{\varrho |v|^2+ C \cdot v} f \|_\infty \Big( \nu(v)| g(v)| + \int_{\R^3} k_{\vartheta}(v,v_*)  |g(v_*)| \dd v_* \Big).\label{est_Carl:Gamma}
\Ee
\end{lemma}
\begin{proof}
The decomposition (\ref{dec:L_t}) with (\ref{def:L_t}) comes from a direct computation to (\ref{nu_K}) and $
\p_t (Lf_R) =  \p_t (L( \mathbf{I}-\mathbf{P})f_R) =L ( \mathbf{I}-\mathbf{P})\p_tf_R + L_t( \mathbf{I}-\mathbf{P})f_R + L (-\mathbf{P}_t f_R).
$ On the other hand, from (\ref{est:k}) it is easy to check that, for any $0< \varrho < 1/4$,  
\begin{align*} 
|\nabla_v \nu_0 (v)| = \Big| \iint_{\R^3 \times \S^2}   \mathfrak{u}  \frac{(v-v_*) \cdot \mathfrak{u}}{|(v-v_*) \cdot \mathfrak{u}|}\mu_0 (v_*) \dd \mathfrak{u} \dd v_* \Big|\lesssim 1,\\  
|\nabla_v \mathbf{k}_0 (v,v_*)| +  |\nabla_{v_*} \mathbf{k}_0 (v,v_*)|  \lesssim 
\Big(|v-v_*|^{-1}+
 \nu_0 (v)^2 |v-v_*| 
\Big)
k_{\varrho/2} (v,v_*).
\end{align*}
These estimates above combining with (\ref{est:int_k}) and $\mathbf{P}L=0$ yield the first two estimates of (\ref{est:L_t}).

We derive $\Gamma_t$ as in the third identity of (\ref{def:L_t}) from a direct computation to (\ref{Gamma}) and \hide

We read the form of $\Gamma$ from (\ref{L_Gamma}), using $\mu(v^\prime)\mu(v_*^\prime)=\mu(v )\mu(v_* )$,
  Taking a temporal derivative, $\p_t \Gamma(f,f) (t,v)$ equals $2\Gamma(f,\p_t f) (t,v)$ plus a remainder which takes the same form of (\ref{Gamma}) where $\sqrt{\mu(v_* )}$ is replaced by $\p_t \sqrt{\mu(v_* )}$. 

From $\mu(v) = \mu_0 (v-\e u)$, it follows that
\unhide
$
 \p_t \sqrt{\mu(v_*)} =  - \e \p_t u \cdot \nabla_v \sqrt{\mu_0 (v_* - \e u )} = \frac{1}{2}  \e \p_t u \cdot  (v_*-\e u ) \sqrt{\mu_0 (v_* - \e u )} 
.
$
Then it is standard (see Lemma 2.13 in \cite{EGKM2} for example) to have the last estimate in (\ref{est:L_t}).

The first bound of (\ref{est_infty:L_t}) is a direct consequence of applying (\ref{kw/w}) and (\ref{est:int_k}) to the first identity of (\ref{def:L_t}). For the second bound of (\ref{est_infty:L_t}) we bound it as 
\begin{align*}
&\e |\p_t u |
 \| e^{\varrho |v|^2+ C \cdot v} f(t,v) \|_{L^\infty_v}  \| e^{\varrho |v|^2+ C \cdot v} g(t,v) \|_{L^\infty_v} 
\frac{1}{e^{    \varrho |v |^2+ C \cdot v } }  \iint_{\R^3 \times \S^2} 
|(v-v_*) \cdot \mathfrak{u}| 
e^{-  \varrho |v_* |^2- C \cdot v_*} 
\dd \mathfrak{u} \dd v_*\\
\lesssim&  \    
 \frac{\nu(v)}{e^{    \varrho |v |^2+ C \cdot v } }\e |\p_t u |
 \| e^{\varrho |v|^2+ C \cdot v} f(t,v) \|_{L^\infty_v}  \| e^{\varrho |v|^2+ C \cdot v} g(t,v) \|_{L^\infty_v} 
\end{align*}
For the last bound of (\ref{est_infty:L_t}) we recall a standard estimate (e.g. \cite{gl}) that $$|\frac{1}{\nu_{0}(v)} e^{\varrho |v|^{2}+ C \cdot v} \Gamma_{0}(f,g)(v)|\lesssim \|e^{\varrho |v|^{2}+ C \cdot v}  f(v) \|_{L^{\infty}_{v}}\|e^{\varrho |v|^{2}+ C \cdot v}  g(v) \|_{L^{\infty}_{v}} .$$ 
From the second equality of (\ref{L_0=L}) we deduce the last bound of (\ref{est_infty:L_t}). A bound (\ref{est_Carl:Gamma}) is standard. \end{proof}

\subsection{Proof of Proposition \ref{prop:Hilbert}}

We verify two statements of Section \ref{sec:1.1} Hilbert Expansion. Firstly, we will show that the solvability condition (\ref{Comp1}) implies the incompressible condition (\ref{incomp}).   From (\ref{mu_e}), (\ref{basis}), and direct computations we verify the first identity of (\ref{Comp1}). \hide
\Be\label{trans_mu}
\frac{\e^{-1} (v- \e u )\cdot \nabla_x \mu }{\sqrt{\mu}}= \sum_{\ell, m=1}^3 \p_\ell u_m\varphi_\ell  \varphi_m  \sqrt{\mu}.
\Ee\unhide
Then from the oddness of the integrand with respect to the variable $\varphi_i$ we derive that $\Big\langle \varphi_i \sqrt{\mu}, \frac{ \e^{-1} (v- \e u )\cdot \nabla_x \mu}{\sqrt{\mu}} \Big\rangle =0$ for $i=1,2,3$. For $i=0,4$, we compute that 
 \Be\notag
 \Big\langle \varphi_i \sqrt{\mu}, \frac{ \e^{-1} (v- \e u )\cdot \nabla_x \mu}{\sqrt{\mu}} \Big\rangle
 = \sum_{\ell=1}^3 
 \langle
 \varphi_i \sqrt{\mu}, \varphi_\ell \varphi_\ell  \sqrt{\mu}
  \rangle
  \p_\ell u_\ell = \left\{\delta_{i0}  + \delta_{i4}\sqrt{\frac{2}{3}} \right\} ( \nabla_x \cdot u)  \ \ \text{for} \  i=0,4.
 \Ee
 This shows that (\ref{Comp1}) implies (\ref{incomp}). 

 Secondly, we will verify the following statement of Section \ref{sec:1.1}: the leading order terms of the hydrodynamic part in (\ref{eqtn_f_3}) vanish by solving the Navier-Stokes equations \eqref{NS_k}-\eqref{noslip}. 
 Consider (\ref{eqtn_f_3}). We set  $\mathbf{P}f_2 =\{ \tilde{\rho}  \varphi_0+ \sum_{\ell=1}^3 \tilde{u}_\ell    \varphi_\ell +\tilde{\theta}  \varphi_4  \}\sqrt{\mu}$ whose coefficients will be determined as in (\ref{Pf_2}). Then  the leading order term of  $(\ref{eqtn_f_3})= \frac{1}{\delta} (\ref{H_NS})$ can be decomposed as 
\Be \begin{split}
&-  \frac{1}{\delta}
 \mathbf{P} \Big(  (v- \e u) \cdot    (\p_t u + u \cdot \nabla_x u ) \sqrt{\mu}  \\
 &
 + 
\underbrace{ (v- \e u) \cdot   \big(  \nabla_x  \tilde{\rho} \varphi_0 \sqrt{\mu}-  \sum_{\ell=1}^3  \nabla_x  \tilde{u}_\ell \varphi_\ell \sqrt{\mu}
 +  \nabla_x\tilde{\theta} \varphi_4  \sqrt{\mu} \big)}_{ (\ref{H_NS_1})_* }
 -  
 \sum_{\ell,m=1}^3 \kappa (v- \e u) \cdot A_{\ell m}   \nabla_x\p_\ell    u_m  \Big),\label{H_NS_1}\\
\end{split}\Ee
\Be
\begin{split}
 &- \frac{1}{\delta} (\mathbf{I}-\mathbf{P})  
 \Big(
  (v- \e u) \cdot 
   \big(  \nabla_x  \tilde{\rho} \varphi_0 \sqrt{\mu}+  \sum_{\ell=1}^3  \nabla_x  \tilde{u}_\ell \varphi_\ell \sqrt{\mu}
 +  \nabla_x\tilde{\theta} \varphi_4  \sqrt{\mu} \big)
 \Big)\\
 &  + \frac{1}{\delta} (\mathbf{I} - \mathbf{P})\Big(  \sum_{\ell,m=1}^3 \kappa (v- \e u) \cdot A_{\ell m}   \nabla_x\p_\ell   u_m\Big),\label{H_NS_2}
  \end{split}\Ee
 while the lower order term consists of 
\Be
\begin{split}\label{H_NS_lot}
-\frac{1}{\delta}  (v- \e u) \cdot \nabla_x \Big(   \tilde{\rho} \varphi_0 \sqrt{\mu}+  \sum_{\ell=1}^3   \tilde{u}_\ell \varphi_\ell \sqrt{\mu}
 +   \tilde{\theta} \varphi_4  \sqrt{\mu} \Big) \\
+\frac{1}{\delta}  (v- \e u) \cdot \Big(   \nabla_x \tilde{\rho} \varphi_0 \sqrt{\mu}+  \sum_{\ell=1}^3  \nabla_x  \tilde{u}_\ell \varphi_\ell \sqrt{\mu}
 +  \nabla_x\tilde{\theta} \varphi_4  \sqrt{\mu} \Big) \\
 +  \frac{\kappa }{\delta}\sum_{\ell,m=1}^3(v-\e u) \cdot \nabla_x A_{\ell m} \p_\ell u_m.
\end{split}\Ee

First we focus on a leading order contribution of  $(\ref{H_NS_1})_*$   in (\ref{H_NS_1}). 
A direct computation yields 
\Be\label{p_2}
\begin{split}
&\big\langle  \varphi_i \sqrt{\mu},  (\ref{H_NS_1})_*  \big\rangle \\
&= 
\begin{cases} C_i \nabla_x \cdot \tilde{u}, \quad i =0,4\\
\p_i \tilde{\rho}
\langle \varphi_i \sqrt{\mu}, \varphi_i \sqrt{\mu}\rangle 
+ 
 \p_i \tilde{\theta} \langle \varphi_i\sqrt{\mu}, \varphi_i \varphi_4  \sqrt{\mu}\rangle    =   \p_i \Big( \tilde{\rho} + \sqrt{\frac{2}{3} }\tilde{\theta} \Big), \quad i=1,2,3.
\end{cases}
\end{split}
\Ee
 \hide
 
 $\big\langle \varphi_i \sqrt{\mu} , (v- \e u) \cdot \nabla_x    \mathbf{P} f_2  \big\rangle=C_i \nabla_x \cdot \tilde{u}$ with $C_i\neq 0$ for $i=0,4$. On the other hand we derive that, for  i=1,2,3$, 
 \Be \begin{split}\label{p_2}
\Big\langle  \varphi_i \sqrt{\mu}, 
 (v-\e u) \cdot  \nabla_x   \mathbf{P} f_2 \Big \rangle 
= &
\p_i \tilde{\rho}
\langle \varphi_i \sqrt{\mu}, \varphi_i \sqrt{\mu}\rangle 
+ 
 \p_i \tilde{\theta} \langle \varphi_i\sqrt{\mu}, \varphi_i \varphi_4  \sqrt{\mu}\rangle   
%
  =   
   \p_i \Big( \tilde{\rho} + \sqrt{\frac{2}{3} }\tilde{\theta} \Big).
 \end{split}\Ee\unhide
 Among many other choices we make a special choice $(\tilde{\rho}, \tilde{u}, \tilde{\theta})=(p,0,0)$ which is equivalent to (\ref{Pf_2}). 
From (\ref{Delta_u}), (\ref{p_2}), and (\ref{Pf_2}), it follows that for $(u,p)$ solving (\ref{NS_k})
\Be\label{H_NS_1=0}
(\ref{H_NS_1})= \frac{1}{\delta} (v- \e u ) \sqrt{\mu} 
\cdot \big\{ \p_t u+ u \cdot \nabla_x u - \kappa \eta_0 \Delta u + \nabla_x p \big\}
= \frac{1}{\delta} (v- \e u ) \sqrt{\mu} \cdot (\ref{NS_k}) =0, 
\Ee
which verifies the second statement of Section \ref{sec:1.1}. 

Now we turn to proving the estimates. While the leading order terms vanish in (\ref{eqtn_f_3}), the rest of terms of (\ref{eqtn_f_3}) are bounded as follows. 
 Upon the choice of (\ref{Pf_2}), the first term of (\ref{H_NS_2}) vanishes  and 
the first line of (\ref{H_NS_lot}) are bounded by  
\begin{align}
\Big|
 \frac{\e}{2 \delta
} (v-\e u) \cdot \nabla_x u \cdot (v-\e u) \mathbf{P} f_2\Big|
\lesssim \frac{\e}{\delta} 
|\nabla_x u| 
|p| 
\langle v-\e u \rangle^2  \sqrt{\mu}.\label{est1:H_NS_lot}
\end{align} 
From (\ref{est:A&A_x}) we deduce that the second term of (\ref{H_NS_2}) and the second line of (\ref{H_NS_lot}) are bounded respectively by 
 \Be\label{est2:H_NS_lot}
 \frac{\kappa}{\delta} |\nabla_x^2 u|  |v-\e u|
  e^{- \varrho |v-\e u|^2}, \ 
 \frac{ \e \kappa}{\delta}  |\nabla_x u|
  |v-\e u|
  e^{- \varrho |v-\e u|^2} \ \ \text{for any }
  0 < \varrho < 1/4.
 \Ee 
\hide

$(v - \e u) \cdot \nabla_x  \mathbf{P} f_2$ in (\ref{H_NS}).  From (\ref{mu_e}) we expand the term as 
 \begin{align*}
(v - \e u) \cdot \nabla_x  \mathbf{P} f_2&=  
\sum_{j=1}^3\p_j  \tilde{\rho}  \varphi_j    \sqrt{\mu} 
+ 
\sum_{i,j=1}^3 \p_i \tilde{u}_j \varphi_i \varphi_j \sqrt{\mu}
+ 
\sum_{j=1}^3 \p_j \tilde{\theta} \varphi_j  \varphi_4 \sqrt{\mu} 
 \\
 &
-\e 
\sum_{i,j=1}^3  \tilde{u}_i \p_j u_i \varphi_j\sqrt{\mu}
- \e 
\frac{2}{\sqrt{6}}\sum_{i,j=1}^3 \tilde{\theta}\p_j u_i   \varphi_i \varphi_j \sqrt{\mu}
+ \frac{\e}{2
} (v-\e u) \cdot \nabla_x u \cdot (v-\e u) \mathbf{P} f_2
.
  \end{align*}

  %

 we are able to remove the entire leading order of hydrodynamic part of $(v - \e u) \cdot \nabla_x  \mathbf{P} f_2$ in (\ref{H_NS}).\unhide 
In conclusion we end up with the following result: Assume $(u,p)$ solves (\ref{NS_k})-(\ref{noslip}), and both (\ref{f_2}) and (\ref{Pf_2}) hold. Then 
 \begin{align}
 |(\ref{eqtn_f_3})-(\ref{H_NS_2}) | 
  \lesssim   \frac{\e}{\delta} 
  \{
  |\nabla_x u| 
  |p| 
    + \kappa |\nabla_x u|
  \}
   \langle v-\e u  \rangle^2
   e^{-  \frac{|v-\e u|^2}{4}}
  ,\label{est:Pf_3}
  \\
   | (\mathbf{I} - \mathbf{P}) (\ref{H_NS_2}) |= |  (\ref{H_NS_2}) |
   \lesssim
    \frac{1}{\delta}
 \kappa |\nabla _x^2 u|
   e^{-   \varrho{|v-\e u|^2} }
 .\label{est:I-Pf_3}
 \end{align}
 
\smallskip
 
The term $\p_t (\ref{eqtn_f_3})$ can be bounded similarly. The entire leading order term of $\p_t (\ref{eqtn_f_3})$ can be decomposed as 
 \begin{align}
 - \frac{1}{\delta}
 \mathbf{P} \Big( (v- \e u) \cdot   \p_t (\p_t u + u \cdot \nabla_x u ) \sqrt{\mu}  
 + (v- \e u) \cdot   \Big(  \nabla_x \p_t p \varphi_0 \sqrt{\mu}
 \Big) 
\notag \\
- 
 \sum_{\ell,m=1}^3 \kappa (v- \e u) \cdot A_{\ell m}   \nabla_x\p_\ell \p_t  u_m  \Big),\label{H_NS_t1}\\
  -\frac{1}{\delta} (\mathbf{I}-\mathbf{P})  
 \Big(
  (v- \e u) \cdot  (  \nabla_x \p_t p \varphi_0 \sqrt{\mu}
    ) 
 \Big)  + \frac{1}{\delta} (\mathbf{I} - \mathbf{P})\Big(  \sum_{\ell,m=1}^3 \kappa (v- \e u) \cdot A_{\ell m}   \nabla_x\p_\ell \p_t  u_m\Big).\label{H_NS_t2}
  \end{align}
Following the argument to get (\ref{Delta_u}) and (\ref{p_2}), 
we derive that 
 \Be\label{est:H_NS_t1}
 (\ref{H_NS_t1})= -\frac{1}{\delta} (v- \e u) \sqrt{\mu} \cdot  \p_t (\ref{NS_k})=0. 
 \Ee
 On the other hand,
 \Be\label{est:H_NS_t2}
 |(\ref{H_NS_t2})| \lesssim \frac{1}{\delta}
 \kappa |\nabla_x ^2 \p_t u|
  e^{- \varrho |v-\e u|^2}.
 \Ee 
Now the lower order term $\p_t (\ref{eqtn_f_3})- (\ref{H_NS_t1})- (\ref{H_NS_t2})$ consists of 
 \Be\begin{split}\label{lower:H_NS_t}
 &  \frac{\e}{\delta} \p_t u \cdot \Big\{
 (\p_t u + u \cdot \nabla_x u ) \sqrt{\mu}  
 +   \nabla_x   \mathbf{P} f_2
- \kappa   \nabla_x
\Big( \sum_{\ell,m=1}^3 A_{\ell m} \p_\ell u_m\Big)
 \Big\} \\
 &+ \frac{\kappa}{\delta} (v- \e u ) \cdot \nabla_x \Big( \sum_{\ell,m=1}^3  \p_t A_{\ell m} \p_\ell u_m\Big)
 \\
 & - \frac{1}{\delta}  (v- \e u) \cdot  \nabla_x  \p_t   \big( p \varphi_0 \sqrt{\mu}
 \big)
+\frac{1}{\delta} (v- \e u) \cdot  \big(  \nabla_x \p_t p \varphi_0 \sqrt{\mu}
\big) 
 .
\end{split}\Ee
Since the lower order term of $\p_t (\ref{eqtn_f_3})$ always contains $|\p_t (v-\e u)| \leq  \e |\p_t u|$, they can be bounded by, from (\ref{est:A&A_x}) and (\ref{Pf_2}),  
\Be\label{est:lower:H_NS_t}
\begin{split}
|(\ref{lower:H_NS_t})| \lesssim&  \  \frac{\e}{\delta}|\p_t u| \big\{
|\p_t u| + |u| |\nabla_x u|  + |\nabla_x p| 
+ \e |\nabla_x u| 
|p| 
+ \kappa \e |\nabla_x u| ^2 + \kappa |\nabla_x ^2 u|
\big\} e^{- \varrho |v-\e u|^2}\\
&
+  \frac{  \e}{\delta} 
\{
|\nabla_x \p_t u| (1+ \kappa |\nabla_x u|) + |\p_t p| |\nabla_x u| 
\}
e^{- \varrho |v-\e u|^2}. 
\end{split}
\Ee

 \smallskip

\hide  Direct computations yield $\big\langle \varphi_\ell \sqrt{\mu} , (v- \e u) \cdot \nabla_x    \mathbf{P} f_2  \big\rangle=C_\ell \nabla_x \cdot \tilde{u}$ with $C_\ell\neq 0$ for $\ell=0,4$, and
\Be \begin{split}\label{p_2}
\Big\langle  \varphi_\ell \sqrt{\mu}, 
  (v-\e u) \cdot  \nabla_x   \mathbf{P} f_2 \Big \rangle 
  =    \p_\ell \Big( \tilde{\rho} + \sqrt{\frac{2}{3} }\tilde{\theta} \Big)
  + \frac{\e}{2 }\{
  - \tilde{u} \cdot \p_\ell u + \p_\ell u_\ell \tilde{u}_\ell
  \} \ \ \text{for}  \ \ell=1,2,3.
 \end{split}\Ee
In (\ref{Pf_2}) we will make $\p_\ell \Big( \tilde{\rho} + \sqrt{\frac{2}{3} }\tilde{\theta} \Big)$ equal the pressure $\p_\ell p$ of (\ref{incomp}). \unhide


 Now we consider (\ref{eqtn_f_4}). From (\ref{f_2}), (\ref{Pf_2}), and (\ref{est:A&A_x}), we derive that 
  \begin{align}\label{ptPf2}
 | (\p_t + u \cdot \nabla_x )   \mathbf{P}f_2 |  &\lesssim 
 \Big\{
 |\p_t p| + |u| |\nabla_x p| 
+ \e 
|p|
\{ |\p_t u| + |u| |\nabla_x u|\}
\Big\} \langle v-\e u\rangle ^{ 2} e^{  -\frac{|v-\e u|^2}{4}}
,
\end{align}
\Be\label{ptptPf2}
\begin{split}
&| \p_t  (\p_t  + u \cdot \nabla_x   )\mathbf{P}f_2 | \\
 \lesssim& \ 
  \Big\{   
 |\p_t^2 p| + |\p_t u| |\nabla_x p|+ | u| |\nabla_x \p_t p| 
+ \e 
| \p_t p|
  \{ |\p_t u| + |u| |\nabla_x u|\} \\
& \  \ + \e   |  p|
\{ |\p_t^2 u| + |\p_t u| |\nabla_x u|+  |u| |\nabla_x \p_t  u|\}+ \e |\p_t u| \{\text{r.h.s. of } (\ref{ptPf2}) \}
\Big\}   \langle v-\e u\rangle ^{ 2} e^{  -\frac{|v-\e u|^2}{4}},
\end{split}
\Ee
and, for $0< \varrho < 1/4$, 
\Be
\begin{split}\label{ptI-Pf2}
  | (\p_t + u \cdot \nabla_x ) (\mathbf{I} - \mathbf{P}) f_2|
\lesssim
\kappa\Big\{
 \{ | \nabla_x \p_t u| + |u| |\nabla_x^2 u|\}
+ \e \{|\p_t u| + |u| |\nabla_x u|\} |\nabla_x u|
\Big\} e^{-\varrho |v-\e u|^2}
,
\end{split}
\Ee
\Be
\begin{split}\label{ptI-Pf2_t}
& | \p_t (\p_t + u \cdot \nabla_x ) (\mathbf{I} - \mathbf{P}) f_2|\\
\lesssim& \ \kappa\Big\{ \{ | \nabla_x \p_t^2 u| + |\p_t u| |\nabla_x^2 u|+ |u| |\nabla_x^2 \p_t  u|\}
+ \e \{|\p_t u| + |u| |\nabla_x u|\} |\nabla_x \p_t u|\\
& \ \ \ \ + \e \{|\p_t^2 u| + |\p_t u| |\nabla_x u|+ |u| |\nabla_x \p_t u|\} |\nabla_x u|
+ \e |\p_t u| \{ \text{r.h.s. of } (\ref{ptI-Pf2}) \}
\Big\}  e^{-\varrho |v-\e u|^2}
.
\end{split}
\Ee 


 Next  we consider the last term in (\ref{eqtn_f_4}). From (\ref{mu_e}) and (\ref{incomp})
   \begin{align}
 \frac{(\p_t + \e^{-1} v\cdot \nabla_x) \sqrt{\mu}}{\sqrt{\mu}} 
 & = \frac12  \big[  (v-\e u) \cdot \nabla_x u    \cdot (v-\e u)  
  + \e ( \p_t u + u \cdot \nabla_x u) \cdot (v-\e u)   \big] 
  \notag
  ,
  \\
\p_t \left(   \frac{(\p_t + \e^{-1} v\cdot \nabla_x) \sqrt{\mu}}{\sqrt{\mu}} \right) &=
\frac12 \big[  \e (\p_t^2 u + u \nabla_x \p_t u- \p_t u \cdot \nabla_x u) \cdot (v-\e u)
- \e^2 \p_t u \cdot (\p_t u + u \cdot \nabla_x u) \notag\\
&\qquad+ (v-\e u) \cdot \nabla_x\p_t u    \cdot (v-\e u) \big]  \notag
 , 
  \end{align}
    \hide
  \begin{align}
 \frac{(\p_t + \e^{-1} v\cdot \nabla_x) \sqrt{\mu}}{\sqrt{\mu}} 
 & = (v-\e u) \cdot \nabla_x u    \cdot (v-\e u)  
  + \e ( \p_t u + u \cdot \nabla_x u) \cdot (v-\e u)   
  \notag
  ,
  \\
\p_t \left(   \frac{(\p_t + \e^{-1} v\cdot \nabla_x) \sqrt{\mu}}{\sqrt{\mu}} \right) &=
  \e (\p_t^2 u + u \nabla_x \p_t u- \p_t u \cdot \nabla_x u) \cdot (v-\e u)
- \e^2 \p_t u \cdot (\p_t u + u \cdot \nabla_x u)  \notag
 . 
  \end{align}\unhide
  and hence we derive (\ref{transp:mu}) and (\ref{transp:mu_t}). 
  
  \smallskip
  
  Applying (\ref{est:A&A_x}) to (\ref{f_2}), it follows that, for $0< \varrho < \frac{1}{4}$, 
\begin{align}
|(\mathbf{I} - \mathbf{P}) f_2 | \lesssim    \kappa |\nabla_x u| e^{- \varrho |v-\e u|^2}
, \  \ 
|\p_t (\mathbf{I} - \mathbf{P}) f_2 | \lesssim    \kappa\{ | \p_ t \nabla_x u|
+   \e |\p_t u| |   \nabla_x u|\} e^{- \varrho |v-\e u|^2}.\label{est:I-Pf2}
\end{align}
From (\ref{Pf_2})
\Be\label{est:Pf2}
|\mathbf{P} f_2| \lesssim
|p| 
  e^{- \frac{|v-\e u|^2}{4}}, \ \ 
|\p_t \mathbf{P} f_2| \lesssim 
|\p_t p|
e^{- \frac{|v-\e u|^2}{4}} + \e |\p_t u |
| p| 
 \langle v-\e u \rangle
  e^{- \frac{|v-\e u|^2}{4}}.
\Ee
These estimates give (\ref{est:f2}) and (\ref{est:f2_t}). 

\smallskip

The last term of (\ref{eqtn_f_4}) is bounded as  
\Be\begin{split}\label{est:tras_muf_2}
\frac{\e}{\delta} \Big| \frac{(\p_t + \e^{-1} v\cdot \nabla_x) \sqrt{\mu}}{\sqrt{\mu}}  f_2 
 \Big|\lesssim \frac{\e}{\delta} \{|p| 
 + \kappa |\nabla_x u|
 \}
 \{ |\nabla_x u| + \e (|\p_t u| + |u | |\nabla_x u|) \} e^{-\varrho |v-\e u|^2},
 \end{split}\Ee
  \Be\begin{split}\label{est:tras_muf_2_t}
& \frac{\e}{\delta} \Big|
 \p_t \bigg(
  \frac{(\p_t + \e^{-1} v\cdot \nabla_x) \sqrt{\mu}}{\sqrt{\mu}}  f_2 
\bigg) \Big|\\
\lesssim& \frac{\e}{\delta} \{|\nabla_x u| + \e (|\p_t u| + |u| |\nabla_x u|)\}
\{ |\p_t p| 
+ \kappa |\p_t \nabla_x u|  + \e |\p_t u|  (
|p| 
+\kappa |\nabla_x u|)
\}e^{-\varrho |v-\e u|^2} \\
&+ \frac{\e}{\delta} (|p| 
+ \kappa |\nabla_x u|)
\{
|\nabla_x\p_tu|+\e(|\p_t ^2 u | + |u ||\nabla_x \p_t u| + |\p_t u| |\nabla_x u|)\\
&  \ \ \ \ \  \ \ \ \ \ \ \ \ \ \  \ \ \ \ \ \ \ \ \ \ + \e^2 |\p_t u| (|\p_t u| + |u| |\nabla_x u|)
\}e^{-\varrho |v-\e u|^2}.
 \end{split}
\Ee  

 Lastly from (\ref{est_infty:L_t}), (\ref{Pf_2}), and (\ref{f_2})
\Be\label{est:G22}
\begin{split}
\frac{\e}{\delta \kappa}|\Gamma (f_2,f_2)(v)|
&= \frac{\e}{\delta \kappa}|\Gamma (f_2,f_2)(v)
- \Gamma (\mathbf{P} f_2, \mathbf{P}f_2 )
|\\
& \lesssim \frac{\e  }{\delta  } (|p|+ \kappa| \nabla_x u|) |\nabla_x u|   
\nu(v)
e^{-\varrho |v-\e u|^2},
\end{split}
\Ee
\Be\label{est:G22_t}
\begin{split}
&\frac{\e}{\delta \kappa}|\p_t \Gamma (f_2,f_2)(v)|\\ &\lesssim 
 \frac{\e}{\delta
 }
\{(|p|   + \kappa |\nabla_x u| )
|\p_t \nabla_x u| 
+(|\p_t p| + \kappa |\nabla_x \p_t u| )|\nabla_x   u|
\}
  \nu(v) e^{-\varrho |v-\e u|^2}\\
 & + \frac{\e^2}{\delta  }|\p_t u|(|p|   + \kappa |\nabla_x u| )|\nabla_x u|  \nu(v)  e^{-\varrho |v-\e u|^2},
  \end{split}
\Ee 
where we have used $\Gamma (\mathbf{P} f_2, \mathbf{P}f_2 )=\Gamma(p \sqrt{\mu}, p \sqrt{\mu})=0$ to eliminate the contribution of $p^2$ in (\ref{est:G22}). 

Finally we wrap up the estimates of the source term of (\ref{eqtn_fR}) to show (\ref{est:R1}) and (\ref{est:R2}). The term $(\mathbf{I} - \mathbf{P}) \mathfrak{R}_{1}$ consists of (\ref{H_NS_2}), which is bounded as (\ref{est:I-Pf_3}) and hence we prove (\ref{est:R1}). The rest of terms form $\mathfrak{R}_{2}$, which can be proved to be bounded as (\ref{est:R2}), from (\ref{est:Pf_3}), (\ref{ptPf2}), (\ref{ptI-Pf2}), (\ref{est:tras_muf_2}), and (\ref{est:G22}).  

Now we consider the source term of (\ref{eqtn_fR_t}). The term $(\mathbf{I} - \mathbf{P}) \mathfrak{R}_3$ consists of (\ref{H_NS_t2}), which is bounded as (\ref{est:H_NS_t2}). From (\ref{est:H_NS_t1}), 
(\ref{est:lower:H_NS_t}), (\ref{ptptPf2})-(\ref{ptI-Pf2_t}), (\ref{transp:mu}), (\ref{transp:mu_t}), (\ref{est:tras_muf_2_t}), (\ref{est:G22_t}), and (\ref{est_infty:L_t}), we prove (\ref{est:R4}).

\hide 
 \Be
 \begin{split}
   &\Big[ \p_t + \frac{1}{\e} v\cdot \nabla_x + \frac{1}{\e^2 \kappa} L \Big]f_R
  +(\ref{transp:mu}) f_R - \frac{2}{\kappa} \Gamma(f_2, f_R)
 - \frac{\delta}{\e \kappa} \Gamma(f_R,f_R) - \frac{\e}{\delta \kappa} \Gamma(f_2,f_2)
   \\
   =& (\ref{est:Pf_3}) + (\ref{est:I-Pf_3})
   + \frac{\e}{\delta} (\ref{ptPf2}) \langle v-\e u \rangle^2 e^{- \frac{|v-\e u|^2}{4}}
   + \frac{\e}{\delta \kappa} (\ref{ptI-Pf2}) e^{- \varrho |v-\e u|^2}
   + \frac{\e}{\delta} 
 \end{split}
 \Ee

 \Be
 \begin{split}
 |(\ref{eqtn_f_4})| \lesssim& \frac{\e}{\delta} \Big\{
|\p_t ( \tilde{\rho}, \tilde{u}, \tilde{\theta})| + |u|  |\nabla( \tilde{\rho}, \tilde{u}, \tilde{\theta})|
+ \e |  ( \tilde{\rho}, \tilde{u}, \tilde{\theta})|( |\p_t u| + |u| |\nabla u|\\
& \ \ \ 
+ \kappa (|\p_t \nabla u| + \e |\nabla u| |\p_t u|)
+ \kappa |u| (|  \nabla^2 u| + \e |\nabla u| ^2)
 \Big\} \langle v-\e u \rangle^6 \sqrt{\mu}  .\label{r_4}
 \end{split}
 \Ee

  \textit{Expansion of (\ref{eqtn_f_5}):} Directly we compute and bound
  \Be
  \begin{split}\label{r_5}
  (\ref{eqtn_f_5}) & \lesssim 
 \Big\{
 \frac{\e^2}{\delta} |\p_t u| + \frac{\e}{\delta} |\nabla_x u|
  \Big\} \langle v-\e u \rangle^2  |f_2|\\
  &\lesssim \frac{\e}{\delta \kappa^{\frac{1}{2}}} \Big[
  \e| \kappa^{\frac{1}{2}} \p_t u  |  + | \kappa^{\frac{1}{2}} \nabla  u  | 
   \Big]
   \Big[
   \kappa^{\frac{1}{2}} |   \kappa^{\frac{1}{2}}  \nabla u| + |(\tilde{\rho}, \tilde{u}, \tilde{\theta})|
   \Big]  \langle v-\e u \rangle^4 \sqrt{\mu}.
 \end{split} \Ee

 Finally we conclude that

  From the average lemma and the $L^6$-bound of (\ref{L6}), we are able to bound (\ref{nonlinear_1}) by 
  \Be
  \| P f_R \|_{L^\infty_t L^6_x}
  \Ee

 as 
\Be
 \sqrt{  \int^t_0 \delta \kappa^{-\frac{1}{2}} \| \mathbf{P} f_R \|_{L^4_{x,v}}^2
 }
 \sqrt{  \int^t_0 \| 
 \e^{-1} \kappa^{-\frac{1}{2}}
  (\mathbf{I }-\mathbf{P}) f_R\|_{\nu}^2}
\Ee

order to bound (\ref{nonlinear_1}) we consider a bound of $\mathbf{P}f_R$.


which implies the leading order term of $(v- \e u) \cdot \nabla_x    \mathbf{P} f_2$ of (\ref{H_NS}) contributes the pressure $p$ of (\ref{NS_k}).


\hide
We claim that 
\begin{align}
 \langle \varphi_\ell \sqrt{\mu},  (\ref{H_NS}) \rangle
   = \frac{1}{\delta} \{
  \p_t u_\ell  + u  \cdot \nabla_x u_\ell  -  \eta_0   \kappa \Delta_x u _\ell + \p_\ell p 
  \}
   + \frac{\e}{2\delta}\{
  - \tilde{u} \cdot \p_\ell u + \p_\ell u_\ell \tilde{u}_\ell
  \} 
    \ \ \text{for } \ell=1,2,3,
  \label{NS_ell}
 \\
  \langle  \varphi_\ell \sqrt{\mu},  (\ref{H_NS}) \rangle 
  =  \langle  \varphi_\ell \sqrt{\mu}, \frac{\kappa}{\delta} R_1\rangle 
  = \frac{\kappa \e }{\delta} O(|\nabla_x u|^2)
    \ \ \text{for } \ell=0,4
  . \label{Bous_ell}
\end{align}\unhide

Clearly 
\Be\label{est:f_2}
|f_2| \lesssim \{\kappa^{\frac{1}{2}} |\kappa^{\frac{1}{2}} \nabla u| + |(\tilde{\rho}, \tilde{u}, \tilde{\theta})|\} \langle v-\e u \rangle^2 \sqrt{\mu}.
\Ee

Now we consider the last term of (\ref{H_NS}). 
 Expand 
 \[
 (v-\e u )\cdot \nabla_x L^{-1} \left(  \sum_{i,j=1}^3 \varphi_i  \varphi_j  \sqrt{\mu}  - \frac{|v-\e u |^2}{3} \delta_{ij}   \sqrt{\mu} \right)\p_i u_j 
 =  \sum_{i,j,k=1}^3  \varphi_k  L^{-1}  
\Big(
 \varphi_i  \varphi_j  \sqrt{\mu} - \frac{|v-\e u |^2}{3} \delta_{ij}   \sqrt{\mu}
 \Big)
 \p_k \p_i u_{  j} + r_1,
 \]
where $r_1:=\sum_{i,j,k=1}^3  \varphi_k  L^{-1}  
 \p_k\Big(
 \varphi_i  \varphi_j  \sqrt{\mu} - \frac{|v-\e u |^2}{3} \delta_{ij}   \sqrt{\mu}
 \Big)
 \p_i u_{  j}$. Clearly $|r_1| \lesssim \e |\nabla_x u|^2 \langle v-\e u \rangle^4 \sqrt{\mu}$. For evenness/oddness $\langle \varphi_\ell \sqrt{\mu}, r_1 \rangle =0$ for all $\ell=1,2,3$.

  Note that $L^{-1}$ is symmetric under any orthonormal transformation, which implies that $L^{-1}$ preserves evenness/oddness. From this fact, for $\ell=0,4$, $(\ref{f_2:main})=0$ since $\varphi_\ell \varphi_k A_{ij}$ is always odd in at least one among $\varphi_1, \varphi_2, \varphi_3$ when $i,j,k  \in \{1,2,3\}.$ For $\ell=1,2,3$
 \begin{align}
(\ref{f_2:main})
%
 %
 =  \sum_{i ,  j,k=1}^3 \langle  \varphi_\ell  \varphi_k \sqrt{\mu},  A_{ij}
\rangle  \p_{k} \p_i u_{ j}
=   \sum_{i , j,k=1}^3 \left\langle 
 \varphi_\ell  \varphi_k \sqrt{\mu}- \frac{|v- \e u|^2}{3} \delta_{\ell k}
 \sqrt{\mu}
,  A_{ij}
\right\rangle  \p_{k} \p_i u_{ j} =  \sum_{i ,  j, k=1}^3 \left\langle 
L A_{\ell k}
,  A_{ij}
\right\rangle  \p_{k} \p_i u_{  j}. 
\label{u_xx}
 \end{align}

 \hide
 
  \begin{align*}
 &\langle |v |^2  \sqrt{\mu}, v\cdot \nabla_x L^{-1} (\mathbf{I} - \mathbf{P}) (v\cdot \nabla_x u_\kappa \cdot v \sqrt{\mu})\rangle \\
  =& 2 \sum_i \left\langle  v_i \frac{|v |^2-5}{2}  \sqrt{\mu},  \p_i L^{-1} (\mathbf{I} - \mathbf{P}) (v\cdot \nabla_x u_\kappa \cdot v \sqrt{\mu})\right\rangle \\
   =& 2 \sum_i \left\langle B_i,   (v\cdot \nabla_x  \p_iu_\kappa \cdot v \sqrt{\mu})\right\rangle\\
   =& 2 \sum_{i,j,k} \langle B_i, v_k v_j \sqrt{\mu} \rangle  \p_k \p_i u_{\kappa,j}\\
   =&0, 
 \end{align*}
where we have used that $B_i$ is odd in $v_i$. Hence we check (\ref{Bous_ell}) for $v\cdot \nabla_x L^{-1} (\mathbf{I} - \mathbf{P}) (v\cdot \nabla_x u_\kappa \cdot v \sqrt{\mu})$ contribution. \unhide

%
%
\hide
From $L^{-1} \Gamma (\mathbf{P}f, \mathbf{P}f)= \frac{( \mathbf{P} f)^2}{2\sqrt{\mu}}$, we derive that 
\begin{align}\notag
 L^{-1} \Gamma (u_\kappa \cdot v \sqrt{\mu}, u_\kappa \cdot v \sqrt{\mu})  
 = (\mathbf{I} - \mathbf{P}) ( \frac{1}{2 \sqrt{\mu}} \{u_\kappa \cdot v \sqrt{\mu}\}^2) 
 =  \frac{1}{2}(\mathbf{I} - \mathbf{P})
\left(  \sum_{i,j} u_{\kappa,i} u_{\kappa, j} v_i v_j \sqrt{\mu} \right) 
=  \frac{1}{2} 
  \sum_{i,j} u_{\kappa,i} u_{\kappa, j} \{v_i v_j - \frac{|v|^2}{3} \delta_{ij} \}\sqrt{\mu}  ,
\end{align}
where we have used   
 \[
 (\mathbf{I} - \mathbf{P})( v_i v_j \sqrt{\mu} ) 
=  \Big\{
v_iv_j - \frac{|v|^2}{3}  \delta_{ij}
\Big\} \sqrt{\mu}.
 \]

Clearly from oddness of the functions we derive that 
 \begin{align*}
 \left\langle
 \sqrt{\mu}, v\cdot \nabla_x L^{-1} (\Gamma (u_\kappa\cdot v \sqrt{\mu},u_\kappa\cdot v \sqrt{\mu} ))
 \right\rangle 
 = \left\langle
 \sqrt{\mu}, v\cdot \nabla_x   \frac{1}{2} 
\left(  \sum_{i,j} u_{\kappa,i} u_{\kappa, j} \{v_i v_j - \frac{|v|^2}{3} \delta_{ij} \} \sqrt{\mu} \right)\right\rangle 
=  0,\\
 \left\langle
|v|^2 \sqrt{\mu}, v\cdot \nabla_x L^{-1} (\Gamma (u_\kappa\cdot v \sqrt{\mu},u_\kappa\cdot v \sqrt{\mu} ))
 \right\rangle 
 = \left\langle
|v|^2  \sqrt{\mu}, v\cdot \nabla_x   \frac{1}{2} 
\left(  \sum_{i,j} u_{\kappa,i} u_{\kappa, j} \{v_i v_j - \frac{|v|^2}{3} \delta_{ij} \} \sqrt{\mu} \right)\right\rangle 
=  0,
 \end{align*}
and hence we check (\ref{Bous_ell}) for $v\cdot \nabla_x   \frac{1}{2} 
\left(  \sum_{i,j} u_{\kappa,i} u_{\kappa, j} \{v_i v_j - \frac{|v|^2}{3} \delta_{ij} \} \sqrt{\mu} \right) $ contribution.

Now we compute 
  \begin{align} 
 \left\langle
v_\ell \sqrt{\mu}, \sum_k v_k \p_k L^{-1} (\Gamma (u_\kappa\cdot v \sqrt{\mu},u_\kappa\cdot v \sqrt{\mu} ))
 \right\rangle  &= 
 \sum_k  \left\langle v_\ell v_k \sqrt{\mu}
,  
 (\mathbf{I} - \mathbf{P})
\left(  \sum_{i,j}  \p_k u_{\kappa,i} u_{\kappa, j} v_i v_j \sqrt{\mu} \right) 
 \right\rangle\notag \\
 &=  \sum_k  \left\langle
  \{v_\ell v_k  - \frac{|v|^2}{3}
  \delta_{\ell k}
  \}\sqrt{\mu}
,  
    \sum_{i,j}  \p_k u_{\kappa,i} u_{\kappa, j} \{v_i v_j - \frac{|v|^2}{3} \delta_{ij}
    \}
    \sqrt{\mu} 
 \right\rangle\notag \\
& = \sum_{i,j,k} 
 \left\langle
  \{v_\ell v_k  - \frac{|v|^2}{3}
  \delta_{\ell k}
  \}\sqrt{\mu}
,    \{v_i v_j - \frac{|v|^2}{3} \delta_{ij}
    \}
    \sqrt{\mu} 
 \right\rangle  \p_k u_{\kappa,i} u_{\kappa, j}
 \label{vLG}
 .
 \end{align}
From direct computations and $\int_{\R} \frac{e^{- {(v_1)^2}/{2}}}{(2\pi )^{1/2}} \dd v_1=1=\int_{\R}  (v_1)^2\frac{e^{- {(v_1)^2}/{2}}}{(2\pi )^{1/2}} \dd v_1,$ and $\int_{\R}  (v_1)^4\frac{e^{- {(v_1)^2}/{2}}}{(2\pi )^{1/2}} \dd v_1=3$, we have  
 \begin{align}\notag
 & \left\langle
  \{v_\ell v_k  - \frac{|v|^2}{3}
  \delta_{\ell k}
  \}\sqrt{\mu}
,    \{v_i v_j - \frac{|v|^2}{3} \delta_{ij}
    \}
    \sqrt{\mu} 
 \right\rangle = \frac{4}{3} \delta_{\ell k } \delta_{i j} \delta_{i\ell}   - \frac{2}{3} \delta_{\ell k} \delta_{ij} (1- \delta_{i \ell})
+  \delta_{\ell i} \delta_{k j} (1- \delta_{\ell k}) 
+   \delta_{\ell j} \delta_{ki}(1- \delta_{\ell k}) .
  \end{align}
Then we derive 
  \Be\begin{split}
(\ref{vLG}) &=
\frac{4}{3} \p_\ell u_{\kappa, \ell} u_{\kappa, \ell} - \frac{2}{3} \sum_{j (\neq \ell)} \p_ \ell u_{\kappa, j} u_{\kappa, j} + \sum_{j (\neq \ell)} \p_j u_{\kappa, \ell} u_{\kappa, j} 
+ u_\ell \sum_{i (\neq \ell) } \p_i u_{\kappa, i}
\\
&=(\frac{4}{3} -2 ) \p_\ell u_{\kappa, \ell} u_{\kappa, \ell}- \frac{2}{3} \sum_{j (\neq \ell)} \p_ \ell u_{\kappa, j} u_{\kappa, j}
+ \sum_{j  } \p_j u_{\kappa, \ell} u_{\kappa, j} 
+ u_\ell \sum_{i   } \p_i u_{\kappa, i}
\\
&= - \frac{2}{3} \sum_{j  } \p_ \ell u_{\kappa, j} u_{\kappa, j}
+ \sum_{j  } \p_j u_{\kappa, \ell} u_{\kappa, j} 
+ u_\ell \sum_{i   } \p_i u_{\kappa, i}
\\
&= - \frac{1}{3} \p_\ell|u_\kappa|^2
+  u_\kappa \cdot \nabla_x u_{\kappa,\ell}  +  (\nabla \cdot u_\kappa )u_{\kappa, \ell}.\label{vLG1}
  \end{split}\Ee\unhide
%
%
\hide From we end up with reduced equation 
 \Be
 \begin{split}
 \p_t f_R + \frac{1}{\e} v\cdot \nabla_x f_R+ \frac{1}{\e^2 \kappa } Lf_R 
 \\
 =
  \frac{  \delta }{\e\kappa}\Gamma(f_R,f_R)
 - \frac{\e }{\delta} \p_t f_2  - \frac{\kappa}{\delta} (\mathbf{I} - \mathbf{P}) ( v\cdot \nabla_x L^{-1} (v\cdot \nabla_x u_\kappa \cdot v \sqrt{\mu}) ) \\
 +
  \frac{1}{\delta\kappa} \Gamma(f_1,f_2) + \frac{  1}{\e\kappa} \Gamma(f_1,f_R) 
+    \frac{\e }{\delta\kappa} \Gamma({f_2}, f_2) +  \frac{ 1}{\kappa} \Gamma({f_2}, f_R).
 \end{split}
 \Ee\unhide
 Finally we note that 
\begin{align}
 (\mathbf{I} - \mathbf{P}) \big( (v-\e u) \cdot \nabla_x \mathbf{P} f_2\big)= \{O(|(  \nabla\tilde{\rho},\nabla \tilde{u},\nabla \tilde{\theta})|)  + \e O(|\nabla_x u|  |(\tilde{\rho}
, \tilde{u}, \tilde{\theta})| )\} \langle v-\e u \rangle^4  \mu^{\frac{1}{2}}, \label{r2}\\
 (\mathbf{I} - \mathbf{P}) \big( (v-\e u) \cdot \nabla_x  (\mathbf{I} - \mathbf{P})f_2\big)
= \{O(|\nabla^2 u|) + \e O(|\nabla u|^2)v\} \langle v-\e u \rangle^4 \mu^{\frac{1}{2}}.\label{r3}
\end{align}

\Be
\int_0^t\iint_{\O \times \R^3}\frac{\delta}{\e\kappa} \Gamma( f_R,(\mathbf{I} - \mathbf{P})f_R) (\mathbf{I} - \mathbf{P}) f_R.\label{nonlinear_2}
\Ee
 
  \hide
 
Clearly $\langle v_\ell \sqrt{\mu}, \p_t u_\kappa \cdot v \sqrt{\mu}\rangle = \p_t u_\kappa$. From $ \langle   \sqrt{\mu}, \p_t u_\kappa \cdot v \sqrt{\mu}\rangle=0= \langle  |v|^2 \sqrt{\mu}, \p_t u_\kappa \cdot v \sqrt{\mu}\rangle$, we check (\ref{Bous_ell}) for $\p_t u_\kappa$ contribution.

 Since $u_\kappa$ solves (\ref{NS_k}) we derive (\ref{H3}). Finally we derive an equation for $f_R$, by inserting (\ref{H1}), (\ref{f_2}), and (\ref{H3}) into (\ref{eqtn_f}), as 

 where $f_1, f_2, R_1, R_2$ are defined in (\ref{f_1}), (\ref{f_2}), (\ref{H3}) respectively. \unhide

 \hide

First we expand 
\Be\begin{split}\notag
&\p_t \mu_\e + \frac{1}{\e} {v}\cdot \nabla_x  \mu_\e\\
 =&  \e \p_t u_\kappa \cdot (v-\e u_\kappa)\mu_\e + (v\cdot \nabla_x u_\kappa) \cdot (v- \e u_\kappa) \mu_\e\\
 =&v\cdot \nabla_x u_\kappa \cdot v \mu_\e +  \e \{ \p_t u_\kappa \cdot v- v \cdot \nabla_x u_\kappa \cdot u_\kappa \} \mu_\e + \e^2 (-\p_t u_\kappa \cdot u_\kappa) \mu_\e 
\end{split}\Ee
We denote 
\begin{align}
\frac{\p_t \mu_\e + \frac{1}{\e} {v}\cdot \nabla_x  \mu_\e}{\sqrt{\mu_\e}} =&
\underbrace{v\cdot \nabla_x u_\kappa \cdot v \sqrt{\mu_\e}}_{=(I-\mathbf{P})\mathfrak{M}_1} +  \e \underbrace{ \{ \p_t u_\kappa \cdot v- v \cdot \nabla_x u_\kappa \cdot u_\kappa \} \sqrt{\mu_\e }+ \e  (-\p_t u_\kappa \cdot u_\kappa) \sqrt{\mu_\e}}_{\mathbf{P}\mathfrak{M}_2}
\notag
\\
 =& (I-\mathbf{P})\mathfrak{M}_1 + \e \mathbf{P}\mathfrak{M}_2\label{M1M2}
\end{align}

\hide\Be
n(x) \cdot \f{1}{\e}\int_{\R^3}
F_\e (t,x,v)  v   
 \dd v \rightarrow u^{(1)}(t,x).\label{hydro_limit}
\Ee
Define 
\Be
\mu_\e : = M_{1+ \e \rho, \e u, 1+ \e \theta }.
\Ee

Expanding by $\e$,
\Be
\Bs\label{mu_e}
\mu_\e = \mu_0+ \e \Big\{\rho + u \cdot v + \theta \frac{|v|^2- 3}{2}\Big\} \mu_0+ \e^2 \varphi_2\mu_0 + \e^{3} \mu_{\e, R} .
\end{split}
\Ee
where 
\Be
\mu_{\e, R}
:=\e^{-3}\int^\e_0 \frac{1}{2!} \frac{\p^3 \mu_\tau}{\p \tau^3} |\e - \tau|^2 \dd \tau 
=  O_{\rho,u, \theta}(\langle v\rangle^6 ).
\Ee

The solution $F_\e$ is found as a form of 
\Be
F_\e = \mu_\e
 + 
  \e^{2+  {\kappa} } F_{ 2} 
 + 
  \e^{3+  {\kappa} } F_{ 3} + \e^{1+ \f{3}{2} \kappa}F_{ R} .
\Ee
The equation of $F_R^\e$ is given by 
\begin{eqnarray*}
&&\p_t F_R+ \frac{1}{\e} v \cdot \nabla_x F_R   - \frac{1}{\e^{2+\kappa}} Q(\mu_\e, F_R)- \frac{1}{\e^{1- \frac{\kappa}{2}}} Q(F_R,F_R)
\label{F_R}
\\
&=&  \frac{\e^{-1}}{\e^{  \frac{3\kappa}{2}  }}\Big\{
- v\cdot \nabla_x \varphi_1
+ Q(\mu_\e, F_2)  
\Big\}\label{F_R1}\\
&+&   \frac{1}{\e^{  \frac{3\kappa}{2}  }}
\Big\{
-\p_t \varphi_1 - v\cdot \nabla_x \varphi_2 -\e^\kappa v\cdot \nabla_x F_2 + Q(\mu_\e, F_3)
\Big\}\label{F_R2}\\
&+&  \frac{\e}{\e^{\frac{3\kappa}{2}}}
 \Big\{- \p_t \varphi_2 - v\cdot \nabla_x \mu_R
\Big\}\label{F_R3}\\
&+& Q(F_2,F_R)+ \e Q(F_3,F_R)
\\
&+& \e^{1- \f{\kappa}{2}} Q(F_2,F_2) +  \e^{2- \f{\kappa}{2}} Q(F_2,F_3)-  \e^{1- \f{\kappa}{2}}\p_t F_2
\\
&+&  \e^{3- \frac{\kappa}{2}}Q(F_3,F_3)- \e^{2- \frac{\kappa}{2}} \p_t F_3
-  \e^{1- \frac{\kappa}{2}} v\cdot \nabla_x  F_3 - \e^{2- \frac{3\kappa}{2}} \p_t \mu_{  R}.\label{F_R4}
\end{eqnarray*}

\hide

 \Be\label{Hilbert_R}
 \Bs
 &{\color{red} \e^{\delta}} \p_t  F^\e_R + {\color{red} \e^{\delta}} \e^{-1} v\cdot \nabla_x F^\e_R
 - \e^{-2-\delta} \left\{Q(\mu_\e, {\color{red} \e^{\delta}}F^\e_R+ \e  F_2
 + \e^{2 } F_3
 )+ Q({\color{red} \e^{\delta}}F^\e_R+ \e F_2+ \e^{2 } F_3,\mu_\e)
 \right\}\\
 =& \ \e^{-1  } Q({\color{red} \e^{\delta}}F^\e_R+ \e  F_2 +  \e^{2 } F_3    , {\color{red} \e^{\delta}}F^\e_R+ \e  F_2+  \e^{2 } F_3 ) \\
 & - \e^{- 1-  \delta  } 
 [
 \p_t   + \e^{-1} v\cdot \nabla_x]
 \left\{
  \mu_\e + \e^{2+\delta}F_2 + \e^{3+ \delta} F_3
 \right\}
 .
 \end{split}
 \Ee
 Note that for $\delta>0$
 \Be\Bs\label{Hilbert_small}
 &- \e^{- 1-  {\delta} } 
 [
 \p_t   + \e^{-1} v\cdot \nabla_x]
 \left\{
   \e^{2+\delta}F_2 + \e^{3+ \delta} F_3
\right\}\\
=& -   v\cdot \nabla_x F_2- \e^{1 }\{\p_t F_2 +v\cdot \nabla_x F_3\}
- \e^{ 2 } \p_t F_3 
=   O(1).
\end{split}
\Ee
 \unhide
  \unhide

 \hide

Then 
\Be
\begin{split}\label{transp_mu}
\f{\p_t \mu_\e  + \f{1}{\e} v\cdot \nabla_x \mu_\e}{\mu_\e}= & \e \p_t \varphi^{(1)}_1 + \f{\e^2}{2} \p_t \varphi^{(2)}_1 - \e^2 \p_t \varphi^{(1)}_2\\
&+ v\cdot \nabla_x \varphi^{(1)}_1 + \f{\e}{2} v\cdot \nabla_x \varphi^{(2)}_1
 -\e \varphi_2^{(1,1)}- \f{\e^2}{2} \varphi_2^{(1,2)} + \e^2 \varphi_3 
+ \e^3 \varphi_R,
\end{split}
\Ee
where
\hide Using
$\f{1}{1+ \e a  } = 
\sum_{j=0}^{k-1}   ( \e a   )^j  + O(\e^k).$
 we have 
\begin{equation}\notag
\begin{split}
&\f{ [\p_t + \f{1}{\e} v\cdot \nabla_x] \mu_\e  }{\mu_\e}\\
=& \Big(\e \p_t \rho^{(1)} + \f{\e^2}{2} \p_t \rho^{(2)}\Big)
\Big\{1 -\Big(\e \p_t \rho^{(1)} + \f{\e^2}{2} \p_t \rho^{(2)}\Big) 
\Big\} \\
&+  \Big(\e \p_t \theta^{(1)} + \f{\e^2}{2} \p_t \theta^{(2)}\Big) \Big\{1 - \Big(\e \p_t \theta^{(1)} + \f{\e^2}{2} \p_t \theta^{(2)}\Big)\Big\} \frac{
\left|
\f{v- \e u^{(1)} -\f{\e^2}{2} u^{(2)}}{ \sqrt{1+ \e \theta^{(1)}+\f{\e^2}{2} \theta^{(2)} }}
\right|^2
-3}{2}
\\
&+O(\e^3)
\end{split}
\end{equation}

We denote the expansion of $\frac{ [\p_t + \f{1}{\e} v\cdot \nabla_x] \mu_\e  }{\mu_\e}$ in $\e$  
\Be
[\p_t + \f{1}{\e} v\cdot \nabla_x ]\mu_\e = v\cdot \nabla_x \varphi^{(1)} +   \e \p_t \varphi^{(1)}_1 \mu_\e+ \f{\e}{2} v\cdot \nabla_x \varphi^{(2)}_1
\Ee
where  

\unhide
\begin{equation}\notag
\begin{split}
\varphi^{(i)}_1: =
  &  \ 
 \rho^{(i)} + u^{(i)} \cdot v + \theta^{(i)} \f{|v|^2-3}{2} , \\
 \varphi^{(i,j)}_2: = &      \ 
  v\cdot 
\rho^{(i) } \nabla \rho^{(j)}  
+ v \cdot \nabla u^{(j)} \cdot u^{(i)}
+  v\cdot \theta^{(i)} \nabla u^{(j)} \cdot v 
\\
&  +  
\Big\{
2 \theta^{(i)} \f{|v|^2-3}{2} + u^{(i) } \cdot v + \f{3}{2} \theta^{(i)}
\Big\}
 \nabla \theta^{(j)} \cdot v,\\
 \varphi^{(1)}_2 : = &  
\f{(\rho^{(1)})^2}{2} + \f{|u^{(1)}|^2}{2} + \f{3 (\theta^{(1)})^2}{4}
+   (\theta^{(1)} u^{(1)}) \cdot v
+ 
  (\theta^{(1)})^2
\f{|v|^2-3}{2}\\
\varphi_3 : =  &     
 - \f{1}{2} \rho^{(2)} \nabla \rho^{(1)} \cdot v- \f{1}{2} u^{(2)} \cdot \nabla u^{(1)} \cdot v+ 
 \theta^{(1) } u^{(1)} \cdot \nabla_x u^{(1)} \cdot v
 - \f{1}{2}  v \cdot \theta^{(2)} \nabla  u^{(1)} \cdot v
 \\
   &+  
   \Big\{
   \f{|u^{(1)}|^2}{2} -\f{3}{4} \theta^{(2)} + 3 (\theta^{(1)})^2 + \big(2 \theta^{(1)} u^{(1)} - \f{1}{2} u^{(2)}\big)\cdot v
   - \big(
   (\theta^{(1)})^2 + \theta^{(2)}
   \big) \f{|v|^2-3}{2}
   \Big\} \nabla_x \theta^{(1)} \cdot v.
\end{split}
\end{equation}
\hide

\Be 
 \f{\f{1}{\e} v\cdot \nabla_x \mu_\e}{\mu_\e} 
 =  v\cdot \nabla_x \varphi^{(1)}_1 + \f{\e}{2} v\cdot \nabla_x \varphi^{(2)}_1
 -\e \varphi_2^{(1,1)}- \f{\e^2}{2} \varphi_2^{(1,2)} + \e^2 \varphi_3 
\hide    
 & - \e 
\Big\{
v \cdot \rho^{(1)}\nabla_x \rho^{(1)} + v\cdot \nabla_x u^{(1)} \cdot u^{(1)}
+ v\cdot \theta^{(1)} \nabla_x u^{(1)} \cdot v\Big\}
\\
& 
- \e  \Big\{
2 \theta^{(1)} \f{|v|^2-3}{2} + u^{(1)} \cdot v + \f{3}{2} \theta^{(1)}
\Big\} \nabla_x \theta^{(1)} \cdot v
\\
&  - \f{\e^2}{2} 
\Big\{ v\cdot 
\rho^{(1) } \nabla \rho^{(2)}  
+ v \cdot \nabla u^{(2)} \cdot u^{(1)}
+  v\cdot \theta^{(1)} \nabla u^{(2)} \cdot v\Big\}
\\
&  - \frac{\e^2}{2}
\{
2 \theta^{(1)} \f{|v|^2-3}{2} + u^{(1) } \cdot v + \f{3}{2} \theta^{(1)}
\}
 \nabla \theta^{(2)} \cdot v\\
 & + \e^2\Big\{
 - \f{1}{2} \rho^{(2)} \nabla \rho^{(1)} \cdot v- \f{1}{2} u^{(2)} \cdot \nabla u^{(1)} \cdot v+ 
 \theta^{(1) } u^{(1)} \cdot \nabla_x u^{(1)} \cdot v
 - \f{1}{2}  v \cdot \theta^{(2)} \nabla  u^{(1)} \cdot v
 \Big\}\\
   &+ \e^2 
   \Big\{
   \f{|u^{(1)}|^2}{2} -\f{3}{4} \theta^{(2)} + 3 (\theta^{(1)})^2 + \big(2 \theta^{(1)} u^{(1)} - \f{1}{2} u^{(2)}\big)\cdot v
   - \big(
   (\theta^{(1)})^2 + \theta^{(2)}
   \big) \f{|v|^2-3}{2}
   \Big\} \nabla_x \theta^{(1)} \cdot v\\  & \unhide
 + O(\e^3)
 \Ee

\unhide

\unhide

The equation of $f_R$ is given by 
\begin{eqnarray*}
\delta \e \sqrt{\mu_\e}\{
 \p_t f_R+ \frac{1}{\e} \munderbar{v}\cdot \nabla_x f_R \}  
 + \p_t \mu_\e + \frac{1}{\e} v\cdot \nabla_x \mu_\e 
 + \delta \e \frac{ \p_t \mu_\e + \frac{1}{\e} v\cdot \nabla_x \mu_\e}{2 \sqrt{\mu_\e}} f_R
 - \frac{2}{\kappa\e^{2 }} Q(\mu_\e , \delta \e \sqrt{\mu_\e} f_R)\\ =
  \frac{1}{\kappa \e^2
} Q( \delta \e \sqrt{\mu_\e} f_R, \delta \e \sqrt{\mu_\e} f_R).
 \end{eqnarray*}
 or equivalently 
 \Be
 \begin{split}
  \p_t f_R+ \frac{1}{\e} \munderbar{v}\cdot \nabla_x f_R - \frac{2}{\kappa\e^{2 }} \frac{1}{\sqrt{\mu_\e}} Q(\mu_\e ,  \sqrt{\mu_\e} f_R)\\ =
  \frac{\delta }{\kappa \e 
} 
 \frac{1}{\sqrt{\mu_\e}}
Q(  \sqrt{\mu_\e} f_R,  \sqrt{\mu_\e} f_R)
+ \frac{ \p_t \mu_\e + \frac{1}{\e} v\cdot \nabla_x \mu_\e }{\delta \e \sqrt{\mu_\e}} +  \frac{ \p_t \mu_\e + \frac{1}{\e} v\cdot \nabla_x \mu_\e}{2 \mu_\e } f_R
 \end{split}
 \Ee
 or 
  \Be\label{eqtn_f}
 \begin{split}
  \p_t f_R+ \frac{1}{\e}  v\cdot \nabla_x f_R - \frac{1}{\kappa\e^{2 }}
  L_\e f_R \\ =
  \frac{\delta }{\kappa \e 
} 
\Gamma_\e(    f_R,    f_R)
+
\frac{1}{\delta \e } (I-\mathbf{P})\mathfrak{M}_1 + \frac{1}{\delta  }  \mathbf{P}\mathfrak{M}_2
+  \frac{ \p_t \mu_\e + \frac{1}{\e} v\cdot \nabla_x \mu_\e}{2 \mu_\e } f_R
 \end{split}
 \Ee
 
 \hide

where we have used an expansion-in-$\e$, $
\p_t \mu_\e   + \f{1}{\e} \munderbar{v}\cdot \nabla_x \mu_\e =    \varphi_0  + \e \varphi_1 + \e^2 \varphi_2 + \e^3 \varphi_R.$ Here 
\Be\label{varphi_1}
\varphi_0 = \munderbar{v}\cdot \nabla_x \varphi^{(1)}_1  \mu_\e, \ \ \ 
\varphi_1 =  \p_t \varphi_1^{(1)}  \mu_\e- \varphi_2^{(1,1)} \mu_\e
  + \f{1}{2}\munderbar{v}\cdot \nabla_x   \varphi_1^{(2)} \mu_\e,
\Ee
where
\begin{equation}\notag
\begin{split}
\varphi^{(i)}_1: =
  &  \ 
 \rho^{(i)} + u^{(i)} \cdot v + \theta^{(i)} \f{|v|^2-3}{2} , \\
 \varphi^{(i,j)}_2: = &      \ 
  v\cdot 
\rho^{(i) } \nabla \rho^{(j)}  
+ v \cdot \nabla u^{(j)} \cdot u^{(i)}
+  v\cdot \theta^{(i)} \nabla u^{(j)} \cdot v 
 +  
\Big\{
2 \theta^{(i)} \f{|v|^2-3}{2} + u^{(i) } \cdot v + \f{3}{2} \theta^{(i)}
\Big\}
 \nabla \theta^{(j)} \cdot v.
\end{split}
\end{equation}

\color{red}{We can rewrite this as the following:} \color{black}{Let}
\begin{eqnarray}
\varphi_2 ^{(i)} = | u ^{(i)} + \theta^{(i)}v |^2 + \left ( \rho^{(i)} \right )^2 -\frac{3}{2}  \left ( \theta^{(i)} \right )^2, \\
\varphi_3 ^{(1,2)} = \frac{1}{2} \theta^{(1)} | u^{(1)} + \theta^{(1)} v |^2 - \frac{1}{2} (u^{(1)} + \theta^{(1)} v ) \cdot (u^{(2)} + \theta^{(2)}v ) \\
+\left \{ \frac{1}{3} \left ( \rho^{(1)} \right )^3 - \frac{1}{2} \rho^{(1)} \rho^{(2) } \right \} - \frac{3}{2} \left \{ \frac{1}{3} \left ( \theta^{(1)} \right )^3 - \frac{1}{2} \theta^{(1)} \rho^{(2) } \right \}
\end{eqnarray}
then
\begin{eqnarray}
\varphi_0 = \munderbar{v} \cdot \nabla_x \varphi_1 ^{(1)} \mu_\e, \\
\varphi_1 = \left \{\partial_t \varphi_1 ^{(1)} + \munderbar{v} \cdot \nabla_x \left ( \frac{\varphi_1 ^{(2) }- \varphi_2 ^{(1)} }{2} \right ) \right \} \mu_\e, \\
\varphi_2 = \left \{ \partial_t \left ( \frac{\varphi_1 ^{(2) }- \varphi_2 ^{(1)} }{2} \right ) + \munderbar{v} \cdot \nabla_x \varphi_3 ^{(1,2) } \right \} \mu_\e.
\end{eqnarray}

We set all (\ref{F_R1}), (\ref{F_R2}), and (\ref{F_R3}) to be vanished by solving
\begin{eqnarray}
   v\cdot \nabla_x \varphi_1^{(1)} \mu_\e
&\!\!\! =\!\! \! &2Q(\mu_\e, F_2)  
   ,
 \label{H1=0}
 \\
 \p_t \varphi_1^{(1)}  \mu_\e- \varphi_2^{(1,1)} \mu_\e
  + \f{1}{2}v\cdot \nabla_x   \varphi_1^{(2)} \mu_\e 
&\!\! \!=\!\!\!  & 2 Q(\mu_\e,  F_3)
 \label{H2=0},\\
\varphi_2 +\f{\delta}{\e}v\cdot \nabla_x F_2
&\!\! \!=\!\! \! & 2Q(\mu_\e, F_4)    .\label{H3=0}
\end{eqnarray}
By the Fredholm alternative the solvability conditions of (\ref{H1=0}), (\ref{H2=0}), and (\ref{H3=0}) are that the left hand side of each equations are entirely included in $\left (\text{Ker} L \right )^{\color{red}{\perp}}$. (See the information on $L$ and its properties in (\ref{L})-(\ref{P})) In particular we solve
\Bes
&\mathbf{P} ( v\cdot \nabla_x \varphi_1^{(1)}\mu_\e)=0,&\\
&\mathbf{P} ( \p_t \varphi_1^{(1)}  \mu_\e- \varphi_2^{(1,1)} \mu_\e)=0,\ \ \ \mathbf{P} ( v\cdot \nabla_x   \varphi_1^{(2)} \mu_\e ) =0,&\\
&\mathbf{P} (   \varphi_2
 +2\f{\delta}{\e} v\cdot \nabla_x F_2)=0.&
\Ees

From straightforward computations 
\Be
\int_{\R^3} \mu_\e \begin{bmatrix}1 \\ v \\ \frac{|v|^2-3}{2}
\end{bmatrix}  \dd v = \begin{bmatrix}
1+ \e \rho^{(1)} + \f{\e^2}{2} \rho^{(2)}  \\
 (1+ \e \rho^{(1)} + \f{\e^2}{2} \rho^{(2)} ) (\e u^{(1)} + \f{\e^2}{2} u^{(2)})  \\
\frac{  |\e u^{(1)  } + \f{\e^2}{2} u^{(2)}|^2 -3}{2}(1+ \e \rho^{(1)} + \f{\e^2}{2} \rho^{(2)})
+   \frac{ 3}{ 2   }   (1+ \e \theta^{(1)} +  \f{\e^2}{2} \theta^{(2)} ) (1+ \e \rho^{(1)} + \f{\e^2}{2} \rho^{(2)}) 
\end{bmatrix},
\Ee
and 
\Be
\Bs
&\int_{\R^3} \mu_\e v_i \begin{bmatrix}1 \\ v_j \\ \frac{|v|^2-3}{2}
\end{bmatrix}  \dd v\\
 =& \begin{bmatrix}
  (1+ \e \rho^{(1)}+ \f{\e^2}{2} \rho^{(2)}) (\e u^{(1)}_i + \f{\e^2}{2} u^{(2)}_i)
 \\
 \delta_{ij} (1+ \e \theta^{(1)} + \f{\e^2}{2} \theta^{(2)} ) (1+ \e \rho^{(1)} + \f{\e^2}{2} \rho^{(2)}) +   (1+ \e \rho^{(1)} + \f{\e^2}{2} \rho^{(2)}) (\e u_i^{(1)} + \f{\e^2}{2} u_i^{(2)} )(\e u_j^{(1)} + \f{\e^2}{2} u_j^{(2)} )
 \\ 
 5 (1+ \e \rho^{(1)} + \f{\e^2}{2} \rho^{(2)}) (1+ \e \theta^{(1)} + \f{\e^2}{2} \theta^{(2)})  (\e u_i^{(1)}  + \f{\e^2}{2} u_i^{(2)}) 
 - \frac{3}{2} (1+ \e \rho^{(1)} + \f{\e^2}{2} \rho^{(2)} )  (\e u_i^{(1)} + \f{\e^2}{2} u_i ^{(2)})
\end{bmatrix}\\
&+ \begin{bmatrix}
0 \\
0 \\
   (1+ \e \rho^{(1)} + \f{\e^2}{2} \rho^{(2)}) |\e u^{(1)} + \f{\e^2}{2} u^{(2)}|^2
 ( \e u_i^{(1)} + \f{\e^2}{2} u_i^{(2)} )
\end{bmatrix}
\end{split}\Ee

Then one can easily find set of equations which is equivalent to as  
\Be\label{Bousn}
\nabla_x \cdot u^{(i)} =0, \ \ \nabla_x (\rho^{(i)} + \theta^{(i)})=0,\ \ \text{for} \ i=1,2,
\Ee
and  
 \Be\label{Euler1}
 \f{\p}{\p t} \begin{bmatrix}
\rho^{(1)}\\
u^{(1)}\\
 \theta^{(1)}
\end{bmatrix}
+ \begin{bmatrix}
u^{(1)} \cdot \nabla_x \rho^{(1)}\\
u^{(1)} \cdot \nabla_x u^{(1)} +\nabla_x p^{(1)}\\
-  u^{(1)} \cdot \nabla_x \rho^{(1)}
\end{bmatrix}=0,
\Ee
and
\Be\label{Euler2}
\f{\p}{\p t} 
\begin{bmatrix}
\rho^{(2)} \\
u^{(2)} + 2 \rho^{(1)} u^{(1)} \\
\theta^{(2)} + 2 \theta^{(1)} \rho^{(1)} + \f{2}{3} |u^{(1)}|^2
\end{bmatrix}
+  
\begin{bmatrix}
u^{(2)} \cdot \nabla_x \rho^{(1)} + u^{(1)} \cdot \nabla_x \rho^{(2)}\\
u^{(1) } \cdot \nabla_x u^{(2)} + u^{(2)} \cdot \nabla_x u^{(1)} + 2 u^{(1)} \cdot \nabla_x  (\rho^{(1)} u^{(1)})
+\nabla_x p^{(2)}\\
- u^{(2)} \cdot \nabla_x  \rho^{(1)} 
- u^{(1)} \cdot \nabla_x  \rho^{(2)}  
+ \f{4}{3} u^{(1)} \cdot \nabla_x |u^{(1)}|^2
\end{bmatrix}=0,
\Ee
where $p^{(1)}:= \frac{1}{2} \rho^{(2)} + \f{1}{2} \theta^{(2)} + \theta^{(1)} \rho^{(1)}$ and $p^{(2)}
:=\theta^{(1)} \rho^{(2)} + \theta^{(2)} \rho^{(1)}
$.  

Note that (\ref{Euler2}) is a linear system. It is a direct application of the energy estimate that the classical solution of (\ref{Euler2}) exists as long as the classical solution of (\ref{Euler}) persists {\color{red}(Check!!)}. Finally by the Fredholm alternative we have unique solutions $F_2,F_3,F_4 \in (\text{Ker} L)^\perp$ to (\ref{H1=0}), (\ref{H2=0}), (\ref{H3=0}) respectively.

\unhide
 \hide

\Bes
|(\mathbf{I} - \mathbf{P}) F_2| \lesssim 1,\\
|\varphi^{(2)}_1| \lesssim \f{\delta}{\e},\\
|(\mathbf{I} - \mathbf{P}) F_4| \lesssim  \f{\delta}{\e},\\
|\varphi^{(1)}_1| \lesssim1,\\
|(\mathbf{I} - \mathbf{P}) F_3| \lesssim1.
\Ees

Focus on (\ref{H3=0}). For the solvability 
\Be
\mathbf{P} \big( \p_t ( \varphi_2^{(1)} + \varphi_2^{(2)} ) +
  v\cdot \nabla_x ( \varphi_3^{(2)}+\varphi^R_3) +\e^{\kappa-1} v\cdot \nabla_x F_2\big) =0,
\Ee
which is equivalent to 
\Be
 \mathbf{P} \p_t\varphi_2^{(2)} + \nabla_x \cdot (\mathbf{P} v \varphi^{(2)}_3)
=- \mathbf{P} \p_t \varphi_2^{(1)} - \nabla_x \cdot (\mathbf{P} v \varphi^{R}_3)- \e^{\kappa-1} \nabla_x \cdot (\mathbf{P} v F_2).
\Ee

Let us choose that $F_2, F_3 \in N^\perp$ so that 
\Be\begin{split}
&\iint_{\O \times \R^3} F^\e(t) \begin{bmatrix}
1 \\ 
|v|^2
\end{bmatrix} \dd v \dd x\\
 = & 
\iint_{\O \times \R^3} F^\e(0) \begin{bmatrix}
1 \\ 
|v|^2
\end{bmatrix} \dd v \dd x \\
=& \iint_{\O \times \R^3} \mu^\e \begin{bmatrix}
1 \\ 
|v|^2
\end{bmatrix} \dd v \dd x + \e^{1+ \frac{\delta}{2}}\iint_{\O \times \R^3}F^\e_R \begin{bmatrix}
1 \\ 
|v|^2
\end{bmatrix} \dd v \dd x \\
= & \begin{bmatrix}
1+ \e \rho
\\
(1+ \e \rho) \e |u|^2 + 3 (1+ \e \rho) (1+ \e \theta)
\end{bmatrix}+ \e^{1+ \frac{\delta}{2}}\iint_{\O \times \R^3}F^\e_R \begin{bmatrix}
1 \\ 
|v|^2
\end{bmatrix} \dd v \dd x.
\end{split}\Ee
We set 
\Be
\iint_{\O \times \R^3}F^\e_R \begin{bmatrix}
1 \\ 
|v|^2
\end{bmatrix} \dd v \dd x= \begin{bmatrix}
0\\ 
0
\end{bmatrix}.
\Ee

The solvability of (\ref{H1}) is guaranteed if (Fredholm) 
\Be
\int_{\R^3} v\cdot \nabla_x  \Big\{\rho + u \cdot v + \theta \frac{|v|^2- 3}{2}\Big\} \mu_0\begin{bmatrix}
 1 \\
 v\\
 \frac{|v|^2- 3}{2}
\end{bmatrix}
\dd v = 0 ,
\Ee
which can be read as 
\Be\label{incompress_Bousnq}
\nabla_x \cdot u =0, \ \ \nabla_x (\rho+ \theta ) =0. 
\Ee

To solve (\ref{H2}) we need 
\Be\label{Fredholm}
\Bs
0= &\int_{\R^3} \left\{\p_t \Big\{\rho + u \cdot v + \theta \frac{|v|^2- 3}{2}\Big\}\mu_0
 + v\cdot \nabla_x \varphi_2 \mu_0 
\right\} \begin{bmatrix}
 1 \\
 v\\
 \frac{|v|^2- 3}{2}
\end{bmatrix}
 \dd v\\
 =&\p_t  \int_{\R^3}\Big\{\rho + u \cdot v + \theta \frac{|v|^2- 3}{2}\Big\}\mu_0 \begin{bmatrix}
 1 \\
 v\\
 \frac{|v|^2- 3}{2}
\end{bmatrix}
 \dd v
 + \nabla_x \cdot \int_{\R^3}  \varphi_2 v \mu_0  \begin{bmatrix}
 1 \\
 v\\
 \frac{|v|^2- 3}{2}
\end{bmatrix}
 \dd v   \\
 = & \p_t \int_{\R^3}   \begin{bmatrix}
 1 \\
 v\\
 \frac{|v|^2- 3}{2}
\end{bmatrix}\f{\p \mu_\e}{\p \e}\Big|_{\e=0} \dd v
+ \nabla_x \cdot \int_{\R^3} v  \begin{bmatrix}
 1 \\
 v\\
 \frac{|v|^2- 3}{2}
\end{bmatrix} \frac{1}{2} \f{\p^2 \mu_\e}{\p \e^2}\Big|_{\e=0}
\dd v 
\end{split}\Ee
where we have used (\ref{mu_e}) at the last line.

 Therefore the condition of (Fredholm) is equivalent to 
 \Be\label{Euler_1}\Bs
\p_t \begin{bmatrix}
\rho \\
u \\
\frac{3}{2} \theta
 \end{bmatrix}
 +  \begin{bmatrix}
 \nabla_x \cdot \{\rho u \}\\
u\cdot \nabla_x u  + \nabla_x \{\theta \rho\}\\
\nabla_x \cdot \{5 (\rho+ \theta) u - \frac{3}{2} \rho u\}
 \end{bmatrix}=0.
 \end{split} 
 \Ee
From (\ref{incompress_Bousnq}) and the first and the last row of (\ref{Euler_1}) we derive $\rho = - \theta +C$ for some constant $C$ and
\Be\label{Euler_theta}
\p_t \theta + u \cdot \nabla_x \theta =0.
\Ee  
From the second to fourth row of (\ref{Euler_1}), 
\Be
\label{Euler_u}
\p_t u + u \cdot \nabla_x u + \nabla_x p =0,
\Ee
where $p := \theta \rho = \theta(- \theta + C)$.

Now we can read (\ref{Hilbert_R}), using (\ref{Hilbert_small}), (\ref{H1}), and (\ref{H2}),
\Be\label{eqtn_R}
\Bs
&\p_t F^\e_R + \e^{-1} v\cdot \nabla_x F^\e_R
- \e^{-2-\delta} \left\{Q(\mu_\e, F^\e_R )+ Q(F^\e_R ,\mu_\e)
\right\}\\
=& \ \e^{-1  } Q(F^\e_R , F^\e_R ) - \e^{- 1-  {\delta} } 
[
\p_t   + \e^{-1} v\cdot \nabla_x] 
\{ \e^{2+ \delta} F_2 + \e^{3+ \delta} F_3\}
- \e^{ 1- {\delta} } \{  \p_t \varphi_2 \mu_0 + \e  \p_t \mu^\e_R\}
- \e^{1 -  {\delta} } 
  v\cdot \nabla_x \mu^\e_R
\\
&+ \{
Q(F_R^\e, F_2) + Q(F_2, F^\e_R)
\}+ \e \{ Q(F_2,F_2) 
+ Q(F^\e_R, F_3) + Q(F_3, F^\e_R)
\}
+ \e^2 \{ 
Q(F_2, F_3) + Q(F_3,F_2)
\}
+ \e^3Q(F_3,F_3)
.
\end{split}
\Ee
\unhide

\section{Perturbation background
}
A linearized operator with $\mu_\e$ is 
 \Be\label{L}
 L_\e f = \frac{1}{\sqrt{ \mu_\e}} Q(\mu_\e, \sqrt{\mu_\e} f). 
 \Ee
From (\ref{collision_inv}) the kernel of $L$ has five orthonormal basis $\text{Ker} L= \langle \{\varphi_i \sqrt{\mu_\e}\}_{i=1}^5
 \rangle_{L^2_v(\R^3)}$ with 
 \Be\label{basis}
\varphi_0 := \frac{1}{\sqrt{R}},   \  \ \ \varphi_i: = \frac{1}{\sqrt{R}} \frac{v_i -U }{\sqrt{T}}
 \ \ \text{for} \ i=1,2,3 
,   \  \ \ \varphi_4: =  \frac{1}{ \sqrt{R}}
 \frac{ \big|\f{v-U}{\sqrt{T}}\big|^2-3}{\sqrt{6}}.
\Ee

We denote an $L^2_v$-projection $\mathbf{P}$ on $\text{Ker} L$ such as 
\Be\label{P}
\mathbf{P} g:= \sum_{j=0}^4  ( {P}_j g) \varphi_j \sqrt{\mu_\e}, \ \  
P g:= (P_0 g, P_1 g, P_2 g, P_3 g, P_4 g ), \ \ 
 {P}_j g:= \int_{\R^3} g \varphi_j \sqrt{\mu_\e} \dd v \ \ \text{for } j=0,1, \cdots,4.
\Ee

\hide

From the previous section we derive an equation of $f_R$ as
\Be\begin{split}\label{eqtn_f}
 &\p_t f_R+ \frac{1}{\e} v \cdot \nabla_xf_R   - \frac{1}{\delta\e^{2 }} Lf_R- \frac{\delta^{\f{1}{2}}}{\e
} \Gamma(f_R,f_R)  + \f{1}{2} \f{[\p_t + \f{1}{\e}v\cdot\nabla_x ] \mu_\e}{\mu_\e} f_R
\\
 =&   -   \f{\e}{\delta^{\f{1}{2}}} \f{1}{\sqrt{\mu_\e}}\p_t (F_2+ \e F_3 + \e^2 F_4) 
- 
\f{\e }{\delta^{\f{1}{2}}}\f{1}{\sqrt{\mu_\e}}
v\cdot \nabla_x  (    F_3 + \e  F_4)   \\
&+ \f{2\e}{\sqrt{\mu_\e}}  Q(F_2 + \e F_3 + \e^2 F_4, \sqrt{\mu_\e}f_R)  +\f{\e^2 }{\delta^{\f{1}{2}}} \f{1}{\sqrt{\mu_\e}}Q(F_2+ \e F_3 + \e^2 F_4,F_2+ \e F_3 + \e^2 F_4)  
 - 
 \f{\e^2}{\delta^{\f{3}{2}}}\f{\varphi_R }{\sqrt{\mu_\e}} .
\end{split}\Ee


It is well-known that {\color{red}(this has to be checked!!)}
\Be
|F_i (t,x,v)| \lesssim |v|^5 \mu_\e.
\Ee

\unhide

\hide
The boundary term from energy estimate is 
\Bes
&&\int^t_0
\frac{1}{\e}\iint_{n(x) \cdot v>0}
|F_R(t,x,v)|^2
(n(x) \cdot v)
\dd S_x \dd v
 \dd s \\
 &+& 
\f{1}{\e} \int^t_0
\iint_{n(x) \cdot v<0}
|F_R(t,x,R_x v) + \e^{1/2} r(t,x,v) \sqrt{\mu_\e}|^2
(n(x) \cdot v)
\dd S_x \dd v
 \dd s\\
 &=&   \int^t_0
\iint_{n(x) \cdot v<0} | r(t,x,v)|^2  \mu_\e 
(n(x) \cdot v)
\dd S_x \dd v
 \dd s\\
 &+&  \int^t_0
\iint_{n(x) \cdot v<0}
 \f{2}{\e^{1/2}} F_R(t,x,R_x v)    r(t,x,v) \sqrt{\mu_\e} 
(n(x) \cdot v)
\dd S_x \dd v
 \dd s\\
 &\lesssim&
 \int^t_0
\iint_{n(x) \cdot v<0} | r(t,x,v)|^2  \mu_\e  
\dd S_x \dd v
 \dd s\\
 &+& \frac{o(1)}{\e} \int^t_0\iint_{n(x) \cdot v>0} |F_R(t,x,v)|^2\sqrt{\mu_\e} |n(x) \cdot v|^2 \dd S_x \dd v \dd s  
 +  \int^t_0 \| r(s)\|_\infty \dd s^2.
\Ees

Trace theorem says 
\Bes
&&\int^t_0 \int_{\gamma_+} |f(s,x,v)|  |n(x) \cdot v| \dd  \gamma \dd s\\
&\lesssim& \iint_{\O \times \R^3} |f(0,x,v)| \dd v \dd x + \int^t_0 \iint_{\O \times \R^3} |f(s,x,v)| \dd v \dd x \dd s\\
&+& \int^t_0 \iint_{\O \times \R^3} \big| \p_t f(s,x,v) + \frac{1}{\e} v\cdot \nabla_x f(s,x,v)\big| \dd v \dd x \dd s
\Ees
\begin{proof}
From 
\Bes
&&\f{d}{d\tau} f(t+\tau,X(t+ \frac{\tau}{\e};t,x,v), V(t+ \frac{\tau}{\e};t,x,v))\\
&=& [\p_t + \e^{-1} v\cdot \nabla_x] f(t+\tau,X(t+ \frac{\tau}{\e};t,x,v), V(t+ \frac{\tau}{\e};t,x,v)).
\Ees
Integrating along $\tau \in [s,   0]$ for $-\e \tb(t,x,v) \leq s \leq 0$
\Bes
&&\mathbf{1}_{t>\e \tb(t, x,v)}f(t,x,v)\\
&=&\mathbf{1}_{t>\e \tb(t, x,v)}  f(t+s,X(t+ \frac{s}{\e};t,x,v), V(t+ \frac{s}{\e};t,x,v))\\
&+& \mathbf{1}_{t>\e \tb(t, x,v)}  \int^0_{s} [\p_t + \e^{-1} v\cdot \nabla_x] f(t+\tau,X(t+ \frac{\tau}{\e};t,x,v), V(t+ \frac{\tau}{\e};t,x,v)) \dd \tau .
\Ees
Integrating over $s \in [- \e \tb(t,x,v),0]$
\Bes
&&\e \tb(t,x,v) \mathbf{1}_{t>\e \tb(t, x,v)}f(t,x,v)\\
&=& \int^0_{- \e \tb(t,x,v) }\mathbf{1}_{t>\e \tb(t, x,v)}  f(t+s,X(t+ \frac{s}{\e};t,x,v), V(t+ \frac{s}{\e};t,x,v)) \dd s \\
&+ &\int^0_{- \e \tb(t,x,v) }\mathbf{1}_{t>\e \tb(t, x,v)}  \int^0_{s} [\p_t + \e^{-1} v\cdot \nabla_x] f(t+\tau,X(t+ \frac{\tau}{\e};t,x,v), V(t+ \frac{\tau}{\e};t,x,v)) \dd \tau
\dd s\\
&\leq &\int^0_{- \e \tb(t,x,v) }\mathbf{1}_{t>\e \tb(t, x,v)}  f(t+s,X(t+ \frac{s}{\e};t,x,v), V(t+ \frac{s}{\e};t,x,v)) \dd s \\
&+ & \e \tb(t,x,v)   \mathbf{1}_{t>\e \tb(t, x,v)}  \int^0_{-\e \tb(t,x,v)} [\p_t + \e^{-1} v\cdot \nabla_x] f(t+\tau,X(t+ \frac{\tau}{\e};t,x,v), V(t+ \frac{\tau}{\e};t,x,v)) \dd \tau 
\Ees
From $\tb(t,y,u) \gtrsim \frac{|n(y) \cdot u|}{|u|^2}$ we conclude that 
\Bes
&&  \mathbf{1}_{t>\e \tb(t, x,v)}|n(x) \cdot v| \mu_0^{1/2} |f(t,x,v)|\\
&\leq & \frac{ \mathbf{1}_{t>\e \tb(t, x,v)} |n(x) \cdot v| \mu_0^{1/2} }{\e \tb(t,x,v)}\int^0_{- \e \tb(t,x,v) } | f(t+s,X(t+ \frac{s}{\e};t,x,v), V(t+ \frac{s}{\e};t,x,v)) |\dd s\\
&+& \mathbf{1}_{t>\e \tb(t, x,v)} |n(x) \cdot v| \mu_0^{1/2}  \int^0_{-\e \tb(t,x,v)}\big| [\p_t + \e^{-1} v\cdot \nabla_x] f(t+\tau,X(t+ \frac{\tau}{\e};t,x,v), V(t+ \frac{\tau}{\e};t,x,v)) \big|\dd \tau \\
&\lesssim&  \frac{ \mathbf{1}_{t>\e \tb(t, x,v)} |v|^2 \mu_0^{1/2} }{\e  }\int^0_{- \e \tb(t,x,v) } | f(t+s,X(t+ \frac{s}{\e};t,x,v), V(t+ \frac{s}{\e};t,x,v)) |\dd s\\
&+& \mathbf{1}_{t>\e \tb(t, x,v)} |n(x) \cdot v| \mu_0^{1/2}  \int^0_{-\e \tb(t,x,v)}\big| [\p_t + \e^{-1} v\cdot \nabla_x] f(t+\tau,X(t+ \frac{\tau}{\e};t,x,v), V(t+ \frac{\tau}{\e};t,x,v)) \big|\dd \tau .
\Ees

From $\tb(t,y,u) \gtrsim \frac{|n(y) \cdot u|}{|u|^2}$
\Bes
&&\mathbf{1}_{t>\e \tb( y,u)}
\e |n(y) \cdot u|\sqrt{\mu_0} |f(t,y,u)|\\
&\lesssim& \mathbf{1}_{t>\e \tb( y,u)}
\frac{  |n(y) \cdot u| \sqrt{\mu_0}}{  \tb(y,-u) } \int^0_{ - \e \tb(y,-u) } |f(t+s, X(\e^{-1}s;0,y,u))| \dd s\\
&+& \mathbf{1}_{t>\e \tb( y,u)} \frac{  |n(y) \cdot u| \sqrt{\mu_0}}{  \tb(y,-u) } 
\int^0_{ - \e \tb(y,-u) }
\int^t_s |[\p_t + \frac{1}{\e} v\cdot \nabla_x] f(t+\tau, X(\e^{-1} \tau;0,y,u), V(\e^{-1} \tau;0,y,u))| \dd \tau \dd s\\
&\lesssim&
 |u|^2\sqrt{\mu_0 (u)} \int^0_{ - \e \tb(y,-u) } |f(t+s, X(\e^{-1}s;0,y,u))| \dd s\\
&+& |u|^2\sqrt{\mu_0 (u)} 
\int^0_{ - \e \tb(y,-u) }
\int^t_s |[\p_t + \frac{1}{\e} v\cdot \nabla_x] f(t+\tau, X(\e^{-1} \tau;0,y,u), V(\e^{-1} \tau;0,y,u))| \dd \tau \dd s.
\Ees

If 
\Bes
&&\mathbf{1}_{t< \e \tb(y,u)} |f(t,y,u)|\\
&\leq& |f(0,X(\e^{-1}t;0,y,u), V(\e^{-1}t;0,y,u))|\\
&+& \int^0_{-t} |[\p_t + \e^{-1} v\cdot \nabla_x]f(t+s, X(\e^{-1}t;0,y,u), V(\e^{-1}t;0,y,u))| \dd s
\Ees
\end{proof}

By the trace theorem,
\Bes
&& \frac{o(1)}{\e} \int^t_0\iint_{n(x) \cdot v>0} |F_R(t,x,v)|^2\sqrt{\mu_\e} |n(x) \cdot v|^2 \dd S_x \dd v \dd s  \\
&\lesssim&\frac{o(1)}{\e} \int^t_0\iint - \frac{1}{\e^{2+\kappa}}Q(\mu_\e, F_R) 
F_R
\dd   x \dd v \dd s 
\Ees
\unhide

\unhide


\unhide

 \section{A priori estimates for $f_R$}\label{sec:B}
For each $\e>0$ an existence of a unique 
 solution in a time interval $[0,\infty)$ can be found in \cite{EGKM2}. Thereby we only focus on a priori estimates of $f_R$ in different spaces. For the sake of simplicity at times we will use simplified notations 
\Be\label{short_notation}
\| g(t,x,v) \|_{L^{p_1}_t L^{p_2}_x L^{p_3}_v}: = 
\Big\|  
\big\|
\|
g(t,x,v)\|_{L^{p_3}_v(\R^3)}
\big\|_{L^{p_2}_x(\O)}
\Big\|_{L^{p_1}_t ([0,T])}, \ \ \| g \|_{L^p_{t,x,v}} := \| g \|_{L^p_t L^p_x L^p_v}. 
\Ee 
Recall the boundary integral and the norms in (\ref{bdry_int}). Also recall $\mathfrak{w}=\mathfrak{w}_{\varrho, \ss}(x,v)$ in (\ref{weight}) and  $\mathfrak{w}'=\mathfrak{w}_{\varrho', \ss}(x,v) $ for $0<\varrho'<\varrho$.

\hide
  For the rest of terms, we only need to consider \bcb commutator estimates for \ec 
 $\p_t \Gamma(f_R,f_R) - 2 \Gamma(f_R, \p_t f_R),$ and $\p_t \Gamma(f_2, f_R) - \Gamma(f_2, \p_t f_R)- \Gamma( \p_t f_2,  f_R)$. \bcb  
 \hide
We derive that 
\Be\label{dec:L_t}
\p_t (Lf_R) =  \p_t (L( \mathbf{I}-\mathbf{P})f_R) =L ( \mathbf{I}-\mathbf{P})\p_tf_R + L_t( \mathbf{I}-\mathbf{P})f_R + L (-\mathbf{P}_t f_R)=
L\p_tf_R + L_t( \mathbf{I}-\mathbf{P})f_R - L( \mathbf{P}_t f_R),
\Ee
where 
\Be
\begin{split}\label{def:L_t}
L_t g (t,v):= - \e \p_t u \cdot \nabla_v \nu_0 (v-\e u ) g(t, v)
+ \e \p_t u \cdot  \int_{\R^3} \{\nabla_v \mathbf{k}_0(v-\e u, v_* - \e u) + \nabla_{v_*} \mathbf{k}_0(v-\e u, v_* - \e u) \}  g(t,v_*)\dd v_*,
\\
\mathbf{P}_t g:=- \e \sum _{j=0}^4
  ( {P}_j g) \p_tu \cdot \nabla_v( \varphi_j \sqrt{\mu })
  -\e  \sum _{j=0}^4
  \langle g ,    \p_t  u \cdot  \nabla_v (\varphi_j \sqrt{\mu }) \rangle \varphi_j \sqrt{\mu },
\end{split}
\Ee
\bcb From (\ref{est:k}) and (\ref{est:int_k}) \ec it is easy to check that, for any $0< \varrho < 1/4$,  
\begin{align*} 
|\nabla_v \nu_0 (v)| = \Big| \iint_{\R^3 \times \S^2}   \mathfrak{u}  \frac{(v-v_*) \cdot \mathfrak{u}}{|(v-v_*) \cdot \mathfrak{u}|}\mu_0 (v_*) \dd \mathfrak{u} \dd v_* \Big|\lesssim 1,\\  
|\nabla_v \mathbf{k}_0 (v,v_*)| +  |\nabla_{v_*} \mathbf{k}_0 (v,v_*)|  \lesssim 
\Big(|v-v_*|^{-1}+
 \nu_0 (v)^2 |v-v_*| 
\Big)
k_{\varrho/2} (v,v_*).
\end{align*}

Upon the change of variables $\tilde{v}_*:= v- v_*$, the exponent can be expanded as 
\Be
\begin{split}
- \vartheta |\tilde{v}_*|^2 - \vartheta \frac{\big||\tilde{v}_*|^2 - 2 v \cdot\tilde{v}_* \big|^2}{|\tilde{v}_*|^2}
-  \big\{\varrho|v-\tilde{v}_*|^2+ \varrho^\prime |v-\tilde{v}_*| a\cdot (v-\tilde{v}_*)
 - \varrho|v|^2-   \varrho^\prime |v| a\cdot v\big\}
\end{split}
\Ee \unhide


First two differences are bounded as in (\ref{est_infty:L_t}). For the last difference we follow the proof of (\ref{est_infty:L_t}) to have  
\begin{align}
|\p_t \Gamma(f_2, f_R)(t,v) - \Gamma(f_2, \p_t f_R)(t,v)- \Gamma( \p_t f_2,  f_R)(t,v) | \lesssim 
\e |\p_t u| e^{C \e^2 |u|^2} \nu(v)^2 e^{- \varrho |v-\e u|^2} 
\{
|p| + |\tilde{u}|  + \kappa |\nabla_x u |
\}
\| e^{\varrho |v|^2} f_R(t,v) \|_{L^\infty_v}. \notag
\end{align}

\ec

\unhide

 \subsection{$L^2$-Energy estimate} 
Our starting point is a basic $L^2$-energy estimate for the Boltzmann remainder $f_R$ and its temporal derivative $\p_t f_R$ in which the dissipation (\ref{dissipation}) plays an important role in the nonlinear estimate.

\hide
{\color{red} JJ: In $d_2$ and $d_{2,t}$ below, do we want to use $\sigma_0$ from \eqref{s_gap} in place of $\frac12$?}
\unhide

\begin{proposition}\label{prop:energy} Under the same assumptions in Proposition \ref{prop:Hilbert}, we have 
\Be
\begin{split}\label{est:Energy}
&  \| f_R (t)\|_{L^2_{x,v}}^2+  {d}_{2}
 \int^t_0 
  \| \kappa^{-\frac{1}{2}} \e^{-1} \sqrt{\nu} (\mathbf{I} - \mathbf{P})f_R \|_{L^2_{x,v}}^2
  + \int^t_0 | \e^{-\frac{1}{2}}
f_R
|_{  L^2_\gamma}^2
  \\
  \lesssim & \ \| f_R (0)\|_{L^2_{x,v}}^2 
   +
  (1+ \| (\ref{transp:mu})\|_{L^\infty_{t,x}} )\int^t_0 \| P f_R (s) \|_{L^2_x}^2 \dd s 
  +
     \frac{ \delta^2}{\kappa^3 }
 \| \kappa^{1/2}   {P} f_R(s) \|_{L^\infty_tL^6_{x }}^2  
\| \kappa^{1/2}P  f_R \|^2_{L^2_tL^3_{x}} \\
&
 + \frac{\e\kappa^2}{ \delta^2}  
 |  \nabla_x u|^2_{L^2_t L^2 (\p\O)}
 +  \kappa  \e^{1/8}
  \| 
 \nabla_x u  
  \|_{L^2_{t,x }}  ^2 
+ \kappa  \e^2  
  \|(\ref{est:R1})
  \|^2_{L^2_{t,x}}  + \|(\ref{est:R2})  \|^2_{L^2_{t,x}},
  \end{split}
\Ee
where 
\Be\label{d2}
{d}_{2}
:=  \frac{\sigma_0}{2}- \delta \e \| \mathfrak{w} f_R \|_{L^\infty_{t,x,v}}
-( \e^{\frac{15}{16}} \| \mathfrak{w} f_R \|_{L^\infty_{t,x,v}} )^2
  - \e^2 \| (\ref{est:f2})  \|_{L^\infty_{t,x}}  -
 \frac{\e^2}{\kappa }  \| (\ref{est:f2}) \|_{L^\infty_{t,x}}^2
- \e  \kappa^{1/2}   \|  (\ref{transp:mu})_*  \|_{L^\infty_{t,x}}
 .
 \Ee
Here $L^p_t=L^p_t([0,t])$ in particular. Note that a weighted $L^\infty$-bound of $f_R$ is involved in this energy estimate, where the weight $\mathfrak{w}=\mathfrak{w}_{\varrho, \ss}(x,v)$ is defined in (\ref{weight}).
 
 We also have 
 \Be\begin{split}\label{est:Energy_t}
&\| \p_t f_R (t) \|_{L^2_{x,v}}^2 + {d}_{2,t}
 \| \kappa^{-\frac{1}{2}} \e^{-1} \sqrt{\nu}(\mathbf{I} - \mathbf{P}) \p_t f_R\|_{L^2_{t,x,v}}^2
 +  |\e^{- \frac{1}{2}} \p_t f_R|_{L^2_tL^2_\gamma}^2
-  \e \| \p_t  u \|_{L^\infty_{t,x} } |      f_R
|_{L^2 _t L^2_{\gamma}}^2
 \\
\lesssim& \ \| \p_t f_R (0) \|_{L^2_{x,v}}^2
+\kappa^{-1} \delta^2 
\| Pf_R \|_{L^\infty_t L^6_x}^2 
\{\| P f_R \|_{L^2_t L^3_x}^2
+
\| P \p_t f_R \|_{L^2_t L^3_x}^2
\}\\
&+ \Big\{  
 \kappa^{-1 } \| \p_t u \|_{L^\infty_{t,x }} ^2
 + \| \nabla_x \p_t u \|_{L^\infty_{t,x}}
 + \e \| \p_t^2 u \|_{L^2_{t} L^\infty_x} 
+
\e \kappa^{-1/2} (1+  \| \p_t u \|_{L^\infty_{t,x}} )\| (\ref{est:f2}) \|_{L^\infty_{t,x}} \\
& \ \ \ \ \
+\| (\ref{transp:mu}) \|_{L^\infty_{t,x} } 
+ \| (\ref{transp:mu_t})_* \|_{L^\infty_{t,x} } 
\Big\}  \times 
\int^t_0
  \|     P \p_tf_R(s) \|_{L^{2}_x }^2 \dd s \\
  &
  +  \Big\{   \kappa^{-1} \| \p_t u \|_{L^\infty_{t,x,v}}^2
  + \| \nabla_x \p_t u \|_{L^\infty_{t,x}}
 + \e \| \p_t^2 u \|_{L^2_{t} L^\infty_x} 
  + (\e \| (\ref{est:f2})\|_{L^\infty_{t,x}})^2\\
& \ \ \ \ \ 
    + (\e\kappa^{-1/2} \| (\ref{est:f2_t})\|_{L^\infty_{t,x}})^2
  + \| (\ref{transp:mu_t})_*\|_{L^\infty_{t,x}} 
  \Big\} \times 
  \int^t_0
  \|     P  f_R(s) \|_{L^{2}_x }^2 \dd s  \\
  &+
 \Big\{
 \e (1+ \e \| (\ref{est:f2})\|_{L^\infty_{t,x}}) \| \p_t u \|_{L^\infty_{t,x}}
 + \e \kappa \| \nabla_x \p_t u \|_{L^\infty_{t,x}}
 + \e^2 \kappa \|   \p_t ^2u \|_{L^2_tL^\infty_{ x}}
 \\
 & \ \ \ \ \
 + (\e \kappa^{1/2} \| (\ref{transp:mu_t})_*\|_{L^\infty_{t,x}})^2
 +( \e \delta \| \mathfrak{w} f_R \|_{L^\infty_{t,x,v}}) ^2
 \Big\} \times 
  \| \e^{-1}\kappa^{-1/2} \sqrt{\nu} (\mathbf{I} -\mathbf{P}) f_R \|_{L^2_{t,x,v}}^2
  \\ 
  &+ e^{- \frac{\varrho}{4 \e^2}} \{\| (\ref{transp:mu})\|_{L^\infty_{t,x}}^2 + \| \nabla_x \p_t u \|_{L^2_t L^\infty_x}^2 \}
  + (\e \kappa^{1/2} \| (\ref{est:R3}) \|_{L^2_{t,x}})^2 +  \|(\ref{est:R4})\|_{L^2_{t,x}}^2\\
  &+ 
 \frac{\e \kappa^2}{ \delta^2} 
 | |\p_t \nabla_x u| + \e |\p_t u| |\nabla_x u| |_{L^2_tL^2 (\p\O) }^2
 + \frac{\e^3 \kappa^2  }{\delta} 
|\nabla_x u|^2_{L^2_t L^2(\p\O)}  \| \p_t  u \|_{L^\infty_{t,x}} ,
 \end{split}\Ee
\hide

\Be\begin{split}\label{est:Energy_t}
&\| \p_t f_R (t) \|_{L^2_{x,v}}^2 + d_{1, \e, \delta }
 \| \kappa^{-\frac{1}{2}} \e^{-1} \sqrt{\nu}(\mathbf{I} - \mathbf{P}) \p_t f_R\|_{L^2_{t,x,v}}^2
 +  |\e^{- \frac{1}{2}} \p_t f_R|_{L^2_tL^2_\gamma}^2
 \\
\lesssim& \ \| \p_t f_R (0) \|_{L^2_{x,v}}^2
+ \Big( 
(\kappa^{-1/2} \| \p_t u \|_{L^\infty_{t,x,v}})^2
+ \e^2 \| (\ref{est:f2}) \|_{L^\infty_{t,x,v}}^2
+
\| (\ref{transp:mu})
 \|_{L^\infty_{t,x} 
 }+ \frac{\e^2}{\delta^2 \kappa^2}\Big) \int^t_0
  \|     P \p_tf_R(s) \|_{L^{2}_x }^2 \dd s \\
  &
  + \bcb \kappa^{-1} \| \p_t u \|_{L^\infty_{t,x,v}}^2\int^t_0
  \|     P  f_R(s) \|_{L^{2}_x }^2 \dd s  \ec
  +\kappa \e^2  e^{-\frac{\varrho}{4\e^2}}  \| 
(\ref{transp:mu})
 \|_{L^\infty_{t,x} 
 }^2  \int^t_0 \| \mathfrak{w} \p_t f_R(s) \|_{ L^\infty_{x,v}}^2 \dd s\\
 &+
 \big\{
 \frac{\delta^2}{\kappa^3} \| \kappa^{1/2} Pf_R \|_{L^\infty_tL^6_x}^2 +  \frac{\delta}{\kappa^{5/2}} \|   (\ref{est:f2}) \|_{L^\infty_tL^6_x}^2
 \big\}
  \| \kappa^{1/2} P \p_t f_R \|_{L^2_t L^3_x}^2
  +   \kappa^{-1}\e^2  \| (\ref{est:f2_t}) \|_{L^2_{t,x }} ^2
  \| \mathfrak{w} f_R \|_{L^\infty_{t,x,v}}^2
  \\
  &
  +( \frac{\e}{\delta^2} + \frac{\e^3}{\delta}\| \p_t u \|_{L^\infty} )| (\ref{est:f2_t}) |_{L^2_tL^2_{\gamma }}^2
+
  \kappa \e^2 \| (\ref{est:R3}) + (\ref{est:R4})\|^2_{L^2_{t,x}}
 +
 \|(\ref{transp:mu_t})\|_{L^\infty_{t,x}}
\{ 
\| f_R\|_{L^\infty_t L^2_{x,v}}^2 + \| \p_t f_R\|_{L^\infty_t L^2_{x,v}}^2\}\\
&
\bcb
+ 
\big\{1+
\e^4\| \p_t u \|_{L^\infty_{t,x,v}}^2 
 (  \delta ^2
  \| \mathfrak{w} f_R \|_{L^\infty_{t,x,v}}^2
  +\| (\ref{est:f2})\|_{L^\infty_{t,x}}^2
   )
\big\}
 \| \kappa^{-\frac{1}{2}} \e^{-1} \sqrt{\nu} (\mathbf{I} - \mathbf{P}) f_R \|_{L^2_{t,x,v}} ^2
 \ec
 \\
 & 
+ \frac{\delta^2}{\kappa^3} \| \p_t u \|_{L^\infty_{t,x,v}}^2\| \kappa^{1/2} P f_R \|_{L^\infty_tL^6_x}^2 \|\kappa^{1/2}  Pf_R \|_{L^2_t L^3_x}^2\\
&+  \frac{\e \kappa^2}{ \delta^2} 
 | |\p_t \nabla_x u| + \e |\p_t u| |\nabla_x u| |_{L^2_tL^2 (\p\O) }^2
 +  \| \p_t  u \|_{\infty}\frac{\e^3}{\delta} 
\kappa^2 |\nabla_x u|^2_{L^2_t L^2(\p\O)} 
, 
\end{split}\Ee\unhide
where $\p_t f_R (0,x,v) := f_{R,t}(0,x,v)$ is defined in \eqref{initial_f}. 
Here
\Be\label{d2t}
\begin{split}
{d}_{2,t}&:= \frac{\sigma_0}{2}- \e( \kappa^{-1/2} + \e  \| \p_t u\|_{L^\infty_{t,x}}) \| (\ref{est:f2})\|_{L^\infty_{t,x}}- 
\e \kappa \| (\ref{transp:mu}) \|_{L^\infty_{t,x}}
- (\e\kappa^{1/2} \| (\ref{transp:mu_t})_* \|_{L^\infty_{t,x}})^2\\
& \  \  \ - \e \kappa \| \nabla_x \p_t u \|_{L^\infty_{t,x}}
- \e^2 \kappa \|   \p_t^2 u \|_{L^2_tL^\infty_{x}}
+ \e \| \p_t u \|_{L^\infty_{t,x}} (1+ \e \| \p_t u \|_{L^\infty_{t,x}}) \\
&
- 
\e^{- \frac{\varrho}{4 \e^2}}
(\e \kappa^{1/2} \| \mathfrak{w}_{\varrho^\prime, \ss} \p_t f_R \|_{L^2_t L^\infty_{x,v}})^2 -  \e \delta (1+ \e \| \p_t u \|_{L^\infty_{x,v}}) \| \mathfrak{w} f_R \|_{L^\infty_{t,x}}\\
&
- (\e \kappa^{1/2} \| \mathfrak{w} f_R \|_{L^\infty_{t,x,v}})^2,
\end{split}
\Ee
where $0<\varrho^\prime < \varrho$. 

\begin{remark}
We utilize several different time-space norms to control the fluid source terms, which possess the initial-boundary and boundary layers as in Theorem \ref{thm_bound}.  
\end{remark}

\hide

\Be
\begin{split}\label{est:Energy_t}
&  \|\p_t  f_R (t)\|_2^2+ d_{  \e, \delta}
 \int^t_0 
  \| \kappa^{-\frac{1}{2}} \e^{-1} \sqrt{\nu} (\mathbf{I} - \mathbf{P}) \p_t f_R \|_2^2
  + \int^t_0 | \e^{-\frac{1}{2}}
(1- P_{\gamma_+}) \p_t f_R
|_{  2, {\gamma_+}}^2
  \\
  \lesssim & \ \| \p_t  f_R (0)\|_2^2 + \frac{\e}{2\delta^2} \int^t_0 \int_{\gamma_+} |(1- P_{\gamma_+}) \p_t f_2|^2
  + \e \| \p_t u \|_\infty
   \int^t_0 |  f_R
|_{  2, {\gamma_+}}^2
   \\
   &+
  \e^2 \big\|| \p_t  p| +| \p_t  \tilde{u} | + \kappa |\nabla_x  \p_t   u| \big\|_{L^\infty_{t,x}}    \int^t_0 
  \| \kappa^{-\frac{1}{2}} \e^{-1} \sqrt{\nu} (\mathbf{I} - \mathbf{P})  f_R \|_2^2
   \\
  & +
  \| |\nabla_x u | + \e |\p_t u| + \e |u| |\nabla_x u|\|_{L^2_tL^{3/2}_{x}} \| {P}f_R  \|_{L^\infty_tL^{6}_{x} } ^2
  +
   \Big(\frac{\e}{\kappa^2}+  
   \frac{ \delta^2}{\kappa^3 }
 \|  \kappa^{1/2}  {P} f_R(s) \|_{L^\infty_tL^6_{x }}^2  
\Big)
\| \kappa^{1/2}  P \p_t f_R \|^2_{L^2_tL^3_{x}} 
\\
&+ \frac{\kappa  \e^2 }{\delta^2} \|\mathfrak{q}(|\nabla_x \p_t \tilde{u}|, |\nabla_x^2 \p_t u|)\|^2_{L^2_{t,x}}\\
&  +  \frac{\e^4}{\delta^2 \kappa } \big\|
(1+ 
|\p_t^2 p| + |\nabla_x \p_t^2 u| 
) 
\\
& \ \ \ \ \ \ \ \ \  \times 
 \mathfrak{q} (
|p|, |\nabla_x p|, |\p_t p|,|\nabla_x \p_t p | , |u| , |\nabla_x u|, |\p_t u|, |\nabla_x \p_t u|, |\nabla_x^2 u|, |\nabla_x^2 \p_t u|, 
|\tilde{u}|, |\nabla_x \tilde{u}|,|\p_t \tilde{u}| ,|\nabla_x \p_t \tilde{u}|, |\p_t ^2 \tilde{u}|
) 
 \big\|_{L^2_{t,x}}^2,
  \end{split}
\Ee
 where  
 \Be
d_{2,\e, \delta}:= d_{1,\e, \delta}- 
  \e^2 \big\|| \p_t  p| +| \p_t  \tilde{u} | + \kappa |\nabla_x  \p_t   u| \big\|_{L^\infty_{t,x}}
 - \e  \||\p_t  p| +|\p_t  \tilde{u} | + \kappa |\nabla_x\p_t   u| \|_{L^\infty_{t}L^6_x}^2
 \Ee
 \unhide

\end{proposition}

The following trace theorem is useful to control the boundary terms. 
\begin{lemma}[Trace theorem]
 
\Be\label{trace}
\frac{1}{\e}\int^t_0 \int_{\gamma_+^N} |h| \dd \gamma \dd s \lesssim_N  \iint_{\O \times \R^3} |h(0)|  +   \int^t_0 \iint_{\O \times \R^3} |h|  + \int^t_0 \iint_{\O \times \R^3 } | \p_t h +\frac{1}{\e} v\cdot \nabla_x h|,
\Ee
where $\gamma_+^N:= \{(x,v) \in \gamma_+: |n(x) \cdot v|> 1/N \ \text{and} \ 1/N<|v|<N \}$.\end{lemma}
The proof is standard (for example see Lemma 3.2 in \cite{EGKM2} or Lemma 7 in \cite{CKL}). 

\begin{proof}[\textbf{Proof of Proposition \ref{prop:energy}}]

First we prove (\ref{est:Energy}). An energy estimate to (\ref{eqtn_fR}) and (\ref{bdry_fR}) reads as 
\begin{align}
&\frac{1}{2}\|f _R(t)\|_{L^2_{x,v}}^2 - \frac{1}{2} \|f _R (0)\|_{L^2_{x,v}}^2
+ \frac{1}{ \kappa \e^{2 }} \int^t_0 \iint _{\O \times \R^3}f _R L  f_R\label{Energy_LHS} \\
& +  \frac{1}{2\e} \int^t_0 \int_{\gamma_+} |f_R|^2 
-  \frac{1}{2\e} \int^t_0 \int_{\gamma_-} | 
P_{\gamma_+}   f_R- \frac{\e}{\delta} (1- P_{\gamma_+})
(\mathbf{I} -\mathbf{P})
 f_2
|^2 \label{Energy_bdry}
 \\
=  
&   \frac{\delta }{\kappa \e } \int^t_0 \iint_{\O \times \R^3} 
\Gamma (f_R,f_R) (\mathbf{I} - \mathbf{P}) f_R 
\label{Energy_Gamma}
\\
&+
 \frac{2}{  \kappa} \int^t_0 \iint_{\O \times \R^3} \Gamma(f_2, f_R) (\mathbf{I} - \mathbf{P}) f_R \label{Energy_Gf_2}
\\
&+ \int^t_0 \iint_{\O \times \R^3} ( \mathbf{I} -  \mathbf{P}) \mathfrak{R}_1  ( \mathbf{I} -  \mathbf{P}) f_R \label{Energy_I-PR_1}\\
&+ \int^t_0 \iint_{\O \times \R^3}  \mathfrak{R}_2  f_R \label{Energy_R2}\\
&+ \int^t_0 \iint_{\O \times \R^3}
\frac{- (\p_t + \e^{-1} v\cdot \nabla_x ) \sqrt{\mu}}{\sqrt{\mu}} |f_R|^2.\label{Energy_v^3}
\end{align}
Among others two terms (\ref{Energy_Gamma}) and (\ref{Energy_v^3}) are most problematic.

We start with (\ref{Energy_LHS}). 
From the spectral gap estimate in (\ref{s_gap}), we have 
\Be\label{est:EL}
(\ref{Energy_LHS})\geq \frac{1}{2}\|f _R(t)\|_{L^2_{x,v}}^2 - \frac{1}{2} \|f _R (0)\|_{L^2_{x,v}}^2
+  \sigma_0
\| \kappa^{-\frac{1}{2}} \e^{-1} \sqrt{\nu}(\mathbf{I} - \mathbf{P}) f_R\|_{L^2_{t,x,v}}^2. 
\Ee

Now we consider (\ref{Energy_Gamma}), in which we need integrability gain of $\mathbf{P}f_R$ in $L^6_x$ 
of the next sections. From decomposition $f_R= \mathbf{P} f_R + (\mathbf{I} - \mathbf{P}) f_R$ and $\Gamma=\Gamma_+- \Gamma_-$ in (\ref{Gamma}),  we derive 
\Be
\begin{split}
\label{est:EG}
|(\ref{Energy_Gamma})| \lesssim &  \ 
\frac{\delta}{\kappa \e}  \sum_{i=\pm }\int^t_0 \iint_{\O \times \R^3} | \nu^{-\frac{1}{2}} \Gamma_{i}( |f_R|, (\mathbf{I} -\mathbf{P}) f_R)| | \sqrt{\nu}(\mathbf{I} - \mathbf{P}) f_R|\\
&
+\frac{\delta}{\kappa \e}  \sum_{i=\pm } \int^t_0 \iint_{\O \times \R^3}  | \nu^{-\frac{1}{2}} \Gamma_{i}(|\mathbf{P}f_R|, |\mathbf{P} f_R|)| | \sqrt{\nu}(\mathbf{I} - \mathbf{P}) f_R|
\\
 \lesssim &   \ \delta \e \|   \mathfrak{w}f_R\|_{L^\infty_{x,v}}  \| \kappa^{-\frac{1}{2}} \e^{-1} \sqrt{\nu}
(\mathbf{I} - \mathbf{P}) f_R \|^2_{L^2_{t,x,v}} 
\\
&+  \frac{ \delta
}{\kappa^{3/2}}
 \|\kappa^{1/2}     {P} f_R \|_{L^\infty_tL^6_{x }}
 \| \kappa^{1/2}   {P} f_R \|_{L^2_tL^3_{x }}
\| \kappa^{-\frac{1}{2}} \e^{-1} \sqrt{\nu}
(\mathbf{I} - \mathbf{P}) f_R \|_{L^2_{t,x,v}}.
\end{split}\Ee

From (\ref{est:I-Pf2}) and (\ref{est:Pf2}),   
\Be
\begin{split}
&|(\ref{Energy_Gf_2})|\\ &\leq  
\frac{\e}{\kappa^{1/2}}
  \|  (\ref{est:f2}) \|_{L^\infty_{t,x}} 
  \{
   \|  P f_R \|_{L^2_{t,x}} 
+  \kappa^{ \frac{1}{2}} \e  \| \kappa^{-\frac{1}{2}} \e^{-1}  \sqrt{\nu}(\mathbf{I} - \mathbf{P}) f_R \|_{L^2_{t,x,v}} 
  \}\\
  & \ \ \ \ \ \ \ \times
 \| \kappa^{-\frac{1}{2}} \e^{-1}  \sqrt{\nu}(\mathbf{I} - \mathbf{P}) f_R \|_{L^2_{t,x,v}} \\
 &\lesssim  
\{\e^2 \| (\ref{est:f2})  \|_{L^\infty_{t,x}}
+ \frac{\e^2}{\kappa}  \| (\ref{est:f2})  \|_{L^\infty_{t,x}}^2
\}
  \| \kappa^{-\frac{1}{2}} \e^{-1}  \sqrt{\nu} (\mathbf{I} - \mathbf{P}) f_R \|_{L^2_{t,x,v}}^2
 + 
\|  P  f_R \|_{L^2_tL^2_{x}}^2.
  \label{est:Gf_2}
\end{split}
\Ee

From (\ref{est:R1}) and (\ref{est:R2}) we derive that 
\Be\begin{split}\label{est:I-PR_1R2}
|(\ref{Energy_I-PR_1})| &\lesssim 
 \kappa^{ 1/2} \e   \| (\ref{est:R1})
 \|_{L^2_{t,x}} \| \kappa^{-1/2} \e^{-1} (\mathbf{I} - \mathbf{P} ) f_R \|_{L^2_{t,x,v}},
\\
|(\ref{Energy_R2})|& \lesssim   
\| (\ref{est:R2})\|_{L^2_{t,x,v}}^2 +   \| \mathbf{P}f_R\|_{L^2_{t,x,v}}^2  +
 \kappa^{ 1/2} \e \| (\ref{est:R2})\|_{L^2_{t,x,v}} 
\|
  \kappa^{-1/2} \e^{-1}
  (\mathbf{I} -\mathbf{P})f_R\|_{L^2_{t,x,v}} 
 .
\end{split}\Ee
\hide

\textit{Estimate of (\ref{Energy_PR_1}):} From (\ref{PR_1})
\Be
|(\ref{Energy_PR_1})| \lesssim  \frac{\e}{\delta \kappa}\left\|
 |\tilde{u}| |\kappa^{\frac{1}{2}} \nabla u|
\right\|_{L^2_{t,x,v}} \left\|
\kappa^{\frac{1}{2}} \mathbf{P} f_R 
\right\|_{L^2_{t,x,v}} .\label{est:PR_1}
\Ee

\textit{Estimate of (\ref{Energy_I-PR_1}):} From (\ref{I-PR_1})
\Be\begin{split}
|(\ref{Energy_I-PR_1})| \lesssim & \frac{\e}{\delta \kappa^{\frac{1}{2}}}\left\|
 |\kappa \nabla^2 u|  + \kappa^{\frac{1}{2}} |\kappa^{\frac{1}{2}} \nabla u| 
 \big[ \e |(\tilde{\rho}, \tilde{u}, \tilde{\theta})| + \e \kappa^{-\frac{1}{2}}|\kappa^{\frac{1}{2}} \nabla_x u |\big]
 + \e |(\tilde{\rho}, \tilde{u}, \tilde{\theta})|^2 + \kappa^{\frac{1}{2}} 
 |\kappa^{\frac{1}{2}}  \nabla(\tilde{\rho}, \tilde{u}, \tilde{\theta}) |
\right\|_{L^2_{t,x,v}}\\
& \times  \left\| 
\kappa^{-\frac{1}{2}} \e^{-1} (\mathbf{I}-\mathbf{P}) f_R 
\right\|_{L^2_{t,x,v}}.
\label{est:I-PR_1}
 \end{split}
\Ee

\textit{Estimate of (\ref{Energy_R2}):} From (\ref{R2})
\Be\begin{split}
|(\ref{Energy_R2})| \lesssim & \frac{\e}{ \delta \kappa}
\| \frac{\delta \kappa^{1/2}}{\e} \mathfrak{R}_2 \|_{L^2_{t,x,v}}
 \| \kappa^{\frac{1}{2}} f_R \|_{L^{2}_{t,x,v}}.\label{est:R2}
\end{split}
\Ee
\unhide

Next using (\ref{transp:mu}) it follows that 
\Be\label{est:Energy_v^3}
\begin{split}
&|(\ref{Energy_v^3})|\\ 
& \lesssim \ 
  \kappa^{1/2} \e
  \| \mathfrak{w}^{-1}{  \nu^{3/2} } 
 \nabla_x u   
  \|_{L^2_{t,x,v}}  
 \| \mathfrak{w}f_R\|_{ L^\infty_{t,x,v}}
\|\kappa^{-\frac{1}{2}} \e^{-1} \sqrt{\nu} (\mathbf{I} - \mathbf{P})f_R\|_{L^2_{t,x,v}}\\
& \ \ + \| (\ref{transp:mu})\|_{L^\infty_{t,x}} \int^t_0 \| P f_R (s) \|_{L^2_x}^2 \dd s
+ \e  \kappa^{1/2}  \| (\ref{transp:mu})_* \|_{L^\infty_{t,x}}  \|\kappa^{-\frac{1}{2}} \e^{-1} \sqrt{\nu} (\mathbf{I} - \mathbf{P})f_R\|_{L^2_{t,x,v}}^2\\
&\lesssim  \ 
\big\{( \e^{\frac{15}{16}} \| \mathfrak{w} f_R \|_{L^\infty_{t,x,v}} )^2
+  \e  \kappa^{1/2}   \|  (\ref{transp:mu})_*  \|_{L^\infty_{t,x}}
\big\}
 \|\kappa^{-\frac{1}{2}} \e^{-1} \sqrt{\nu} (\mathbf{I} - \mathbf{P})f_R\|_{L^2_{t,x,v}}^2
\\
& \ \ + \| (\ref{transp:mu})\|_{L^\infty_{t,x}} \int^t_0 \| P f_R (s) \|_{L^2_x}^2 \dd s +  (\kappa^{\frac{1}{2}} \e^{\frac{1}{16}}
  \| 
 \nabla_x u  
  \|_{L^2_{t,x }} )^2.
\end{split}
\Ee

Finally we control the boundary term (\ref{Energy_bdry}) using a trace theorem (\ref{trace1}).
\hide
 (Lemma 3.2 in \cite{EGKM2})
\Be\label{trace1}
\frac{1}{\e}\int^t_0 \int_{\gamma_+^N} |h| \dd \gamma \dd s \lesssim_N  \iint_{\O \times \R^3} |h(0)|  +   \int^t_0 \iint_{\O \times \R^3} |h|  + \int^t_0 \iint_{\O \times \R^3 } | \p_t h +\frac{1}{\e} v\cdot \nabla_x h|,
\Ee
where $\gamma_+^N:= \{(x,v) \in \gamma_+: |n(x) \cdot v|> 1/N \ \text{and} \ 1/N<|v|<N \}$. 
\unhide
First we have, from (\ref{bdry_fR}),
\Be
\begin{split}
(\ref{Energy_bdry}) &= \frac{1}{2\e} \int_0^t \int_{\gamma_+} \{
|f_R|^2 - |P_{\gamma_+} f_R|^2
\} - \frac{\e}{2\delta^2} \int^t_0 \int_{\gamma_-} |(1- P_{\gamma_+})
 (\mathbf{I}-\mathbf{P}) 
 f_2|^2\\
 &
 \ \ -  \int^t_0 
\int_{\gamma_-}  \frac{1}{\e^{1/2}}P_{\gamma_+ } f_R \frac{\e^{1/2}}{\delta}(1-  {P}_{\gamma_+}) (\mathbf{I}-\mathbf{P})  f_2 
\\
&\geq  \frac{1}{2 } 
| \e^{-\frac{1}{2}}
(1- P_{\gamma_+}) f_R
|_{L^2_t L^2_{\gamma_+}}^2  
-\frac{1}{8C} | \e^{-\frac{1}{2}}
  P_{\gamma_+}   f_R
|_{L^2_t L^2_{\gamma_+}}^2\\
& \ \ \ 
 -( \frac{\e}{2\delta^2}
  + 2C \frac{\e}{\delta^2}
 )
  \int^t_0 \int_{\gamma_+} |(1- P_{\gamma_+}) 
   (\mathbf{I}-\mathbf{P}) 
   f_2|^2
  \ \ \text{for} \ C\gg 1,
\label{est:EB}
\end{split}
\Ee
where we have used the fact $|P_{\gamma_+}f_R|_{L^2_{\gamma_+}}= |P_{\gamma_+}f_R|_{L^2_{\gamma_-}}$ from $P_{\gamma_+}f_R(t,x,v)$ being a function of $(t,x,|v|)$ due to $u|_{\p\O}=0$.

Now we estimate $P_{\gamma_+}f_R$. 
Since $P_{\gamma_+}$ in (\ref{bdry_fR}) is a projection of $c_\mu\sqrt{\mu}$ on $\gamma_+$, it follows $\int_{\gamma_+} |P_{\gamma_+}f|^2 \leq 2 \int_{\gamma_+^N} |P_{\gamma_+}f|^2$ for large enough $N>0$, where $\gamma_+^N:= \{(x,v) \in \gamma_+: |n(x) \cdot v|> 1/N \ \text{and} \ 1/N<|v|<N \}$. Setting $h=|f|^2$ in (\ref{trace}) and using (\ref{eqtn_fR}), (\ref{est:R1}), and (\ref{est:R2}) we derive 
\Be \label{est:bdry1}
\begin{split}
&\frac{1}{\e}\int^t_0 \int_{\gamma_+ } |f|^2 \dd \gamma \dd s\\
&
 \leq C_N  \iint_{\O \times \R^3} |f(0)|^2  +   \int^t_0 \iint_{\O \times \R^3} |f|^2  \\
& \ \ \ + \int^t_0 \iint_{\O \times \R^3 }\Big|\Big[
 - \frac{1}{\e^2 \kappa} Lf_R  + \frac{1}{\kappa} \Gamma({f_2}, f_R)  
+    \frac{   \delta }{ \e\kappa}\Gamma(f_R, f_R) \\
&  \ \ \ \ \ \ \ \ \ \ \ \  \ \ \ \ \ \ \ \ \ \ \ \ 
-  \frac{(  \p_t + 
\e^{-1} v\cdot \nabla_x) \sqrt{\mu}}{\sqrt{\mu}}  f_{R} 
  + 
 (\mathbf{I}- \mathbf{P})\mathfrak{R}_1  +  \mathfrak{R}_2\Big] f_R\Big|\\
 &\leq C_N \big\{  \|f(0)\|^2_{L^2_{x,v}}  + \|    \mathbf{P}  f_R \|_{L^{2}_{t,x,v}}^2
 + \| \e^{-1} \kappa^{-1/2} \sqrt{\nu} (\mathbf{I} - \mathbf{P}) f_R \|_{L^{2}_{t,x,v}}^2\\
 &  \ \ \ \ \ \ \ \ \ \ \ \  \ \ \ 
 + (\ref{est:EG}) + (\ref{est:Gf_2}) +(\ref{est:I-PR_1R2})
 + (\ref{est:Energy_v^3})\big\}.
 \end{split}
\Ee
Furthermore from (\ref{bdry_fR}) and (\ref{est:bdry1}) 
\Be\label{f:+-}
| f_R|_{L^2_t L^2_{\gamma_-}
 }^2\lesssim 
| f_R|_{L^2_t L^2_{\gamma_+}
 }^2 + 
 \frac{\e^2}{\delta^2} | (1- P_{\gamma_+})  (\mathbf{I}-\mathbf{P}) f_2| _{L^2_t L^2_{\gamma_-}
 }
 ^2
 = | f_R|_{L^2_t L^2_{\gamma_+}
 }^2 + 
 \frac{\e^2\kappa^2}{\delta^2} | \nabla_x u| _{L^2_t L^2 (\p\O)
 }
 ^2 . 
\Ee


Finally we collect the terms as 
\Be\notag
\begin{split}
&\text{r.h.s of} \ (\ref{est:EL})+ (\ref{est:EB})
+ \frac{1}{4C}  | \e^{-\frac{1}{2}}
  P_{\gamma_+}   f_R
|_{L^2_t L^2_{\gamma_+}}^2 + \frac{\e^{-1}}{16C} | f_R|_{ L^2_t L^2_{\gamma_-}}^2\\
&\leq \text{r.h.s of} \  (\ref{est:EG}) + (\ref{est:Gf_2}) +  (\ref{est:I-PR_1R2}) + (\ref{est:Energy_v^3}) + \frac{1}{4C} \times  \text{r.h.s of} \   (\ref{est:bdry1})
+ \frac{\e^{-1}}{16C}  \times  \text{r.h.s of} \    (\ref{f:+-}). 
\end{split}
\Ee
We choose large $N$ and then large $C$ so that $\frac{C_N}{4C}\ll \sigma_0$. Using Young's inequality for products, and then moving contributions of $ \| \kappa^{-\frac{1}{2}} \e^{-1} \sqrt{\nu} (\mathbf{I} - \mathbf{P}) f_R \|_{L^2_{t,x,v}}^2$ to l.h.s., we derive (\ref{est:Energy}).


\smallskip

Next we prove (\ref{est:Energy_t}). An energy estimate to (\ref{eqtn_fR_t}) and (\ref{bdry_fR_t}) lead to (\ref{est:Energy_t})
\begin{align}
&\frac{1}{2}\| \p_t f _R(t)\|_2^2 - \frac{1}{2} \| \p_t  f _R (0)\|_2^2
+ \frac{1}{ \kappa \e^{2 }} \int^t_0 \iint _{\O \times \R^3}\p_t  f _R L  \p_t f_R\label{Energy_LHS_t} \\
& +  \frac{1}{2\e} \int^t_0 \int_{\gamma_+} |\p_t f_R|^2 \notag\\
&-  \frac{1}{2\e} \int^t_0 \int_{\gamma_-} | 
  P_{\gamma_+} \p_t f_R- \frac{\e}{\delta}  (1-P_{\gamma_+})  \p_t   (\mathbf{I} -\mathbf{P}) f_2 
+r_{\gamma_+} (f_R)- \frac{\e}{\delta}r_{\gamma_+}  (
 (\mathbf{I} -\mathbf{P})
f_2)
|^2 \label{Energy_bdry_t}
 \\
= &
  - \frac{1}{\e^2 \kappa } \int^t_0 \iint_{\O \times \R^3} L_t (\mathbf{I} - \mathbf{P})f_R
 \p_t  f_R +\frac{1}{\e^2 \kappa }  \int^t_0 \iint_{\O \times \R^3} L(\mathbf{P}_t f_R)(\mathbf{I} - \mathbf{P}) \p_t f_R
  \label{Energy_L_t}
 \\
& +  \frac{2\delta }{\kappa \e } \int^t_0 \iint_{\O \times \R^3} 
\Gamma (f_R,\p_t f_R) (\mathbf{I} - \mathbf{P}) \p_t f_R 
+  \frac{2 }{  \kappa} \int^t_0 \iint_{\O \times \R^3} \Gamma( f_2,\p_t f_R) (\mathbf{I} - \mathbf{P})\p_t  f_R
\label{Energy_Gamma_t}
\\
&
+
 \frac{2}{  \kappa} \int^t_0 \iint_{\O \times \R^3} \Gamma(\p_t f_2, f_R) (\mathbf{I} - \mathbf{P})\p_t  f_R 
 \label{Energy_Gf_2_t}
\\
&+
 \frac{2}{\kappa} \int^t_0 \iint_{\O \times \R^3} \Gamma_t(  {f_2}, f_R)
\p_t f_R+ \frac{\delta}{\e \kappa }   \int^t_0 \iint_{\O \times \R^3}  \Gamma_t (f_R,f_R)\p_tf_R
\label{Energy_p_tGamma}
\\
&+ \int^t_0 \iint_{\O \times \R^3} ( \mathbf{I} -  \mathbf{P}) \mathfrak{R}_3  ( \mathbf{I} -  \mathbf{P}) \p_t f_R \label{Energy_I-PR_1_t}
\end{align}
\begin{align}
&+ \int^t_0 \iint_{\O \times \R^3}  \mathfrak{R}_4  \p_t f_R \label{Energy_R2_t}\\
&+ \int^t_0 \iint_{\O \times \R^3}
\frac{- (\p_t + \e^{-1} v\cdot \nabla_x ) \sqrt{\mu}}{\sqrt{\mu}} |\p_t f_R|^2
+ \int^t_0 \iint_{\O \times \R^3}\p_t \Big(
\frac{ -(\p_t + \e^{-1} v\cdot \nabla_x ) \sqrt{\mu}}{\sqrt{\mu}} \Big)f_R  \p_t f_R .\label{Energy_v^3_t}
\end{align}

We consider the first term of (\ref{Energy_v^3_t}). We decompose $\p_t f_R = \mathbf{P} \p_t f_R + (\mathbf{I} - \mathbf{P}) \p_t f_R$. The contribution of $\mathbf{P} \p_t f_R$ can be bounded above as, from (\ref{transp:mu}),  
\Be\label{est:Energy_v^3_t:1}
  \| (\ref{transp:mu}) 
 \|_{L^\infty_{t,x}  }
  \int^t_0
  \|     P \p_tf_R(s) \|_{L^{2}_x }^2 \dd s  .
\Ee
For the contribution of $(\mathbf{I} - \mathbf{P}) \p_t f_R$ we utilize an extra decomposition $\mathbf{1}_{|v|\leq \e^{-1}} + \mathbf{1}_{|v|\geq  \e^{-1}}$ Then it is bounded as 
\Be\label{est:Energy_v^3_t:2}
\begin{split}
&\|(\ref{transp:mu}) \|_\infty
\Big\{
\iiint  \mathbf{1}_{|v|\leq  \e^{-1}} |v|   \nu(v) |(\mathbf{I} - \mathbf{P}) \p_t f_R|^2\\
& \ \ \ \ \ \ \ \ \ \ \ \ \ \ 
+ \iiint  \mathbf{1}_{|v|\geq  \e^{-1}} \frac{|v|^{3/2}}{\mathfrak{w}^\prime (v)} \mathfrak{w}^\prime (v) \p_t f_R(v) \sqrt{\nu(v)} |(\mathbf{I} - \mathbf{P}) \p_t f_R |
\Big\}\\
\lesssim & \ \|(\ref{transp:mu}) \|_\infty
\Big\{ \e^{-1} \| \sqrt{\nu} (\mathbf{I} -\mathbf{P}) \p_t f_R \|_{L^2_{t,x,v}}^2
+ e^{-\frac{\varrho}{4\e^2}}
\| \mathfrak{w}^\prime \p_t f_R \|_{L^2_t L^\infty_{x,v}} \| \sqrt{\nu} (\mathbf{I} -\mathbf{P}) \p_t f_R \|_{L^2_{t,x,v}}
\Big\}.
\end{split}
\Ee 

For the second term of (\ref{Energy_v^3_t}) using (\ref{transp:mu_t}) we bound it by \\
\Be
\begin{split}
& 
 \| (\ref{transp:mu_t})_* \|_{L^\infty_{t,x}}
\| \sqrt{\nu} (\mathbf{I} -\mathbf{P})  f_R \|_{L^2_{t,x,v}}
 \| \sqrt{\nu} (\mathbf{I} -\mathbf{P}) \p_t f_R \|_{L^2_{t,x,v}}\\ 
 &
+ e^{-\frac{\varrho}{4\e^2}}
\| \nabla_x \p_t u \|_{L^2_t L^\infty_{ x}}
\| \mathfrak{w}   f_R \|_{ L^\infty_{t,x,v}} \| \sqrt{\nu} (\mathbf{I} -\mathbf{P}) \p_t f_R \|_{L^2_{t,x,v}}
 \\
&+\{ \| \nabla_x \p_t u \|_{L^\infty_{t,x}}+ \|(\ref{transp:mu_t})_*\|_{L^\infty_{t,x}}\}
\Big\{
\int^t_0 \| Pf_R (s)\|_{L^2_x}^2 \dd s  + \int^t_0 \|   P\p_t f_R  (s)\|_{L^2_x}^2 \dd s  
\Big\}.
 \label{est:Energy_v^3_t_second}
\end{split}
\Ee

Using (\ref{est:L_t}) we bound (\ref{Energy_L_t}) and (\ref{Energy_p_tGamma}) as 
\Be\begin{split}
|(\ref{Energy_L_t})|&\lesssim \kappa^{-\frac{1}{2}} \|\p_t u\|_{L^\infty_{t,x}} 
\| \kappa^{-\frac{1}{2}} \e^{-1} \sqrt{\nu} (\mathbf{I} - \mathbf{P}) f_R \|_{L^2_{t,x,v}} \\
& \ \ \ \ \ \ \ \ \ \times
\{\| P \p_t f_R \|_{L^2_{t,x}}
+ \kappa^{ \frac{1}{2}} \e 
\| \kappa^{-\frac{1}{2}} \e^{-1} \sqrt{\nu} (\mathbf{I} - \mathbf{P}) \p_t f_R \|_{L^2_{t,x,v}} 
\} \\
& \ \ + \kappa^{-\frac{1}{2}} \|\p_t u\|_{L^\infty_{t,x}} 
\| \kappa^{-\frac{1}{2}} \e^{-1} \sqrt{\nu} (\mathbf{I} - \mathbf{P})\p_t  f_R \|_{L^2_{t,x,v}}
\| P   f_R \|_{L^2_{t,x}}
,
\label{est:Energy_L_t} 
\end{split}\Ee
\Be\begin{split}\label{est:Energy_p_tGamma}
|(\ref{Energy_p_tGamma})|& \lesssim  
\kappa^{-\frac{1}{2}}\e\|\p_t u \|_{L^\infty_{t,x}} \| (\ref{est:f2}) \|_{L^\infty_{t,x}} \{ \| \sqrt{\nu}(\mathbf{I} - \mathbf{P}) f_R \|_{L^2_{t,x,v}} + \| P f_R \|_{L^2_{t,x}}\} \| \p_t f_R \|_{L^2_{t,x,v}}\\
& \ \ + \delta \kappa^{-1} \|\p_t u\|_{L^\infty_{t,x}}
\{
\| P \p_t f_R \|_{L^2_{t,x}}+ \| \sqrt{\nu} (\mathbf{I} -\mathbf{P}) \p_t f_R \|_{L^2_{t,x,v}}
\}\\
&  \ \ \ \ \ \ \ \ \ \  \times 
\{
\| P f_R \|_{L^\infty_tL^6_x} \| Pf_R \|_{L^2_t L^3_x}
+ \| \kappa^{-\frac{1}{2}} \e^{-1} (\mathbf{I} -\mathbf{P}) f_R \|_{L^2_{t,x,v}} 
\kappa^{ \frac{1}{2}} \e \|  \mathfrak{w}f _R\|_{L^\infty_{t,x,v}}
\}.
\end{split}\Ee

\hide
From (\ref{transp:mu}), we bound  by
\Be\label{est:Energy_v^3_t}
 \begin{split}
 &
 \| 
 |\nabla_x u| +  |\p_t u| + \e |u||\nabla_x u|  
 \|_{L^\infty_{t,x} ([0,T ] \times \O)}
  \\
 & \times 
 \{ \kappa^{ \frac{1}{2}} \e   \|w \p_t f_R\|_{L^2_t ((0,T) ; L^\infty( \O \times \R^3))}
  \| \kappa^{-\frac{1}{2}} \e^{-1} (\mathbf{I} - \mathbf{P}) \p_t f_R\|_{L^2 ([0,T] \times \O \times \R^3)}
  +
  \int^t_0
  \|     P \p_tf_R(s) \|_{L^{2}_x(\O)}^2 \dd s  \}.
 \end{split}
\Ee\unhide

The rest of terms can be controlled similarly as in the proof of (\ref{est:Energy}):
\begin{align}
(\ref{Energy_LHS_t})\geq &  \  \frac{1}{2}\| \p_t f _R(t)\|_{L^2_{x,v}}^2 - \frac{1}{2} \|\p_t f _R (0)\|_{L^2_{x,v}}^2
+  \sigma_0
\| \kappa^{-\frac{1}{2}} \e^{-1} \sqrt{\nu}(\mathbf{I} - \mathbf{P}) \p_t f_R\|_{L^2_{t,x,v}}^2,\label{est:EL_t} \\
|(\ref{Energy_Gamma_t})|  \lesssim &   \ 
\{  \delta \e \|   \mathfrak{w}f_R\|_{L^\infty_{t,x,v}} +  \e^2 \|   (\ref{est:f2})\|_{L^\infty_{t,x }}   \} \| \kappa^{-\frac{1}{2}} \e^{-1} \sqrt{\nu}
(\mathbf{I} - \mathbf{P}) \p_t f_R \|^2_{L^2_{t,x,v}} \notag
\\
&+  \frac{  \delta
}{\kappa^{3/2}}
  \|\kappa^{1/2}     {P} f_R \|_{L^\infty_tL^6_{x }} 
 \| \kappa^{1/2}   {P} \p_t f_R \|_{L^2_tL^3_{x }}
\| \kappa^{-\frac{1}{2}} \e^{-1} \sqrt{\nu}
(\mathbf{I} - \mathbf{P}) \p_t f_R \|_{L^2_{t,x,v}}\label{est:EG_t}\\
&+ \frac{\e}{\kappa^{1/2}}   \|    {P} f_2 \|_{L^\infty_{t,x} }
\|  {P} \p_t f_R \|_{L^2_{t,x} }
\| \kappa^{-\frac{1}{2}} \e^{-1} \sqrt{\nu}
(\mathbf{I} - \mathbf{P}) \p_t f_R \|_{L^2_{t,x,v}},\notag
\end{align}
\begin{align}
\\
|(\ref{Energy_Gf_2_t})| 
 \lesssim & \  \kappa^{-\frac{1}{2}}\e
   \| \sqrt{\nu} \p_t f_2 \|_{L^\infty_{t,x,v}}
 \{
 \|P f_R \|_{L^2_{t,x}}
 +  \|\sqrt{\nu}(\mathbf{I} -\mathbf{P}) f_R \|_{L^2_{t,x,v}}
 \}
  \| \kappa^{-\frac{1}{2}} \e^{-1} (\mathbf{I} - \mathbf{P})  \p_t f_R \|_{L^2_{t,x,v}} 
,
  \label{est:Gf_2_t}\\
  |(\ref{Energy_I-PR_1_t})| \lesssim & \  
 \kappa^{ 1/2} \e  \|
 (\ref{est:R3})
 \|_{L^2_{t,x}} \| \kappa^{-1/2} \e^{-1} (\mathbf{I} - \mathbf{P} )\p_t f_R \|_{L^2_{t,x,v}},
\label{est:Energy_I-PR_1_t}
\\
|(\ref{Energy_R2_t})| \lesssim   & \ 
\| (\ref{est:R4}) \|_{L^2_{t,x}}
 \big\{
  \|  {P } \p_t f_R\|_{L^2_{t,x}}+ 
  \kappa^{1/2}\e
  \|
  \kappa^{-1/2} \e^{-1}
  (\mathbf{I} -\mathbf{P}) \p_tf_R\|_{L^2_{t,x,v}} \big\}.\label{est:Energy_R2_t}
\end{align}
 \hide We recall (\ref{est:f2_t}).
We remark that, from (\ref{f_2}), (\ref{est:A&A_x}), and (\ref{Pf_2}), 
\Be\label{est:f2_t}
|\p_t (\mathbf{I} - \mathbf{P}) f_2| \lesssim \kappa ( |\nabla_x\p _t u| + 
 \e |\p_t u| |\nabla_x u|) e^{-\varrho |v-\e u|^2}, \ |\p_t Pf_2| \lesssim 
 (|\p_t p| + |\p_t \tilde{u} |) + \e |\p_t u| (|p| + |\tilde{u}|).
\Ee \unhide

Lastly we estimate (\ref{Energy_bdry_t}) and the first term of (\ref{Energy_v^3_t}). As in (\ref{est:EB}) we derive that (\ref{Energy_bdry_t}) is bounded from below by
 \Be
 \begin{split}\label{est:EB_t}
 \frac{1}{2 } 
&| \e^{-\frac{1}{2}}
(1- P_{\gamma_+}) \p_t  f_R
|_{L^2((0,T); L^2_{\gamma_+})}^2  
-\frac{1}{8C} | \e^{-\frac{1}{2}}
  P_{\gamma_+} \p_t   f_R
|_{L^2 ((0,T); L^2_{\gamma_+})}^2\\
&
 -
 C
 \Big\{
  \frac{\e}{ \delta^2} 
 |(1- P_{\gamma_+}) \p_t  (\mathbf{I} -\mathbf{P})   f_2|_{L^2 ((0,T); L^2_{\gamma_-})}^2
 +  \e \| \p_t  u \|_{\infty } |      f_R
|_{L^2 ((0,T); L^2_{\gamma_+})}^2\\
& \ \ \ \ \ \  +  \| \p_t  u \|_{\infty}\frac{\e^3}{\delta} |  (\mathbf{I} -\mathbf{P})    f_2
|_{L^2 ((0,T); L^2_{\gamma_+})}^2
  \Big\}\\
\geq  & \frac{1}{2 } 
 | \e^{-\frac{1}{2}}
(1- P_{\gamma_+}) \p_t  f_R
|_{L^2((0,T); L^2_{\gamma_+})}^2  
-\frac{1}{8C} | \e^{-\frac{1}{2}}
  P_{\gamma_+} \p_t   f_R
|_{L^2 ((0,T); L^2_{\gamma_+})}^2\\
&
 -
 C
 \Big\{
  \frac{\e \kappa^2}{ \delta^2} 
 | |\p_t \nabla_x u| + \e |\p_t u| |\nabla_x u| |_{L^2_tL^2 (\p\O) }^2
 +  \e \| \p_t  u \|_{\infty } |      f_R
|_{L^2 ((0,T); L^2_{\gamma_+})}^2\\
& \ \ \ \ \ \ \ \ \ +  \| \p_t  u \|_{\infty}\frac{\e^3}{\delta} 
\kappa^2 |\nabla_x u|^2_{L^2_t L^2(\p\O)} 
  \Big\}
  \ \ \text{for} \ C\gg 1,
  \end{split}
 \Ee
 where we have used $|r_{\gamma_+} (g)|_{L^2(\gamma_-)}\lesssim  \e \| \p_t u \|_\infty
 | g|_{L^2(\gamma_-)}$ from (\ref{bdry_fR_t}).  Now we bound $P_{\gamma_+} \p_t f_R$ using (\ref{trace}). Following the argument arriving at (\ref{est:bdry1}) and setting $h=|\p_t f|^2$ we derive 
 \Be
 \begin{split}
& \frac{1}{\e} \int^t_0 \int_{\gamma_+} |   \p_t f  |^2 \dd \gamma \dd s \\
 &\lesssim_N
 \| \p_t f(0) \|_{L^2_{x,v}} +  \| \p_t f  \|_{L^2_{t,x,v}} 
 + \int_0^t \iint_{\O \times \R^3} \Big|\Big( - \frac{1}{\e^2 \kappa} L \p_t f_R  + \text{r.h.s of } (\ref{eqtn_fR_t})\Big)\p_t f_R  \Big|\\
 &\lesssim_N \| \p_t f(0) \|_{L^2_{x,v}} ^2+  \| P \p_t f  \|_{L^2_{t,x}} ^2
 + \| \e^{-1} \kappa^{-1/2} \sqrt{\nu} (\mathbf{I}-\mathbf{P}) \p_t f  \|_{L^2_{t,x,v}} ^2\\
 & \ \ \  \ \ \   + (\ref{est:Energy_v^3_t:1}) + \cdots + (\ref{est:Energy_p_tGamma})+
 (\ref{est:EG_t})+  \cdots + (\ref{est:Energy_R2_t}).
 \end{split}\Ee

We conclude (\ref{est:Energy_t}) by collecting the terms.\end{proof}

\hide

\Be
\begin{split}\label{est:f2}
|(\ref{Energy_f2})| & \leq \{ \e \kappa^{\frac{3}{2}} \delta^{-1} 
\| \nabla_x^2 u_\kappa \|_{L^2_{t,x,v}}
+\e \kappa^{\frac{1}{2}} \delta^{-1} \| |u_\kappa||\nabla_x u_\kappa|\|_{L^2_{t,x,v}}
\}
\| \kappa^{-\frac{1}{2}} \e^{-1} \sqrt{\nu} (\mathbf{I}-\mathbf{P}) f_R \|_{L^{2}_{t,x,v}}\\
& \leq  C_{\sigma_0} 
\| \e \kappa^{\frac{3}{2}} \delta^{-1}  \nabla_x^2 u_\kappa \|_{L^2_{t,x,v}}^2
+C_{\sigma_0} \|\e \kappa^{\frac{1}{2}} \delta^{-1} |u_\kappa||\nabla_x u_\kappa|\|_{L^2_{t,x,v}}^2
+ \frac{\sigma_0}{10}\| \kappa^{-\frac{1}{2}} \e^{-1} \sqrt{\nu} (\mathbf{I}-\mathbf{P}) f_R \|_{L^{2}_{t,x,v}}^2.
\end{split}\Ee

\textit{Estimate of (\ref{Energy_f3}):} From (\ref{f_2}), 
\Be
\begin{split}\label{est:f3}
|(\ref{Energy_f3})| &\leq \e \delta^{-1} \| \p_t\mathbf{P}f_2 \|_{L^2_{t,x,v}} \| \mathbf{P} f_R \|_{L^2_{t,x,v}}\\
&+ 
\{  
\e^2 \kappa^{3/2} \delta^{-1}
\|  \p_t \nabla_x u_\kappa \|_{L^2_{t,x,v}}
+
\e^2 \kappa^{1/2} \delta^{-1}
\| |\p_t u_\kappa | | u_\kappa |\|_{L^2_{t,x,v}}
\}
\| \kappa^{-1/2} \e^{-1}\sqrt{\nu}(\mathbf{I} - \mathbf{P}) f_R\|_{L^2_{t,x,v}} \\
&+ \delta^{-1}\kappa^{ \frac{1}{2}} \e   \| \nabla_x (\rho_2, u_2, \theta_2) \|_{L^2_{t,x,v}} \| \kappa^{-\frac{1}{2}} \e^{-1} \sqrt{\nu} (\mathbf{I} - \mathbf{P}) f_R \|_{L^2_{t,x,v}}
\end{split}
\Ee

\hide Furthermore  
\begin{align}
|(\ref{Energy_Gamma}) | \leq& \   {\delta}{ \kappa ^{-1/2} } \int^t_0  \|\nu^{-1/2} \Gamma(f_R,f_R) \|_2 
\| 
\kappa^{-1/2} \e^{-1} \nu^{1/2} (\mathbf{I} - \mathbf{P}) f_R \|_2 \notag\\
 \leq & \ o(1) \int^t_0 \| 
\kappa^{-1/2} \e^{-1} \nu^{1/2} (\mathbf{I} - \mathbf{P}) f_R \|_2^2 \notag\\
&+ C \delta^2 \kappa^{-1} \int_0^t\|\nu^{-1/2} \Gamma(f_R,f_R) \|_2 ^2\label{Gamma_L2}
\end{align} 
and 
\begin{align}
(\ref{Energy_forcing})  \leq & \ o(1) \int^t_0 \| 
\kappa^{-1/2} \e^{-1} \nu^{1/2} (\mathbf{I} - \mathbf{P}) f_R \|_2^2 \notag\\
&+ C  \kappa^{3 } \e^2 \delta^{-2} \int^t_0 \|(\mathbf{I} - \mathbf{P}) ( v\cdot \nabla_x L^{-1} (v\cdot \nabla_x u_\kappa \cdot v \sqrt{\mu}) ) \|_2^2\label{Forcing_L2}
\end{align} \unhide
\unhide

\hide

Thus we might have  
\Be\begin{split}
&\| f_R(t) \|_2^2 + \int^t_0\| \e^{-1 } \kappa^{-\frac{1}{2}} (\mathbf{I} - \mathbf{P}) f_R\|_2^2\\
\lesssim  & \| f_R (0) \|_2^2 + \int^t_0 \frac{\delta}{\kappa^{1/2}}  \| \Gamma(f_R, f_R) \|_2\| \e^{-1 } \kappa^{-\frac{1}{2}}  (\mathbf{I} - \mathbf{P}) f_R\|_2 
\\
&+ \int_0^t \iint_{\O \times \R^3} \frac{\kappa^{1/2}}{\delta} \| 
(\mathbf{I} - \mathbf{P} )\mathfrak{M}_1 \|_2 \| 
\e^{-1} \kappa^{-1/2} 
(\mathbf{I} - \mathbf{P}) f_R \|_2
\\
&+ \frac{1}{\delta} \int^t_0 \iint |\mathbf{P} \mathfrak{M}_2 \mathbf{P} f_R|+ \int^t_0 \iint_{\O \times \R^3}\frac{[\p_t + \e^{-1} v\cdot \nabla_x ]\mu_\e}{\mu_\e}
|f_R|^2  
\end{split}\Ee
As long as the lower order term can be bounded, the major task is to have a good bound of the $\Gamma$-contribution in the last line. By decomposition $f_R = \mathbf{P} f_R + ( \mathbf{I} - \mathbf{P})f_R$, we have 
\Be\begin{split}\label{est_Gamma}
& \int^t_0 \delta^2   \|  \Gamma(f_R, f_R) \|_2^2\\
\lesssim&  \int^t_0 \delta^2   \|  \Gamma(\mathbf{P}f_R, \mathbf{P}f_R) \|_2^2 + 
  \int^t_0 \delta^2   \|  \Gamma(( \mathbf{I} -\mathbf{P})f_R,  |f_R|) \|_2^2 \\
  \lesssim & \int^t_0 \| \delta^{\frac{1}{2}}  \mathbf{P} f_R\|_4^4
  + \int^t_0\| \e^{-1 } \delta^{-\frac{1}{2}} (\mathbf{I} - \mathbf{P}) f_R\|_2^2 \| \e \delta^{\frac{3}{2}} f_R \|_\infty^2.
\end{split}\Ee

This estimate holds for any dimension. But the spaces are chosen for 3D. 
\unhide

\hide
\section{$\mathbf{P}f^\e$-estimate}
From now we set
\Be\label{delta_e}
\delta= \e^{\kappa} \ for \ some \ \kappa>0.
\Ee

Roughly we have 
\Bes
 \e^{\f{\kappa}{2}}  v\cdot \nabla_x f_R = - \e^{1+  \f{\kappa}{2}}\p_t f_R   
- \e^{-1-\f{\kappa}{2}}  {L}  f_R  +   \e^{\f{\kappa}{2}} \Gamma (f_R , f_R ) {\color{red}+(l.o.t)}
\Ees
The weak formulation is 
\Bes
 \iint_{\O \times \R^3} - \e^{\f{\kappa}{2}}   f_R v\cdot \nabla_x \phi
= \iint_{\O \times \R^3} \{ - \e^{1+  \f{\kappa}{2}}\p_t f_R  
- \e^{-1-\f{\kappa}{2}}  {L}  f_R +   \e^{\f{\kappa}{2}} \Gamma (f_R, f_R)\} \phi
\Ees
Set {\color{red}(Check the set of $\psi$'s in \cite{EGKM2})}, for some integer $m>0$,
\Be
\phi = \psi(v) \cdot \nabla_x (-\Delta_x)^{-1} | \e^{\f{\kappa}{2}} \mathbf{P}f_R|^{(m-1)} 
\Ee
Then 
\Bes
  \iint |\e^{\f{\kappa}{2}} \mathbf{P}f_R |^m \lesssim  
\| \e^{-1 - \f{\kappa}{2}} (\mathbf{I} - \mathbf{P}) f_R\|_2  
\| \nabla_x (-\Delta_x)^{-1} | \e^{\f{\kappa}{2}} \mathbf{P}f_R|^{m-1}  \|_2  + \cdots
\Ees
Note that if $h \in L^m$ then $|h|^{m-1} \in L^{\frac{m}{m-1}}$ and hence $\nabla_x (-\Delta_x)^{-1} |h|^{m-1} \in W^{1,\frac{m}{m-1} }$. From Gagliardo-Nirenberg-Sobolev inequality in 2D
\Be
\|\nabla_x (-\Delta_x)^{-1} |h|^{m-1}  \|_{2}\lesssim \| \nabla_x \nabla_x (-\Delta_x)^{-1} |h|^{m-1}  \|_{\frac{m}{m-1}} \lesssim 
 \|   |h|^{m-1}  \|_{\frac{m}{m-1}} \lesssim \| h  \|_{m}^{m-1}
\Ee 
with $\frac{1}{2} = \frac{1}{m/(m-1)}- \frac{1}{2}$ which implies that for any $m<\infty$
\Bes
\| \nabla_x (-\Delta_x)^{-1} | \e^{\f{\kappa}{2}} \mathbf{P}f_R|^{m-1}  \|_2
\lesssim
\| \e^{\f{\kappa}{2}} \mathbf{P}f_R\|_{L^m(\R^3_x)}^{m-1}
\Ees
Therefore we have, for any $m<\infty$ 
\Be\label{Lm}
\sup_t\| \e^{\f{\kappa}{2}} \mathbf{P} f_R \|_{L^m}  \lesssim \| \e^{-1 - \f{\kappa}{2}} (\mathbf{I} - \mathbf{P}) f_R \|_2.
\Ee

\unhide

\subsection{$L^6_x$-integrability gain for $\mathbf{P}f_R$}

\begin{proposition}\label{prop:L6}
Under the same assumptions in Proposition \ref{prop:Hilbert}, we have  for all $t  \in [0,T]$
\Be\label{L6}
\begin{split}
&d_6
\|   {P} f_R(t) \|_{L^6_x} \\
\lesssim & \  
 (\e\| (\ref{transp:mu})\|_{L^\infty_{t,x}}
+  \e  \kappa^{-1} 
 \|(\ref{est:f2})\|_{L^\infty_{t,x}}
)
 \| f_R(t) \|_{L^2_{x,v}}
%
%
+ \e \| \p_t f_R(t) \|_{ {L^2_{x,v}}} 
\\
&+o(1)  (\kappa \e)^{1/2}
 \| \mathfrak{w} f_R (t) \|_{L^\infty_{x,v}} 
+  \frac{\e}{\delta}| (\ref{est:f2})|_{L^4(\p\O)} 
  + \e \| 
 (\ref{est:R1})  \|_{L^2_{x,v}}+  \e \| (\ref{est:R2})
 \|_{L^2_{x,v}} \\
&
  + 
\Big(
\frac{1}{\e \kappa }+
 \frac{\delta}{\kappa} \| \mathfrak{w}_{\varrho,\ss} f_R(t) \|_{L^\infty_{x,v}}
 \Big)
 {
\big\{
\| (\mathbf{I} - \mathbf{P}) f_R  \|_{ {L^2_{t,x,v}}} + \| (\mathbf{I} - \mathbf{P})\p_t  f_R  \|_{ {L^2_{t,x,v}}} + \e \| \p_t u \|_{L^\infty_{t,x}} \| P f_R \|_{L^2_{t,x}}
\big\}
}
\\
&+   \| \mathfrak{w}_{\varrho, \ss} f_R(t) \|_{L^\infty _{x,v}}^{1/2}
{
\big\{
|  f_R |_{L^2_tL^2({\gamma_+})}^{1/2}
+ |  \p_t f_R  |_{L^2_tL^2({\gamma_+})}^{1/2}
\big\}
}
,
\end{split}
\Ee
\hide

\Be\label{L6}
\begin{split}
&d_6
\|   {P} f_R(t) \|_{L^6_x} \\
\lesssim & \  
 (\e\| (\ref{transp:mu})\|_{L^\infty_{t,x}}
+  \e  \kappa^{-1} 
 \|(\ref{est:f2})\|_{L^\infty_{t,x}}
)
 \| f_R(t) \|_{L^2_{x,v}}
%
%
+ \e \| \p_t f_R(t) \|_{ {L^2_{x,v}}} 
\\
&+o(1)  (\kappa \e)^{1/2}
 \| \mathfrak{w} f_R (t) \|_{L^\infty_{x,v}} 
+  \frac{\e}{\delta}| (\ref{est:f2})|_{L^4(\p\O)} 
  + \e \| 
 (\ref{est:R1})  \|_{L^2_{x,v}}+  \e \| (\ref{est:R2})
 \|_{L^2_{x,v}} \\
&
  + 
\Big(
\frac{1}{\e \kappa }+
 \frac{\delta}{\kappa} \| \mathfrak{w}_{\varrho,\ss} f_R(t) \|_{L^\infty_{x,v}}
 \Big)
\| (\mathbf{I} - \mathbf{P}) f_R (t) \|_{ {L^2_{x,v}}} +  |(1- P_{\gamma_+}) f_R(t)|_{L^2({\gamma_+})}^{1/2} \| \mathfrak{w}_{\varrho, \ss} f_R(t) \|_{L^\infty _{x,v}}^{1/2},
\end{split}
\Ee\unhide
where
\Be\label{d6}
d_6:=1-\Big[ \frac{\delta}{\kappa} \| Pf_R(t)\|_{L^6_x}^{1/2} \| Pf_R (t) \|_{L^2_x}^{1/2} + \e \| u(t) \|_{L^\infty_x}\Big]^{1/6}.
\Ee

\end{proposition}

\begin{proof} For the sake of simplicity we use notations (\ref{short_notation}) throughout this subsection.

We view (\ref{eqtn_fR}) as a weak formulation for a test function $\psi$
\Be\label{weak_form}
\begin{split}
 &\underbrace{
 \iint_{\O \times\R^3}    f_R   v\cdot \nabla_x \psi}_{(\ref{weak_form})_1}
- \underbrace{
\int_\gamma f_R \psi}_{(\ref{weak_form})_2}
-\underbrace{
 \iint_{\O \times\R^3}      \e \p_t f_R   \psi }_{(\ref{weak_form})_3}
\\ 
&=
   \iint_{\O \times \R^3} \psi \left\{ \frac{1}{ \e \kappa} L
  f_R
- \frac{2\e}{\kappa} \Gamma({f_2}, f_R)
-    \frac{   \delta }{ \kappa}\Gamma(f_R, f_R)
+ \frac{( \e\p_t + 
  v\cdot \nabla_x) \sqrt{\mu}}{\sqrt{\mu}} f_{R}
 -
\e(\mathbf{I}- \mathbf{P})\mathfrak{R}_1 - \e\mathfrak{R}_2\right\}.   
\end{split}
\Ee
The proof of the lemma is based on a recent test function method in the weak formulation (\cite{EGKM, EGKM2}). We define
\Be\label{P0f}
\mathbf{\tilde{P}}f_R :=  \Big\{ a + b \cdot v + c \frac{|v|^2-3}{\sqrt{6}} \Big\} \sqrt{\mu_0} \ \text{and} \  \tilde{P}f_R := (a,b,c),
\Ee
where $a:= \langle f_R, \sqrt{\mu_0} \rangle, b:=  \langle f_R,v \sqrt{\mu_0} \rangle$, and $c:= \langle f_R,  \frac{|v|^2-3}{\sqrt{6}}\sqrt{\mu_0} \rangle$. We choose a family of test functions as %
\begin{align}
 \psi _{a }&:= (|v|^{2}-\beta _{a})v\sqrt{
\mu_0}\cdot \nabla _{x}\varphi_{a}
,  \label{phia}\\
 \psi _{b,1 }^{i,j}&:= (v_{i}^{2}-\beta _{b})\sqrt{\mu_0}\partial _{j}%
\varphi _{b }^{j},\quad i,j=1,2,3,  \label{phibj}\\
\psi^{i,j}_{b,2} &:=|v|^{2}v_{i}v_{j}\sqrt{\mu_0}\partial _{j}\varphi _{b }^{i} ,\quad i\neq
j,  \label{phibij} \\
\psi_{c }&:= (|v|^{2}-\beta _{c})v\sqrt{\mu_0}\cdot \nabla _{x}\varphi%
_{c } ,  \label{phic}
\end{align}
where we choose $\beta_a=10, \beta_b=1, \beta_c=5$ such that 
\Be
0 = \int_{\R^3} (|v|^2 - \beta_a) \frac{|v|^2-3}{\sqrt{6}} (v_1)^2 \mu_0 (v) \dd v 
=
\int_{\R} (v_1^2 - \beta_b) \mu_0 (v_1) \dd v_1 = \int_{\R^3} (|v|^2-\beta_c) v_i^2 \mu_0 (v) \dd v .\label{beta_abc}
\Ee
Here, 
\begin{align}
-\Delta _{x}\varphi_{a }  =  a^{5}&   \ \ \text{with} \ \  \frac{\partial \varphi_{a }}{%
\partial n}\Big|_{\partial \Omega }=0,%
\label{phi_a}
\\ 
-\Delta _{x}\varphi _{b }^{j}  =b^{5}_{j}&   \ \ \text{with} \ \  \varphi _{b}^{j}|_{\partial
\Omega }=0, \label{phi_b}\\
 -\Delta _{x}\varphi_{c }  =c^{5}&  \ \ \text{with} \ \  \varphi %
_{c }|_{\partial \Omega }=0.\label{phi_c}
\end{align} 
A unique solvability to the above Poisson equations when $(a,b,c) \in L^6(\O)$ and an estimate 
\Be
\| \nabla_x^2 \varphi_{(a,b,c)}\|_{L^{6/5}(\O)} + \| \nabla_x \varphi _{(a,b,c)}\|_{L^2(\O)}
+ \| \varphi_{(a,b,c)}\|_{ L^6(\O) } \lesssim \| |\tilde{P}f_R|^5\|_{L^{6/5}(\O)}
\lesssim \| \tilde{P}f_R\|_{L^6(\O)}^5
. \label{elliptic:varphi}
\Ee
 is a direct consequence of Lax-Milgram and suitable extension  (extend $a^5$ of (\ref{phi_a}) evenly in $x_3 \in \R$, and $b^5$ and $c^5$ of (\ref{phi_b}) and (\ref{phi_c}) oddly in $x_3 \in \R$, then solve the Poisson equation, and then restrict the whole space solutions to the half space $x_3>0$)  and a standard elliptic estimate $(L^{\frac{6}{5}}(\O) \rightarrow \dot{W}^{2,\frac{6}{5}} (\O) \cap \dot{W}^{1,2} (\O) \cap L^6(\O))$. 

From $M_{1,\e u , 1}(v) = M_{1,0 , 1}(v) + O( \e) |u| | v-\e u | M_{1,\e u , 1}(v) $ we can easily check that 
\Be\label{diff:P-tP}
|\mathbf{P} f_R (t,x,v)- \mathbf{\tilde{P}}f_R (t,x,v)|\lesssim \e |u(t,x)| |v-\e u| \sqrt{\mu }
 |
f_R(t,x,v)|.
\Ee
Therefore we have 
\Be\label{P-tildeP}
\begin{split}
\| P f_R(t) \|_{ L^6_x} \lesssim&
\| \mathbf{P} f_R (t) \|_{L^6_{x,v}}  \lesssim \| \mathbf{\tilde{P}} f_R (t) \|_{L^6_{x,v}} 
+ \e \| u(t) \|_{\infty} \{
\|  {P} f_R (t) \|_{L^6_{x }}  + \| ( \mathbf{I}-\mathbf{P}) f_R (t) \|_{L^6_{x,v}} 
\}\\
\lesssim & (1+ \e \| u \|_\infty)  \|  {\tilde{P}} f_R (t) \|_{L^6_{x }} 
+ \e \| u(t) \|_{\infty}   \| ( \mathbf{I}-\mathbf{P}) f_R (t) \|_{L^6_{x,v}} 
.
\end{split}
\Ee
Note that $ \| ( \mathbf{I}-\mathbf{P}) f_R (t) \|_{L^6_{x,v}} \leq  \| ( \mathbf{I}-\mathbf{P}) f_R (t) \|_{L^\infty_{x,v}} ^{2/3} \| ( \mathbf{I}-\mathbf{P}) f_R (t) \|_{L^2_{x,v}} ^{1/3}
\lesssim o(1)  (\kappa \e)^{1/2}  \| \mathfrak{w} f_R (t) \|_{L^\infty_{x,v}}  + (\kappa \e)^{-1} \| ( \mathbf{I}-\mathbf{P}) f_R (t) \|_{L^2_{x,v}}$. Hence to prove the lemma and (\ref{L6}) it suffices to prove the same bound for $\|  \tilde{P}f_R \|_{L^6_{x,v}}:=\| ( a,b,c) \|_{L^6_{x }} $.
%

Following the direct computations in the proof of Lemma 2.12 in \cite{EGKM2} we derive that 
\begin{equation}\label{est:wf1}  (\ref{weak_form})_1 =
\begin{cases}
 -5  \| a(t) \|_{6}^6 +o(1)  \|   \mathbf{\tilde{P}} f_R(t) \|_6^6
+O(1) \| (\mathbf{I} - \mathbf{ {P}}) f_R(t) \|_6^6
& \text{if } \psi=\psi_a,\\ 
-2 \int_\O   b_i  \p_i \p_j  \varphi_b^j 
+o(1)  \|   \mathbf{\tilde{P}} f_R(t) \|_6^6 +O(1) \| (\mathbf{I} - \mathbf{ {P}}) f_R(t) \|_6^6  & \text{if } \psi=\psi_{b,1}^{i,j},\\ 
\int_\O b_j \p_{i} \p_j  \varphi_b^i
 + \int_\O b_i \p_{j} \p_j  \varphi_b^i +O(1) \| (\mathbf{I} - \mathbf{ {P}}) f_R(t) \|_6^6
 & \text{if } \psi=\psi_{b,2}^{i,j} \ \text{and} \ i \neq j,\\
 5  \| c(t)\|_6^6 +o(1)  \|   \mathbf{\tilde{P}} f_R(t) \|_6^6 + O(1) \| (\mathbf{I} - \mathbf{ {P}}) f_R(t) \|_6^6  & \text{if } \psi=\psi_c.
    \end{cases}
\end{equation}
For $\|  b_i\|_{6}^6$, using the second and third estimate of (\ref{est:wf1}) we deduce that 
\Be\label{est:wf1_b}
\begin{split}
\|  b_i\|_{L^6(\O)}^6&= - \int_\Omega b_i \Delta_x\varphi_b^i \dd x =  - \int_\Omega b_i \p_i^2\varphi_b^i \dd x  - 
\sum_{j(\neq i)} \int_\Omega b_i \p_j^2\varphi_b^i \dd x
\\
&= \frac{1}{2}\sum_{j}(\ref{weak_form})_1|_{\psi^{j,i}_{b,1}}  - \sum_{j(\neq i)}{(\ref{weak_form})_1|_{\psi^{i,j}_{b,2}}} 
+ o(1)  \|   \mathbf{\tilde{P}} f_R(t) \|_6^6 +O(1) \| (\mathbf{I} - \mathbf{ {P}}) f_R(t) \|_6^6.
\end{split}\Ee
\hide From the second and third estimate of (\ref{est:wf1}) and 
\Be\label{est:wf1_b}
\begin{split}
\|  b_i\|_{L^6(\O)}^6&=
\|b_i \p_i^2 \varphi_b^i \|_{L^1(\O)} +\sum_{j(\neq i)} \|  b_i \p_j^2\varphi_b^i \|_{L^1(\O)} 
\lesssim  -\sum_{j(\neq i)} \|  b_j \p_i \p_j \varphi_b^i \|_{L^1(\O)}
+ o(1)  \|   \mathbf{\tilde{P}} f_R(t) \|_6^6 +O(1) \| (\mathbf{I} - \mathbf{ {P}}) f_R(t) \|_6^6\\
&\lesssim  o(1)  \|  {\tilde{P}} f_R(t) \|_6^6 +O(1) \| (\mathbf{I} - \mathbf{ {P}}) f_R(t) \|_6^6.
\end{split}\Ee\unhide

Now we consider the boundary term $(\ref{weak_form})_2$. 
From (\ref{phia})-(\ref{phic}) and (\ref{beta_abc})
\begin{equation}   \label{no_Pg}
  \int_{\gamma}  \psi P_{\gamma_+} f=
\begin{cases}
 \int_{\p\O}\p_n \varphi_a  \int_{\R^3} (|v|^2 - \beta_a) (v\cdot n)^2 \mu_0 \dd v \dd S_x  =0
& \text{if } \psi=\psi_a,\\
0& \text{if } \psi=\psi_{b,1}^{i,j} \ \text{or} \ \psi_{b,2}^{i,j},\\
 \int_{\p\O}\p_n \varphi_c  \int_{\R^3} (|v|^2 - \beta_c) (v\cdot n)^2 \mu_0 \dd v \dd S_x  =0
& \text{if } \psi=\psi_c. 
\end{cases}
\end{equation}
Here we have used the Neumann boundary condition of (\ref{phi_a}) for $\psi_a$, and the last identity in (\ref{beta_abc}) for $\psi_c$. For $\psi_{b,1}^{i,j}$ or $\psi_{b,2}^{i,j}$ we used the fact that the integrands are odd in $v$. From (\ref{bdry_fR}),  we decompose $f|_{\gamma}= P_{\gamma_+} f + \mathbf{1}_{\gamma_+} (1- P_{\gamma_+})f - \mathbf{1}_{\gamma_-}\frac{\e}{\delta} (1- P_{\gamma_+}) f_2$. From (\ref{no_Pg}) together with (\ref{est:I-Pf2}) and (\ref{est:Pf2}) we have  
\Be
\begin{split}\label{est:wf2}
| (\ref{weak_form})_2|  &= \Big| \cancel{\int_{\gamma}  \psi P_{\gamma_+} f} +    \int_{\gamma}  \psi\{ \mathbf{1}_{\gamma_+} (1- P_{\gamma_+})f - \mathbf{1}_{\gamma_-}\frac{\e}{\delta} (1- P_{\gamma_+}) f_2\} \Big|\\
&\lesssim
 |\nabla_x \varphi| _{L^{4/3}(\p\O)} \big\{|(1-P_{\gamma_+}) f_R|_{4, \gamma_+} 
+ \frac{\e }{\delta }   |(\ref{est:f2})| _{L^{4}(\p\O) }  \big\}
\end{split}
\Ee
where we have used $|\int_{\gamma_+} \psi (1-P_{\gamma_+}) f|\lesssim  |\nabla_x \varphi|_{L^{ {4} / {3}}(\p\O) } |(1-P_{\gamma_+}) f|_{4, \gamma_+}$ at the last line. Here $\varphi \in \{  \varphi_a, \varphi_b, \varphi_c\}$. For the first term of (\ref{est:wf2}) we interpolate 
\Be\label{inter:bdry}
|(1-P_{\gamma_+}) f_R|_{4, \gamma_+}\lesssim | \e^{-\frac{1}{2}}(1-P_{\gamma_+}) f_R|_{2, \gamma_+}^{1/2} \e^{\frac{1}{4}} \|\mathfrak{w}_{\varrho, \ss} f_R\|_\infty^{1/2}.
\Ee
 For the second term of (\ref{est:wf2}), we use (\ref{elliptic:varphi}) 
 and a trace theorem $(\dot{W}^{1,\frac{6}{5}} (\mathbb{T}^2\times \R_+) \cap L^2(\mathbb{T}^2\times \R_+) \rightarrow  {W}^{1- \frac{1}{6/5},\frac{6}{5}} (\mathbb{T}^2 )) $, and the Sobolev embedding $(W^{\frac{1}{6},\frac{6}{5}}( \mathbb{T}^2) \rightarrow L^{4/3} (\mathbb
 {T}^2))$ to conclude that
 \Be\label{4/3bdry}
 |\nabla_x \varphi|_{L^{\frac{4}{3}}( \mathbb{T}^2)}\lesssim  |\nabla_x \varphi |_{W^{\frac{1}{6},\frac{6}{5}}(  \mathbb{T}^2)}\lesssim \|  \nabla_x \varphi \|_{\dot{W}^{1, \frac{6}{5}}( \mathbb{T}^2 \times \R_+)
 \cap L^2(\O)
 }
 \lesssim \| \tilde{P}f_R \|_{L^6(\mathbb{T}^2 \times \R_+)}^5 
 . 
 \Ee

Next we consider $(\ref{weak_form})_3$. For $\psi$ of (\ref{phia})-(\ref{phic}) and $\varphi$ of (\ref{phi_a})-(\ref{phi_c}), using (\ref{elliptic:varphi}), it follows that 
\Be\begin{split}
|(\ref{weak_form})_3|&\lesssim \e \| \p_t f_R \|_{L^{2}_{x,v}} \| \psi \|_{L^{2}_{x,v}}
\lesssim \e \| \p_t f_R \|_{L^{2}_{x,v}}\| \nabla_x \varphi \|_{L^{2}_{x }}
\lesssim  \e \| \p_t f_R \|_{L^{2}_{x,v}} \| \tilde{P}f_R \|_{L^6_x}^{5}\\
&
\leq O(1) [\e \| \p_t f_R \|_{L^{2}_{x,v}}]^6 + o(1) \|  \tilde{P}f_R \|_{L^6_x}^{6}. \label{est:wf3}
\end{split}\Ee

Lastly we consider the right hand side of (\ref{weak_form}). From (\ref{nu_K}), (\ref{est:int_k}), (\ref{est_Carl:Gamma}), and (\ref{elliptic:varphi}), it follows 
\Be\label{est:wf_L}
\begin{split}
&\Big|\iint_{\O \times \R^3} \psi \frac{1}{\e \kappa } Lf_R\Big|  =\Big| \iint_{\O \times \R^3} \psi \frac{1}{\e \kappa } L (\mathbf{I} - \mathbf{P})f_R\Big|\\
&\lesssim 
\frac{1}{\e \kappa} \int_\O \int_{\R^3}
|\nabla_x \varphi_{(a,b,c)}(x)|
 \mu(v)^{1/4}
\Big[ \nu(v ) |(\mathbf{I} - \mathbf{P})f_R (x,v ) |\\
& \ \ \ \ \ \ \ 
+\int_{\R^3} k_\vartheta (v,v_*)  |(\mathbf{I} - \mathbf{P})f_R (x,v_*) | \dd v_*\Big]
\dd v \dd x \\
&\lesssim \frac{1}{\e  \kappa}  \| \nabla_x \varphi_{(a,b,c)} \|_{L^2_x} \| (\mathbf{I} - \mathbf{P}) f_R\|_{L^2_{x,v}} \lesssim  \frac{1}{\e  \kappa} \|  \tilde{P}f \|_{L^6_x}^5\| (\mathbf{I} - \mathbf{P}) f_R\|_{L^2_{x,v}}\\
&\leq o(1)  \|  \tilde{P}f \|_{L^6_x}^6+ \big[  \e^{-1}  \kappa^{-1}  \| (\mathbf{I} - \mathbf{P}) f_R\|_{L^2_{x,v}}\big]^6
.
\end{split}
\Ee
Note that, from (\ref{est_Carl:Gamma}), $|\Gamma(\frac{\e}{\kappa} {f_2}, f_R)
|\lesssim 
\frac{\e}{\kappa} \| \mathfrak{w}_{\varrho, \ss} f_2 \|_{\infty} 
\mathfrak{w}_{\varrho, \ss}(v)^{-1}
\big[\nu(v) f_R(v)+
\int_{\R^3} k_{\vartheta} (v,v_*) f_R(v_*) \dd v_*\big].$ Then from (\ref{est:I-Pf2}) and (\ref{est:Pf2})
\Be\label{est:wf_Gamma1}
\begin{split}
\Big|\iint_{\O \times \R^3} \psi 
\frac{\e}{\kappa} \Gamma({f_2}, f_R)
\Big|  &\lesssim   \| \nabla_x \varphi_{(a,b,c)} \|_{L^2_x}
 \frac{\e}{\kappa}
 \|(\ref{est:f2})\|_\infty
\| f_R \|_{L^2_{x,v}}\\
&\leq o(1)  \|  \tilde{P}f \|_{L^6_x}^6+ \big[  \e  \kappa^{-1} 
 \|(\ref{est:f2})\|_\infty
  \|   f_R\|_{L^2_{x,v}}\big]^6.
\end{split}
\Ee
For the contribution of $\Gamma(f_R,f_R)$ we decompose $f_R= \mathbf{P} f_R + (\mathbf{I} - \mathbf{P}) f_R$. From (\ref{est_Carl:Gamma}) (or (\ref{est_infty:L_t})) 
\Be
\begin{split}
&|\Gamma(f_R,f_R)(v)|\\
& \lesssim |\Gamma(\mathbf{P}f_R,\mathbf{P}f_R)(v)| +   |\Gamma(( \mathbf{I}-\mathbf{P})f_R,( \mathbf{I}-\mathbf{P})f_R)(v)|\\
&\lesssim\nu(v) 
|Pf_R |^2\\
& \ \ \ \  + \| \mathfrak{w}_{\varrho,\ss} f_R \|_\infty 
\Big\{
\nu(v) |( \mathbf{I}-\mathbf{P})f_R)(v)|  + \int_{\R^3} k_\vartheta (v,v_*) |( \mathbf{I}-\mathbf{P})f_R)(v_*)|\dd v_*
\Big\}.
\end{split}
\Ee
Then from (\ref{phia})-(\ref{phic}), (\ref{est:int_k}), and the H\"older's inequality ($1=1/2+ 1/3+ 1/6$)
\Be
\begin{split}\label{est:wf_Gamma2}
&\Big|\iint_{\O \times \R^3} \psi  
   \frac{   \delta }{  \kappa}\Gamma(f_R, f_R) 
\Big| \\ &\lesssim \frac{\delta}{ \kappa} \| \nabla_x \varphi_{(a,b,c)}\|_{L^2_x}\Big\{ \| Pf_R \|_{L^3_x}  \| Pf_R \|_{L^6_x} 
+
\| \mathfrak{w}_{\varrho,\ss} f_R \|_{L^\infty_{x,v}} \|( \mathbf{I}-\mathbf{P})f_R  \|_{L^2_{x,v}}
 \Big\}\\
 & \lesssim \frac{\delta}{  \kappa} 
 \| \tilde{P}f_R \|_{L^6_x}^{5}  \| Pf_R\|_{L_x^6}^{3/2}
  \| Pf_R \|_{L_x^2}^{1/2} +  \frac{\e \delta}{  \kappa^{1/2}} \| \tilde{P}f_R \|_{L^6_x}^{5} \| \mathfrak{w}_{\varrho,\ss} f_R \|_{L^\infty_{x,v}} \|\e^{-1} \kappa^{-1/2}( \mathbf{I}-\mathbf{P})f_R  \|_{L^2_{x,v}}, 
\end{split}
\Ee
where we have used an interpolation $\| Pf_R\|_{L^3}\leq \| Pf_R\|_{L^6}^{1/2} \| Pf_R \|_{L^2}^{1/2}$ and (\ref{elliptic:varphi}) at the last step. A contribution of the rest of terms in the r.h.s of (\ref{weak_form}) can be easily bounded as, from (\ref{est:R1}) and (\ref{est:R2}), 
\Be
\begin{split}\label{est:wf_Gamma3}
&  \iint_{\O \times \R^3} |\psi| \left| 
 \frac{( \e\p_t + 
  v\cdot \nabla_x) \sqrt{\mu}}{\sqrt{\mu}} f_{R}
 -
\e(\mathbf{I}- \mathbf{P})\mathfrak{R}_1 - \e\mathfrak{R}_2\right|\\ \lesssim & \  
\| Pf_R \|_{L^6_x}^{5}\Big\{
\e\| (\ref{transp:mu})\|_{\infty}
 \| f_R \|_{L^2_{x,v}}
 + \e \| 
 (\ref{est:R1}) + (\ref{est:R2})
 \|_{L^2_{x,v}}
\Big\}.
\end{split}
\Ee

In conclusion, collecting the terms from (\ref{est:wf1}) with (\ref{est:wf1_b}), (\ref{est:wf2}) with (\ref{inter:bdry}) and (\ref{4/3bdry}), (\ref{est:wf3}), (\ref{est:wf_L}), (\ref{est:wf_Gamma1}), (\ref{est:wf_Gamma2}), (\ref{est:wf_Gamma3}), and utilizing (\ref{P-tildeP}), and two facts from (\ref{Sob_1D}):
{
 \Be \label{I-P:expansion}
 \begin{split}
 \sup_{0 \leq s \leq t} \|(\mathbf{I} -\mathbf{P}) f_R(s ) \|_{L^2_{x,v}} 
 &\lesssim 
 \|(\mathbf{I} -\mathbf{P}) f_R  \|_{L^2_{t,x,v}} 
 +  \|(\mathbf{I} -\mathbf{P})\p_t f_R  \|_{L^2_{t,x,v}} 
 +  \e \| \p_t u \|_{L^\infty_{t,x}}\|  {P} f_R  \|_{L^2_{t,x}} ,\\
 \sup_{0 \leq s \leq t}  |(1- P_{\gamma_+}) f_R(s)|_{L^2({\gamma_+})}
&\lesssim \sup_{0 \leq s \leq t}   |  f_R(s)|_{L^2({\gamma_+})}
  \lesssim  
   |  f_R |_{L^2_tL^2({\gamma_+})}
+  |\p_t  f_R |_{L^2_tL^2({\gamma_+})},
     \end{split}
 \Ee
} 
we prove (\ref{L6}).\hide

In conclusion, collecting the terms from (\ref{est:wf1}) with (\ref{est:wf1_b}), (\ref{est:wf2}) with (\ref{inter:bdry}) and (\ref{4/3bdry}), (\ref{est:wf3}), (\ref{est:wf_L}), (\ref{est:wf_Gamma1}), (\ref{est:wf_Gamma2}), (\ref{est:wf_Gamma3}), and utilizing (\ref{P-tildeP}), and \bcb
\begin{align*}
\|  (\mathbf{I} - \mathbf{P}) f_R (t) \|_{ {L^2 (\O \times \R^3)}}
 \lesssim  
   \e \kappa^{ 1/2} 
   \{ \| \e^{-1} \kappa^{-1/2} \sqrt{\nu} (\mathbf{I} - \mathbf{P}) f_R \|_{L^2_{t,x,v}}
   + \| \e^{-1} \kappa^{-1/2} \sqrt{\nu} (\mathbf{I} - \mathbf{P})\p_t  f_R \|_{L^2_{t,x,v}}
   \}\\
   |(1- P_{\gamma_+}) f_R(t)|_{L^2({\gamma_+})}
  \lesssim     \e^{1/2} \{ \| \e^{-1} \kappa^{-1/2} \sqrt{\nu} (\mathbf{I} - \mathbf{P}) f_R \|_{L^2_{t,x,v}}
   + \| \e^{-1} \kappa^{-1/2} \sqrt{\nu} (\mathbf{I} - \mathbf{P})\p_t  f_R \|_{L^2_{t,x,v}}
   \}
   , 
   \end{align*}
 which come from (\ref{Sob_1D}), \ec we prove (\ref{L6}).
 \unhide
\hide
Taking $\mathbf{\tilde{P}}$ to (\ref{eqtn_fR}), we derive following equations (in a weak form), for some non-zero constant $d_1,d_2,d_3,d_4$, 
\Be\label{local_conserv}
\begin{split}
\begin{bmatrix}
\p_t a + \frac{d_1}{\e} \nabla_x\cdot b\\
\p_t b + \frac{d_2}{\e}\nabla_x a + \frac{d_3}{\e} \nabla_x c\\
\p_t c +  \frac{d_4}{\e}\nabla_x \cdot b
\end{bmatrix}  =\mathbf{\tilde{P}}  \Big(- \frac{1}{\e^2\kappa}Lf_R +   \{\text{r.h.s of }  (\ref{eqtn_fR})\} \Big).
\end{split}
\Ee
This right hand side can be estimated effectively by using cancellation in the action of $\mathbf{P}$ as 
\begin{align}\notag
\frac{1}{\e^2\kappa}\mathbf{\tilde{P}} Lf_R&= \frac{1}{\e^2\kappa} \cancel{\mathbf{P} Lf_R} + \frac{1}{\e^2\kappa} (\mathbf{\tilde{P}}- \mathbf{P}) L (\mathbf{I} - \mathbf{P})
f_R=  \frac{O(1)}{\e \kappa} |u| |v-\e u| \sqrt{\mu} \| \sqrt{\nu}(\mathbf{I} - \mathbf{P}) f_R\|_{L^2_v}.\notag
\end{align}
Then using (\ref{diff:P-tP}), (\ref{nu_K}), and (\ref{est:int_k}) we bound the rest as 
\Be\label{local_con1}
\begin{split}
\frac{1}{\e^2\kappa} |(\mathbf{\tilde{P}}- \mathbf{P}) L (\mathbf{I} - \mathbf{P})f_R| &  \lesssim 
\frac{1}{\e^2 \kappa} \e|u| |v-\e u|\sqrt{\mu} \int_{\R^3} \langle v-\e u^2 \rangle  \sqrt{\mu(v)}
\Big\{ \nu(v ) |(\mathbf{I} - \mathbf{P})f_R (v ) |
+\int_{\R^3} k_\vartheta (v,v_*)  |(\mathbf{I} - \mathbf{P})f_R (v_*) | \dd v_*\Big\}
\dd v \\
&\lesssim \frac{1}{\e  \kappa} |u| |v-\e u|\sqrt{\mu}  \| (\mathbf{I} - \mathbf{P}) f_R\|_{L^2_v}.
\end{split}
\Ee
Similarly $\mathbf{\tilde{P}} \{\text{r.h.s of }  (\ref{eqtn_fR})\}  = \mathbf{P} \{\text{r.h.s of }  (\ref{eqtn_fR})\}  +
(\mathbf{\tilde{P}}- \mathbf{P}) \{\text{r.h.s of }  (\ref{eqtn_fR})\}$ where, from (\ref{est:R2}),.
\Be\begin{split} \label{local_con2}
| \mathbf{P} \{\text{r.h.s of }  (\ref{eqtn_fR})\}  |
 &= \Big| -\mathbf{P}\Big( \frac{( \p_t + 
\e^{-1} v\cdot \nabla_x) \sqrt{\mu}}{\sqrt{\mu}} f_{R} +  \mathfrak{R}_2\Big)\Big|\\
&
\lesssim  \big\{ |\nabla_x u |\|f_R\|_{L^2_v}+ \frac{\e}{\delta \kappa } \mathfrak{q}(|u|, |\nabla_{x } u |, |\p_{t}u|, |\nabla_{x}^{2} u|, |\nabla_{x} \p_{t} u|, |p|, |\nabla_{x } p |, |\p_{t} p|, |\tilde{u}|, |\nabla_{x } \tilde{u}|, |\p_{t} \tilde{u}|) \big\}
\langle v-\e u \rangle^2  \sqrt{\mu} 
,
\end{split}
\Ee
and, from (\ref{est:R1}), (\ref{est:R2}), (\ref{est_Carl:Gamma}), (\ref{est:I-Pf2}), and (\ref{est:Pf2}), 
\Be
\begin{split}
|(\mathbf{\tilde{P}}- \mathbf{P}) \{\text{r.h.s of }  (\ref{eqtn_fR})\}|
 \lesssim&  \  \e |u|  |v-\e u| \sqrt{\mu}  
\Big\{ (\| \mathfrak{w}_{\varrho, \ss} f_2 \|_{\infty} + \| \mathfrak{w} _{\varrho, \ss}  f_R \|_\infty ) \| f_R \|_{L^2_v}  + |\nabla_x u| + \frac{1}{\delta}
\mathfrak{q}(
|\nabla_{x}\tilde{u}|, 
|\nabla_x^2   u |
)\\
& \ \ \ \ \ \ \   \ \ \ \ \ \ \  \ \ \ \ \ \ \    + \frac{\e}{\delta \kappa } \mathfrak{q}(|u|, |\nabla_{x } u |, |\p_{t}u|, |\nabla_{x}^{2} u|, |\nabla_{x} \p_{t} u|, |p|, |\nabla_{x } p |, |\p_{t} p|, |\tilde{u}|, |\nabla_{x } \tilde{u}|, |\p_{t} \tilde{u}|)
\Big\}\\
\lesssim& \  \e |u|  |v-\e u| \sqrt{\mu}  
\big\{ (|p| +|\tilde{u}| + \kappa |\nabla_x  u| + \| \mathfrak{w} _{\varrho, \ss}  f_R \|_\infty) \| f_R\|_{L^2_v}\\
&+ \frac{ \e}{\delta} |u|  |v-\e u| \sqrt{\mu}   \mathfrak{q}(|u|, |\nabla_{x } u |, |\p_{t}u|, |\nabla_{x}^{2} u|, |\nabla_{x} \p_{t} u|, |p|, |\nabla_{x } p |, |\p_{t} p|, |\tilde{u}|, |\nabla_{x } \tilde{u}|, |\p_{t} \tilde{u}|)
\big\}.\label{local_con3}
\end{split} \Ee

Now we view $\psi$ of (\ref{phia})-(\ref{phic}) as the same combination with $(\p_t \varphi_a, \p_j \p_t \varphi_b^j,\p_j \p_t \varphi_b^i, \nabla_x \p_t \varphi_c)$. Clearly from (\ref{phi_a})-(\ref{phi_c}), such $\p_t \varphi$'s can be viewed as unique solutions in $L^6(\O) \cap \dot{W}^{1,2} (\O) \cap \dot{W}^{2, \frac{6}{5}} (\O)$ of 
\begin{align}
-\Delta _{x} \p_t \varphi_{a }  =  \p_t( a^{5})&   \ \ \text{with} \ \  \frac{\partial \p_t \varphi_{a }}{%
\partial n}\Big|_{\partial \Omega }=0,%
\label{phi_at}
\\ 
-\Delta _{x} \p_t\varphi _{b }^{j}  = \p_t(b^{5}_{j})&   \ \ \text{with} \ \   \p_t\varphi _{b}^{j}|_{\partial
\Omega }=0, \label{phi_bt}\\
 -\Delta _{x} \p_t\varphi_{c }  = \p_t c^{5}&  \ \ \text{with} \ \   \p_t\varphi %
_{c }|_{\partial \Omega }=0.\label{phi_ct}
\end{align} 
\Be
\sup_{\nabla\phi \in }\int\p_t (a^5) \phi
\Ee
Note that $\p_t (a)^5= 5 \p_t a a^4= -5 \frac{d_1}{\e} a^4  \nabla_x \cdot b $. 
\Be
\| \nabla_x \p_t \varphi \|_{  L^{6/5}}  \lesssim \| \p_t (a^5) \|_{W^{-1,6/5}} 
\Ee
 \unhide
\end{proof}

\subsection{Average in Velocity} 
 We prove a version of velocity lemma when a suitable bound for source terms is only known in a finite time interval. In this section we often specify domains in which an $L^p$-norm is taken while the simplified notation (\ref{short_notation}) will be used only when the domain is $[0,T] \times \O\times \R^3$.

\begin{proposition}\label{prop:average}
 Assume the same assumptions in Proposition \ref{prop:Hilbert}. Then we have, for $2<p<3$,  
\Be\begin{split}\label{average_3D}
&d_3\big\|    {P} f_R
 \big\|_{L^2_t L^p_x }
 \\
\lesssim & \ 
(1+ 
\e \|
(\ref{transp:mu})
 \|_{L^2_t  L^\infty_x }
)
 \| f_R \|_{L^\infty_t L^2 _{x,v}}
\\
& +\Big\{
\frac{1}{\e \kappa }
+ \frac{\delta}{\kappa} \| \mathfrak{w}_{\varrho, \ss} f_R \|_{L^\infty_{t,x,v} } +   \| \mathfrak{w}_{\varrho, \ss } f_R   \|_{ L^\infty _{t,x,v}}^{\frac{p-2}{p}}
\Big\}
 \| \sqrt{\nu} (\mathbf{I} - \mathbf{P})f_R \|_{L^2_{t,x,v} }
 \\
 &
+ \| f_R (0) \|_{L^2_\gamma} +
 \e \|
 (\ref{est:R3})\| _{L^2_{t,x}    }+\e  \|(\ref{est:R4})
  \|_{L^2_{t,x}    },
\end{split}\Ee
with 
\Be\label{d3}
d_{3}:= 1- O(\e) \|u\|_{L^\infty_{t,x}}
- 
 \frac{\e}{\kappa} \|(\ref{est:f2}) \|_{L^\infty_t   L_x^{\frac{2p}{p-2}} } - \frac{\delta}{\kappa} \| Pf_R \|_{L^\infty_tL_x^{6} }^{\frac{3(p-2)}{p}} 
\|  \mathfrak{w}_{\varrho,\ss}  f_R \|_{L_{t,x,v}^{\infty} }^{\frac{6-2p}{p}},
\Ee
and for $\varrho'<\varrho$ 
\Be\begin{split}\label{average_3Dt}
&d_{3,t} \big\|    {P}  \p_t f_R
 \big\|_{L^2_t L^p_x }
 \\
\lesssim & 
 \frac{1}{\kappa}  \| \p_t u \|_{ L^\infty_{t, x}  }\big(
1+  \e^2   \| (\ref{est:f2}) \|_{L^\infty_{t,x}}
\big)\| Pf_R \|_{  L^2_{t, x}  } 
+ \frac{\delta \e}{\kappa} 
\| \p_t u \|_{L^\infty_{t,x} }\|    {P} f_R \|_{L^\infty_tL^6_{x }}^{\frac{3(p-2)}{p}}\| \mathfrak{w} f_R\|_{L^\infty_{t,x,v}}^{\frac{6-2p}{p}}\|    {P} f_R \|_{L^2_tL^p_{x }} \\
&+  \frac{\e}{\kappa}  \| \p_t u \|_{L^\infty_{t,x} }
(\delta \| \mathfrak{w} f_R \|_{L^\infty_{t,x,v}} + \| (\ref{est:f2}) \|_{L^\infty_{t,x}})
 \| \sqrt{\nu} (\mathbf{I} -\mathbf{P}) f_R \|_{L^2_{t,x,v}}  
\\ &+ (\kappa \e ) ^{\frac{2}{p-2}} \|  \mathfrak{w}_{\varrho^\prime,\ss} \p_t  f_R \|_{L^2_t  L^\infty_{x,v}  }+
\| \p_t f_R \|_{L^\infty_t L^2_{x,v}}+
\e  \| (\ref{transp:mu_t})\|_{L^2_tL^\infty_{x}} \| f_R\|_{L^\infty_t L^2_{x,v}} 
\\
&+
\Big\{\frac{1}{\kappa \e} + \frac{\delta}{\kappa} \|   \mathfrak{w}_{\varrho,\ss} f_R\|_{L^\infty_{t,x,v}} +  \frac{\e}{\kappa} \|   (\ref{est:f2})\|_{L^\infty_{t,x}}  
+ \e \| 
(\ref{transp:mu})
 \|_{L^\infty_{t,x}  }
\Big\}
\| \sqrt{\nu} (\mathbf{I} - \mathbf{P}) \p_t f_R \|_{L^2_{t,x,v}}
\\
&+ \|\p_t  f_R (0) \|_{L^2_\gamma} + 
\frac{\e}{\kappa}  \| 
 (\ref{est:f2_t}) \|_{L^2_{t,x,v}} 
  \| \mathfrak{w}_{\varrho,\ss} f_R \|_{L^\infty_{t,x,v}}
+ \e \| (\ref{est:R3})\|_{L^2_{t,x}}
+ \e  \| (\ref{est:R4})\|_{L^2_{t,x}},
\end{split}\Ee
with  
\Be\label{d3t}
d_{3,t}:= 1- O(\e) \|u\|_{L^\infty_{t,x}} 
- 
 \frac{\e}{\kappa} \| 
 (\ref{est:f2})
  \|_{L^\infty_t   L_x^{\frac{2p}{p-2}}  }-\e
\| (\ref{transp:mu})\|_{L^\infty_{t } L^{\frac{2p}{p-2}}_x }
  - \frac{\delta}{\kappa} \| P f_R \|_{L^\infty_t  L_x^{6} }^{\frac{3(p-2)}{p}} 
\|  \mathfrak{w}_{\varrho,\ss} f_R \|_{L^{\infty} _{t,x,v}}^{\frac{6-2p}{p}},
\Ee
where both bounds are uniform-in-$p$ for $2<p<3$. 

\end{proposition}

We prove the proposition by several steps. 

\
 

\textbf{Step 1: Extension.} We define a subset 
\Be
\tilde{\O} := (0,2\pi)\times (0,2\pi) \times (0,\infty) \subset \R^3.\label{def:tilde_O}
\Ee 
We regard $\tilde{\O}$ as an open subset but not a periodic domain as $\O$. Without loss of generality we may assume that $f_R(0,x,v)$ is defined in $\R^3 \times \R^3$ and $\|f_R(0)
\|_{L^p(\R^3 \times \R^3)} \lesssim \|f_R(0)
\|_{L^p(\tilde{\O} \times \R^3)}$ for all $1 \leq p \leq \infty$.  
Then we extend a solution for whole time $t \in \R$ as 
\Be\label{ext:f1}
f_I (t,x,v) : = \mathbf{1}_{t \geq 0 } f_R (t,x,v) + \mathbf{1}_{t \leq 0} \chi_1 (t) f_R (0, x
,v ), 
\Ee
where a smooth non-negative function $\chi_1$ satisfies $\chi_1 (t)\equiv  1$ for $t \in [-1,0]$, $\chi_1 (t) \equiv 0$ for $t<-2$, and $0 \leq \frac{d}{dt}\chi_1 \leq 4$.

 A closure of $\tilde{\O}$ is given as $cl(\tilde{\O})= [0,2\pi]\times [0,2\pi] \times [0,\infty)$. Let us define $\tilde{t}_B(x,v) \in \R$ for $(x,v)  \in  (\R^3 \backslash \tilde{\O}) \times \R^3$. 
We consider $\tilde{B}(x,v):=\{s \in \R:x+ s v \in \R^3 \backslash cl(\tilde{\O})  \}$. Clearly if $\tilde{B}(x,v) \neq \emptyset$ then $ \{s>0\}\subset \tilde{B} (x,v)$ or $  \{s<0\}\subset\tilde{B} (x,v)$ exclusively. We define
\Be\label{tB}
\tilde{t}_B(x,v)  = \begin{cases} 0 & \text{if} \ \  x \in \p\tilde{\O}, 
\\
\inf \tilde{B}(x,v)
& \text{if} \ \ x \in  
\R^3 \backslash cl(\tilde{\O})
\ \text{and} \ \tilde{B}(x,v) \neq \emptyset \ \text{and} \ 
 \{s>0\}\subset\tilde{B}(x,v) 
,\\
\sup \tilde{B}(x,v)
&  \text{if} \ \ x \in  
\R^3 \backslash cl(\tilde{\O})
\ \text{and} \ \tilde{B}(x,v) \neq \emptyset \ \text{and} \ 
\{s<0\} \subset\tilde{B}(x,v) ,\\
  - \infty& \text{if} \  \ \tilde{B}(x,v) = \emptyset \ \text{and} \  x \notin \p\tilde{\O}. 
\end{cases}
\Ee
Using (\ref{tB}) we define 
\Be\label{ext:f2}
f_E (t,x,v): =  \mathbf{1}_{(x,v) \in (\R^3 \backslash \tilde{\O}) \times \R^3 } f_I (t+ \e \tilde{t}_B(x,v), \tilde{x}_B(x,v),v) \ \ \text{with} \ \ \tilde{x}_B(x,v):= x+\tilde{t}_B(x,v)v.
\Ee
 It is easy to see that $\e \p_t f_E + v\cdot \nabla_x f_E=0$ in the sense of distributions.


Next we define two cutoff functions. For any $N>0$ we define smooth non-negative functions as
\Be
\begin{split}
 &\chi_2 (x) \equiv 1 \ \text{for} \  x \in [-\pi, 3\pi] \times [-\pi, 3\pi] \times [-\pi, \infty), \\ 
& \chi_2 (x ) \equiv 0 \ \text{for} \  x \notin [-2\pi, 4\pi] \times [-2\pi, 4\pi] \times [-2\pi, \infty)   , \ \  
|\nabla_{x}\chi_2  |  \leq 10   
,\label{chi1}\\
\end{split}\Ee
\Be
\begin{split}
&\chi _3 (v) \equiv 1 \ \text{for}  \ 
|v| \leq N-1,  \ \text{and} \ |v_i|\geq 2/N  \ \text{for all } i=1,2,3,\\
&\chi_3  (v) \equiv 0 \ \text{for} \ |v| \geq  N  \ \text{or} \ |v_i| \leq 1/N \ \text{for any } i=1,2,3, \ \ |\nabla_v \chi _3  | \leq 10  
.\label{chi_v}
\end{split}\Ee
We denote 
\Be\label{def:UV}
U : = [-2\pi, 4\pi] \times [-2\pi, 4\pi] \times [-2\pi, \infty), \ \   V: = 
\{ v\in \R^3: |v| \leq N \} \cap \bigcap_{i=1,2,3} \{ v\in \R^3: |v_i| \geq 1/N  \} 
\Ee

We define an extension of cut-offed solutions   
\Be\label{ext:f+}
 \bar{f}_{R}(t,x,v) :=
 \chi_2 (x) \chi_3 (v)  \big\{
  \mathbf{1}_{\tilde{\O}}(x)
 f_I(t,x,v)
 +  f_E (t,x,v)\big\}
 \  \ \text{for} \ (t,x,v) \in (-\infty,T] \times \R^3 \times \R^3
 .
\Ee
\hide
\Be\label{ext:f+}
 \bar{f}_R(t,x,v) = 
 \begin{cases}
 \mathbf{1}_{|v_3|\geq 1/N}
 \big[\mathbf{1}_{t\geq 0} e^{-t}f_R(t,x,v) 
+\mathbf{1}_{t<0} e^tf_R(0,x- \frac{t}{\e} v,v)
 \big]& \text{if } x_3\geq 0,\\ 
 \mathbf{1}_{|v_3|\geq 1/N}
 \big[ 
 \mathbf{1}_{t- \tb(x,v)
 \geq 0}
 e^{-(t-\tb(x,v)
 )}
 f_R(t-\tb(x,v)
 , \xb(x,v)
 , v)
 +  \mathbf{1}_{t- \tb(x,v) 
 < 0}
 e^{ t- \tb(x,v) 
  }
 f_R(0, 
 x- \frac{t}{\e} v
 , v)
 \big]
 & \text{if } x_3<0.
 \end{cases}
\Ee\unhide
\hide

\Be
\begin{split}
f_1(t,x,v): =  \mathbf{1}_{(x_1,x_2) \in [-\pi, \pi]^2}
\Big\{& 
\mathbf{1}_{x_3\geq 0 } f_R(t,x,v)
 + \mathbf{1} _{x_3<0}   \mathbf{1}_{t- \e \frac{x_3}{v_3}
 \geq 0}
\mathbf{1}_{v_3<0}
 f_R(t-\e\frac{x_3}{v_3}
 ,  x-  \frac{x_3}{v_3} v
 , v)
\\
 & 
  +\mathbf{1} _{x_3<0}   \big[ \mathbf{1}_{t-  \e\frac{x_3}{v_3}
 < 0}  \mathbf{1}_{v_3<0}
 + 
   \mathbf{1}_{v_3>0}
\big] 
f_R(0, x- \frac{t}{\e} v 
 , v) \Big\}
  \end{split}
\Ee
\Be
\begin{split}
 f_2 (t,x,v ): =    \mathbf{1}_{(x_1,x_2) \notin [-\pi, \pi]^2}\Big\{&
 \end{split}
\Ee
\Be
\begin{split}
\end{split}
\Ee\unhide
We note that in the sense of distributions $\bar{f}_R$ solves
\Be\label{eqtn:bar_f}
\begin{split}
& \e \p_t \bar{f}_R + v\cdot \nabla_x \bar{f}_R 
 %
=      \bar{g}
 \ \ \text{in} \ (-\infty,T] \times \R^3 \times \R^3,\\
 & \bar{g}:=     \frac{v\cdot \nabla_x \chi_2  }{\chi_2  } \bar{f}_R + \mathbf{1}_{t \geq 0} \mathbf{1}_{\tilde{\O}} (x) \chi_2(x)  \chi_3(v) 
[\e \p_t + v\cdot \nabla_x] f_R\\
&  \ \ \ \ \ \ +\mathbf{1}_{t \leq 0} \{
\e\p_t \chi_1(t) f_R (0,x,v) + \chi_1 (t) v\cdot \nabla_x f_R (0,x,v) 
\}
 %
  \end{split}
\Ee\hide
with
\begin{align}
\bar{g} 
:=& \   \mathbf{1}_{t\geq 0}   \mathbf{1}_{\O}(x) \chi_1 (x) \chi_2 (v) 
 (\e \p_t   + v\cdot \nabla_x+ \frac{\nu }{\kappa \e}  )f_R(t,x,v).
\label{etqn_barf1} 
\end{align}\unhide
\hide
  \\
  &+ 
     \mathbf{1}_{\O}(x)
     \chi
      (v)
\big[
  \mathbf{1}_{t\geq 0} 
  f_R(t,x,v)
 + 
 \mathbf{1}_{t<0}
e^{- \frac{\nu t}{\kappa \e^2}} 
 f_R(0,x- \frac{t}{\e} v,v) \big]
  \label{etqn_barf2}
    \\ 
  &+
    \e  \chi_1^\prime(t)   
   \chi_2 (v)   \mathbf{1}_{t-  \e \frac{x_3}{v_3}
 \geq 0}\mathbf{1}_{\O^c}(x)   
\mathbf{1}_{v_3<0}
 f_R(t-  \e\frac{x_3}{v_3}
 ,  x-  \frac{x_3}{v_3} v
 , v)\label{etqn_barf3}
   \\
    &+  
       \e  \chi_1^\prime(t) 
         \chi_2 (v)   \mathbf{1}_{\O^c}(x)  \big[ \mathbf{1}_{t-  \e\frac{x_3}{v_3}
 < 0}  \mathbf{1}_{v_3<0}
 + 
   \mathbf{1}_{v_3>0}
\big] f_R(0, x- \frac{t}{\e} v 
 , v).  \label{etqn_barf4} 
\unhide
Here we have used the fact that $\bar{f}_R$ in (\ref{eqtn:bar_f}) is continuous along the characteristics across $\p \tilde{\O}$ and $\{t=0\}$. 
We derive that, using (\ref{eqtn:bar_f}),
\Be\label{f_R:duhamel}
\bar{f}_R(t,x,v) =
 \frac{1}{\e} \int^t_{-\infty} 
 \bar{g} (s, x- \frac{t-s}{\e} v, v)  \dd s \ \ \text{for} \ (t,x,v) \in 
(-\infty,T] \times \R^3
 \times \R^3.
\Ee

\hide
$\bar{f}_R$ is continuous across $\{x_3=0\}$, $\{t=0\}$, and $\{t-\tb(x,v)=t- \e  \frac{x_3}{v_3}=0\}$, and the fact of $(\e \p_t   + v\cdot \nabla_x   )f_R(t-\e \frac{x_3}{v_3}, x- \frac{x_3}{v_3} v, v)=0$,  
$(\e \p_t   + v\cdot \nabla_x  )f_R(0,x- \frac{t}{\e} v,v)=0$, and $(\e \p_t + v\cdot \nabla_x) (t-\e \frac{x_3}{v_3})=0$ in the sense of distributions.
\unhide

Recall $\tilde{\varphi}_i \in \{\tilde{\varphi}_0, \cdots \tilde{\varphi}_4\}$ in (\ref{basis}). From (\ref{ext:f+}) we note that   
\Be\begin{split}
 &
\left\|\int_{\R^3} \bar{f}_R(t,x,v) \tilde{\varphi}_i (v) \sqrt{\mu_0(v)}    \dd v \right\|_{L^2_t ((0,T); L^p_x ( \tilde{\O}))} \\ 
&= 
 \left\|\int_{\R^3}
\chi_2 (x) \chi_3 (v)
 f_R(t,x,v)  \tilde{\varphi}_i (v) \sqrt{\mu_0(v)}    \dd v   \right\|_{L^2_t ((0,T); L^p_x (\tilde{\O}))}\label{decomp1:L2L3}
 \end{split}\Ee
 \hide
 \begin{align}
 \\
&+
 \left\|\int_{\R^3} 
\chi_2 (x) \chi_3 (v)  \mathbf{1}_{t-\e \frac{x_3}{v_3}\geq 0}
 \mathbf{1}_{v_3<0}
 e^{- \frac{\nu }{\kappa \e^2 } \frac{ \e x_3}{v_3}}
 f_R(t-\e\frac{x_3}{v_3}
 ,  x-  \frac{x_3}{v_3} v
 , v)
 \tilde{\varphi}_i (v) \sqrt{\mu_0(v)}    \dd v
 \right\|_{L^2_t ((0,T); L^p_x (\O^c))}\label{decomp2:L2L3}\\  
 &+
 \left\|\int_{\R^3}  
\chi_2 (x) \chi_3 (v) \big[ \mathbf{1}_{t-  \e\frac{x_3}{v_3}
 < 0}  \mathbf{1}_{v_3<0}
 + 
   \mathbf{1}_{v_3>0}
\big] e^{- \frac{\nu t}{\kappa \e^2}} f_R(0, x- \frac{t}{\e} v 
 , v)  
  \tilde{\varphi}_i (v) \sqrt{\mu_0(v)}    \dd v
 \right\|_{L^2_t ((0,T); L^p_x (\O^c))}\label{decomp3:L2L3}.
\end{align} 
\unhide
From (\ref{P}), we decompose 
\Be\notag
\begin{split}
(\ref{decomp1:L2L3})\geq &   \  \Big\| 
\sum_j     \chi_2 (x)  \tilde{P}_jf_R(t,x)
\int_{\R^3}  \chi_3 (v)    \tilde{\varphi}_j (v)    \tilde{\varphi} _i (v)  \mu_0  (v)    \dd v \Big\|_{L^2_t ((0,T); L^p_x (\tilde{\O}))}
\\&
 -  \left\|\int_{\R^3}
 \chi_3 (v) 
(\mathbf{I} - \mathbf{\tilde{P}}) f_R(t,x,v)  \tilde{\varphi}_i (v) \sqrt{\mu_0(v)}    \dd v   \right\|_{L^2_t ((0,T); L^p_x (\tilde{\O}))}.
\end{split}\Ee
We consider the right hand side of above terms. From (\ref{diff:P-tP}), $\int \tilde{\varphi}_i\tilde{\varphi}_j \mu_0 = \delta_{ij}$, and (\ref{chi_v}), the first term can be bounded below by $\big(1-O( \e) \|u\|_\infty- O(\frac{1}{N})\big)
 \big\|   \chi_2   {P} f_R
 \big\|_{L^2_t ((0,T); L^p_x (\tilde{\O}))}$. For the second term we use (\ref{diff:P-tP}), $L^2_t (0,T)  \subset L^p_t( 0,T )$, and $L^1(\{|v|\leq N\}) \subset L^p(\{|v|\leq N\})$ to bound it above by $C_{T,N} \| (\mathbf{I} - \mathbf{P}) f_R \|_{L^p((0, T) \times \tilde{\O} \times \R^3)}+ \big( O( \e) \|u\|_\infty+ O(\frac{1}{N})\big)
 \big\|   {P} f_R
 \big\|_{L^2_t( (0,T)  ;L^p_x(\tilde{\O}))}$. Hence we derive 
\Be
\begin{split} 
&(\ref{decomp1:L2L3}) \\
\geq  & \  \big(1- O(\e )\|u\|_\infty- O(\frac{1}{N})\big)
 \big\|   {P} f_R
 \big\|_{L^2_t( (0,T)  ;L^p_x(\tilde{\O}))} - C_{T,N} \| (\mathbf{I} - \mathbf{P}) f_R \|_{L^p((0, T) \times\tilde{\O} \times \R^3)}
  \\
 \geq  &  \ \big(1- O(\e) \|u\|_\infty- O(\frac{1}{N})\big)  \big\|    {P} f_R
 \big\|_{L^2_t((0,T);L^p_x(\tilde{\O}))}\\
 &
 -C_{T,N}   \| \mathfrak{w}_{\varrho, \ss } f_R (t) \|_{ L^\infty((0,T) \times \tilde{\O} \times \R^3)}^{\frac{p-2}{p}} \| (\mathbf{I} - \mathbf{P}) f_R \|_{L^2((0, T) \times \tilde{\O} \times \R^3)}^{\frac{2}{p}}
.\label{est:decomp1:L2L3}
\end{split}
\Ee 

\

\textbf{Step 2: Average lemma}. Recall $\tilde{\varphi}_i \in \{\tilde{\varphi}_0, \cdots \tilde{\varphi}_4\}$ in (\ref{basis}). We choose $\tilde{\varphi}(v)$ such that 
\Be\label{tilde_varphi}
\begin{split}
 \chi_3 (v) 
 |\tilde{\varphi}_i(v)|  \sqrt{\mu_0 (v)}  \leq \tilde{\varphi}(v) , \ \ \tilde{\varphi}(v) \in  C^\infty_c (\R^3) \\
 \ \ \text{and} \ \ \tilde{\varphi}(v)\equiv 0 \ \  \text{for} \ \ |v|\geq N \ \ \text{or} \ \ |v_i| \leq 1/N \ \text{for any } i=1,2,3. 
\end{split}\Ee
\hide
 $
 \bar{f}_R(t,x,v)=
 \frac{1}{\e} \int^t_{-\infty} 
 \bar{g} (s, x- \frac{t-s}{\e} v, v)  \dd s$ for $(t,x,v) \in 
 [0,T] \times \tilde{\O}
 \times \R^3.$ If $|v_i| \leq 1/N$ for $i=1$ or $i=2$ then $\bar{g} (s, x- \frac{t-s}{\e} v, v)=0$ from (\ref{chi_v}) and (\ref{eqtn:bar_f}). Now we consider the case of $|v_i| > 1/N$ for $i=1,2$. For $(t,x) \in [0,T] \times \tilde{\O}$ then $x_i- \frac{t-s}{\e} v_i\notin [-2\pi, 4\pi]$ if $s<t- 10N \e $. On the other hand from (\ref{chi1}) and (\ref{eqtn:bar_f}) we derive that $\bar{g} (s, x- \frac{t-s}{\e} v, v)=0$ if $x_i- \frac{t-s}{\e} v_i \notin [-2\pi, 4\pi]$ for either $i=1$ or $i=2$. Therefore 
\Be\label{f_R:duhamel}
 \bar{f}_R(t,x,v) =
 \frac{1}{\e} \int^t_{ t-10N \e } 
  \bar{g} (s, x- \frac{t-s}{\e} v, v)   \dd s \  \  \text{for} \ 
 (t,x,v) \in 
(-\infty,T] \times \tilde{\O}
 \times \R^3.
\Ee
\unhide

\begin{lemma}\label{lemma:average}
We define 
\Be\label{def:S}
S(\bar{g})(t,x):=
\frac{1}{\e} \int^t_{-\infty
}\int_{\R^3}
| \bar{g}(s, x- \frac{t-s}{\e} v, v)  | \tilde{\varphi} (v) \dd v \dd s \ \  \text{for} \ (t,x) \in(-\infty , T ] \times 
\R^3. 
\Ee
Then, for $p<3$ and $1\ll N$,
\Be\label{bound:S}
\| S(\bar{g}) \|_{L^2_t((0,T);  L^p_x(\mathbb{T}^2 \times \R))} \lesssim_{N} 
 \| \mathbf{1}_{(t,x,v) \in \mathfrak{D}_T}   \bar{g}\|_{L^2  ((0,T)\times (\mathbb{T}^2 \times \R) \times \{|v| \leq N \})},
\Ee
 where the bound (\ref{bound:S}) only depends on $N$ but can be independent on $p<3$. 
\end{lemma}
We remark that from (\ref{f_R:duhamel}) and (\ref{def:S}) $ \int_{\R^3}\bar{f}_R(t,x,v) \tilde{\varphi}_i (v) \dd v \leq S(\bar{g})(t,x)$.
\begin{proof}[\textbf{Proof of Lemma \ref{lemma:average}}]
We prove (\ref{bound:S}) by a $TT^*$($SS^*$ for our case) method. First we derive a dual of $S$ in the following equalities:\hide. We consider a function space 
\Be\label{space:G}
\begin{split}
\mathcal{G}:= \left\{ g(t,x,v) \in L^2((-\infty,T] \times \R^3\times \R^3)\bigg| 
\begin{array}{c}
 \text{supp} (g) \in (-\infty,T] \times U\times V ,
\\
 \text{supp} (S(g)) \subset (-\infty,T] \times U   \ \text{and} \ 
S(g) (t,x) \in L^2((-\infty,T] \times U ) 
\end{array}  \right\}.
\end{split}
\Ee
For $g \in \mathcal{G}$\unhide
\Be\begin{split}\label{derive:S*}
& \int^T_{-\infty} \int_{
\R^3
 }
S(\bar{g}) (t,x) h(t,x) 
 \dd x \dd t  \\
&=\int^T_{-\infty} \int_{\R^3
} \frac{1}{\e} \int^t_{-\infty}\int_{\R^3} 
%
%
|\bar{g}(s, x- \frac{t-s}{\e} v, v) | \tilde{\varphi} (v)  h(t,x) 
\dd v \dd s  \dd x \dd t \\
&=\int^T_{-\infty} \int_{\R^3
}\int_{\R^3} |\bar{g}(s, x , v)  |
\left[
\frac{1}{\e}  \int^T_{s} 
 h(t,x+ \frac{t-s}{\e} v) \tilde{\varphi} (v)
\dd t  \right]
\dd v  \dd x \dd s \\
&=\int^T_{-\infty} \iint_{\R^3\times \R^3} |\bar{g}(t, x , v)  |
\left[
\frac{1}{\e}  \int^T_{t} 
h(s,x+ \frac{s-t}{\e} v) \tilde{\varphi} (v)
\dd s  \right]
\dd v  \dd x \dd t \\
&= \int^T_{-\infty}  \iint_{\R^3\times \R^3} | \bar{g}(t, x , v)  |
S^*(h) (t,x,v)
\dd v  \dd x \dd t ,
\end{split}\Ee
where we have defined 
\Be\label{def:S*} 
S^*(h) (t,x,v) : = \frac{1}{\e}  \int^T_{t} 
h(\tau,x+ \frac{\tau-t}{\e} v) \tilde{\varphi} (v)
\dd \tau. 
\Ee 
Here, in the second equality of (\ref{derive:S*}) we have used the Fubini theorem for changing order of $s$ and $t$ integrations, and then used a change of variables $x \mapsto x- \frac{t-s}{\e} v$. 
In the third equality of (\ref{derive:S*}) we have used a change of variable $(t,s) \mapsto (s,t)$ and the fact $\text{Supp} (g) \subset (-\infty,T] \times U \times V$.

On the other hand, for $1/p+1/q=1$
, following the argument of (\ref{derive:S*}) with $h(t,x) = \mathbf{1}_{x \in \tilde{\O}} h(t,x)$ we derive that 
\Be \label{S<S*:1}
\begin{split}
\| S(\bar{g}) \|_{L^2_t((-1,T];  L^p_x(\tilde{\O}))} &= \sup_{\| h \|_{L^2_t((-1,T];  L^q_x(\tilde{\O}))} \leq 1} \int^T_{-1} \int_{\tilde{\O}}
S(\bar{g}) (t,x) h(t,x) 
 \dd x \dd t\\
 &= 
 \sup_{\| h \|_{L^2_t((-1,T];  L^q_x(\tilde{\O}))} \leq 1} \int^T_{-1} \iint_{U\times V}
 |\bar{g} (t,x,v) |S^*(h)(t,x,v) 
\dd v \dd x \dd t. 
\end{split}
\Ee
It is important to check the integral region in space of the last term of (\ref{S<S*:1}). From (\ref{def:S*}), we note that if $x + \frac{\tau-t}{\e} v \notin  cl( \tilde{\O})$ for all $\tau \in [t,T]$ then the last term would vanish since $\text{supp} (h) \subset (-\infty,T] \times \tilde{\O}$. Therefore we can exclude $(t,x,v)$ from the last integration in (\ref{S<S*:1}) if $L(t,x,v) \cap  \tilde{\O}  = \emptyset$ for $L(t,x,v):=\{x + \frac{\tau-t}{\e} v : \tau \in [t,T] \}$. Now we define 
\Be\label{def:mathcalU}
\mathfrak{D}_T:= 
\big\{ (t,x,v) \in  (-1, T] \times U \times V:  L(t,x,v) \cap  \tilde{\O}  \neq \emptyset \big\}. 
\Ee
Then we can write 
\Be \label{S<S*:2}
\begin{split}
(\ref{S<S*:1})&  = 
 \sup_{\| h \|_{L^2_t((-1,T];  L^q_x(\tilde{\O}))} \leq 1} \int^T_{-1} \iint_{U\times V}
 \mathbf{1}_{(t,x,v) \in \mathfrak{D}_T} |\bar{g}(t, x , v)  |
S^*(h) (t,x,v)
\dd v  \dd x \dd t\\
 &\leq  \|   \mathbf{1}_{(t,x,v) \in \mathfrak{D}_T} \bar{g}\|_{L^2  ((-1,T]\times U \times V)} \sup_{\| h \|_{L^2_t((-1,T];  L^q_x(\tilde{\O}))} \leq 1} 
 \big\| 
 S^* (h)(t,x,v) \big\|_{L^2 ((-1,T]\times U \times V)}.
\end{split}
\Ee
Therefore to prove (\ref{bound:S}) it suffices to show that
\Be\label{bound:S*}
\|  
S^* (h) \|_{L^2  ((-1,T]\times U \times V)} \lesssim
 \| h \|_{L^2_t((-1,T];  L^q_x(\tilde{\O}))}. 
\Ee
Note that since $\text{supp}(h) \subset(-1,T] \times U$ and $\text{supp} (\tilde{\varphi}) = V$ for $(x,v) \in U\times V$, we have, with $x= (x_1,x_2,x_3), v=(v_1,v_2,v_3)$ $$|x_1 + \frac{\tau- t}{\e} v_1|\geq \frac{|\tau-t|}{\e}|v_1| - |x_1|
\geq \frac{10 \pi  N \e}{\e} \frac{1}{N} - 4\pi > 4\pi  \ \ \text{if} \ \tau\geq t+ 10 \pi  N \e.
$$  
Hence we can rewrite (\ref{def:S*}) as 
\Be\label{def:S*_1}
\begin{split}
S^*(h) (t,x,v)   =& \frac{1}{\e}  \int^{\min \{T, t+ 10 \pi N\e \}}_{t} 
h(\tau,x+ \frac{\tau-t}{\e} v) \tilde{\varphi} (v)
\dd \tau \\
&
 \ \ \text{for} \  (x,v) \in U\times V, 
\ \ \text{if} \   \text{supp}(h) \subset (-1,T] \times U. 
\end{split}\Ee

On the other hand, from (\ref{derive:S*}), we have for $\text{supp}(h) \in (-1,T] \times \tilde{\O}$, 
\Be\notag
\begin{split}
 \|  
 S^* (h) \|_{L^2  ((-1,T]\times U \times V)}^2 
=&\int^T_{-1} \iint_{U\times V}  
S^* (h)(t,x,v)
S^* (h)(t,x,v)
 \dd v  \dd x \dd  t\\
 = & \int^T_{-1} \iint_{U\times V} SS^* (h) (t,x) h(t,x) \dd x \dd t \\
 \leq &  \     \| SS^* (h)\|_{L^2_t((-1,T);  L^p_x(U))} \| h\| _{L^2_t((-1,T];  L^q_x(\tilde{\O}))}.
\end{split}
\Ee
Therefore to show (\ref{bound:S*}) (which will imply (\ref{bound:S})) we only need to prove that, for $\text{supp}(h) \subset (-1,T] \times \tilde{\O}$, 
\Be\label{claim:SS*}
 \| SS^* (h)\|_{L^2_t((-1,T];  L^p_x(U))}\lesssim
\| h\| _{L^2_t((-1,T];  L^q_x(\tilde{\O}))}. 
\Ee

Now we prove (\ref{claim:SS*}). From (\ref{def:S}) and (\ref{def:S*_1}), we read 
\Be\notag\label{form:SS*}
\begin{split}
 SS^*(h) (t,x) =& \  
 \frac{1}{\e } \int^t_{-1} \int_{\R^3}  S^*(h) (s, x- \frac{t-s}{\e} v,v ) \tilde{\varphi} (v) \dd v \dd s\\
=& \  
 \frac{1}{\e^2} \int^{t}_{-1} \int_{\R^3} 
\int^{\min \{T, s+ 10 \pi N\e \}}
_s 
h(\tau, x- \frac{t-s}{\e} v + \frac{\tau-s}{\e} v )
 \dd \tau
 (\tilde{\varphi} (v)) ^2\dd v \dd s\\
 =& \  
 \frac{1}{\e^2} \int^{t}_{-1} 
\int^{\min \{T, s+ 10 \pi N\e \}}
_s\int_{\R^3} 
h(\tau, x
 + \frac{\tau-t}{\e} v )
 (\tilde{\varphi} (v)) ^2  \dd v  \dd \tau \dd s. 
 \end{split}
\Ee 
Now for the same reason to restrict $\tau$-integration in (\ref{def:S*_1}) we rewrite the above expression as 
\Be\label{form:SS*}
 SS^*(h) (t,x) =\frac{1}{\e^2} \int^{t}_{
 \max\{-1, t- 10\pi N \e
 \}
 } 
\int^{\min \{T, s+ 10 \pi N\e \}}
_s\int_{\R^3}  
h(\tau, x 
 + \frac{\tau-t}{\e} v )
 (\tilde{\varphi} (v)) ^2  \dd v  \dd \tau \dd s.
\Ee

We consider a map with the change of variables
\Be\label{COV_average}
v \in  V
\mapsto     y:=x+ \frac{\tau-t}{\e } v \in  \R^3,  \ \ 
 \dd v=  {\dd y}\Big/{\left|\frac{\p y}{\p v}\right|}  =\frac{ \e^3}{|\tau- t|^3}\dd y.
\Ee
\hide
Note that $|y| \leq \sqrt{2}\pi + \frac{N}{N^2}< 2 \pi $ from $\tau -s < \frac{\e}{N^2}$. Hence this map is not one-to-one but the image of the map (\ref{map:v_y}) can cover $\mathbb{T}^2 \times \R$ at most $9$ times. More precisely for each $n_i  \in \{-1,0,1\}$ for $i=1,2$
\Be \notag
y_i= x_i + \frac{\tau-s }{\e} v_i - n_i \pi \in \mathbb{T}
 \ \ \text{for}  \ 
 \ v_i   \in I_{n_i}:= \Big[
\frac{1}{\frac{\tau-s }{\e} } \{(n_i-1)\pi  -  x_i\}  , \frac{1}{\frac{\tau-s }{\e} }\{(n_i+1) \pi  -  x_i\}\Big].  
\Ee
Therefore   
\Be\label{COV_average}
 \mathbf{1}_{(v_1,v_2) \in I_{n_1} \times I_{n_2}} 
 \dd v= \frac{1}{\left|\frac{\p y}{\p v}\right|} \dd y \leq \frac{ \e^3}{|\tau- t|^3}\dd y. 
\Ee\unhide
Now we apply (\ref{COV_average}) to (\ref{form:SS*}) 
and derive that 
\Be\label{bound:SS*}
\begin{split}
|SS^*(h) (t,x)| &\leq 
\frac{1}{\e^2} \int^t_{t- 10 \pi N\e }\int^{
\min \{ s+ 10 \pi N\e\}
 }_s   \int_{\tilde{\O}}
| 
\mathbf{1}_{ \tau \in [ -1, T] } 
h(\tau, y )| 
 \frac{\e^3}{|\tau-t|^3}  
 \tilde{\varphi}  \Big(\e\frac{|y-x| }{|\tau-t|}
 \Big)^2
 \dd y   \dd \tau   \dd s.
 \end{split}
\Ee

First using the Minkowski's inequality and the Young's inequality to a convolution in $y$ with $1+ 1/p=1/q + 1/(p/2)$ we have 
\Be\label{bound1:SS*}
\begin{split}
&\left\| \frac{1}{\e^2} \int^t_{t- 10\pi N \e }\int^{s+ 10 \pi N \e }_s 
\int_{\tilde{\O}}
\mathbf{1}_{ \tau \in [ -1, T] }   |h(\tau, y )|
 \frac{\e^3}{|\tau-t|^3} 
 \tilde{\varphi} \Big(\e \frac{|y-x|}{|\tau-t|}\Big)^2
 \dd y\dd \tau   \dd s\right\|_{L^p_x(\tilde{\O})}\\
 \leq & \  \frac{1}{\e^2} \int^t_{t- 10\pi N \e }\int^{s+ 10 \pi N \e }_s   \left\| \int _{\tilde{\O}}
\mathbf{1}_{ \tau \in [ -1, T] }   |h(\tau, y )|
 \frac{\e^3}{|\tau-t|^3} 
 \tilde{\varphi} \Big(\e \frac{|y-x|}{|\tau-t|}\Big)^2
 \dd y \right\|_{L^p_x(\tilde{\O})} \dd \tau   \dd s\\
  \leq & \  \frac{1}{\e^2} \int^t_{t- 10\pi N \e }\int^{s+ 10 \pi N \e }_s  \|
  \mathbf{1}_{ \tau \in [ -1, T] } 
  h(\tau, \cdot  )\|_{L^q_x(
  \tilde{\O})}\underbrace{\Big\| 
 \frac{\e^3}{|\tau-t|^3} 
 \tilde{\varphi} \Big(\e \frac{|\cdot|}{|\tau-t|}\Big)^2
 \Big\|_{L^{p/2}_x(\tilde{\O})}}_{(\ref{bound1:SS*})_*} \dd \tau   \dd s.
\end{split}\Ee
From the properties of $\tilde{\varphi} \in C_c^\infty$, it follows that 
\Be\notag
(\ref{bound1:SS*})_* 
\leq 
 \frac{\e^3}{|\tau-t|^3}  
\left( \frac{|\tau-t|^3}{\e^3}\right)^{\frac{2}{p}}
 \left[ \int_{ \R^3}\left| \tilde{\varphi}(\tilde{y}) \right|^{p } \dd \tilde{y}
 \right]^{\frac{1}{p/2}}
 \lesssim \bigg(\frac{\e}{|\tau-t|}\bigg)^{3-\frac{6}{p}}, 
\Ee
where $\tilde{y}=  \frac{\e  }{|\tau-t|}(y-x)$ with $\dd \tilde{y}=  \frac{\e^3 }{|\tau-t|^3}\dd y$. Therefore we derive that
\Be\label{bound2:SS*}
\begin{split}
\|SS^*(h)(t, \cdot) \|_{L^p_x} &\lesssim \frac{1}{\e^2} \int^t_{t- 10 \pi N \e }
\int^{s+ 10 \pi N \e}_s  
 \|
 \mathbf{1}_{ \tau \in [ -1, T] } 
 h(\tau, \cdot  )\|_{L^q_x(\tilde{\O})} \bigg(\frac{\e}{|\tau-t|}\bigg)^{3-\frac{6}{p}}    \dd \tau \dd s 
 . 
\end{split}\Ee
Using the Minkowski's inequality and the Young's inequality, finally we prove (\ref{claim:SS*}) as 
\begin{align*}
 &\big\|\|SS^*(h)(t, \cdot) \|_{L^p_x} \big\|_{L^2_t (0, T)}  \\
\lesssim & \  \frac{1}{\e^2} 
\Big\| \mathbf{1}_{[t- 10 \pi N \e , t]} (s)
\Big\|_{L^1_s  } 
\left\|
\sup_{s \in [t- 10 \pi N \e , t]}
\int^{s+ 10 \pi N \e}_{s}  
 \|
 \mathbf{1}_{ \tau \in [ -1, T] } 
 h(\tau, \cdot  )\|_{L^q_x(\tilde{\O})} \bigg(\frac{\e}{|\tau-t|}\bigg)^{3-\frac{6}{p}}  
 \dd \tau \right\|_{L^2_t}\\
 \lesssim & \  \frac{1}{\e^2} 
10 \pi N \e 
\left\|
\int^{t+ 10 \pi N \e}_{t- 10 \pi N \e}  
 \| \mathbf{1}_{ \tau \in [ -1, T] } h(\tau, \cdot  )\|_{L^q_x(\tilde{\O})} \bigg(\frac{\e}{|\tau-t|}\bigg)^{3-\frac{6}{p}}  
 \dd \tau \right\|_{L^2_t}\\
 \lesssim & \  \frac{1}{\e^2}   10 \pi N  \e 
\big\|\|h(\tau, \cdot  )\|_{L^q_x(\tilde{\O})}\big\|_{L^2_\tau ((-1, T])}
\left\|\bigg(\frac{\e}{| t|}\bigg)^{3-\frac{6}{p}}  
\right\|_{L^1_t ((0,  10 \pi N \e ))}\\
\lesssim  & \ 
N^{-1 + \frac{6}{p}} 
\| h \|_{L^2_t ((-1 , T];L^q_x(\tilde{\O}) )} .
\end{align*} \end{proof}
\hide
Here we have used $\int^T_0 e^{- \frac{C_0 t}{\kappa \e^2 } }\dd t \leq \kappa \e^2 \{1-  e^{- \frac{C_0 T}{\kappa \e^2 } }\} \leq \kappa \e^2$ and 
\Be\begin{split}
\int_0^{T  } s^{\frac{6}{p}-3} e^{- \frac{C_0 s}{\kappa \e^2 } }\dd s
&= \int_0^{T } \frac{1}{\frac{6}{p}-2}\frac{d}{ds} \big(s^{\frac{6}{p}-2}\big) e^{- \frac{C_0 s}{\kappa \e^2 } }\dd s\\
&=\frac{1}{\frac{6}{p}-2}  T ^{\frac{6}{p}-2} e^{- \frac{C_0}{\kappa\e^2}  T }
+ \frac{1}{\frac{6}{p}-2} \int^{T }_0 (\kappa \e^2)^{\frac{6}{p}-2}
\left(\frac{s}{\kappa \e^2}\right)^{\frac{6}{p}-2} \frac{C_0}{\kappa \e^2} e^{- \frac{C_0 s}{\kappa \e^2 } }\dd s\\
& \lesssim  (\kappa \e^2)^{\frac{6}{p}-2} \ \ \text{for} \ \ p<3. \notag
\end{split}\Ee\unhide

\

\textbf{Step 3: Applying Lemma \ref{lemma:average}.}  
Now we apply Lemma \ref{lemma:average} to (\ref{f_R:duhamel}) and derive that 
\begin{align}
& \left\| \int_{\R^3}\bar{f}_R (t,x,v) \tilde{\varphi} (v)  \dd v \right\|_{L^2_t( (-1,T]; L^p_x (\tilde{\O}))}\notag\\ 
\lesssim& \  \| \mathbf{1}_{(t,x,v) \in \mathfrak{D}_T}  \bar{g} \|_{L^2  ((-1,T]\times U \times V)}\notag \\
\lesssim & \   \| 
 f_R (t,x,v)\|_{L^2  ((0,T]\times \tilde{\O} \times V)}  + \| 
 f_R (0,x,v) \|_{L^2  (  \tilde{\O} \times V)}
 \notag
 \\
&+ \| 
 \mathbf{1}_{(t,x,v) \in \mathfrak{D}_T}
 f_I (t+ \e \tilde{t}_B(x,v),  \tilde{x}_B(x,v) ,v)
\|_{L^2  ((-1,T]\times (U \backslash  \tilde{\O})\times V)}
\label{g2}\\
&+  \| [\e\p_t + v\cdot \nabla_x ] f_R
\|_{L^2  ((0,T]\times \tilde{\O} \times V)},\label{g3}
\end{align}
where we have used (\ref{ext:f+}), (\ref{ext:f1}), (\ref{ext:f2}), and the fact that $|v\cdot \nabla_x \chi_2 (x)| \lesssim_N 1$ on $v \in V$. 

First we consider (\ref{g2}). We split the cases of (\ref{g2}) according to (\ref{tB}). For $x \in \p\tilde{\O}$, which has a zero measure in $L^2  ((-1,T]\times (U \backslash  \tilde{\O})\times V)$, we have $\tilde{t}_B(x,v)=0$ from the first line of (\ref{tB}). If $\tilde{B}(x,v) = \emptyset$ and $x \notin \p\tilde{\O}$ then $\tilde{t}_B(x,v)=-\infty$ from the last line of (\ref{tB}) and hence $\bar{f}_R(-\infty)=0$ since $\chi_1(-\infty)=0$ in (\ref{ext:f1}). Therefore we derive that 
\begin{align}
(\ref{g2})\leq
&
 \   \| 
\mathbf{1}_{\{s<0\} \subset \tilde{B} (x,v)}
 \mathbf{1}_{(t,x,v) \in \mathfrak{D}_T} 
 f_I (t+ \e \tilde{t}_B(x,v), \tilde{x}_B(x,v)
 ,v)
\|_{L^2  ((-1,T]\times (U \backslash  \tilde{\O})\times V)}
\label{g2:1}
\\
&+  \| 
\mathbf{1}_{\{s>0\} \subset \tilde{B} (x,v)}
 \mathbf{1}_{(t,x,v) \in \mathfrak{D}_T} 
 f_I (t+ \e \tilde{t}_B(x,v), \tilde{x}_B(x,v),v)
\|_{L^2  ((-1,T]\times (U \backslash  \tilde{\O})\times V)}.
\label{g2:2} 
\end{align}
We need a special attention to (\ref{g2:1}). Since $(t,x,v) \in \mathfrak{D}_T$ we know that 
$\inf \{\tau \geq t: x+  \frac{\tau - t}{\e} v \in cl(\tilde{\O})\} \leq T.$ 
If $\{s<0\} \subset \tilde{B} (x,v)$ then, from the third line of (\ref{tB}), $\tilde{t}_B(x,v) = \sup \tilde{B} (x,v)= \sup \{s \in \R: x+ sv \in \R^3 \backslash cl (\tilde{\O})\} \leq (T-t)/\e$. Therefore the argument of $f_I$ in (\ref{g2:1}) is confined as
\Be\label{t+tB<}
(t+ \e \tilde{t}_B(x,v), \tilde{x}_B(x,v),v) \in (- \infty,T] \times \p\tilde{\O} \times V .
\Ee
For (\ref{g2:2}), from the second line of (\ref{tB}), $\tilde{t}_B (x,v)= \inf \tilde{B}(x,v) = \inf 
\{s \in \R: x+ sv \in \R^3 \backslash cl (\tilde{\O})\} \leq 0$. Therefore $t+ \e \tilde{t}_B (x,v) \leq t \leq T$ and hence the argument of $f_I$ in (\ref{g2:2}) is confined as in (\ref{t+tB<}). Now we apply the Minkowski's inequality in time, change of variables $t+ \e \tilde{t}_B (x,v) \mapsto t$, and use (\ref{t+tB<}) to derive that 
\Be\label{est1:g2}
(\ref{g2:1}) + (\ref{g2:2}) \lesssim
\Big\|
 \| f_I(t, \tilde{x}_B(x,v),v)\|_{L^2_t ((-1, T])}
 \Big\|_{L^2_{x,v}((U \backslash \tilde{\O}) \times V)}.
\Ee

Let us define an outward normal $\tilde{n}(x)$ on $\p\tilde{\O}$. More precisely 
\Be\label{tilde_n}
\tilde{n}(x)= \begin{cases}
(0,0,-1) &\ \ \text{if} \ x_3=0 \ \text{and} \ x \in \p\tilde{\O},\\
((-1)^{\frac{x_1}{2\pi}+1},0,0) &\ \ \text{if} \ x_1 \in \{0, 2\pi\} \ \text{and} \ x \in \p\tilde{\O},\\
(0,(-1)^{\frac{x_2}{2\pi}+1},0) &\ \ \text{if} \  x_2 \in \{0, 2\pi\} \ \text{and} \ x \in \p\tilde{\O}.
\end{cases}
\Ee
From (\ref{def:UV}) we have therefore $(x ,v) \in (U \backslash \tilde{\O}) \times V$ then $|\tilde{n}( \tilde{x}_B (x,v) ) \cdot v|\geq 1/N$. We consider maps 
\Be\begin{split}\label{bdry_map}
(x_1,x_3) &\mapsto  \tilde{x}_B (x,v) 
\in  (0,2\pi) \times (0, 2\pi) \times \{x_3=0\},\\
& \ \ \text{with} \ \ \Big|\det\Big(\frac{\p (  \tilde{x}_{B,1} (x,v) ,     \tilde{x}_{B,2} (x,v) )}{\p (x_1,x_3)}\Big)\Big|= \Big|\frac{ v_2}{v \cdot \tilde{n}}\Big|,\\
(x_i,x_3) &\mapsto  (\tilde{x}_{B,i} (x,v), \tilde{x}_{B,3} (x,v))  
\in  (0,2\pi)   \times (0,\infty), \\
&\ \ \text{with} \ \ \Big|\det\Big(\frac{\p (  \tilde{x}_{B,i} (x,v) ,     \tilde{x}_{B,3} (x,v) )}{\p (x_1,x_3)}\Big)\Big|= \Big|\frac{ v_i }{v \cdot \tilde{n}}\Big|,  \ \ \text{for} \ i =1,2. 
\end{split}\Ee
 Note that if $v \in V$ of (\ref{def:UV}) then $|v_i| \geq 1/N$ for all $i=1,2,3.$ We define 
\Be\label{def:tilde_gamma}
\tilde{\gamma} := \p\tilde{\O} \times \R^3, \ \ \ \tilde{\gamma}^{N} := \p\tilde{\O} \times (\R^3\backslash V). 
\Ee
We apply the change of variables (\ref{bdry_map}) to (\ref{est1:g2}):
\Be\label{est2:g2}
\begin{split}
(\ref{est1:g2})&= 
\bigg\| 
\bigg[
\int_{-2\pi}^{4\pi}
\int_{-2\pi}^\infty \int_{-2\pi}^{4\pi}
 \| f_I (t, \tilde{x}_B(x,v),v)\|_{L^2_t ((-1, T])}^2
 \dd x_1 \dd x_3 \dd x_2
\bigg]^{1/2} 
\bigg\|_{L^2_{v}(V)}\\
&\leq   \bigg\|  \bigg[ 5 \times 6 \pi N   \int_{\p\tilde{\O}}\int_{-1}^T  |f_I (t,  y,v) |^2 |v \cdot \tilde{n}(y)| \dd t  \dd y\bigg]^{1/2} \bigg\|_{L^2_{v}(V)}\\
&\lesssim \|f_R\|_{L^2((0,T) \times \tilde{\gamma} \backslash \tilde{\gamma}^{  N   } )}
+ \|f_R(0)\|_{L^2(  \tilde{\gamma} \backslash \tilde{\gamma}^{   N   } )}.
\end{split}
\Ee
We recall the trace theorem:
\Be\label{trace1}
\begin{split}
\int^T_{0} \int_{\tilde{\gamma}\backslash \tilde{\gamma}^{1/N}} |h| \dd \gamma \dd s 
\lesssim& \sup_{t \in [0,T]} \| h(t) \|_{L^1(\tilde{\O} \times V)}  + \int^T_{0} \| h(s) \|_{L^1(\tilde{\O} \times V)} \dd s \\
&
+ \int^T_{0} \| [\e \p_t + v\cdot \nabla_x ] h \|_{L^1(\tilde{\O} \times V)} \dd s.
\end{split}
\Ee
We apply (\ref{trace1}) with $h=f^2$ and derive an estimate 
\Be\begin{split}
\label{est:bdry}
 &\|f_R\|_{L^2((0,T) \times \tilde{\gamma} \backslash \tilde{\gamma}^{  N   } )}^2 \\
 &\lesssim \sup_{t \in [0,T]} \| f_R(t) \|_{L^2(\tilde{\O} \times V)}^2  + \int^T_{0} \| f_R(s) \|_{L^2(\tilde{\O} \times V)}^2 \dd s 
+ \int^T_{0}  \iint_{\tilde{\O}\times V}\big| f_R[\e \p_t + v\cdot \nabla_x ] f_R \big| \dd x \dd v  \dd s\\
&\lesssim _T 
 \| f_R \|_{L^\infty ([0,T];L^2( {\O} \times \R^3))}^2
 + \big\|  [\e \p_t+ v\cdot \nabla_x ]   f_R  \big\|_{L^2( [0,T] \times  {\O} \times \R^3)}.
\end{split}\Ee
\hide
where the last term can be bounded 
from (\ref{eqtn_fR}) as 
\Be
\begin{split}
\int^T_{0}  \iint_{\O \times \R^3}
\Big|&
- \frac{1}{\e \kappa} (\mathbf{I} - \mathbf{P}) f_R  L(\mathbf{I} - \mathbf{P}) f_R 
+ \frac{\e}{\kappa} \Gamma(f_2, f_R) (\mathbf{I} - \mathbf{P}) f_R
+ \frac{\delta}{\kappa} \Gamma(f_R, f_R) (\mathbf{I} - \mathbf{P}) f_R
- \frac{(\e \p_t + v\cdot \nabla_x ) \sqrt{\mu}}{\sqrt{\mu}} |f_R|^2\\
&+ \e (\mathbf{I} - \mathbf{P}) \mathfrak{R}_1  (\mathbf{I} - \mathbf{P}) f_R
+ \e   \mathfrak{R}_2  f_R
\Big|
\dd x \dd v  \dd t
\end{split}
\Ee\unhide
Finally we conclude a bound of (\ref{g2}) as below via (\ref{g2:1}), (\ref{g2:2}), (\ref{est1:g2}), (\ref{est2:g2}), and (\ref{est:bdry}) 
\Be\label{est3:g2}
(\ref{g2})\lesssim \| f_R (0)\|_{L^2_\gamma} + \| f_R \|_{L^\infty ([0,T];L^2( {\O} \times \R^3))} + \underbrace{\big\|  [\e \p_t+ v\cdot \nabla_x ]   f_R  \big\|_{L^2( [0,T] \times  {\O} \times \R^3)}}_{(\ref{est3:g2})_*}. 
\Ee

Next we estimate (\ref{g3}) (and $(\ref{est3:g2})_*$). 
Using (\ref{eqtn:bar_f}) and (\ref{eqtn_fR}) we conclude that 
\Be\notag
\begin{split}
&(\ref{g3})+ (\ref{est3:g2})_*
 \\
& \lesssim  
  \bigg\|  
 - \frac{1}{\e \kappa}    L(\mathbf{I} - \mathbf{P}) f_R 
 +  \frac{\e}{\kappa} \Gamma(f_2, f_R) + \frac{\delta}{\kappa} \Gamma(f_R, f_R)  \\
 & \ \ \ \ \ 
  -  \frac{( \e \p_t + 
  v\cdot \nabla_x) \sqrt{\mu}}{\sqrt{\mu}} f_{R}
  +  
\e (\mathbf{I}- \mathbf{P})\mathfrak{R}_1 +\e \mathfrak{R}_2
    \bigg\|_{L^2 ((0,T]\times \O \times
V)}.
 \end{split}
\Ee
Following the arguments of (\ref{est:EG})-(\ref{est:Energy_v^3}), and (\ref{est:R1}), (\ref{est:R2}), we derive that 
\Be \label{est:bar_f3_1}
\begin{split}
&(\ref{g3})+(\ref{est3:g2})_*\\
 &\lesssim   
 \  \Big\{
 \frac{\e}{\kappa} \|  (\ref{est:f2})  \|_{L^\infty_t ((0,T); L_x^{\frac{2p}{p-2}}(\O)) } + \frac{\delta}{\kappa} \| Pf_R \|_{L^\infty_t ((0,T); L_x^{\frac{2p}{p-2}}(\O)) } 
 \Big\}
 \| P f_R \|_{L^2_t ((0,T); L_x^{p}(\O))}\\
&+\Big\{
\frac{1}{\e \kappa }
+ \frac{\delta}{\kappa} \| \mathfrak{w}_{\varrho, \ss} f_R \|_{L^\infty_t ((0,T) \times \O \times \R^3) }
\Big\}
 \| (\mathbf{I} - \mathbf{P})f_R \|_{L^2_t ((0,T) \times \O \times \R^3) }
 \\
 &+\e \big\|  (\ref{transp:mu}) \big\|_{L^2_t ((0,T); L^\infty_x(\O))} \|  f_R(t) \|_{L^\infty_t ((0,T);L^2(\O \times \R^3))} \\
 &+  \e  \{  \|(\ref{est:R1})\| _{L^2_t ((0,T); L_x^{2}(\O))}
 +  \|(\ref{est:R2})\| _{L^2_t ((0,T); L_x^{2}(\O))}
 \} 
,
\end{split}
\Ee
where we further bound 
\Be\label{est:bar_f3_1:1}
\| Pf_R \|_{  L_x^{\frac{2p}{p-2}}(\O)  }
\leq \| Pf_R \|_{L_x^{6}(\O)}^{\frac{3(p-2)}{p}} 
\|  \mathfrak{w}_{\varrho,\ss} f_R \|_{L_x^{\infty}(\O)}^{\frac{6-2p}{p}}.
\Ee
\hide
and conclude that 
\Be
\begin{split}(\ref{g3})
 \lesssim& \ 
\Big\{
  \frac{\e}{\kappa}
 \| \mathfrak{w}_{\varrho, \ss} f_2 \|_{L^\infty_t ((0,T); L_x^{\frac{2p}{p-2}}(\O)) } + \frac{ \delta}{\kappa}  \| Pf_R \|_{L^\infty_t ((0,T);L_x^{6}(\O))}^{\frac{3(p-2)}{p}} 
\|  \mathfrak{w}_{\varrho,\ss} f_R \|_{L_ {\infty}((0,T) \times\O\times\R^3)}^{\frac{6-2p}{p}}
\Big\}
\| P f_R \|_{L^2_t ((0,T); L_x^{p}(\O))}\\
  &+ 
\e \big\| |\nabla_x u| + \e |\p_t u| + \e |u| |\nabla_x u| \big\|_{L^2_t ((0,T); L^\infty_x(\O))}
 \|  f_R(t) \|_{L^\infty_t ((0,T);L^2(\O \times \R^3))}\\
 &+ \frac{\delta}{\kappa}
  \| \mathfrak{w}_{\varrho, \ss} f_R \|_{L^\infty_t ((0,T) \times \O \times \R^3) } 
  \| (\mathbf{I} - \mathbf{P})f_R \|_{L^2_t ((0,T) \times \O \times \R^3) }\\
 &+
 \frac{\e}{\delta}
 \|\mathfrak{q}(|\nabla_x \tilde{u}|, |\nabla_x^2 u|)\| _{L^2_t ((0,T); L_x^{2}(\O))}
 \\
 &  + \frac{\e^2}{\delta \kappa} 
 \| \mathfrak{q}(|u|, |\nabla_{x } u |, |\p_{t}u|, |\nabla_{x}^{2} u|, |\nabla_{x} \p_{t} u|, |p|, |\nabla_{x } p |, |\p_{t} p|, |\tilde{u}|, |\nabla_{x } \tilde{u}|, |\p_{t} \tilde{u}|) \|_{L^2_t ((0,T); L_x^{2}(\O))}.
\end{split}
\Ee
\unhide

\ 

\textbf{Step 4. Proof of (\ref{average_3D}). } First we use (\ref{est:decomp1:L2L3}) and then (\ref{g2}) and (\ref{g3}). We bound (\ref{g2}) via (\ref{est1:g2}) and (\ref{est2:g2}), which are bounded by (\ref{est:bdry}) and (\ref{est:bar_f3_1}) respectively. These conclude that, for $p<3$, 
\Be\begin{split}\label{final_est:3D}
&\big(1- O(\e) \|u\|_\infty- O(\frac{1}{N})\big)  \big\|    {P} f_R
 \big\|_{L^2_t((0,T);L^p_x(\tilde{\O}))}\\
 &
 -C_{T,N}   \| \mathfrak{w}_{\varrho, \ss } f_R (t) \|_{ L^\infty((0,T) \times \tilde{\O} \times \R^3)}^{\frac{p-2}{p}} \| (\mathbf{I} - \mathbf{P}) f_R \|_{L^2((0, T) \times \tilde{\O} \times \R^3)}^{\frac{2}{p}}\\
\leq & \  \left\|\int_{\R^3} \bar{f}_R(t,x,v) \tilde{\varphi}_i (v) \sqrt{\mu_0(v)}    \dd v \right\|_{L^2_t ((0,T); L^p_x ( \tilde{\O}))}\\
\leq & \  \left\|\int_{\R^3} \bar{f}_R(t,x,v) \tilde{\varphi}  (v)     \dd v \right\|_{L^2_t ((0,T); L^p_x ( \tilde{\O}))}\\
\lesssim & \ 
 \| f_R \|_{L^\infty ([0,T];L^2( {\O} \times \R^3))}
+ \| f_R (0) \|_{L^2_\gamma}+ \text{r.h.s. of }  (\ref{est:bar_f3_1}) \text{ with } (\ref{est:bar_f3_1:1}). 
\end{split}\Ee
Then we move a contribution of $\| P f_R \|_{L^2_t ((0,T); L_x^{p}(\O))}$ to the l.h.s and use (\ref{est:bar_f3_1:1}). This concludes (\ref{average_3D}).

\

\textbf{Step 5: Sketch of proof for (\ref{average_3Dt}). } We follow the same argument for (\ref{average_3D}). Thereby we only pin point the difference of the proof of (\ref{average_3Dt}). Recall $\p_t f_R (0,x,v) = f_{R,t}(0,x,v)$ from \eqref{initial_f}. \hide From (\ref{eqtn_fR}) we define 
\Be\label{f:initial_t}
\begin{split}
 f_{R,t} (0,x,v):= & - \frac{1}{\e} v\cdot \nabla_x   f_R 
  -   \frac{1}{ \e^2\kappa} L
  f_R +  \frac{1}{\kappa} \Gamma({f_2} , f_R )
+    \frac{   \delta }{\e\kappa}\Gamma(f_R , f_R )\\
&
-  \frac{( \p_t + 
\e^{-1} v\cdot \nabla_x) \sqrt{\mu}}{\sqrt{\mu}} f_{R }
  + 
(\mathbf{I}- \mathbf{P})\mathfrak{R}_1  + \mathfrak{R}_2 
\Big|_{t=0},
\end{split}
\Ee
where we read all terms in r.h.s at $t=0$. \unhide We regard $\tilde{\O}$ as an open subset but not a periodic domain as $\O$. Without loss of generality we may assume that $f_{R,t}(0,x,v)$ is defined in $\R^3 \times \R^3$ and $\|f_{R,t}(0)
\|_{L^p(\R^3) \times \R^3} \lesssim \|f_{R,t}(0)
\|_{L^p(\tilde{\O}) \times \R^3}$ for all $1 \leq p \leq \infty$. Then we extend a solution for whole time $t \in \R$ as 
\Be\label{ext:f1_t}
  f_{  I,t} (t,x,v) : = \mathbf{1}_{t \geq 0 } \p_t f_R (t,x,v) + \mathbf{1}_{t \leq 0} \chi_1 (t)  f_{R,t} (0, x
,v ).
\Ee
Using $\tilde{t}_B(x,v) $ in (\ref{tB}) we define 
\Be\label{ext:f2_t}
 f_{E,t} (t,x,v): =  \mathbf{1}_{(x,v) \in (\R^3 \backslash \tilde{\O}) \times \R^3 }   f_{I,t} (t+ \e \tilde{t}_B(x,v), \tilde{x}_B(x,v),v).
\Ee
We define an extension of cut-offed solutions   
\Be\label{ext:f+_t}
  \bar{f}_{R,t}(t,x,v) :=
 \chi_2 (x) \chi_3 (v)  \big\{
  \mathbf{1}_{\tilde{\O}}(x)
 f_{I,t}(t,x,v)
 +  f_{E,t} (t,x,v)\big\}
 \  \ \text{for} \ (t,x,v) \in (-\infty,T] \times \R^3 \times \R^3
 . 
\Ee
 We note that in the sense of distributions $\bar{f}_{R,t}$ solves
\Be\label{eqtn:bar_f_t}
\begin{split}
& \e \p_t \bar{f}_{R,t} + v\cdot \nabla_x \bar{f}_{R,t} =     \   \bar{g}_t 
 \ \ \text{in} \ (-\infty,T] \times \R^3 \times \R^3, \ \ \text{where}\\
& \bar{g}_t :  =     \frac{v\cdot \nabla_x \chi_2  }{\chi_2  } \bar{f}_{R,t} + \mathbf{1}_{t \geq 0} \mathbf{1}_{\tilde{\O}} (x) \chi_2(x)  \chi_3(v) 
[\e \p_t + v\cdot \nabla_x] \p_t f_R \\
 &\quad \ \ +\mathbf{1}_{t \leq 0} 
\chi_2 (x) \chi_3 (v) \{
\e\p_t \chi_1(t) 
f_{R,t} (0,x,v) + \chi_1 (t) v\cdot \nabla_x f_{R,t} (0,x,v) 
\}.
 %
  \end{split}
\Ee
Here we have used the fact that $\bar{f}_{R,t}$ in (\ref{eqtn:bar_f_t}) is continuous along the characteristics across $\p \tilde{\O}$ and $\{t=0\}$. 
We derive that, using (\ref{eqtn:bar_f_t}),
\Be\label{f_R:duhamel_t}
\bar{f}_{R,t}(t,x,v) =
 \frac{1}{\e} \int^t_{-\infty} 
 \bar{g}_t (s, x- \frac{t-s}{\e} v, v)  \dd s \ \ \text{for} \ (t,x,v) \in 
(-\infty,T] \times \R^3
 \times \R^3.
\Ee

Now we apply Lemma \ref{lemma:average} to (\ref{f_R:duhamel_t}) and derive that, for $p<3$,
\Be\label{bound:S_t}
\begin{split}
&\| S(\bar{g_t}) \|_{L^2_t((0,T);  L^p_x(\mathbb{T}^2 \times \R))} \\
\lesssim &
 \| \mathbf{1}_{(t,x,v) \in \mathfrak{D}_T}   \bar{g}_t\|_{L^2  ((0,T)\times (\mathbb{T}^2 \times \R) \times \{|v| \leq N \})}\\
 \lesssim & \ \| f_{R,t} (0)\|_{L^2(\O \times \R^3)} + \| \e \p_t  f_R + v\cdot \nabla_x f_R \|_{L^2 ((0,T) \times \tilde{\O} \times V)} \\
 &+  \| 
 \mathbf{1}_{(t,x,v) \in \mathfrak{D}_T}
 f_{I,t} (t+ \e \tilde{t}_B(x,v),  \tilde{x}_B(x,v) ,v)
\|_{L^2  ((-1,T]\times (U \backslash  \tilde{\O})\times V)}.
 \end{split}
\Ee
Following the same argument of (\ref{est3:g2})-(\ref{est:bar_f3_1}) we deduce that 
\Be\label{est:bar_f3_1_t}
\begin{split}
(\ref{bound:S_t})
 \lesssim  &
 \|\p_t  f_R \|_{L^\infty ([0,T];L^2( {\O} \times \R^3))}  +  \|\p_t  f_R (0)\|_{L^2_\gamma} 
\\
& +\big\|  - \frac{1}{\e \kappa}    L(\mathbf{I} - \mathbf{P}) \p_t f_R  + \e \times \text{r.h.s. of } (\ref{eqtn_fR_t}) 
    \big\|_{L^2 ((0,T]\times \O \times
V)}.
 \end{split}
\Ee
From (\ref{est:Energy_v^3_t:1})-(\ref{est:Energy_v^3_t_second}), the last term of (\ref{est:bar_f3_1_t}) is bounded above by 
\Be
\begin{split}\label{est:bar_f3_1_t:last}
&
\Big\{\frac{1}{\kappa} \| \p_t u \|_{L^\infty_{t,x}}
\Big(
1+ \delta \e   \| \mathfrak{w}f_R \|_{L^\infty_{t,x,v}}
\Big)
+ \frac{\e^2}{\kappa} \|  (\ref{est:f2}) \|_{L^\infty_{t,x}}
\Big\} 
\Big\{
\| Pf_R \|_{L^2_{t,x}}+ \| \sqrt{\nu} (\mathbf{I} -\mathbf{P}) f_R \|_{L^2_{t,x,v}}
\Big\}
\\
&+
\Big\{\frac{1}{\kappa \e} + \frac{\delta}{\kappa} \|   \mathfrak{w}f_R\|_{L^\infty_{t,x,v}} +  \frac{\e}{\kappa} \|(\ref{est:f2})  \|_{L^\infty_{t,x}}  + \e \| 
 (\ref{transp:mu})
 \|_{L^\infty_{t,x} }
\Big\}
\| \sqrt{\nu} (\mathbf{I} - \mathbf{P}) \p_t f_R \|_{L^2_{t,x,v}}
 \\
 &+   
\Big\{ 
\frac{\delta}{\kappa}\|    {P} f_R \|_{L^\infty_tL^6_{x }}^{\frac{3(p-2)}{p}}
\| \mathfrak{w} f_R\|_{L^\infty_{t,x,v}}^{\frac{6-2p}{p}}
+ \frac{\e}{\kappa}\|  (\ref{est:f2}) \|_{L^\infty_tL^{\frac{2p}{p-2}}_{x }}   
+  \e \| 
 (\ref{transp:mu})
 \|_{L^\infty_{t } L^\frac{2p}{p-2}_x
 }
\Big\}
 \|    {P} \p_t f_R \|_{L^2_tL^p_{x }}
\\ 
&+ \frac{\e}{\kappa}  \| 
 (\ref{est:f2_t}) \|_{L^2_{t,x,v}}
  \| \mathfrak{w} f_R \|_{L^\infty_{t,x,v}}
 +\e 
\|(\ref{transp:mu_t})\|_{L^2_tL^\infty_{x}}
\| f_R\|_{L^\infty_t L^2_{x,v}} 
 +   \e \{  \|
 (\ref{est:R3})
 \|_{L^2_{t,x}}
 +
 \|  (\ref{est:R4}) \|_{L^2_{t,x}}
 \}.
\end{split}
\Ee
Here the most singular term comes from $\frac{1}{\e^2 \kappa} L(\mathbf{P}_t f_R)$ in the r.h.s. of (\ref{eqtn_fR_t})
.

On the other hand from (\ref{ext:f+_t}) and the argument of (\ref{decomp1:L2L3}) we derive 
\Be\begin{split}\label{decomp1:L2L3_t}
&\| S(\bar{g_t}) \|_{L^2_t((0,T);  L^p_x(\mathbb{T}^2 \times \R))}\gtrsim\left\|\int_{\R^3}  \bar{f}_{R,t}(t,x,v) \tilde{\varphi}_i (v) \sqrt{\mu_0(v)}    \dd v \right\|_{L^2_t ((0,T); L^p_x ( \tilde{\O}))} 
\\
\gtrsim & \ \big(1- O(\e )\|u\|_\infty- O(\frac{1}{N})\big)
 \big\|   {P} \p_t  f_R
 \big\|_{L^2_t( (0,T)  ;L^p_x(\tilde{\O}))}\\
 &- (\kappa \e ) ^{\frac{2}{p-2}} \|  \mathfrak{w}^\prime \p_t  f_R \|_{L^2_t ((0,T);L^\infty_{x,v} (  \O \times \R^3))}-\frac{1}{\kappa \e }\|  (\mathbf{I} - \mathbf{P}) \p_t  f_R  \|_{L^2 ((0,T) \times   \O \times \R^3)}
 %
.
 \end{split}\Ee
Here we have used 
\Be\label{I-P_infty,2}
\begin{split}
& \left\|\int_{\R^3}
\chi_2 (x) \chi_3 (v)
 (\mathbf{I} - \mathbf{P}) \p_t  f_R(t,x,v)  \tilde{\varphi}_i (v) \sqrt{\mu_0(v)}    \dd v   \right\|_{L^2_t ((0,T); L^p_x (\tilde{\O}))} \\
 \leq &  \left\|  
 (\mathbf{I} - \mathbf{P}) \p_t  f_R(t,x,v)      \right\|_{L^2_t ((0,T); L^p_{x,v} (\tilde{\O} \times \R^3))} \\
 \lesssim & \ \Big\| \|  \mathfrak{w}^\prime \p_t  f_R \|_{L^\infty_{x,v} (  \O \times \R^3)}^{\frac{p-2}{p}}
 \|  (\mathbf{I} - \mathbf{P}) \p_t  f_R  \|_{L^2_{x,v} (  \O \times \R^3)}^{\frac{2}{p}}
 \Big\|_{L^2_t((0,T))}\\
 \lesssim& \  \Big\| \|  \mathfrak{w} ^\prime\p_t  f_R \|_{L^\infty_{x,v} (  \O \times \R^3)}^{\frac{p-2}{p}}\Big\|_{L^{\frac{2p}{p-2}}_t((0,T))}
 \Big\| \|  (\mathbf{I} - \mathbf{P}) \p_t  f_R  \|_{L^2_{x,v} (  \O \times \R^3)}^{\frac{2}{p}}\Big\|_{L^p_t((0,T))} \\
 \lesssim& \  (\kappa \e)^{\frac{2}{p}} \|  \mathfrak{w}^\prime \p_t  f_R \|_{L^2_t ((0,T);L^\infty_{x,v} (  \O \times \R^3))}^{\frac{p-2}{p}} (\kappa \e)^{-\frac{2}{p}} \|  (\mathbf{I} - \mathbf{P}) \p_t  f_R  \|_{L^2 ((0,T) \times   \O \times \R^3)}^{\frac{2}{p}}\\
 \lesssim & \ (\kappa \e ) ^{\frac{2}{p-2}} \|  \mathfrak{w}^\prime \p_t  f_R \|_{L^2_t ((0,T);L^\infty_{x,v} (  \O \times \R^3))}+ (\kappa \e)^{-1} \|  (\mathbf{I} - \mathbf{P}) \p_t  f_R  \|_{L^2 ((0,T) \times   \O \times \R^3)}.
\end{split}
\Ee 

Combining (\ref{decomp1:L2L3_t}), (\ref{bound:S_t}), (\ref{est:bar_f3_1_t}), and (\ref{est:bar_f3_1_t:last}) and choosing $N\gg1$ we conclude (\ref{average_3Dt}).

\hide

consider (\ref{ext:f1}), (\ref{tB}), and (\ref{def:mathcalU}).

From the change of variables $x- \frac{t}{\e} \mapsto x$ we bound the second term of (\ref{g1}) by

\begin{align}
 \chi_2 (x) \chi_3 (v)  \big\{
  \mathbf{1}_{\tilde{\O}}(x)
 f_1(t,x,v)
 +  f_2 (t,x,v)\big\}
\end{align}

Therefore we derive that 
\begin{align}
&\left\|\int_{\R^3} \bar{f}_R(t,x,v) \tilde{\varphi}_i (v) \sqrt{\mu_0(v)}    \dd v \right\|_{L^2_t ((0,T); L^p_x (\mathbb{T}^2 \times \R ))}\label{bar_f1}
\\
\leq & \  \left\| \int_{\R^3} e^{-   \frac{C_\nu }{\kappa \e^2 }   t}  f_R (0,x- \frac{t}{\e} v , v)    \tilde{\varphi}_i (v) \sqrt{\mu_0(v)}    \dd v \right\|_{L^2_t ((0,T); L^p_x (\mathbb{T}^2 \times \R ))}\label{bar_f2}\\
&+\left\|  \frac{1}{\e} \int^t_{0} \int_{\R^3}
e^{-  \frac{C_\nu }{\kappa \e^2 } (t-s) }
|\bar{g} (s, x- \frac{t-s}{\e} v, v)| \tilde{\varphi}(v) \dd v  \dd s
\right\|_{L^2_t ((0,T); L^p_x (\mathbb{T}^2 \times \R ))}. \label{bar_f3}
\end{align}

For (\ref{bar_f2}), using the Minkowski's inequality, the change of variables $x- \frac{t}{\e}v\mapsto x$, and (\ref{ext:f+}), it follows that 
\Be\label{est:bar_f2}
(\ref{bar_f2}) \leq    \int_{\R^3}  \Big\|e^{-   \frac{C_\nu }{\kappa \e^2 }   t} \Big\|_{L^2_t ( 0,T )} \Big\| f_R (0,\cdot , v) \Big\|_{L^p_x (\mathbb{T}^2 \times \R )}   \tilde{\varphi}_i (v) \sqrt{\mu_0(v)}    \dd v  
\lesssim 
\big\| f_R ( 0 ) \big\|_{L^p_{x,v} (\O\times \R^3 )}. 
\Ee 

For (\ref{bar_f3}) we decompose 
\Be\notag
(\ref{bar_f3})
= \underbrace{\left\|  \frac{1}{\e} \int^t_{0}  \mathbf{1}_{t-s < \frac{\e}{N^2}} \cdots
\right\|_{L^2_t ((0,T); L^p_x (\mathbb{T}^2 \times \R ))}}_{(\ref{bar_f3})_1}+ \underbrace{ \left\|  \frac{1}{\e} \int^t_{0}  \mathbf{1}_{t-s \geq  \frac{\e}{N^2}} \cdots
\right\|_{L^2_t ((0,T); L^p_x (\mathbb{T}^2 \times \R ))}}_{(\ref{bar_f3})_2}. 
\Ee

For $(\ref{bar_f3})_2$ using $\mathbf{1}_{t-s \geq  \frac{\e}{N^2}} e^{-  \frac{C_\nu }{\kappa \e^2 } (t-s) } \leq e^{-  \frac{C_\nu  }{N^{ 2} } \frac{1}{\kappa \e}  }$ we derive 
\Be\label{est:bar_f3_2}
\begin{split}
(\ref{bar_f3})_2 \lesssim& 
\   \frac{e^{-  \frac{C_\nu  }{N^{ 2} } \frac{1}{\kappa \e}  }}{\e}
\big\|      \mathbf{1}_{t\geq 0}   \mathbf{1}_{\O}(x)  
 (\e \p_t   + v\cdot \nabla_x+ \frac{\nu }{\kappa \e}  )f_R(t,x,v)\big\|_{L^2_t ((0,T); L^p_x (\mathbb{T}^2 \times \R \times \{|v| \leq N\} ))}\\
 \lesssim & \  \kappa  e^{-  \frac{C_\nu  }{2 N^{ 2} }  \frac{1}{\kappa \e } }   \left\|   \frac{1}{\e \kappa} Kf_R + \frac{\e}{\kappa} \Gamma(f_2, f_R) + \frac{\delta}{\kappa} \Gamma(f_R, f_R) 
  -  \frac{( \e \p_t + 
  v\cdot \nabla_x) \sqrt{\mu}}{\sqrt{\mu}} f_{R}
  +  
\e (\mathbf{I}- \mathbf{P})\mathfrak{R}_1 +\e \mathfrak{R}_2
    \right\|_{L^2_t ((0,T); L^p (\O \times
 \{|v| \leq N \}))}
 \end{split}
\Ee

Finally collecting the estimates from (\ref{est:bar_f2}), (\ref{est:bar_f3_1}), and (\ref{est:decom:L2L3}) for (\ref{bar_f1}), we prove (\ref{average_3D}). 

\hide
  Applying the standard Average lemma (e.g. \cite{Laure,gl}) and the Sobolev embedding $H^{1/2}(\O) \subset L^3(\O)$, we derive that 
 \Be\label{VA}
 \begin{split}
 \left\|\int_{\R^3} \bar{f}_R(t,x,v) \tilde{\varphi}_i (v) \sqrt{\mu_0(v)}    \dd v \right\|_{L^2_t (\R)L^3_x(\mathbb{T}^2 \times \R)}&\lesssim  \left\|\int_{\R^3} \bar{f}_R(t,x,v) \tilde{\varphi}_i (v) \sqrt{\mu_0(v)}   \dd v \right\|_{L^2_tH^{1/2}_x}\\
&\lesssim_N    \| \bar{f}_R \|_{L^2_{t,x,v}} +
 \| (\ref{etqn_barf1})
 \|_{L^2_{t,x,v}} +  \| (\ref{etqn_barf2})
 \|_{L^2_{t,x,v}} +  \| (\ref{etqn_barf3})
 \|_{L^2_{t,x,v}}.
\end{split}
\Ee
\unhide

For (\ref{decomp2:L2L3}), using $L^2_t(0,T) \subset L^p_t(0,T)$ and $L^1(\{|v|\leq N\}) \subset L^p(\{|v|\leq N\})$ and $t-\e \frac{x_3}{v_3} \mapsto t$, it follows that 
\Be\label{est1:decomp2:L2L3}
 \begin{split}
(\ref{decomp2:L2L3}) &
\lesssim_{N,T}   \Big\| 
\int_{\R^3} 
 \chi(v) \mathbf{1}_{t-\e \frac{x_3}{v_3}\geq 0}
 \mathbf{1}_{v_3<0}
 e^{- \frac{\nu }{\kappa \e^2 } \frac{ \e x_3}{v_3}}
 f_R(t-\e\frac{x_3}{v_3} 
 ,  x-  \frac{x_3}{v_3} v
 , v) 
 \tilde{\varphi}_i (v) \sqrt{\mu_0(v)}    \dd v
%
 \Big\|_{L^p ((0,T) \times \O \times \R^3)} \\
 &\lesssim_{N,T}  
  \bigg\| \mathbf{1}_{|v| \leq N } \mathbf{1}_{|v_3| \geq 1/N}
\Big\| 
 f_R(t , x-  \frac{x_3}{v_3} v, v)
\Big\|_{L^p_t(0,T)}   \bigg\|_{L^p_{x,v}(  \O \times \R^3)}.
   \end{split}
\Ee
 Now we use a map, for given $|v| \leq N $ and $|v_3|\geq 1/ N$,
\Be
\begin{split}\label{COV_bdry}
(x_1,x_2 ) \in \mathbb{T}^2  \mapsto 
 \Big(
 x_1-  \frac{x_3}{v_3} v_1, x_2-  \frac{x_3}{v_3} v_2  \Big) \in \mathbb{T}^2 
  \ \  \text{with} \  \ 
\left|\frac{\p   \xb(x,v) }{\p (x_1,x_2 )}\right|
=\left|
\begin{bmatrix}
1 & 0
\\
0& 1 
\end{bmatrix} 
\right|=1.
\end{split}\Ee
The image of the map can cover $\mathbb{T}^2$ at most $N^2$-times. By the change of variables of (\ref{COV_bdry}) we derive that 
\Be\label{est2:decomp2:L2L3}
\begin{split}
(\ref{decomp2:L2L3}) \lesssim _{N,T} (\ref{est1:decomp2:L2L3}) &\lesssim_{N,T} 
\bigg[ \int_{
\substack{
|v|\leq N , \\ |v_3|\geq1/N}
} 
\int_0^{ T }  \int_{\p\O}  |
 f_R(t,y, v) |^p   \dd S_y \dd t
  \dd v \bigg]^{1/p} 
  \lesssim_{N,T} \bigg[ 
\int_0^{T }   \int_{\p\O} \int_{
\substack{
|v|\leq N , \\ |v_3|\geq1/N}
} 
 |
 f_R(t,y, v) |^p  |v_3|   \dd v  \dd S_y \dd t
\bigg]^{1/p}.
  \end{split}
\Ee
Finally we utilize an $L^\infty$-bound of $f_R$ to conclude 
\Be\label{est3:decomp2:L2L3}
(\ref{decomp2:L2L3}) \lesssim _{N,T} (\ref{est2:decomp2:L2L3}) \lesssim _{N,T} \sup_{0 \leq t \leq T }\big[ \e \|  \mathfrak{w}_{\varrho, \ss} f_R \|_\infty\big]^{\frac{p-2}{p}}
| \e^{-1/2} f_R| _{L^2_t((  0,T );L^2  (\gamma_+))}^{\frac{2}{p}}. 
\Ee
\hide
By using the same map (\ref{COV_bdry}) and following the argument to get (\ref{est3:decomp2:L2L3}), we derive that 
\Be\label{est:decomp4:L2L3}
(\ref{decomp4:L2L3})\lesssim_{N,T}  \sup_{0 \leq t \leq T+1 }\big[ \e \|  \mathfrak{w}_{\varrho, \ss} f_R \|_\infty\big]^{1/3}
| \e^{-1/2} f_R| _{L^2_t( [0,T+1])L^2_{\gamma_+}}^{2/3}. 
\Ee
\unhide

For (\ref{decomp3:L2L3}), using the Minkowski's inequality then $x- \frac{t}{\e} v \mapsto v$ and (\ref{ext_f0}), it follows 
\Be\begin{split}\label{est:decomp3:L2L3}
(\ref{decomp3:L2L3}) &\lesssim_{N,T}  \left\|\int_{\R^3}  
 \chi  (v) \tilde{\varphi}_i (v) \sqrt{\mu_0(v)}  
\Big\| f_R(0
 , x- \frac{t}{\e} v
 , v) \Big\|_{L^p_x (\O^c)}   \dd v
 \right\|_{L^2_t(0,T ) }\\
 &\lesssim_{N,T}
 \left\|\int_{\R^3} 
 \chi (v) \tilde{\varphi}_i (v) \sqrt{\mu_0(v)}  
 \| f_R(0
 , \cdot 
 , v)  \|_{L^p_x (\O)}   \dd v
 \right\|_{L^2_t(0,T) }
 \lesssim_{N,T}   \| f_R(0 )  \|_{L^p_{x,v}    } . 
\end{split}\Ee

Finally from (\ref{decomp1:L2L3})-(\ref{decomp3:L2L3}), (\ref{est:decomp1:L2L3}), (\ref{est1:decomp2:L2L3}), (\ref{est3:decomp2:L2L3}), and (\ref{est:decomp3:L2L3}) we conclude that, for $p<3$, 
\Be\label{est:decom:L2L3}
\begin{split}
&\big(1-O( \e) \|u\|_\infty- O(\frac{1}{N})\big)
 \big\|    {P} f_R
 \big\|_{L^2_t ((0,T); L^p_x (\O ))}\\
 \lesssim_{T,N} &\left\|\int_{\R^3} \bar{f}_R(t,x,v) \tilde{\varphi}_i (v) \sqrt{\mu_0(v)}    \dd v \right\|_{L^2_t ((0,T); L^p_x (\mathbb{T}^2 \times \R ))} +   \sup_{0 \leq t \leq T } \| \mathfrak{w}_{\varrho, \ss } f_R (t) \|_\infty^{\frac{p-2}{p}} \| (\mathbf{I} - \mathbf{P}) f_R \|_{L^2((0, T) \times \O \times \R^3)}^{\frac{2}{p}}\\
 &+ \sup_{0 \leq t \leq T  }\big[ \e \|  \mathfrak{w}_{\varrho, \ss} f_R \|_\infty\big]^{\frac{p-2}{p}}
| \e^{-1/2} f_R| _{L^2_t((  0,T );L^2  (\gamma_+))}^{\frac{2}{p}} +  \| f_R(0 )  \|_{L^p_{x,v}    }. 
\end{split}
\Ee
\unhide


\subsection{$L^\infty$-estimate} In this section we develop a unified $L^\infty$-estimate in the local Maxwellian setting. We devise the weight functions to control an extra growth in $|v|$ comes from $\frac{(\p_t + \e^{-1} v\cdot \nabla_x) \sqrt{\mu}}{\sqrt{\mu}}$ and its temporal derivative:
\Be\notag
\mathfrak{w}_{\varrho, \ss}(x,v)=\mathfrak{w} := \exp\{\varrho |v|^2 - 
\mathfrak{z}_{\ss}(x_3) (x \cdot v) 
\} \ \ \text{for} \ \ 0< \ss \ll \frac{\varrho}{2\pi} \ \text{and} \ 0 < \varrho < \frac{1}{4},
\Ee 
where $\mathfrak{z}_{\ss}: \R_+ \rightarrow \R_+$ is defined as, for $\ss>0$ 
\Be\notag\label{z_ss}
\begin{split}
\mathfrak{z}_{\ss}(x_3)= \ss   \ \ \text{for} \ \  x_3 \in \big[0, \frac{1}{\ss}-1 \big],  \ \ \text{and} \ \ 
\mathfrak{z}_{\ss}(x_3)=\frac{1}{1+ x_3}  \  \ \text{for} \ \  x_3\in \big[\frac{1}{\ss}-1 , \infty\big).
\end{split}
\Ee 
\hide

Recall a weight function (\ref{weight}). We use the weight function   Let us denote $\mathfrak{z}_{\ss}: \R_+ \rightarrow \R_+$ defined as, for $\ss>0$ 
\Be\label{z_ss}
\begin{split}
\mathfrak{z}_{\ss}(x_3)= \ss   \ \ \text{for} \ \  x_3 \in \big[0, \frac{1}{\ss}-1 \big],  \ \ \text{and} \ \ 
\mathfrak{z}_{\ss}(x_3)=\frac{1}{1+ x_3}  \  \ \text{for} \ \  x_3\in \big[\frac{1}{\ss}-1 , \infty\big).
\end{split}
\Ee 
We define a weight function as 
\Be\label{weight}
\mathfrak{w}=\mathfrak{w}_{\varrho, \ss}(x,v) := \exp\{\varrho |v|^2 - 
\mathfrak{z}_{\ss}(x_3) (x \cdot v) 
\} \ \ \text{for} \ \ 0< \ss \ll \frac{\varrho}{2\pi} \ \text{and} \ 0 < \varrho < \frac{1}{4}. 
\Ee \unhide
We often abuse the notation of $\mathfrak{w}_{\varrho, \ss}$ and $\mathfrak{w}$. We compute to have 
\begin{align*}
&\frac{v\cdot \nabla_x \mathfrak{w}_{\varrho, \ss}(x,v) }{\mathfrak{w}_{\varrho, \ss}(x,v)}\\
&=   - \mathfrak{z}_{\ss} (x_3) |v |^2  -  v_3 \p_{x_3}\mathfrak{z}_{\ss}(x_3) 
( x_1 v_1+ x_2 v_2 + x_3 v_3 )    
\\
 &=
  - \mathfrak{z}_{\ss} (x_3) |v_3|^2  -  x_3 \p_{x_3}\mathfrak{z}_{\ss}(x_3)   |v_3|^2 
- \mathfrak{z}_{\ss} (x_3) (|v_1|^2 + |v_2|^2)  
- \p_{x_3} \mathfrak{z}_{\ss} (x_3)(x_1v_1 + x_2 v_2) v_3 
  \\
&= -\ss \mathbf{1}_{[0, \ss ^{-1}-1]} (x_3) |v |^2 
-  \mathbf{1}_{[ \ss ^{-1}-1, \infty)} (x_3) (1+ x_3)^{-2}|v_3|^2  
- \mathbf{1}_{[ \ss ^{-1}-1, \infty)} (x_3)  \frac{1}{1+ x_3}
(|v_1|^2 + |v_2|^2)  \\
&  \ \ \ - \p_{x_3} \mathfrak{z}_{\ss} (x_3)(x_1v_1 + x_2 v_2) v_3  ,
\end{align*}
where we have used $\p_{x_3} \mathfrak{z}_{\ss} (x_3)= \mathbf{1}_{[ \ss ^{-1}-1,\infty)} (x_3) \frac{-1}{(1+ x_3)^2}$. The last term, the sole term without a sign, can be bounded as 
\begin{align*}
&|- \p_{x_3} \mathfrak{z}_{\ss} (x_3)(x_1v_1 + x_2 v_2) v_3 |\\ \leq& \ 
2\sqrt{2 } \pi \mathbf{1}_{[ \ss ^{-1}-1,\infty)} (x_3)  (1+ x_3)^{-2} (|v_1|^2 + |v_2|^2)^{1/2}  |v_3|  \\
\leq & \   4 \pi^2 \mathbf{1}_{[ \ss ^{-1}-1,\infty)} (x_3)  (1+ x_3)^{-2}
(|v_1|^2 + |v_2|^2) 
+  \frac{1}{2}\mathbf{1}_{[ \ss ^{-1}-1,\infty)} (x_3)  (1+ x_3)^{-2}
|v_3|^2 .
\end{align*}
Therefore we conclude that 
\begin{equation}\label{est:vw_x}
\begin{split}
- v\cdot \nabla_x \mathfrak{w}_{\varrho, \ss}(x,v)  
\geq& \  \Big\{ \ss \mathbf{1}_{[0, \ss ^{-1}-1]} (x_3) |v |^2  
+  \frac{1}{2} \mathbf{1}_{[ \ss ^{-1}-1, \infty)} (x_3) (1+ x_3)^{-2}|v_3|^2  
 \\
& \ \ +(1- 4 \pi^2 \ss) \mathbf{1}_{[\ss^{-1}-1, \infty)} (x_3) \frac{1}{1+ x_3} (|v_1|^2 + |v_2|^2) 
\Big\} \mathfrak{w}_{\varrho, \ss}(x,v)
 \\
\geq & \ \frac{\mathfrak{z}_{\ss} (x_3)}{2}  |v|^2 \mathfrak{w}_{\varrho, \ss}(x,v).
\end{split}
\end{equation}


\hide
From the choice of parameters and the fact $(x_1,x_2) \in (-\pi,\pi)^2$, it follows that $e^{\varrho|v|^2/2} \leq \mathfrak{w}_{\varrho, \ss}(x,v) \leq e^{ 2\varrho|v|^2}$. This particular form of the weight function (\ref{weight}) is designed 
\unhide

We consider 
\Be\label{h}
h(t,x,v) = \mathfrak{w}_{\varrho, \ss}(x,v) f_R(t,x,v).
\Ee
An equation for $h$ can be written from (\ref{eqtn_fR}) and (\ref{bdry_fR}) as 
\Be \label{eqtn:h}
\p_t  h+ \frac{1}{\e} v\cdot \nabla_x h+  \frac{\nu_{\ss}  }{\e^2 \kappa}  
h =  \frac{1}{\e^2 \kappa}K_{\mathfrak{w}} h  + \mathcal{S}_h
,
\Ee 
 \Be\label{bdry:h}
 h|_{\gamma_-} = \mathfrak{w} P_{\gamma_+} \Big(\frac{h}{\mathfrak{w}}\Big)+ r.
 \Ee
 For (\ref{h}), we have $r=- \frac{\e}{\delta} \mathfrak{w} (1- P_{\gamma_+}) f_2$ and $\mathcal{S}_h: = \frac{\delta}{\kappa \e}  \Gamma_{ {\mathfrak{w} }}( {h} ,  {h} )
+\frac{2}{\kappa} \Gamma_{\mathfrak{w}} ( \mathfrak{w} f_2,h) + 
\mathfrak{w} (\mathbf{I} - \mathbf{P}) \mathfrak{R}_1 +\mathfrak{w}  \mathbf{R}_2,
$ and 
\Be\label{nu_ss}
\nu_{\ss}    :=
 \nu(v)
 - \e \kappa \frac{v \cdot \nabla_x  \mathfrak{w}_{\varrho, \ss}}{ \mathfrak{w}_{\varrho, \ss}} +  \e^2 \kappa \frac{(\p_t + \e^{-1} v\cdot \nabla_x) \sqrt{\mu}}{\sqrt{\mu}} 
,
  \Ee
where we denote $\Gamma_{\mathfrak{w}}(\cdot, \cdot )(v) := \mathfrak{w} (v)\Gamma(\frac{\cdot}{\mathfrak{w}}, \frac{\cdot}{\mathfrak{w}})(v)$ and $K_{\mathfrak{w}}(\cdot ):= \mathfrak{w} K( \frac{\cdot}{\mathfrak{w}})$. 
 
If we have
\Be\label{decay:u}
\e^{5/2} \kappa |\p_t u| +  
\e^{1/2}\sup_{x \in \O}(1+ x_3)  |\nabla_x u(t,x)|  <\infty,
\Ee
then for sufficiently small $\e, \kappa>0$, from (\ref{est:vw_x}), 
\Be
\nu_{\ss}  \geq \nu(v) +\frac{ \e \kappa}{2}   \mathfrak{z}_{\ss} (x_3) |v|^2 
 - \e^2\kappa \{
 \e |\p_t u|  + |\nabla_x u| |v| 
 \}|v-\e u| \geq \frac{\nu(v) }{2} + \frac{ \e \kappa}{4}   \mathfrak{z}_{\ss} (x_3) |v|^2 .
 \label{lower:nu_ss}
 \Ee

 From (\ref{L_Gamma}), (\ref{Gamma}), and (\ref{weight})
\Be\label{est:Gamma_w}
\begin{split}
 & |\mathfrak{w}(v) \Gamma( \frac{h}{\mathfrak{w}}, \frac{h}{\mathfrak{w}}) (v)| \\
&\leq   \iint_{\R^3 \times \S^2} 
|(v-v_*) \cdot \mathfrak{u}| 
\sqrt{ \mu(v_* )} e^{- \varrho |v_*|^2+ \frac{\varrho}{2 } |v_*|}
\big\{
|h( v ^\prime)|| h( v_*^\prime)|
+| h( v  )||h( v_* )|
\big\}
\dd \mathfrak{u} \dd v_* \\
&\lesssim_\varrho \nu(v) \|h\|_{L^\infty_v}^2.
\end{split}
\Ee
 From (\ref{est:k}) clearly we have 
 \Be\label{k_w}
\mathbf{k}(v,v_*) \frac{\mathfrak{w}_{\varrho , \ss }(v)}{\mathfrak{w}_{\varrho , \ss }(v
 _*)} \leq   \mathbf{k}_{\mathfrak{w}} (v,v_*): =  \frac{2C_{2}}{|v-v_*|}  e^{- \frac{|v-v_*|^2}{8}
- \frac{1}{8} \frac{(|v-\e u|^2 - |v_*- \e u|^2)^2}{|v-v_*|^2}}
 \frac{\mathfrak{w}_{\varrho , \ss }(v)}{\mathfrak{w}_{\varrho , \ss }(v
 _*)} .
 \Ee
As in (\ref{est:int_k}) we derive  
 \Be\label{est:K_w}
 \int_{\R^3}
 \mathbf{k}_{\mathfrak{w}} (v,v_*)
  \dd v_* \lesssim \frac{1}{1+ |v|}.
 \Ee

 \begin{proposition}\label{est:Linfty}Recall $\mathfrak{w}_{\varrho,\ss}$ in (\ref{weight}).
Assume the same assumptions in Proposition \ref{prop:Hilbert}. In addition we assume (\ref{decay:u}), and the conditions of $\varrho$ and $\ss$ in (\ref{weight}). Then  
 \Be\label{Linfty_3D}
 \begin{split}
&d_\infty
 \| \mathfrak{w}_{\varrho, \ss} f _R  \|_ {L^\infty_{t,x,v}}  \\
 \lesssim &  \ 
  \|  \mathfrak{w}_{\varrho, \ss} f (0)\|_ {L^\infty_{ x,v}}   
 +\frac{\e}{\delta}
 \| (\ref{est:f2}) \|_{L^\infty_{t,x}}
 + 
 \e^2 \kappa( \| (\ref{est:R1})\|_{L^\infty_{t,x}}+ \| (\ref{est:R2})\|_{L^\infty_{t,x}})
\\
&+ \frac{1}{\e^{1/2} \kappa^{1/2}}   \|  P  f_R  \|_{L^\infty_tL^6_
{x}} + \frac{1}{\e^{3/2} \kappa^{3/2}}\Big\{ \|  \sqrt{\nu}(\mathbf{I} -\mathbf{P}) f_R  \|_{L^2_{t,x,v}}
+  \| \sqrt{\nu}(\mathbf{I} -\mathbf{P}) \p_t f_R  \|_{L^2_{t,x,v}}
\Big\}\\
&
{
+\frac{1}{\e^{1/2}\kappa^{3/2}} \| \p_t u \|_{L^\infty_{t,x}} \| Pf_R \|_{L^2_{t,x}}
}
,
 \end{split}\Ee%
 %
where 
\Be\label{dinfty}
d_{\infty}:= 1-  \e^2
\| (\ref{est:f2}) \|_{L^\infty_{t,x}}  - 
  \e\delta 
 \| \mathfrak{w}_{\varrho, \ss} f_R    \|_{L^\infty_{t,x,v}}  
 .
\Ee
\end{proposition}

\begin{proposition}\label{prop:Linfty_t}
Assume the same assumptions of Proposition \ref{est:Linfty}. We denote 
\Be\label{w_prime}
\mathfrak{w}^\prime(x,v) : =   \mathfrak{w}_{\varrho^\prime, \ss}(x,v) \ \ \text{for} \ \varrho^\prime<\varrho.
\Ee
Let $p<3$. Then 
\Be\label{Linfty_3D_t}
\begin{split}
&d_{\infty, t}
\| \mathfrak{w} ^\prime  \p_t  f_R  \|_{L^2_t ((0,T);L^{\infty}_{x,v} (\O \times \R^3))}\\
\lesssim & \ 
\e\kappa^{1/2}
\| \mathfrak{w} ^\prime  \p_t  f_R(0 ) \|_{ L^{\infty}_{x,v}}+ 
 \frac{1}{\e^{3/p} \kappa^{3/p}} \| P \p_t f  \|_{L^2_t  L^p_{x} }
 +  \frac{1}{\e^{3/2} \kappa^{3/2} }
 \| \sqrt{\nu} ( \mathbf{I} - \mathbf{P}) \p_t f  \|_{L^2 _{t,x,v}}
\\
&+ \frac{\e}{\delta}
 \|(\ref{est:f2_t})\|_{ L^{\infty}_{x,v}} +  \frac{\e^2}{\delta} \|\p_t u\|_{ L^{\infty}_{x,v}} \|(\ref{est:f2}) \|_{ L^{\infty}_{x,v}}
  + \e^2 \kappa \|(\ref{est:R3})\|_{ L^{\infty}_{x,v}}+ \e^2 \kappa \|(\ref{est:R4})  \|_{ L^2_tL^{\infty}_{x}}\\
  & +\e \big(  \|\p_t u \|_{L^\infty_{t,x}}+ \e 
\|(\ref{est:f2_t})\|_{L^\infty_{t,x}}+
  \e  \kappa\| (\ref{transp:mu_t}) \|_{L^2_tL^\infty_{x}}
 \big)
 \| \mathfrak{w} f_R\|_{L^\infty_{t,x,v}}\\
  &
  +
\e     \Big(
\e\| (\ref{est:f2_t})\|_{L^\infty_{t,x}} + \e  \kappa  \|(\ref{transp:mu_t})\|_{L^2_tL^\infty_{x}} \\
& \ \ \ \ 
 + \|\p_t u\|_{L^\infty_{t,x}}  \big(1+\e^2\|(\ref{est:f2}) \|_{L^\infty_{t,x}}  +  \e \delta   \| \mathfrak{w} f_R \|_{L^\infty_{t,x,v}} \big)
 \Big) \| \mathfrak{w} f_R\|_{L^\infty_{t,x,v}}
,
\end{split}
\Ee
with 
\Be\label{dinftyt}
d_{\infty, t}: = 
1-
\e^2 \| (\ref{est:f2})\|_{L^\infty_{t,x}} - 
\e \delta \| \mathfrak{w} f_R\|_{L^\infty_{t,x,v}}.
\Ee
\end{proposition}

 In the proof of propositions, for simplicity, we often use $\| \ \cdot \ \|_\infty$ for $\| \ \cdot \ \|_{L^\infty_{t,x,v}}$, $\| \ \cdot \ \|_{L^\infty_{x,v}}$ or $\| \ \cdot \ \|_{L^\infty_{x}}$ if there would be no confusion. 

\begin{proof}[\textbf{Proof of Proposition \ref{est:Linfty}}]
 We define backward exit time and position as 
 \Be\label{tb}
 \tb(x,v) : = \e \frac{x_3}{v_3}, \ 
  \ \ 
  \xb(x,v) := x-   \frac{x_3}{v_3} v \ \ \text{for} \ \ (x,v) \in  \O  \times \R^3. 
 \Ee
 %
Since the characteristics for (\ref{eqtn:h}) are given by $(x- \frac{t-s}{\e}v, v)$, we have, for $0 \leq t-s <  \tb(x,v)$,   
 \Be\label{Duhamel}
 \frac{d}{ds}\Big\{  e^{-\int^t_s \frac{\nu_{\ss}  }{\e^2 \kappa}  } h(s,x- \frac{t-s}{\e}v, v )
\Big\}=e^{-\int^t_s \frac{\nu_{\ss}  }{\e^2 \kappa}  } 
\Big\{ \frac{1}{\e^2 \kappa}K_{\mathfrak{w}} h 
+\mathcal{S}_h
\Big\}(s,x- \frac{t-s}{\e}v, v ).
 \Ee
Here $e^{-\int^t_s \frac{\nu_{\ss}  }{\e^2 \kappa} }=e^{-\int^t_s \frac{\nu_{\ss} (\tau , x- \frac{t-\tau}{\e}v, v )}{\e^2 \kappa} \dd \tau}$. We regard $(x_1- \frac{t-s}{\e} v_1, x_2- \frac{t-s}{\e} v_2) \in \mathbb{R}^2$ belongs to $\mathbb{T}^2$ without redefining them in $[- \pi, \pi]^2$. 
 
 Now we represent $h$ using (\ref{Duhamel}) and (\ref{bdry:h}) as 
 \begin{align}
 h(t,x,v) =& \mathbf{1}_{t-\tb(x,v)<0}e^{-\int^t_0 \frac{\nu_{\ss}  }{\e^2 \kappa}  } h(0,x- \frac{t }{\e}v, v )
\notag
 \\
 &+ \int^t_{\max\{0, t-\tb(x,v)\}}
 e^{-\int^t_s \frac{\nu_{\ss}  }{\e^2 \kappa}  } 
 \frac{1}{\e^2 \kappa}K_{\mathfrak{w}} h   (s,x- \frac{t-s}{\e}v, v )
 \dd s
 \label{h:K}
  \\
  &+ \int^t_{\max\{0, t-\tb(x,v)\}}
 e^{-\int^t_s \frac{\nu_{\ss}  }{\e^2 \kappa}  } 
\mathcal{S}_h(s,x- \frac{t-s}{\e}v, v )
 \dd s
\notag
  \\
& + \mathbf{1}_{t-\tb(x,v)\geq 0}e^{-\int^t_{t-\tb(x,v)} \frac{\nu_{\ss}  }{\e^2 \kappa}  } 
 h(t-\tb(x,v),\xb(x,v), v )
.\label{h:bdry}
 \end{align}
 Since the integrand of (\ref{h:bdry}) reads on the boundary, using the boundary condition (\ref{bdry:h}) and (\ref{Duhamel}) again, we represent it as 
 \begin{align}
&
h(t-\tb(x,v),\xb(x,v), v )\notag
\\
=&
\mathfrak{w} (\xb(x,v), v )
c_\mu   \sqrt{\mu(v)}
 \int_{\mathfrak{v}_3<0} h(t-\tb(x,v),\xb(x,v), \mathfrak{v})  \frac{ \sqrt{\mu(\mathfrak{v})} |\mathfrak{v}_3| }{\mathfrak{w} (\xb(x,v), \mathfrak{v}) }\dd \mathfrak{v}
+ r(t-\tb(x,v),\xb(x,v),  v ) 
 \notag
\\
=&
\mathfrak{w} (\xb(x,v), v )
c_\mu   \sqrt{\mu(v)}
 \int_{\mathfrak{v}_3<0} 
 e^{
- \int^{t-\tb(x,v)}_0 \frac{\nu_{\ss}  }{\e^2 \kappa}  
} 
 h(0,\xb(x,v)- \frac{t-\tb(x,v)}{\e} \mathfrak{v}, \mathfrak{v})  \frac{ \sqrt{\mu(\mathfrak{v})} |\mathfrak{v}_3| }{\mathfrak{w} (\xb(x,v), \mathfrak{v}) }\dd \mathfrak{v}
\notag
\\
&+\mathfrak{w} (\xb(x,v), v )
c_\mu   \sqrt{\mu(v)}
 \int_{\mathfrak{v}_3<0}\notag \\
 & \ \ \ \times 
 \int^{t-\tb(x,v)}_0 
 e^{- \int^{t-\tb(x,v)}_s
 \frac{\nu_{\ss}}{\e^2 \kappa} 
 } \frac{1}{\e^2 \kappa} K_{\mathfrak{w}} h (s, \xb(x,v) -
 \frac{t-\tb(x,v)-s}{\e} \mathfrak{v}, \mathfrak{v}
  )
   \frac{ \sqrt{\mu(\mathfrak{v})} |\mathfrak{v}_3| }{\mathfrak{w} (\xb(x,v), \mathfrak{v}) }
   \dd s
   \dd \mathfrak{v}\label{bdry:K}\\
   &+\mathfrak{w} (\xb(x,v), v )
c_\mu   \sqrt{\mu(v)}
 \int_{\mathfrak{v}_3<0}\notag \\
 &  \ \ \ \ \ \ \ \ \ \ \ \times
 \int^{t-\tb(x,v)}_0 
 e^{- \int^{t-\tb(x,v)}_s
 \frac{\nu_{\ss}}{\e^2 \kappa} 
 }  \mathcal{S}_h (s, \xb(x,v) -
 \frac{t-\tb(x,v)-s}{\e} \mathfrak{v}, \mathfrak{v}
  )
   \frac{ \sqrt{\mu(\mathfrak{v})} |\mathfrak{v}_3| }{\mathfrak{w} (\xb(x,v), \mathfrak{v}) }
   \dd s 
   \dd \mathfrak{v}\notag
  \\&  +  r (t-\tb(x,v),\xb(x,v),  v )  ,\notag
 \end{align}
 where $r= - \frac{\e}{\delta} \mathfrak{w} (1- P_{\gamma_+}) f_2$ and $e^{- \int^{t-\tb(x,v)}_0 \frac{\nu_{\ss}  }{\e^2 \kappa}  
}:=  e^{-\int^{t-\tb(x,v)}_0 \frac{1}{\e^2 \kappa} 
{\nu_{\ss} (\tau , x- \frac{ \tb(x,v) }{\e}v
- \frac{t-\tb(x,v) -s }{\e} \mathfrak{v}
, \mathfrak{v} )}\dd \tau}$. 

Note that, from (\ref{est:R1}), (\ref{est:R2}), (\ref{est:I-Pf2}), (\ref{est:Pf2}), and (\ref{est:Gamma_w}), 
\Be\label{est:S_h}
\begin{split}
|\mathcal{S}_h(s, x- \frac{t-s}{\e} v,v)|&\lesssim   \nu(v)\frac{\delta}{\kappa \e}\| h  \|_\infty^2 + \frac{\nu(v)}{\kappa}
\|(\ref{est:f2})\|_\infty
\| h \|_\infty   +\| (\ref{est:R1})\|_\infty+  \| (\ref{est:R2})\|_\infty
 ,\\
| \mathfrak{w} (1- P_{\gamma_+}) f_2| & \lesssim   
\|(\ref{est:f2})\|_\infty.
  \end{split}
\Ee

We derive a preliminary estimate as 
\begin{align}
&|h(t,x,v) |
 \lesssim e^{- \frac{\nu }{2\e^2 \kappa}t} \| h(0)\|_\infty
\notag
\\
 &+  \e\delta  \sup_{0 \leq s \leq t}\| h(s)  \|_\infty^2
 + \e^2
 \sup_{0 \leq s \leq t} \|(\ref{est:f2})\|_\infty \| h(s) \|_\infty\notag\\
 &
 +\frac{\e}{\delta}  \sup_{0 \leq s \leq t} \|(\ref{est:f2})\|_\infty
  + 
 {\e^2 \kappa  } (
\| (\ref{est:R1})\|_\infty+  \| (\ref{est:R2})\|_\infty)  
 \label{infty1} 
 \\
&+  \int^t_0  \frac{e^{- \frac{\nu }{2\e^2 \kappa}(t-s)}}{\e^2 \kappa}
\int_{\R^3} \mathbf{k}_{\mathfrak{w}} (v,v_*) |h(s,x- \frac{t-s}{\e}, v_*)| \dd v_*
 \dd s
 \label{K1}
 \\
 &  +\mathfrak{w} (\xb(x,v), v )
c_\mu   \sqrt{\mu(v)}
 \int_{\mathfrak{v}_3<0} 
 \int^{t-\tb(x,v)}_0 
 \frac{ e^{-  
 \frac{\nu }{ 2\e^2 \kappa} (t-s)
 } }{\e^2 \kappa}\notag\\
 &  \ \ \ \ \ \times 
 \int_{\R^3}
  \mathbf{k}_{\mathfrak{w}}(\mathfrak{v} ,v_*) |h (s, \xb(x,v) -
 \frac{t-\tb(x,v)-s}{\e} \mathfrak{v}, v_*
  )| \dd v_* \dd s 
   \frac{ \sqrt{\mu(\mathfrak{v})} |\mathfrak{v}_3| }{\mathfrak{w} (\xb(x,v), \mathfrak{v}) }\dd \mathfrak{v}.\label{K2} 
\end{align}
We note that $|h(s,x- \frac{t-s}{\e}, v_*)|$ has the same upper bound. Then we bound (\ref{K1}) by a summation of  $(\ref{infty1})$ and
\Be 
\begin{split}\label{K2_1}
\sup_{
\substack{(\xb, v) \in \p\O \times \R^3  \\
t-\tb\geq 0
}}
\mathfrak{w} (\xb , v )
&c_\mu   \sqrt{\mu(v)}
 \int_{\mathfrak{v}_3<0} 
 \int^{t-\tb}_0 
 \frac{ e^{-  
 \frac{\nu }{ 2\e^2 \kappa} (t-s)
 } }{\e^2 \kappa}\\
 & \times 
 \int_{\R^3}
  \mathbf{k}_{\mathfrak{w}}(\mathfrak{v} ,v_*) |h (s, \xb  -
 \frac{t-\tb -s}{\e} \mathfrak{v}, v_*
  )| \dd v_* \dd s 
   \frac{ \sqrt{\mu(\mathfrak{v})} |\mathfrak{v}_3| }{\mathfrak{w} (\xb , \mathfrak{v}) }\dd \mathfrak{v},
\end{split}
\Ee
and importantly 
\Be
\begin{split}
\int^t_0 & \frac{e^{- \frac{\nu(v)}{2\e^2 \kappa}(t-s)}}{\e^2 \kappa}
\int_{\R^3} \mathbf{k}_{\mathfrak{w}} (v,v_*) 
\int^s_0
 \frac{e^{- \frac{\nu(v_*)}{2\e^2 \kappa}(s-\tau)}}{\e^2 \kappa}\\ &\times 
 \int_{\R^3}\mathbf{k}_{\mathfrak{w}} (v_*,v_{**}) 
 |h(s,x- \frac{t-s}{\e}v - \frac{s-\tau}{\e}v_* , v_{**})| 
 \dd v_{**}
\dd \tau
\dd v_*
 \dd s. \label{double_K}
\end{split} 
\Ee
We consider (\ref{double_K}). We decompose the integration of $\tau \in [0,s] = [0, s- o(1)\e^2 \kappa] \cup [s- o(1)\e^2 \kappa, s]$. The contribution of $\int^s_{s-o(1) \e^2 \kappa} \cdots \dd \tau$ is bounded as 
\Be\label{bound:small_t}
\frac{2}{\nu(v)}\big(1- e^{- \frac{\nu(v)}{2 \e^2 \kappa}}\big)
 \| \mathbf{k}_{\mathfrak{w}}(v , \cdot)\|_{L^1}
\frac{o(1) \e^2 \kappa}{\e^2 \kappa}
 \| \mathbf{k}_{\mathfrak{w}}(v_*, \cdot)\|_{L^1}
 \sup_{0 \leq s \leq t}\| h(s) \|_\infty\leq o(1) \sup_{0 \leq s \leq t}\| h(s) \|_\infty.
\Ee
For the rest of term we decompose $\mathbf{k}_{\mathfrak{w} }(v_*,v_{**})= \mathbf{k}_{\mathfrak{w},N}(v_*,v_{**}) + \{ \mathbf{k}_{\mathfrak{w} }(v_*,v_{**})-  \mathbf{k}_{\mathfrak{w},N}(v_*,v_{**})\}$ where $\mathbf{k}_{\mathfrak{w},N}(v_*,v_{**}):=\mathbf{k}_{\mathfrak{w}}(v_*,v_{**})$ $\times  
\mathbf{1}_{ \frac{1}{N} <|v_*- v_{**}|< N 
 \  \& 
  \ 
|v_*|< N  
}.
$ From (\ref{est:K_w}), \\ $\int_{\R^3} \mathbf{k}_{\mathfrak{w}}(v_*,v_{**}) \mathbf{1}_{|v_*|\geq N} \dd v_{**} \lesssim 1/N$. Also from the fact $\mathbf{k}_{\mathfrak{w}}(v_*,v_{**}) \leq \frac{e^{-C |v_*- v_{**}|^2}}{|v_*- v_{**}|} \in L^1(\{ v_*- v_{**} \in \R^3 \})$, $\sup_{v_*}\int_{\R^3} \mathbf{k}_{\mathfrak{w}}(v_*,v_{**}) \{\mathbf{1}_{\frac{1}{N}\geq |v_*- v_{**}| } + \mathbf{1}_{  |v_*- v_{**}|\geq N }  \}\dd v_{**} \downarrow 0$ as $N \rightarrow \infty$. Hence for $N\gg1$
\Be\begin{split}
(\ref{double_K})& \leq  \int^t_0  \frac{e^{- \frac{\nu(v)}{2\e^2 \kappa}(t-s)}}{\e^2 \kappa}
\int_{\R^3} \mathbf{k}_{\mathfrak{w} ,N} (v,v_*) 
\int^{s- o(1) \e^2 \kappa}_0
 \frac{e^{- \frac{\nu(v_*)}{2\e^2 \kappa}(s-\tau)}}{\e^2 \kappa}\\
 & \ \ \ \  \times 
 \int_{\R^3}\mathbf{k}_{\mathfrak{w},N } (v_*,v_{**}) 
 |h(s,x- \frac{t-s}{\e}v - \frac{s-\tau}{\e}v_* , v_{**})| 
 \dd v_{**}
\dd \tau
\dd v_*
 \dd s  \\
 &\leq C_N  \int^t_0  \frac{e^{- \frac{\nu(v)}{2\e^2 \kappa}(t-s)}}{\e^2 \kappa}
\int_{|v_*| \leq 2N}  
\int^{s- o(1) \e^2 \kappa}_0
  \frac{e^{- \frac{\nu(v_*)}{2\e^2 \kappa}(s-\tau)}}{\e^2 \kappa}
\\
 & \ \ \ \  \times  \int_{|v_{**}| < 2N} 
 |f_R(s,x- \frac{t-s}{\e}v - \frac{s-\tau}{\e}v_* , v_{**})| 
 \dd v_{**}
\dd \tau
\dd v_*
 \dd s
 \label{int:Kf}
 \\
 & \ \ \ + o(1) \sup_{0 \leq s \leq t} \| h(s)\|_{L^\infty_{x,v}}, 
\end{split}\Ee 
where we have used the fact $\sup_{x}\mathbf{k}_{\mathfrak{w} } (v_*, v_{**}) \mathfrak{w}_{\varrho, \ss}(v_{**}) \leq C_N<\infty$ when $\frac{1}{N} <|v_*- v_{**}|< N$ and $|v_*|< N$ (then $|v_{**}|< 2N$).

Now we decompose $f_R=\mathbf{P}f_R+ (\mathbf{I} -\mathbf{P})f_R$. We first take integrations (\ref{int:Kf}) over $v_{*}$ and $v_{**}$ and use Holder's inequality with $p=6, p=2$ in $1/p+ 1/p^\prime=1$ for $\mathbf{P}f_R , (\mathbf{I} -\mathbf{P})f_R$ respectively to derive 
\Be\label{int:Kf1}
\begin{split}
&(\ref{int:Kf})\\
 \leq& \  (4N)^3C_N
\frac{1}{\nu(v)}
\sup_{\substack{0 \leq s \leq t\\
0 \leq \tau \leq s- o(1) \e^2 \kappa
}}
\left(\iint_{|v_*|\leq N,|v_{**}| \leq 2 N}  |\mathbf{P}f_R(s, x- \frac{t-s}{\e} v - \frac{s-\tau}{\e} v_*, v_{**})|^6  \dd v_{**}\dd v_*\right)^{1/6}   \\
& + (4N)^3C_N
\frac{1}{\nu(v)}\\
& \  \times 
\sup_{\substack{0 \leq s \leq t\\
0 \leq \tau \leq s- o(1) \e^2 \kappa
}}
\left(\iint_{|v_*|\leq N,|v_{**}| \leq 2 N}  | ( \mathbf{I}-\mathbf{P})f_R(s, x- \frac{t-s}{\e} v - \frac{s-\tau}{\e} v_*, v_{**})|^2  \dd v_{**}\dd v_*\right)^{1/2}  .
\end{split}
\Ee
Now we consider a map
\Be\label{COV}
v_* \in \{\R^3: |v_*| \leq N\} \mapsto y:=x- \frac{t-s}{\e} v - \frac{s-\tau}{\e} v_* \in \O, \ \ \text{where} \ \ 
\Big|\frac{\p y}{\p v_*}\Big|= \Big|\frac{s-\tau}{\e}\Big|^3 \gtrsim \e^3 \kappa^3.
\Ee
We note that this mapping is not one-to-one and the image can cover $\O$ at most $N$ times. Therefore we have 
\begin{align*}
&\left(\iint_{|v_*|\leq N,|v_{**}| \leq N}  |\mathbf{P}f_R(s, x- \frac{t-s}{\e} v - \frac{s-\tau}{\e} v_*, v_{**})|^6  \dd v_{**}\dd v_*\right)^{1/6}\\
&\leq N^{1/6} \left(\iint_{|v_*|\leq N,|v_{**}| \leq N}  |\mathbf{P}f_R(s, y, v_{**})|^6  \dd v_{**}
\frac{\dd y}{\e^3 \kappa^3}\right)^{1/6} \leq \frac{N^{1/6}}{\e^{1/2} \kappa^{1/2}} \| \mathbf{P} f_R(s) \|_{L^6_{x,v}},
\end{align*}
\begin{align*}
&\left(\iint_{|v_*|\leq N,|v_{**}| \leq N}  |( \mathbf{I}-\mathbf{P})f_R(s, x- \frac{t-s}{\e} v - \frac{s-\tau}{\e} v_*, v_{**})|^6  \dd v_{**}\dd v_*\right)^{1/6}\\
&\leq \frac{N^{1/2}}{\e^{3/2} \kappa^{3/2}} \| ( \mathbf{I}-\mathbf{P})f_R(s) \|_{L^2_{x,v}}.
\end{align*}
Therefore we conclude that 
\Be\label{est:double_K}
\begin{split}
&(\ref{double_K})\\ &\leq (4N)^3C_N(\ref{int:Kf1}) + o(1)\sup_{0 \leq s \leq t} \| h(s)\|_{L^\infty_{x,v}} \\
&\leq (4N)^4C_N\left\{
\frac{1}{\e^{1/2}\kappa^{1/2}} \sup_{0 \leq s \leq t}\| \mathbf{P}  f_R(s ) \|_{L^6_{x,v}} 
+
  \frac{1}{\e^{3/2}\kappa^{3/2}}  \sup_{0 \leq s \leq t} \|(\mathbf{I} -\mathbf{P}) f_R(s ) \|_{L^2_{x,v}} \right\}\\
  &  \ \ \ \ + o(1)\sup_{0 \leq s \leq t} \| h(s)\|_{L^\infty_{x,v}} \\
  &\lesssim_N  
\frac{1}{\e^{1/2}\kappa^{1/2}} \sup_{0 \leq s \leq t}\| \mathbf{P}  f_R(s ) \|_{L^6_{x,v}} 
+
  \frac{1}{\e^{3/2}\kappa^{3/2}} 
\Big\{ \|(\mathbf{I} -\mathbf{P}) f_R  \|_{L^2_{t,x,v}}
+  \|(\mathbf{I} -\mathbf{P}) \p_t f_R  \|_{L^2_{t,x,v}}
\Big\} \\
& \ \ \ \  {
+ \frac{1}{\e^{1/2} \kappa^{3/2}} \| \p_t u \|_{L^\infty_{t,x}}\|  {P} f_R  \|_{L^2_{t,x}} 
} + o(1)\sup_{0 \leq s \leq t} \| h(s)\|_{L^\infty_{x,v}},
\end{split}
\Ee
where we have used (\ref{Sob_1D}) the Sobolev embedding in 1D at the last line.

Now we consider (\ref{K2}) and (\ref{K2_1}). We decompose $s \in [0, t-\tb] = [0, t-\tb - o(1) \e^2 \kappa] \cap [t-\tb - o(1) \e^2 \kappa, t-\tb]$. The contribution of $\int^{t-\tb}_{t-\tb- o(1) \e^2 \kappa} \cdots $ is bounded as 
\Be
\frac{o(1) \e^2 \kappa}{\e^2 \kappa } \| \mathbf{k}_{\mathfrak{w}}(\mathfrak{v}, \cdot ) \|_{L^1} \sup_{0 \leq s \leq t} \| h(s) \|_\infty
\leq o(1)\sup_{0 \leq s \leq t} \| h(s) \|_\infty. \label{bound:small_t1}
\Ee
For $s \in  [0, t-\tb - o(1) \e^2 \kappa]$ we consider a map as (\ref{COV})
\Be\label{COV_1}
\mathfrak{v}  \in \{\mathfrak{v}  \in \R^3: \mathfrak{v} _3<0 \} \mapsto y:=\xb - \frac{t-\tb -s}{\e} \mathfrak{v} \in \O, \ \  \text{where} \ \ 
\left|\frac{\p y}{\p \mathfrak{v}}\right| = \left|\frac{t-\tb -s}{\e}\right|^3\gtrsim \e^3 \kappa^3. 
\Ee
Following the argument to have (\ref{int:Kf1}) we bound 
\Be\begin{split}
&\text{the contribution of} \ \int^{t-\tb- o(1) \e^2 \kappa}_0 \cdots \ \text{of } \ (\ref{K2_1})\\
& 
\lesssim_N \frac{1}{\e^{1/2} \kappa^{1/2}} \| \mathbf{P} f_R(s) \|_{L^6_
{x,v}} + \frac{1}{\e^{3/2} \kappa^{3/2}} \Big\{ \|(\mathbf{I} -\mathbf{P}) f_R  \|_{L^2_{t,x,v}}\\
& \ \ \ \ \ \ \ 
+  \|(\mathbf{I} -\mathbf{P}) \p_t f_R  \|_{L^2_{t,x,v}}
\Big\}   {
+ \frac{1}{\e^{1/2} \kappa^{3/2}} \| \p_t u \|_{L^\infty_{t,x}}\|  {P} f_R  \|_{L^2_{t,x}} 
}  . \label{est:K_bdry}
\end{split}\Ee

In conclusion, we bound $|h(t,x,v)|$ by $(\ref{infty1})$, (\ref{est:double_K}), (\ref{bound:small_t1}), (\ref{est:K_bdry}) and conclude (\ref{Linfty_3D}) by choosing small enough $o(1)$ in (\ref{est:double_K}) and (\ref{bound:small_t1}). 
\end{proof}
\hide

 From (\ref{est:k}) and (\ref{est_infty:L_t}) it follows easily that  
 \Be\label{eqtn:h}
|  \p_t h +  \e^{-1} v\cdot \nabla_x h +\frac{\nu}{2\e^2 \kappa }  h | \leq \frac{1}{\e^2 \kappa }   \int_{\R^3}
k_{\vartheta, \varrho} (v,v_*)
|h(v_*)| \dd v_* + g, \ \ \ |h||_{\gamma_-} \leq  \mathfrak{w}P_{\gamma_+} | \mathfrak{w}^{-1}h| + r. 
 \Ee
 where $k_{\vartheta, \varrho}(v,v_*)$ has been defined in (\ref{k_the_rho}).
 
  \Be\notag
 K_\varsigma h:= \int_{\R^3} k_\varsigma (v,v_*)h(v_*) \dd v_*, \ \ \ where \ \   k_\varsigma (v,v_*):= C \frac{e^{- \varsigma |v-v_*|^2}}{|v-v_*|}.
 \Ee

 For some later purpose we set a notation 
\Be\label{k_the_rho}
k_{\vartheta, \varrho } (v,v_*) : = e^{- \vartheta{|v-v_*|^2} 
-  \vartheta\frac{(|v|^2 - |v_*|^2)^2}{|v-v_*|^2}}
\frac{e^{\varrho|v|^2}}{e^{\varrho |v_*|^2}} . 
\Ee

  The (scaled) characteristics $[Y,W]$ are given by 
 \Be\label{trajectory_e}\Bs
Y(s;t,x,v) &= X(t- \frac{t-s}{\e};t,x,v),\\
W(s;t,x,v) &= V(t- \frac{t-s}{\e};t,x,v).
\end{split} \Ee
Along the characteristics,
\Be\Bs\label{Duhamel_linear}
&\frac{d}{ds}\Big\{ h(s , Y(s;t,x,v), W(s;t,x,v))
e^{- \frac{\nu(v)}{\e^2 \kappa} (t-s) }
\Big\}\\
=& \  \frac{1}{\e^2 \kappa}K h (s , Y(s;t,x,v), W(s;t,x,v))e^{-\frac{\nu(v)}{\e^2 \kappa}  (t-s) }.
\end{split}\Ee
Then 
\Be\Bs\notag
h(t,x,v) =& \ h_0 (Y(0;t,x,v), W(0;t,x,v)) e^{-  
 \frac{\nu(v)}{\e^2 \kappa}  t  }
 \\
&+ \int^t_0 
\frac{ e^{
-  \frac{\nu(v)}{\e^2 \kappa} (t-s) }}{\e^2 \kappa }
 \int_{\R^3} k_w(W(s;t,x,v),u )
  h (s , Y(s;t,x,v), u) \dd u
 \dd s\\
 & +\cdots 
\end{split}\Ee
Without the boundary we would have 
\Bes
&&h(t,x,v) \\
&\sim&  
\int^t_0 
\frac{e^{-  \frac{\nu_0 (t-s)}{\e^{2 } \kappa} }}{\e^{2 } \kappa}
 \int_{|u| \leq N } 
 \int^s_0 \frac{e^{-  \frac{\nu_0 (s-s^\prime)}{\e^{2 } \kappa} }}{\e^{2 } \kappa} \int_{\R^3}
k_w ( \cdots , u^\prime) h(s^\prime,x- \frac{t-s}{\e}v - \frac{s-s^\prime}{\e}u,u^\prime)
\dd u^\prime \dd  s^\prime
  \dd u
 \dd s\\
 &\sim&\int \int \int^s_{s- o(1) \e^{2 } \kappa} \int + \int \int \int^{s- o(1) \e^{2 } \kappa }_0 \int\\
 &\sim& O(1)\| h\|_\infty+ 
 \int^t_0 
\frac{e^{-  \frac{\nu_0 (t-s)}{\e^{2 } \kappa} }}{\e^{2 } \kappa}\int^{s-o(1) \e^{2 } \kappa}_0 \frac{e^{-  \frac{\nu_0 (s-s^\prime)}{\e^{2 } \kappa} }}{\e^{2 } \kappa} 
\left\| h(s^\prime,x- \frac{t-s}{\e}v - \frac{s-s^\prime}{\e}u,u^\prime)\right\|_{L^p_{u,u^\prime}
(\{|u| \leq N \} \times \R^3)
}
 \dd s^\prime \dd s .
\Ees

Here the dimension enters the business. Assume that the function only depends on $(t, {x},v) \in \R \times \R^3 \times \R^3$. 
Now we consider a change of variables, for $s^\prime \in [0, s- \kappa \e^{2 } \kappa]$,
\Be
 {u} 
\mapsto  {X}: =  {x}- \frac{t-s}{\e} {v} - \frac{s-s^\prime}{\e} {u}.
\Ee
We compute 
\Be
\dd  {u} = \frac{1}{\left|\frac{s-s^\prime}{\e}\right|^3}\dd {X} \lesssim
\frac{1}{\e^3 \kappa^3 }\dd \munderbar{X}.
\Ee
Then 
\Bes
&&\left\| h(s^\prime, {x}- \frac{t-s}{\e} {v} - \frac{s-s^\prime}{\e} {u},u^\prime)\right\|_{L^p_{u,u^\prime}}\\
&\lesssim& \frac{1}{ \e^{3/p} \kappa^{3/p}} \| h(s^\prime, \cdot, \cdot) \|_{L^p}.
\Ees
Therefore for $(x,v) \in \{ |\nabla \xi(x) \cdot v|> o(1) \} \cup \{|v| \leq \frac{1}{o(1)}\}$ we have $C=C(\O, o(1))>0$ such that 
\Be
|h(t,x,v)| \leq
e^{-\f{\nu(v) t}{\e^{2 }\kappa }} |h_0 (Y(0;t,x,v), W(0;t,x,v))|
+
\frac{ C}{ \e^{3/p} \kappa^{3/p}} 
 \int^t_0 
\frac{e^{-  \frac{\nu_0 (t-s)}{\e^{2 }\kappa } }}{\e^{2 }\kappa}
\| h(s) \|_p 
 \dd s 
 .
\Ee
Roughly we will have 
\Be
\| h \|_\infty \lesssim \frac{\e \kappa^{1/2}}{(\e \kappa)^{3/2}} \| \e^{-1} \kappa^{1/2}(\mathbf{I} - \mathbf{P}) f_R \|_2 + \frac{\kappa^{-1/2}}{(\e \kappa)^{3/6}} \| \kappa^{1/2} \mathbf{P}f_R \|_{L^6} \lesssim \frac{1}{\e^{1/2} \kappa}. 
\Ee

\hide

Then, for each time interval $[N, N+1]$, we obtain that 
\Be
\| h(N+1) \|_\infty \leq e^{- \frac{\nu_0}{\e^{2+\kappa}}} \| h(N) \|_\infty + \frac{C}{ \e^{2(1+ \kappa)/p}}
\int^{N+1}_N \frac{e^{- \frac{\nu_0}{\e^{2+ \kappa}}(N+1-s) }}{\e^{2+ \kappa}} \| h(s) \|_p \dd s + C,
\Ee
where $C>0$ does not depend on $N$.

Then inductively we derive 
\Be
\Bs
&\| h(N+1) \|_\infty \\
\leq & \ e^{- \frac{\nu_0 2}{\e^{2+\kappa}}} \| h(N-1) \|_\infty+ \frac{C}{ \e^{2(1+ \kappa)/p}} 
\int^{N+1}_{N-1} \frac{e^{- \frac{\nu_0}{\e^{2+ \kappa}}(N+1-s) }}{\e^{2+ \kappa}} \| h(s) \|_p \dd s 
+ C \sum_{j=0}^1 e^{- \frac{\nu_0}{\e^{2+\kappa}}j}\\
&\cdots
\\
\leq &  \ e^{- \frac{\nu_0 (N+1)}{\e^{2+ \kappa}}} \| h(0) \|_\infty + \f{C}{ \e^{2(1+ \kappa)/p}}
\int^{N+1}_0  \frac{e^{- \frac{\nu_0}{\e^{2+ \kappa}}(N+1-s) }}{\e^{2+ \kappa}} \| h(s) \|_p \dd s 
+ \frac{C}{1- e^{- \frac{\nu_0}{\e^{2+\kappa}}}}.
\end{split}
\Ee

\unhide

\unhide

\begin{proof}[\textbf{Proof of Proposition \ref{prop:Linfty_t}}]
Since many parts of the proof are overlapped with the proof of Proposition \ref{est:Linfty} we only pin point the differences. An equation for $\mathfrak{w}^\prime \p_t f_R$ takes the similar form of (\ref{eqtn:h}) and (\ref{bdry:h}). We can read (\ref{eqtn_fR_t}) for 
\Be\label{h_t}
h(t,x,v) = \mathfrak{w} ^\prime(x,v) \p_t  f_R(t,x,v), \ \ \text{for} \ \varrho^\prime<\varrho,
\Ee
 as (\ref{eqtn:h}) and (\ref{bdry:h}) replacing  
\Be\label{S_h_t}
\begin{split}
\mathcal{S}_h & =  \frac{2}{\kappa} \Gamma_{\mathfrak{w}^\prime}( {\mathfrak{w}^\prime}{f_2},h) 
+    \frac{ 2 \delta }{\e\kappa}\Gamma_{\mathfrak{w}^\prime}( {\mathfrak{w}^\prime} f_R,h)+  \frac{2}{\kappa} \Gamma_{\mathfrak{w}^\prime}( {\mathfrak{w}^\prime}\p_t {f_2}, {\mathfrak{w}^\prime}f_R)  \\
&
 \ \   -
\p_t \Big( \frac{( \p_t +
  \e^{-1} v\cdot \nabla_x) \sqrt{\mu}}{\sqrt{\mu}} \Big)
  \frac{\mathfrak{w}^\prime}{\mathfrak{w} }
  \mathfrak{w} f_{R}  + {\mathfrak{w}^\prime} (\mathbf{I} -  \mathbf{P})\mathfrak{R}_3 + {\mathfrak{w}^\prime} \mathfrak{R}_4\\
 & \ \ 
    - \frac{1}{\e^2 \kappa } \mathfrak{w}^\prime L_t (\mathbf{I} - \mathbf{P}) f_R +\frac{1}{\e^2 \kappa } \mathfrak{w}^\prime  L(\mathbf{P}_t f_R) +  \frac{2}{\kappa} \mathfrak{w}^\prime  \Gamma_t(  {f_2}, f_R)+ \frac{\delta}{\e \kappa }  \mathfrak{w}^\prime \Gamma_t (f_R,f_R)
,
\\
 r&= - \frac{\e}{\delta} {\mathfrak{w}^\prime}  (1-P_{\gamma_+})  \p_tf_2 
+{\mathfrak{w}^\prime}  r_{\gamma_+} (f_R)-{\mathfrak{w}^\prime} \frac{\e}{\delta}r_{\gamma_+}  (f_2),
\end{split}
\Ee
where $r _{\gamma_+}(g) $ has been defined in (\ref{bdry_fR_t}).

We have the same equality of (\ref{h:K}), (\ref{h:bdry}) with (\ref{bdry:K}) for $h$ of (\ref{h_t}) but replacing $S_h$ and $r$ of (\ref{S_h_t}). Note that $\frac{\mathfrak{w}^\prime(x,v)}{\mathfrak{w} (x,v)}\lesssim e^{- (\varrho- \varrho^\prime)|v|^2}$ and hence $\Big|\p_t \Big( \frac{( \p_t +
  \e^{-1} v\cdot \nabla_x) \sqrt{\mu}}{\sqrt{\mu}} \Big)
  \frac{\mathfrak{w}^\prime}{\mathfrak{w} }\Big| \lesssim (\ref{transp:mu_t})
  $ from (\ref{transp:mu_t}). From (\ref{est_infty:L_t}), (\ref{est:R3}), (\ref{est:R4}), (\ref{est:I-Pf2}), (\ref{transp:mu}), (\ref{transp:mu_t}), (\ref{est:Pf2}), and (\ref{est:Gamma_w}), we bound terms of (\ref{S_h_t}) 
  \Be
\begin{split}
|S_h|& \lesssim     \nu(v)
\big\{\frac{1}{\kappa}
 | (\ref{est:f2}) |
 +\frac{\delta}{\kappa \e } \| \mathfrak{w} f_R\|_\infty 
 \big\}
 \| h\|_\infty 
  + (\ref{est:R3}) + (\ref{est:R4})\\
  & \ \ 
 + \Big(
 \frac{\nu(v)}{\kappa} (\ref{est:f2_t}) + (\ref{transp:mu_t})
 + |\p_t u| \big(\frac{1}{\e \kappa } + \frac{\e}{\kappa} (\ref{est:f2})  + \frac{\delta}{ \kappa }\| \mathfrak{w} f_R \|_\infty\big)
 \Big) \| \mathfrak{w} f_R \|_\infty
  ,
  \label{est:S_h_t1}
  \end{split}\Ee
\Be
\begin{split}
|r| \lesssim & \ \frac{\e}{\delta}
 (\ref{est:f2_t}) +  \frac{\e^2}{\delta} |\p_t u|  (\ref{est:f2}) 
+ \e  |\p_t u |\| \mathfrak{w} f_R\|_\infty
.\label{est:S_h_t2}
\end{split}\Ee 
  
  \hide
\Be
\begin{split}
|S_h| \lesssim& \   \nu(v)
\big\{\frac{1}{\kappa} (\| p \|_\infty + \| \tilde{u} \|_\infty + \kappa \| \nabla_x u \|_\infty) +\frac{\delta}{\kappa \e } \| \mathfrak{w} f_R\|_\infty 
 \big\}
 \| h\|_\infty\\
 & 
+ \frac{\nu(v)}{\kappa} \big\{
|\p_t p|  + |\p_t \tilde{u}| + \e |\p_t u| (|p| + |\tilde{u}|) + \kappa (|\nabla_x \p_t u| + \e |\p_t u| |\nabla_x u|)
\big\}\| \mathfrak{w} f_R\|_\infty\\
&+  \e (|\p_t^2 u| + |u||\nabla_x \p_t u| + |\p_t u| |\nabla_x u| )+ \e^2 |\p_t u| (|\p_t u| + |u||\nabla_x u|) \| \mathfrak{w} f_R\|_\infty +  \frac{1}{\delta}
\mathfrak{q}(
|\nabla_{x} \p_t \tilde{u}|, 
|\nabla_x^2 \p_t u | 
)\\
&
+  \frac{\e}{\delta \kappa}\{1+ 
|\p_t^2 p| + |\nabla_x \p_t^2 u| 
\} 
\\
& \ \ \   \ \ \  \times 
 \mathfrak{q} (
|p|, |\nabla_x p|, |\p_t p|,|\nabla_x \p_t p | , |u| , |\nabla_x u|, |\p_t u|, |\nabla_x \p_t u|, |\nabla_x^2 u|, |\nabla_x^2 \p_t u|, 
|\tilde{u}|, |\nabla_x \tilde{u}|,|\p_t \tilde{u}| ,|\nabla_x \p_t \tilde{u}|, |\p_t ^2 \tilde{u}|
) ,\\
|r| \lesssim & \ \frac{\e}{\delta}\{ |\p_t p| +|\p_t \tilde{u}| + \kappa |\nabla_x \p_t u |\}
+\e  |\p_t u | \{ \| \mathfrak{w} f_R\|_\infty+ 
 |  p| +|  \tilde{u}| + \kappa |\nabla_x   u |
\}.\label{est:S_h_t}
\end{split}
 \Ee
 \unhide
 
 Then as in  (\ref{infty1})-(\ref{double_K}) we derive a preliminary estimate as 
 \begin{align}
& |h(t,x,v) |
\notag
\\
& \lesssim    \ e^{- \frac{\nu }{2\e^2 \kappa}t} \| h(0)\|_\infty+ 
 \frac{\e^2 \kappa }{\nu(v)}
 (\ref{est:S_h_t1})
 + (\ref{est:S_h_t2})
 \label{infty1_t}
  \\
&+  \int^t_0  \frac{e^{- \frac{\nu }{2\e^2 \kappa}(t-s)}}{\e^2 \kappa}
\int_{\R^3} \mathbf{k}_{\mathfrak{w}^\prime} (v,v_*) |h(s,x- \frac{t-s}{\e}, v_*)| \dd v_*
 \dd s
 \label{K1_t}
 \\
 &  +\mathfrak{w}^\prime (\xb(x,v), v )
c_\mu   \sqrt{\mu(v)}
 \int_{\mathfrak{v}_3<0} 
 \int^{t-\tb(x,v)}_0 
 \frac{ e^{-  
 \frac{\nu }{ 2\e^2 \kappa} (t-s)
 } }{\e^2 \kappa}\notag\\
 & \ \ \ \ \     \times 
 \int_{\R^3}
  \mathbf{k}_{\mathfrak{w}^\prime}(\mathfrak{v} ,v_*) |h (s, \xb(x,v) -
 \frac{t-\tb(x,v)-s}{\e} \mathfrak{v}, v_*
  )| \dd v_* \dd s 
   \frac{ \sqrt{\mu(\mathfrak{v})} |\mathfrak{v}_3| }{\mathfrak{w}^\prime (\xb(x,v), \mathfrak{v}) }\dd \mathfrak{v}.\label{K2_t} 
\end{align} 
\hide\begin{align}
&|h(t,x,v) |\notag
\\
 \lesssim &  \ e^{- \frac{\nu }{2\e^2 \kappa}t} \| h(0)\|_\infty
 + \frac{\e}{\delta}\{ |\p_t p| +|\p_t \tilde{u}| + \kappa |\nabla_x \p_t u |\}
+\e  |\p_t u | \{ \| \mathfrak{w} f_R\|_\infty+ 
 |  p| +|  \tilde{u}| + \kappa |\nabla_x   u |
\}
\label{infty1_t}
\\
&
 + \e^2  \kappa \big\{ \frac{1}{\kappa}
  (\| p \|_\infty + \| \tilde{u} \|_\infty + \kappa \| \nabla_x u \|_\infty) +\frac{\delta}{\kappa \e } \| \mathfrak{w} f_R\|_\infty 
 \big\}
 \sup_{0 \leq s \leq t}\| h(s)  \|_\infty
 \label{infty2_t}
  \\
 & +  \e^2
  \big\{
|\p_t p|  + |\p_t \tilde{u}| + \e |\p_t u| (|p| + |\tilde{u}|) + \kappa (|\nabla_x \p_t u| + \e |\p_t u| |\nabla_x u|)
\big\}\| \mathfrak{w} f_R\|_\infty
\label{infty3_t}\\
&+  \e^3 \kappa  (|\p_t^2 u| + |u||\nabla_x \p_t u| + |\p_t u| |\nabla_x u| )+ \e^4 \kappa  |\p_t u| (|\p_t u| + |u||\nabla_x u|) \| \mathfrak{w} f_R\|_\infty +  \frac{\e^2 \kappa}{\delta}
\mathfrak{q}(
|\nabla_{x} \p_t \tilde{u}|, 
|\nabla_x^2 \p_t u | 
)
\label{infty4_t} \\
 &+
  \frac{\e^3}{\delta }\{1+ 
|\p_t^2 p| + |\nabla_x \p_t^2 u| 
\} \notag
\\
& \ \ \   \ \ \  \times 
 \mathfrak{q} (
|p|, |\nabla_x p|, |\p_t p|,|\nabla_x \p_t p | , |u| , |\nabla_x u|, |\p_t u|, |\nabla_x \p_t u|, |\nabla_x^2 u|, |\nabla_x^2 \p_t u|, 
|\tilde{u}|, |\nabla_x \tilde{u}|,|\p_t \tilde{u}| ,|\nabla_x \p_t \tilde{u}|, |\p_t ^2 \tilde{u}|
)\label{infty5_t}
 \\
&+  \int^t_0  \frac{e^{- \frac{\nu }{2\e^2 \kappa}(t-s)}}{\e^2 \kappa}
\int_{\R^3} \mathbf{k}_{\mathfrak{w}^\prime} (v,v_*) |h(s,x- \frac{t-s}{\e}, v_*)| \dd v_*
 \dd s
 \label{K1_t}
 \\
 &  +\mathfrak{w}^\prime (\xb(x,v), v )
c_\mu   \sqrt{\mu(v)}
 \int_{\mathfrak{v}_3<0} 
 \int^{t-\tb(x,v)}_0 
 \frac{ e^{-  
 \frac{\nu }{ 2\e^2 \kappa} (t-s)
 } }{\e^2 \kappa}
 \int_{\R^3}
  \mathbf{k}_{\mathfrak{w}^\prime}(\mathfrak{v} ,v_*) |h (s, \xb(x,v) -
 \frac{t-\tb(x,v)-s}{\e} \mathfrak{v}, v_*
  )| \dd v_* \dd s 
   \frac{ \sqrt{\mu(\mathfrak{v})} |\mathfrak{v}_3| }{\mathfrak{w}^\prime (\xb(x,v), \mathfrak{v}) }\dd \mathfrak{v}.\label{K2_t} 
\end{align} \unhide
As (\ref{K2_1}) and (\ref{double_K}), 
we bound (\ref{K1_t}) by a summation of  $(\ref{infty1_t})$ and 
\begin{align}
&\int^t_0  \frac{e^{- \frac{C_\nu
}{2\e^2 \kappa}(t-s)}}{\e^2 \kappa}
\int^{s-o(1) \e^2 \kappa}_0
 \frac{e^{- \frac{C_\nu
 }{2\e^2 \kappa}(s-\tau)}}{\e^2 \kappa}
\int_{|v_*| \leq 2N} 
\notag\\
& \ \ \ \ \ \times \int_{|v_{**} | \leq 2N}
 |h(s,x- \frac{t-s}{\e}v - \frac{s-\tau}{\e}v_* , v_{**})| 
 \dd v_{**}
\dd v_*
\dd \tau
 \dd s , \label{double_K_t}\\
&+\sup_{
\substack{(\xb, v) \in \p\O \times \R^3  \\
t-\tb\geq 0
}}
\mathfrak{w}^\prime (\xb , v )
c_\mu   \sqrt{\mu(v)}
 \int_{\mathfrak{v}_3<0} 
 \int^{t-\tb -o(1) \e^2 \kappa}_0 
 \frac{ e^{-  
 \frac{\nu }{ 2\e^2 \kappa} (t-s)
 } }{\e^2 \kappa}\notag\\
 &\ \ \ \ \ \  \ \ \ \ \ \  \ \ \ \ \ \ \ \ \ \  \times 
 \int_{|v_*| \leq 2N}
   |h (s, \xb  -
 \frac{t-\tb -s}{\e} \mathfrak{v}, v_*
  )| \dd v_* \dd s 
   \frac{ \sqrt{\mu(\mathfrak{v})} |\mathfrak{v}_3| }{\mathfrak{w}^\prime (\xb , \mathfrak{v}) }\dd \mathfrak{v}
   \label{K2_1_t} 
   \\
   &+ o(1) \sup_{0 \leq s \leq t} \| h(s) \|_{L^{\infty}_{x,v}}
   .\label{bound:small_t1_t}
\end{align} 
Then we follow the argument of (\ref{int:Kf1})-(\ref{est:double_K}) to derive that, for $p<3$,
\begin{align}
|(\ref{double_K_t})| \lesssim&\  \int^t_0  \frac{e^{- \frac{C_\nu
}{2\e^2 \kappa}(t-s)}}{\e^2 \kappa}
\int^{s-o(1) \e^2 \kappa}_0
 \frac{e^{- \frac{C_\nu
 }{2\e^2 \kappa}(s-\tau)}}{\e^2 \kappa}
 \frac{N^{1/3}}{\e^{3/p} \kappa^{3/p}}
 \| \mathbf{P} \p_t f (\tau) \|_{L^p_{x,v}}
 \dd \tau \dd s\label{est:Kf_t:1}
 \\
 &+ \int^t_0  \frac{e^{- \frac{C_\nu
}{2\e^2 \kappa}(t-s)}}{\e^2 \kappa}
\int^{s-o(1) \e^2 \kappa}_0
 \frac{e^{- \frac{C_\nu
 }{2\e^2 \kappa}(s-\tau)}}{\e^2 \kappa}
 \frac{N^{1/2}}{\e^{3/2} \kappa^{3/2}}
 \| \mathbf{P} \p_t f (\tau) \|_{L^2_{x,v}}
 \dd \tau \dd s.\label{est:Kf_t:2}
\end{align}
Now we use the Young's inequality for temporal convolution twice to derive that, for $p<3$,  
\Be
\begin{split}\label{est:double_K_t}
&\|(\ref{double_K_t})\|_{L^2_t(0,T)}\\
\lesssim& \ 
\bigg\|\frac{e^{- \frac{C_\nu
}{2\e^2 \kappa}|s |}}{\e^2 \kappa}  \bigg\|_{L^1_s(\R)}
\bigg\|\\
& \times 
\int^s_0
 \frac{e^{- \frac{C_\nu
 }{2\e^2 \kappa}(s-\tau)}}{\e^2 \kappa}
 \bigg(
 \frac{N^{1/3}}{\e^{3/p} \kappa^{3/p}}
 \| \mathbf{P} \p_t f (\tau) \|_{L^p_{x,v}}
 +  \frac{N^{1/2}}{\e^{3/2} \kappa^{3/2}}
 \| (\mathbf{I}- \mathbf{P}) \p_t f (\tau) \|_{L^2_{x,v}}
 \bigg)
 \dd \tau
\bigg\|_{L^2_s (\R)} \\
\lesssim & \ 
\bigg\|\frac{e^{- \frac{C_\nu
}{2\e^2 \kappa}|s |}}{\e^2 \kappa}  \bigg\|_{L^1_s(\R)}
\bigg\|
\frac{e^{- \frac{C_\nu
}{2\e^2 \kappa}|\tau |}}{\e^2 \kappa} \bigg\|_{L^1_\tau (\R)} \\
& \times
\bigg(
\frac{N^{1/3}}{\e^{3/p} \kappa^{3/p} }
 \| P \p_t f  \|_{L^2_t ((0,T);L^p_{x}(\O))}
 + \frac{N^{1/2}}{\e^{3/2} \kappa^{3/2} }
 \|  (\mathbf{I}- \mathbf{P}) \p_t f  \|_{L^2 ((0,T) \times \O\times \R^3)}
 \bigg)\\
 \lesssim_N  & \ 
 \frac{1}{\e^{3/p} \kappa^{3/p} }
 \| P \p_t f  \|_{L^2_t ((0,T);L^p_{x}(\O))}
 +  \frac{1}{\e^{3/2} \kappa^{3/2} }
 \| ( \mathbf{I} - \mathbf{P}) \p_t f  \|_{L^2 ((0,T) \times \O\times \R^3)}
.
\end{split}
\Ee 

As in (\ref{est:K_bdry}), for (\ref{K2_1_t}) we use (\ref{COV_1}) to derive that, for $p<3$,  
\Be\label{est:K_bdry_t}
\begin{split}
&\|(\ref{K2_1_t})\|_{L^2_t (0,T)}\\
 \lesssim & \ 
\bigg\|\int^t_0 \frac{e^{- \frac{C_\nu}{2\e^2\kappa} (t-s)}}{\e^2 \kappa}
\bigg(
\frac{1}{\e^{3/p} \kappa^{3/p} }
 \| \mathbf{P} \p_t f (s) \|_{L^p_{x,v}}
 +  \frac{1}{\e^{3/2} \kappa^{3/2}}
 \| (\mathbf{I}- \mathbf{P}) \p_t f (s) \|_{L^2_{x,v}}
\bigg)
\dd s\bigg\|_{L^2_t (0,T)}
\\
 \lesssim & \ 
 \bigg\|\frac{e^{- \frac{C_\nu
}{2\e^2 \kappa}|s |}}{\e^2 \kappa}  \bigg\|_{L^1_s(\R)}
\Big\{\frac{1}{\e^{3/p} \kappa^{3/p} }
 \| P \p_t f  \|_{L^2_t ((0,T);L^p_{x}(\O))}
 +  \frac{1}{\e^{3/2} \kappa^{3/2} }
 \| ( \mathbf{I} - \mathbf{P}) \p_t f  \|_{L^2 ((0,T) \times \O\times \R^3)}
\Big\}\\
\lesssim & \ 
\frac{1}{\e^{3/p} \kappa^{3/p} }
 \| P \p_t f  \|_{L^2_t ((0,T);L^p_{x}(\O))}
 +  \frac{1}{\e^{3/2} \kappa^{3/2} }
 \| ( \mathbf{I} - \mathbf{P}) \p_t f  \|_{L^2 ((0,T) \times \O\times \R^3)},
\end{split}
\Ee
where we have used the Young's inequality for temporal convolution. 

In conclusion, we bound $\|h\|_{L^2_t L^\infty_{x,v}}$ by $\|(\ref{infty1_t})\|_{L^2_t L^\infty_{x,v}},$ (\ref{est:double_K_t}), (\ref{bound:small_t1_t}), (\ref{est:K_bdry_t}) and conclude (\ref{Linfty_3D_t}) by choosing small enough $o(1)$ in (\ref{bound:small_t1_t}). 
\end{proof}

\subsection{Proof of Theorem \ref{main_theorem:conditional}}

An existence of a unique global solution $F$ for each $\e>0$ can be found in \cite{EGKM,EGKM2,EGKM3,EGM}. Thereby we only focus on the (a priori) estimates (\ref{est:E}). 


 \

\textbf{Step 1. }  Define $T_*>0$ as  
\Be\label{condition:close}
\begin{split}
T_*=\sup   \Big\{ t\geq 0:  \   &
\min \{d_2, d_{2,t}, d_6, d_3, d_{3,t}, d_\infty, d_{\infty, t}\}
\geq \frac{\sigma_0}{4}\\
&  \ \  \text{and} \ \ 
\frac{\delta \e^{1/2}}{\kappa} \sqrt{\mathcal{D}(s)} + 
\e \delta \| \mathfrak{w}_{\varrho, \ss} f(s) \|_{L^\infty_{x,v}} 
+ \frac{\e^{1/2} \delta}{\kappa^{1+ \mathfrak{P} }} 
\| Pf_R(s) \|_{L^2_x}
\ll 1 \\
& \text{ for all}  \ 0 \leq s \leq t 
 \Big\},
%
\end{split}
\Ee
where $d_2, d_{2,t}, d_6, d_3, d_{3,t}, d_\infty, d_{\infty, t}$ are defined in (\ref{d2}), (\ref{d2t}), (\ref{d6}), (\ref{d3}), (\ref{d3t}), (\ref{dinfty}) and (\ref{dinftyt}).

From (\ref{condition:theorem}) and (\ref{condition:close}) we read all the estimates of Proposition \ref{prop:energy}, Proposition \ref{prop:L6}, Proposition \ref{est:Linfty}, Proposition \ref{prop:average}, and Proposition \ref{prop:Linfty_t} in terms of $\mathcal{E} (t)$ and $\mathcal{D} (t)$ as follows. 

\hide Suppose there exists $T>0$ such that 
\Be\begin{split}\label{condition1:close}
  &\frac{\delta}{\kappa} \| Pf_R(t)\|_{L^6(\O)}^{1/2} \| Pf_R (t) \|_{L^2(\O)}^{1/2} 
   + \frac{\delta}{\kappa} \| Pf_R \|_{L_x^{6}(\O)}^{\frac{3(p-2)}{p}} 
\|  \mathfrak{w}_{\varrho,\ss} f_R \|_{L_x^{\infty}(\O)}^{\frac{6-2p}{p}}
+
 ( \e\delta 
+ \e^2
\| (\ref{est:f2}) \|_{L^\infty_{t,x}}
 ) \sup_{0 \leq s \leq t}\| \mathfrak{w}_{\varrho, \ss} f_R (s)  \|_\infty  
  \\
  &  + \e \| u(t) \|_{L^\infty (\O)}  
+
 \frac{1}{N}  +
  \e
\| (\ref{transp:mu})\|_{L^\infty_{t } ([0,T ] ; L^{\frac{2p}{p-2}}_x( \O))}
 +
\e^2 \| (\ref{est:f2})\|_{L^\infty_{t,x}} 
+
 \frac{\e}{\kappa} \| 
 (\ref{est:f2})
  \|_{L^\infty_t ([0,T]; L_x^{\frac{2p}{p-2}}(\O)) } \ll 1,
\end{split}\Ee
and 
\Be\label{condition2:close}
\frac{\e^{1/2} \delta}{\kappa } \mathcal{D}^{1/2} \lesssim 1
\Ee
 and
 \Be
 \e\| (\ref{transp:mu})\|_{\infty}
+  \e  \kappa^{-1} 
 \|(\ref{est:f2})\|_\infty\lesssim \frac{1}{\kappa^{1/2}}
 \Ee
\unhide
From (\ref{Linfty_3D}), (\ref{condition:close}), and (\ref{condition:theorem}) 
\Be 
 \begin{split}
 \label{Linfty_3D_ED}
&\sup_{0 \leq s \leq t} \| \mathfrak{w}_{\varrho, \ss} f _R(s)  \|_{L^\infty_{x,v}}\\
 \lesssim  \  & 
 \frac{1}{\e^{1/2} \kappa^{1/2}
 } \sup_{0 \leq s  \leq t} \| \mathbf{P} f_R(s) \|_{L^6_
{x,v}}
+
 \frac{1}{\e^{1/2} \kappa } \sqrt{ \mathcal{D}(t) } 
 {+ \frac{1}{\e^{1/2}\kappa^{1+ \mathfrak{P}}} 
 \| P f _R\|_{L^2_{t,x}}
 }
 \\
&  
 + \|  \mathfrak{w}_{\varrho, \ss} f (0)\|_\infty
 + \e^{1/2} \exp \Big(\frac{3}{\kappa^{\mathfrak{P}^\prime}}\Big).
%
%
%
%
%
\end{split}
 \Ee
 Now applying (\ref{Linfty_3D_ED}) to (\ref{L6}) we derive that  \hide Using $\|  (\mathbf{I} - \mathbf{P}) f_R (t) \|_{ {L^2 (\O \times \R^3)}}
 \lesssim  
   \e \kappa^{ 1/2} \sqrt{\mathcal{D}(t)},$ and $|(1- P_{\gamma_+}) f_R(t)|_{L^2({\gamma_+})}
  \lesssim     \e^{1/2} \sqrt{\mathcal{D}(t)}$ from (\ref{Sob_1D}), and using (\ref{condition:close}), we derive that \unhide
\Be\label{L6_ED}
\begin{split}
\sup_{0 \leq s \leq t}\|   {P} f_R(s) \|_{L^6_x}\lesssim & \ 
\frac{\e}{\kappa} \exp \Big(\frac{1}{\kappa^{\mathfrak{P}^\prime}}  \Big)\sup_{0 \leq s \leq t} \sqrt{\mathcal{E}(s)}+
\frac{1}{\kappa^{1/2}} \sqrt{\mathcal{D}(t)} {+ \frac{1}{ \kappa^{1/2 + \mathfrak{P}} }
\| P f _R\|_{L^2_{t,x}}
 }\\
&+ ( \e   \kappa)^{\frac{1}{2}}    \|  \mathfrak{w}_{\varrho, \ss} f (0)\|_{L^\infty_{x,v}} +\e^{1/2} \exp \Big(\frac{3}{\kappa^{\mathfrak{P}^\prime}}\Big) . \end{split}
\Ee
\hide
 \Be\notag\label{L6_ED}
\begin{split}
& 
\|   {P} f_R(t) \|_{L^6_x} \\
%
 \lesssim & \ 
\big( o(1) + \frac{\delta \e^{1/2}}{\kappa} \sqrt{\mathcal{D} (t)}\big)
 \|   {P} f_R(t) \|_{L^6_x}
 + (\e+\e\| (\ref{transp:mu})\|_{L^\infty_{t,x}}
+  \e  \kappa^{-1} 
 \|(\ref{est:f2})\|_{L^\infty_{t,x}}
)
\sqrt{\mathcal{E}(t)}\\
&
 + \frac{1}{\kappa^{1/2}}  
 \big(
1+ \delta \e \big\{  \frac{1}{\e^{1/2} \kappa } \sqrt{ \mathcal{D}(t) } + \|  \mathfrak{w}_{\varrho, \ss} f (0)\|_\infty
 +\frac{\e}{\delta}
 \| (\ref{est:f2}) \|_{L^\infty_{t,x}} 
 +
 \e^2  \kappa\| (\ref{est:R1})\|_{L^\infty_{t,x}} + 
\e^2  \kappa\| (\ref{est:R2})\|_{L^\infty_{t,x}} \big\}
 \big)
   \sqrt{\mathcal{D}(t)} \\
   &
 +   
 \frac{\e^{\frac{3}{2}}\kappa^{\frac{1}{2}}  }{\delta}
 \| (\ref{est:f2}) \|_{L^\infty_{t,x}} 
 +  \frac{\e}{\delta}| (\ref{est:f2})|_{L^4(\p\O)}   +  \e\{\| 
 (\ref{est:R1}) \|_{L^2_{x,v}}+  \|(\ref{est:R2})
 \|_{L^2_{x,v}} \}
 + \e^{\frac{5}{2}} \kappa^{\frac{3}{2}} \{ \| (\ref{est:R1})\|_{L^\infty_{t,x}}+ \| (\ref{est:R2})\|_{L^\infty_{t,x}}
 \}
   +( \e   \kappa)^{\frac{1}{2}}    \|  \mathfrak{w}_{\varrho, \ss} f (0)\|_{L^\infty_{x,v}} ,
\end{split}
\Ee
 where we have used $\|  (\mathbf{I} - \mathbf{P}) f_R (t) \|_{ {L^2 (\O \times \R^3)}}
 \lesssim  
   \e \kappa^{ 1/2} \sqrt{\mathcal{D}(t)},$ and $|(1- P_{\gamma_+}) f_R(t)|_{L^2({\gamma_+})}
  \lesssim     \e^{1/2} \sqrt{\mathcal{D}(t)}$ from (\ref{Sob_1D}).
 \unhide
   From (\ref{Linfty_3D_ED}), (\ref{L6_ED}), and (\ref{condition:close}) and (\ref{condition:theorem}) we conclude that
 \Be\label{L6Linfty_ED}
\begin{split}
& \sup_{0 \leq s \leq t} \big\{ \kappa^{\frac{1}{2}} \|   {P} f_R(s) \|_{L^6_x}  + 
 \e ^{\frac{1}{2}}\kappa \| \mathfrak{w}_{\varrho, \ss} f _R(s) \|_{L^\infty_{x,v}}\big\}  \\
 &
 \lesssim   \underbrace{  
\exp \Big(\frac{3}{\kappa^{\mathfrak{P}^\prime}}\Big)
 +  \sqrt{\mathcal{F}_{   p  }  (0)}   +  
\sup_{0 \leq s \leq t} \big\{\sqrt{\mathcal{E}(s)}
 +  \sqrt{\mathcal{D}(s)}
\big\}
{+
\frac{1}{\kappa^{\mathfrak{P}}}  \| P f_R \|_{L^2_{t,x}}
} 
 }_{(\ref{L6Linfty_ED})_*}
.
\end{split}
\Ee

 From (\ref{average_3D}), (\ref{L6Linfty_ED}), (\ref{condition:close}) and (\ref{condition:theorem})
 \Be\begin{split}\label{average_3D_ED}
\kappa^{\frac{1}{2}} \|    {P} f_R
 \|_{L^2_t((0,t);L^p_x )}
\lesssim  
\underbrace{
(\ref{L6Linfty_ED})_* \Big\{1+ 
 \frac{\e^{1/2} \delta}{\kappa}(\ref{L6Linfty_ED})_*
 + \Big( \e^{\frac{p+2}{2(p-2)} }\kappa^{\frac{2}{p-2}} (\ref{L6Linfty_ED})_* \Big)^{\frac{p-2}{p}}
\Big\}}_{(\ref{average_3D_ED})_*}.
\end{split}\Ee  

\hide
\Be\begin{split}\label{average_3Dt}
&\Big(1- O(\e) \|u\|_{L^\infty_{t,x}} 
- 
 \frac{\e}{\kappa} \| 
 (\ref{est:f2})
  \|_{L^\infty_t   L_x^{\frac{2p}{p-2}}  }-\e
\| (\ref{transp:mu})\|_{L^\infty_{t } L^{\frac{2p}{p-2}}_x }
  - \frac{\delta}{\kappa} \| P f_R \|_{L^\infty_t  L_x^{6} }^{\frac{3(p-2)}{p}} 
\|  \mathfrak{w}_{\varrho,\ss} f_R \|_{L^{\infty} _{t,x,v}}^{\frac{6-2p}{p}}
%
%
\Big)  \big\|    {P}  \p_t f_R
 \big\|_{L^2_t L^p_x }
 \\
\lesssim & 
\frac{1}{\kappa}  \| \p_t u \|_{L^2_tL^\infty_{ x}}\big(
1+  \e^2   \| (\ref{est:f2}) \|_{L^\infty_{t,x}}
\big)\| Pf_R \|_{L^\infty_tL^2_{ x}}
+ \frac{\delta \e}{\kappa} 
\| \p_t u \|_{L^\infty_{t,x} }\|    {P} f_R \|_{L^\infty_tL^6_{x }}^{\frac{3(p-2)}{p}}\| \mathfrak{w} f_R\|_{L^\infty_{t,x,v}}^{\frac{6-2p}{p}}\|    {P} f_R \|_{L^2_tL^p_{x }} \\
&+ \frac{\e}{\kappa}  \| \p_t u \|_{L^\infty_{t,x} }
(\delta \| \mathfrak{w} f_R \|_{L^\infty_{t,x,v}} + \| (\ref{est:f2}) \|_{L^\infty_{t,x}})
 \| \sqrt{\nu} (\mathbf{I} -\mathbf{P}) f_R \|_{L^2_{t,x,v}}
\\ &+ (\kappa \e ) ^{\frac{2}{p-2}} \|  \mathfrak{w}_{\varrho,\ss} \p_t  f_R \|_{L^2_t  L^\infty_{x,v}  }+
\| \p_t f_R \|_{L^\infty_t L^2_{x,v}}
\\
&+
\Big\{\frac{1}{\kappa \e} + \frac{\delta}{\kappa} \|   \mathfrak{w}_{\varrho,\ss} f_R\|_{L^\infty_{t,x,v}} +  \frac{\e}{\kappa} \|   (\ref{est:f2})\|_{L^\infty_{t,x}}  
+ \e \| 
(\ref{transp:mu})
 \|_{L^\infty_{t,x}  }
\Big\}
\| \sqrt{\nu} (\mathbf{I} - \mathbf{P}) \p_t f_R \|_{L^2_{t,x,v}}\\
&+ \|\p_t  f_R (0) \|_{L^2_\gamma} + 
\frac{\e}{\kappa}  \| 
 (\ref{est:f2_t}) \|_{L^2_{t,x,v}} 
  \| \mathfrak{w}_{\varrho,\ss} f_R \|_{L^\infty_{t,x,v}}+
\e \| (\ref{transp:mu_t})\|_\infty \| f_R\|_{L^\infty_t L^2_{x,v}} 
+ \e \| (\ref{est:R3})\|_{L^2_{t,x}}
+ \e  \| (\ref{est:R4})\|_{L^2_{t,x}},
\end{split}\Ee
\Be\label{Linfty_3D_t}
\begin{split}
&\Big(
1-
\e^2 \| (\ref{est:f2})\|_{L^\infty_{t,x}} - 
\e \delta \| \mathfrak{w} f_R\|_{L^\infty_{t,x,v}}
\Big)
\| \mathfrak{w} ^\prime  \p_t  f_R  \|_{L^2_t ((0,T);L^{\infty}_{x,v} (\O \times \R^3))}\\
\lesssim & \ 
 \frac{1}{\e^{3/p} \kappa^{3/p}} \| P \p_t f  \|_{L^2_t  L^p_{x} }
 +  \frac{1}{\e^{3/2} \kappa^{3/2} }
 \| \sqrt{\nu} ( \mathbf{I} - \mathbf{P}) \p_t f  \|_{L^2 _{t,x,v}}
 +\e \big(  \|\p_t u \|_{L^\infty_{t,x}}+ \e 
\|(\ref{est:f2_t})\|_{L^\infty_{t,x}}+
  \e  \kappa\| (\ref{transp:mu_t}) \|_{L^\infty_{t,x}}
 \big)
 \| \mathfrak{w} f_R\|_{L^\infty_{t,x,v}}
\\
&+ 
\e\kappa^{1/2}
\| \mathfrak{w} ^\prime(x,v) \p_t  f_R(0 ) \|_{ L^{\infty}_{x,v}}
+ \frac{\e}{\delta}
 \|(\ref{est:f2_t})\|_{ L^{\infty}_{x,v}} +  \frac{\e^2}{\delta} \|\p_t u\|_{ L^{\infty}_{x,v}} \|(\ref{est:f2}) \|_{ L^{\infty}_{x,v}}
  + \e^2 \kappa \|(\ref{est:R3})\|_{ L^{\infty}_{x,v}}+ \e^2 \kappa \|(\ref{est:R4})  \|_{ L^{\infty}_{x,v}}\\
  &
  \bcb
  +
\e     \Big(
\e\| (\ref{est:f2_t})\|_{L^\infty_{t,x}} + \e  \kappa  \|(\ref{transp:mu_t})\|_{L^\infty_{t,x}} 
 + \|\p_t u\|_{L^\infty_{t,x}}  \big(1+\e^2\|(\ref{est:f2}) \|_{L^\infty_{t,x}}  +  \e \delta   \| \mathfrak{w} f_R \|_{L^\infty_{t,x,v}} \big)
 \Big) \| \mathfrak{w} f_R\|_{L^\infty_{t,x,v}}
.\ec
\end{split}
\Ee
\unhide

Using (\ref{L6Linfty_ED}) and (\ref{mathfrak_C'}), from (\ref{average_3Dt}) and (\ref{Linfty_3D_t}), we deduce that, for $p<3$ and $\varrho^\prime<\varrho$,
\Be\begin{split}\label{average_3Dt_ED}
&  \kappa^{ \frac{1}{2}+\ss} \big\|    {P}  \p_t f_R
 \big\|_{L^2_t((0,t); L^p_x) }
 + (\e  \kappa)^{3/p}  \kappa^{ \frac{1}{2}+\ss} \| \mathfrak{w} _{\varrho^\prime,\ss}  \p_t  f_R  \|_{L^2_t((0,t);  L^{\infty}  _{x,v})} 
\\
&\lesssim    \underbrace{ (\ref{average_3D_ED})_*
\Big\{ 1
+
\e^{1- \frac{3-p}{p}}
\delta \kappa^{ - \frac{3}{p}}
\{(\ref{L6Linfty_ED})_*  +(\ref{average_3D_ED})_* \}\Big\}}_{(\ref{average_3Dt_ED})_*}.
\end{split}\Ee


 \hide
\textit{Step 1. }  Define $T_*>0$ as  
\Be\label{condition:close}
T_*=\sup   \Big\{ t\geq 0:  \   
\min \{d_2, d_{2,t}, d_6, d_3, d_{3,t}, d_\infty, d_{\infty, t}\}
\geq \frac{1}{4} \  \  \text{and} \ \ \frac{\delta \e^{1/2}}{\kappa} \sqrt{\mathcal{D}(t)}\ll 1 \Big\},
%
\Ee
where $d_2, d_{2,t}, d_6, d_3, d_{3,t}, d_\infty, d_{\infty, t}$ are defined in (\ref{d2}), (\ref{d2t}), (\ref{d6}), (\ref{d3}), (\ref{d3t}), (\ref{dinfty}) and (\ref{dinftyt}).

From (\ref{condition:theorem}) and (\ref{condition:close}) we read all the estimates of Proposition \ref{prop:energy}, Proposition \ref{prop:L6}, Proposition \ref{est:Linfty}, Proposition \ref{prop:average}, and Proposition \ref{prop:Linfty_t} in terms of $\mathcal{E} (t)$ and $\mathcal{D} (t)$ as follows. 

\hide Suppose there exists $T>0$ such that 
\Be\begin{split}\label{condition1:close}
  &\frac{\delta}{\kappa} \| Pf_R(t)\|_{L^6(\O)}^{1/2} \| Pf_R (t) \|_{L^2(\O)}^{1/2} 
   + \frac{\delta}{\kappa} \| Pf_R \|_{L_x^{6}(\O)}^{\frac{3(p-2)}{p}} 
\|  \mathfrak{w}_{\varrho,\ss} f_R \|_{L_x^{\infty}(\O)}^{\frac{6-2p}{p}}
+
 ( \e\delta 
+ \e^2
\| (\ref{est:f2}) \|_{L^\infty_{t,x}}
 ) \sup_{0 \leq s \leq t}\| \mathfrak{w}_{\varrho, \ss} f_R (s)  \|_\infty  
  \\
  &  + \e \| u(t) \|_{L^\infty (\O)}  
+
 \frac{1}{N}  +
  \e
\| (\ref{transp:mu})\|_{L^\infty_{t } ([0,T ] ; L^{\frac{2p}{p-2}}_x( \O))}
 +
\e^2 \| (\ref{est:f2})\|_{L^\infty_{t,x}} 
+
 \frac{\e}{\kappa} \| 
 (\ref{est:f2})
  \|_{L^\infty_t ([0,T]; L_x^{\frac{2p}{p-2}}(\O)) } \ll 1,
\end{split}\Ee
and 
\Be\label{condition2:close}
\frac{\e^{1/2} \delta}{\kappa } \mathcal{D}^{1/2} \lesssim 1
\Ee
 and
 \Be
 \e\| (\ref{transp:mu})\|_{\infty}
+  \e  \kappa^{-1} 
 \|(\ref{est:f2})\|_\infty\lesssim \frac{1}{\kappa^{1/2}}
 \Ee
\unhide
From (\ref{Linfty_3D}), (\ref{condition:close}), \bcb and (\ref{condition:theorem})\ec
\Be 
 \label{Linfty_3D_ED}
\sup_{0 \leq s \leq t} \| \mathfrak{w}_{\varrho, \ss} f _R(s)  \|_{L^\infty_{x,v}}
 \lesssim   \ 
 \frac{1}{\e^{1/2} \kappa^{1/2}
 } \sup_{0 \leq s  \leq t} \| \mathbf{P} f_R(s) \|_{L^6_
{x,v}}
+
 \frac{1}{\e^{1/2} \kappa } \sqrt{ \mathcal{D}(t) } + \|  \mathfrak{w}_{\varrho, \ss} f (0)\|_\infty
 + \e^{1/2} \exp \Big(\frac{3}{\kappa^{\mathfrak{P}^\prime}}\Big).
%
%
%
%
%
 \Ee
 Now we apply (\ref{Linfty_3D_ED}) to (\ref{L6}). Using $\|  (\mathbf{I} - \mathbf{P}) f_R (t) \|_{ {L^2 (\O \times \R^3)}}
 \lesssim  
   \e \kappa^{ 1/2} \sqrt{\mathcal{D}(t)},$ and $|(1- P_{\gamma_+}) f_R(t)|_{L^2({\gamma_+})}
  \lesssim     \e^{1/2} \sqrt{\mathcal{D}(t)}$ from (\ref{Sob_1D}), and using (\ref{condition:close}), we derive that 
\Be\label{L6_ED}
\sup_{0 \leq s \leq t}\|   {P} f_R(s) \|_{L^6_x}\lesssim 
\frac{\e}{\kappa} \exp \Big(\frac{1}{\kappa^{\mathfrak{P}^\prime}}  \Big)\sup_{0 \leq s \leq t} \sqrt{\mathcal{E}(s)}+
\frac{1}{\kappa^{1/2}} \sqrt{\mathcal{D}(t)}+ ( \e   \kappa)^{\frac{1}{2}}    \|  \mathfrak{w}_{\varrho, \ss} f (0)\|_{L^\infty_{x,v}} +\e^{1/2} \exp \Big(\frac{3}{\kappa^{\mathfrak{P}^\prime}}\Big) . 
\Ee
\hide
 \Be\notag\label{L6_ED}
\begin{split}
& 
\|   {P} f_R(t) \|_{L^6_x} \\
%
 \lesssim & \ 
\big( o(1) + \frac{\delta \e^{1/2}}{\kappa} \sqrt{\mathcal{D} (t)}\big)
 \|   {P} f_R(t) \|_{L^6_x}
 + (\e+\e\| (\ref{transp:mu})\|_{L^\infty_{t,x}}
+  \e  \kappa^{-1} 
 \|(\ref{est:f2})\|_{L^\infty_{t,x}}
)
\sqrt{\mathcal{E}(t)}\\
&
 + \frac{1}{\kappa^{1/2}}  
 \big(
1+ \delta \e \big\{  \frac{1}{\e^{1/2} \kappa } \sqrt{ \mathcal{D}(t) } + \|  \mathfrak{w}_{\varrho, \ss} f (0)\|_\infty
 +\frac{\e}{\delta}
 \| (\ref{est:f2}) \|_{L^\infty_{t,x}} 
 +
 \e^2  \kappa\| (\ref{est:R1})\|_{L^\infty_{t,x}} + 
\e^2  \kappa\| (\ref{est:R2})\|_{L^\infty_{t,x}} \big\}
 \big)
   \sqrt{\mathcal{D}(t)} \\
   &
 +   
 \frac{\e^{\frac{3}{2}}\kappa^{\frac{1}{2}}  }{\delta}
 \| (\ref{est:f2}) \|_{L^\infty_{t,x}} 
 +  \frac{\e}{\delta}| (\ref{est:f2})|_{L^4(\p\O)}   +  \e\{\| 
 (\ref{est:R1}) \|_{L^2_{x,v}}+  \|(\ref{est:R2})
 \|_{L^2_{x,v}} \}
 + \e^{\frac{5}{2}} \kappa^{\frac{3}{2}} \{ \| (\ref{est:R1})\|_{L^\infty_{t,x}}+ \| (\ref{est:R2})\|_{L^\infty_{t,x}}
 \}
   +( \e   \kappa)^{\frac{1}{2}}    \|  \mathfrak{w}_{\varrho, \ss} f (0)\|_{L^\infty_{x,v}} ,
\end{split}
\Ee
 where we have used $\|  (\mathbf{I} - \mathbf{P}) f_R (t) \|_{ {L^2 (\O \times \R^3)}}
 \lesssim  
   \e \kappa^{ 1/2} \sqrt{\mathcal{D}(t)},$ and $|(1- P_{\gamma_+}) f_R(t)|_{L^2({\gamma_+})}
  \lesssim     \e^{1/2} \sqrt{\mathcal{D}(t)}$ from (\ref{Sob_1D}).
 \unhide
   From (\ref{Linfty_3D_ED}), (\ref{L6_ED}), and (\ref{condition:close}) and (\ref{condition:theorem}) we conclude that
 \Be\label{L6Linfty_ED}
\begin{split}
 \sup_{0 \leq s \leq t} \big\{ \kappa^{\frac{1}{2}} \|   {P} f_R(s) \|_{L^6_x}  + 
 \e ^{\frac{1}{2}}\kappa \| \mathfrak{w}_{\varrho, \ss} f _R(s) \|_{L^\infty_{x,v}}\big\}  
 \lesssim   \underbrace{  
\exp \Big(\frac{3}{\kappa^{\mathfrak{P}^\prime}}\Big)
 +  \sqrt{\mathcal{F}_{ \bcb p \ec}  (0)}   +  
\sup_{0 \leq s \leq t} \big\{\sqrt{\mathcal{E}(s)}
 +  \sqrt{\mathcal{D}(s)}
\big\} }_{(\ref{L6Linfty_ED})_*}
.
\end{split}
\Ee

 From (\ref{average_3D}), (\ref{L6Linfty_ED}), (\ref{condition:close}) and (\ref{condition:theorem})
 \Be\begin{split}\label{average_3D_ED}
\kappa^{\frac{1}{2}} \|    {P} f_R
 \|_{L^2_t((0,t);L^p_x )}
\lesssim  
\underbrace{
(\ref{L6Linfty_ED})_* \Big\{1+ 
 \frac{\e^{1/2} \delta}{\kappa}(\ref{L6Linfty_ED})_*
 + \Big( \e^{\frac{p+2}{2(p-2)} }\kappa^{\frac{2}{p-2}} (\ref{L6Linfty_ED})_* \Big)^{\frac{p-2}{p}}
\Big\}}_{(\ref{average_3D_ED})_*}.
\end{split}\Ee  

\hide
\Be\begin{split}\label{average_3Dt}
&\Big(1- O(\e) \|u\|_{L^\infty_{t,x}} 
- 
 \frac{\e}{\kappa} \| 
 (\ref{est:f2})
  \|_{L^\infty_t   L_x^{\frac{2p}{p-2}}  }-\e
\| (\ref{transp:mu})\|_{L^\infty_{t } L^{\frac{2p}{p-2}}_x }
  - \frac{\delta}{\kappa} \| P f_R \|_{L^\infty_t  L_x^{6} }^{\frac{3(p-2)}{p}} 
\|  \mathfrak{w}_{\varrho,\ss} f_R \|_{L^{\infty} _{t,x,v}}^{\frac{6-2p}{p}}
%
%
\Big)  \big\|    {P}  \p_t f_R
 \big\|_{L^2_t L^p_x }
 \\
\lesssim & 
\frac{1}{\kappa}  \| \p_t u \|_{L^2_tL^\infty_{ x}}\big(
1+  \e^2   \| (\ref{est:f2}) \|_{L^\infty_{t,x}}
\big)\| Pf_R \|_{L^\infty_tL^2_{ x}}
+ \frac{\delta \e}{\kappa} 
\| \p_t u \|_{L^\infty_{t,x} }\|    {P} f_R \|_{L^\infty_tL^6_{x }}^{\frac{3(p-2)}{p}}\| \mathfrak{w} f_R\|_{L^\infty_{t,x,v}}^{\frac{6-2p}{p}}\|    {P} f_R \|_{L^2_tL^p_{x }} \\
&+ \frac{\e}{\kappa}  \| \p_t u \|_{L^\infty_{t,x} }
(\delta \| \mathfrak{w} f_R \|_{L^\infty_{t,x,v}} + \| (\ref{est:f2}) \|_{L^\infty_{t,x}})
 \| \sqrt{\nu} (\mathbf{I} -\mathbf{P}) f_R \|_{L^2_{t,x,v}}
\\ &+ (\kappa \e ) ^{\frac{2}{p-2}} \|  \mathfrak{w}_{\varrho,\ss} \p_t  f_R \|_{L^2_t  L^\infty_{x,v}  }+
\| \p_t f_R \|_{L^\infty_t L^2_{x,v}}
\\
&+
\Big\{\frac{1}{\kappa \e} + \frac{\delta}{\kappa} \|   \mathfrak{w}_{\varrho,\ss} f_R\|_{L^\infty_{t,x,v}} +  \frac{\e}{\kappa} \|   (\ref{est:f2})\|_{L^\infty_{t,x}}  
+ \e \| 
(\ref{transp:mu})
 \|_{L^\infty_{t,x}  }
\Big\}
\| \sqrt{\nu} (\mathbf{I} - \mathbf{P}) \p_t f_R \|_{L^2_{t,x,v}}\\
&+ \|\p_t  f_R (0) \|_{L^2_\gamma} + 
\frac{\e}{\kappa}  \| 
 (\ref{est:f2_t}) \|_{L^2_{t,x,v}} 
  \| \mathfrak{w}_{\varrho,\ss} f_R \|_{L^\infty_{t,x,v}}+
\e \| (\ref{transp:mu_t})\|_\infty \| f_R\|_{L^\infty_t L^2_{x,v}} 
+ \e \| (\ref{est:R3})\|_{L^2_{t,x}}
+ \e  \| (\ref{est:R4})\|_{L^2_{t,x}},
\end{split}\Ee
\Be\label{Linfty_3D_t}
\begin{split}
&\Big(
1-
\e^2 \| (\ref{est:f2})\|_{L^\infty_{t,x}} - 
\e \delta \| \mathfrak{w} f_R\|_{L^\infty_{t,x,v}}
\Big)
\| \mathfrak{w} ^\prime  \p_t  f_R  \|_{L^2_t ((0,T);L^{\infty}_{x,v} (\O \times \R^3))}\\
\lesssim & \ 
 \frac{1}{\e^{3/p} \kappa^{3/p}} \| P \p_t f  \|_{L^2_t  L^p_{x} }
 +  \frac{1}{\e^{3/2} \kappa^{3/2} }
 \| \sqrt{\nu} ( \mathbf{I} - \mathbf{P}) \p_t f  \|_{L^2 _{t,x,v}}
 +\e \big(  \|\p_t u \|_{L^\infty_{t,x}}+ \e 
\|(\ref{est:f2_t})\|_{L^\infty_{t,x}}+
  \e  \kappa\| (\ref{transp:mu_t}) \|_{L^\infty_{t,x}}
 \big)
 \| \mathfrak{w} f_R\|_{L^\infty_{t,x,v}}
\\
&+ 
\e\kappa^{1/2}
\| \mathfrak{w} ^\prime(x,v) \p_t  f_R(0 ) \|_{ L^{\infty}_{x,v}}
+ \frac{\e}{\delta}
 \|(\ref{est:f2_t})\|_{ L^{\infty}_{x,v}} +  \frac{\e^2}{\delta} \|\p_t u\|_{ L^{\infty}_{x,v}} \|(\ref{est:f2}) \|_{ L^{\infty}_{x,v}}
  + \e^2 \kappa \|(\ref{est:R3})\|_{ L^{\infty}_{x,v}}+ \e^2 \kappa \|(\ref{est:R4})  \|_{ L^{\infty}_{x,v}}\\
  &
  \bcb
  +
\e     \Big(
\e\| (\ref{est:f2_t})\|_{L^\infty_{t,x}} + \e  \kappa  \|(\ref{transp:mu_t})\|_{L^\infty_{t,x}} 
 + \|\p_t u\|_{L^\infty_{t,x}}  \big(1+\e^2\|(\ref{est:f2}) \|_{L^\infty_{t,x}}  +  \e \delta   \| \mathfrak{w} f_R \|_{L^\infty_{t,x,v}} \big)
 \Big) \| \mathfrak{w} f_R\|_{L^\infty_{t,x,v}}
.\ec
\end{split}
\Ee
\unhide

Using (\ref{L6Linfty_ED}) and (\ref{mathfrak_C'}), from (\ref{average_3Dt}) and (\ref{Linfty_3D_t}), we deduce that, for $p<3$ and $\varrho^\prime<\varrho$,
\Be\begin{split}\label{average_3Dt_ED}
&  \kappa^{ \frac{1}{2}+\ss} \big\|    {P}  \p_t f_R
 \big\|_{L^2_t((0,t); L^p_x) }
 + (\e  \kappa)^{3/p}  \kappa^{ \frac{1}{2}+\ss} \| \mathfrak{w} _{\varrho^\prime,\ss}  \p_t  f_R  \|_{L^2_t((0,t);  L^{\infty}  _{x,v})} 
\\
&\lesssim    \underbrace{ (\ref{average_3D_ED})_*
\Big\{ 1
+
\e^{1- \frac{3-p}{p}}
\delta \kappa^{\bcb - \frac{3}{p} \ec}
\{(\ref{L6Linfty_ED})_*  +(\ref{average_3D_ED})_* \}\Big\}}_{(\ref{average_3Dt_ED})_*}.
\end{split}\Ee

\unhide


\


\textbf{Step 2. } Using the estimates of the previous step we will close the estimate ultimately in the basic energy estimates (\ref{est:Energy}) and (\ref{est:Energy_t}) via the Gronwall's inequality. We note that from (\ref{mathfrak_C'}) the multipliers of $\int^t_0 \| Pf_R(s) \|_{L^2_x}^2 \dd s$ in (\ref{est:Energy}) and $\int^t_0 \| P \p_t f_R(s) \|_{L^2_x}^2 \dd s$ in (\ref{est:Energy_t}) are bounded above by 
\Be\label{multiplier}
O(1)  \kappa^{-\mathfrak{P}} \big(
1+ \e \kappa^{\frac{1}{2}-\mathfrak{P}} + ( \e \kappa^{\frac{1}{2}-\mathfrak{P}})^2
\big) \lesssim  \kappa^{-\mathfrak{P}} ,
\Ee
 where we have used (\ref{choice:delta}). 

In (\ref{est:Energy}) and (\ref{est:Energy_t}) we bound
\Be\label{interp_3}
\begin{split}
\| \kappa^{1/2} Pf_R \|_{L^2_t L^3_x}&\lesssim 
\kappa^{\frac{1}{2}(1- \frac{p}{3})}
\|   P f_R \|_{L^2_t L^\infty_x}^{1- \frac{p}{3}}
\| \kappa^{1/2} P  f_R \|_{L^2_t L^p_x}^{\frac{p}{3}}
\lesssim _T (\e\kappa)^{-\frac{1}{2} (1- \frac{p}{3})} 
|(\ref{L6Linfty_ED})_*|^{1- \frac{p}{3}} |(\ref{average_3Dt_ED})_*|^{\frac{p}{3}}
,\\
\|   P\p_t f_R \|_{L^2_t L^3_x}&\lesssim
\|   P\p_t f_R \|_{L^2_t L^\infty_x}^{1- \frac{p}{3}}
\|  P\p_t f_R \|_{L^2_t L^p_x}^{\frac{p}{3}}
\lesssim 
\e^{-\frac{3}{p}(1-\frac{p}{3}) } \kappa^{
   - \frac{1}{2} - \mathfrak{P} - \frac{3}{p} (1- \frac{p}{3})
   }
|(\ref{average_3Dt_ED})_*|.
\end{split}
\Ee

We can check that the multiplier of $ \| \e^{-1}\kappa^{-1/2} \sqrt{\nu} (\mathbf{I} -\mathbf{P}) f_R \|_{L^2_{t,x,v}}^2$ in (\ref{est:Energy_t}) is bounded as,  from (\ref{condition:theorem}) and (\ref{L6Linfty_ED}), 
\Be\begin{split}\notag
& \Big\{
 \e (1+ \e \| (\ref{est:f2})\|_{L^\infty_{t,x}}) \| \p_t u \|_{L^\infty_{t,x}}
 + \e \kappa \| \nabla_x \p_t u \|_{L^\infty_{t,x}}
 + (\e \kappa^{1/2} \| (\ref{transp:mu_t})_*\|_{L^\infty_{t,x}})^2
 +( \e \delta \| \mathfrak{w} f_R \|_{L^\infty_{t,x,v}}) ^2 \Big\}\\
&\lesssim  \e\kappa^{1/2-\mathfrak{P}}+ \e \delta^2 \kappa^{-2} |(\ref{L6Linfty_ED})_*|^2.   \end{split}\Ee

Applying (\ref{L6Linfty_ED}), (\ref{average_3D_ED}), (\ref{average_3Dt_ED}), (\ref{interp_3}) to $(\ref{est:Energy})+o(1)(\ref{est:Energy_t})$,  using the above bound and (\ref{choice:delta}),  and collecting the terms, we derive that    
\Be \begin{split}\label{energy_final1}
&\sup_{0 \leq s \leq t}\mathcal{E} (s) + (1- \e \delta^2 \kappa^{-2} |(\ref{L6Linfty_ED})_*|^2) \mathcal{D}(t) \\
\lesssim & \      \mathcal{E}(0) + \mathcal{F} (0) + \exp \Big( \frac{6}{\kappa^{\mathfrak{P}^\prime}}\Big)
    + (\ref{multiplier})
\int^T_0 \mathcal{E} (s) \dd s
\\
&+
\delta^2
 \e^{- (1- \frac{p}{3})}\kappa^{- 4 + \frac{p}{3}}
|(\ref{L6Linfty_ED})_*|^{4- \frac{2p}{3}} |(\ref{average_3Dt_ED})_*|^{\frac{2p}{3} }\\
&
   + \delta^2 \e^{- \frac{6}{p} (1- \frac{p}{3})  }
    \kappa^{-3 - 2\mathfrak{P} - \frac{6}{p}  (1- \frac{p}{3})}
   |(\ref{L6Linfty_ED})_*|^{2 } |(\ref{average_3Dt_ED})_*|^{2}.
  \end{split}\Ee  
  Under the assumption of 
\Be
\begin{split}\label{assump_Gronwall}
 \e^{1/2} \delta \kappa ^{-1}(\ref{L6Linfty_ED})_* \ll 1
,\  \   \e^{\frac{p+2}{2(p-2)} }\kappa^{\frac{2}{p-2}} (\ref{L6Linfty_ED})_*
\ll 1, \  \ 
\big[\e^{1- \frac{3-p}{p}}
\delta \kappa^{ - \frac{3}{p} } \big]^{1/2}
 (\ref{L6Linfty_ED})_*  \ll1,\\
 [\delta^2
 \e^{- (1- \frac{p}{3})}\kappa^{- 4 + \frac{p}{3}}]^{1/4}
 (\ref{L6Linfty_ED})_* \ll1, \ \ 
 [\delta^2 \e^{- \frac{6}{p} (1- \frac{p}{3})  }   
    \kappa^{-3 - 2\mathfrak{P} - \frac{6}{p}  (1- \frac{p}{3})}
]^{1/4}
    (\ref{L6Linfty_ED})_*  \ll1, 
\end{split}
\Ee
  we derive that, for some constants $\mathfrak{C}_1>0$ and $\mathfrak{C}_2>0$,  
\Be\label{gronwall}
\sup_{0 \leq s\leq t}\mathcal{E}(s)  + \mathcal{D}(t) 
\leq \mathfrak{C}_1
\Big(\mathcal{E}(0) + \mathcal{F}_p(0)
+\exp\Big(\frac{6}{\kappa^{\mathfrak{P}^\prime}}\Big)
\Big)+ 
   \mathfrak{C}_2
\kappa^{-\mathfrak{P}} \int^t_0\mathcal{E}(s)  \dd s  .
\Ee  \hide
 \\
&+ \Big(
\frac{\e}{\kappa^2}
+\frac{\delta}{\kappa^{\frac{5}{2}} }\| (\ref{est:f2})\|_{L^\infty_t L^6_x}^2
+ \frac{\e^{2- \frac{6}{p}}e^{- \frac{\varrho}{3 \e^2}} }{\kappa^{\frac{6}{p}}}
\| (\ref{transp:mu}) \|_{L^\infty_{t,x}}^2
+ \|(\ref{transp:mu_t})\|_\infty
+
 \frac{\delta^2}{\kappa^3}\sup_{0 \leq t \leq T}\mathcal{E} (t) 
\Big)\Big[\sup_{0 \leq t \leq T}\mathcal{E} (t)+ \mathcal{D}(T) \Big]
\label{final_E0}
\\
&+ \frac{\delta^2}{\kappa^3} \big\{ (\ref{L6Linfty:force})^4 + (\ref{average_3D_ED:force})^4
+ (\ref{average_3Dt_ED:force})^4
\big\}
+ \frac{\e}{\kappa^2} (\ref{average_3D_ED:force})^2
+
\Big\{\frac{\delta}{\kappa^{\frac{5}{2}} }\| (\ref{est:f2})\|_{L^\infty_t L^6_x}^2
+
 \frac{\e^{2- \frac{6}{p}}e^{- \frac{\varrho}{3 \e^2}} }{\kappa^{\frac{6}{p}}}
\| (\ref{transp:mu}) \|_{L^\infty_{t,x}}^2
\Big\}
 (\ref{average_3Dt_ED:force})^2\label{final_E1}
 \\
 &
+ \kappa  \e^2  
  \|(\ref{est:R1})+ (\ref{est:R2})+ (\ref{est:R3}) + (\ref{est:R4})
  \|^2_{L^2_{t,x}}   
  + \frac{\e}{2\delta^2}  
 | (\ref{est:f2}) |^2_{L^2_t L^2_\gamma}
   +( \frac{\e}{\delta^2} + \frac{\e^3}{\delta}\| \p_t u \|_{L^\infty} )| (\ref{est:f2_t}) |_{L^2_tL^2_{\gamma }}^2+ \mathcal{E}(0).\label{final_E2}
\end{align} 
\unhide
 Note that among others the last condition condition is the strongest in (\ref{assump_Gronwall}), which can be read as, from $\delta=\sqrt{\e}$ of (\ref{choice:delta}),
\Be\label{assump_Gronwall_S}
\delta^{\frac{1}{2}- \frac{3}{p} \big(1- \frac{p}{3}\big) }
\kappa^{- \frac{3}{4} - \frac{\mathfrak{P}}{2} - \frac{3}{2 p} (1- \frac{p}{3}) }
 (\ref{L6Linfty_ED})_* \ll 1.
\Ee

Applying the Gronwall's inequality to (\ref{gronwall}) (we may redefine $\mathcal{E}(t)$ as $\sup_{0 \leq s \leq t}\mathcal{E}(s)$ if necessary), we derive that 
\Be\notag
\begin{split}
\sup_{0 \leq s \leq t}\mathcal{E} (s)
 &\leq 
  \mathfrak{C} _1\Big(\mathcal{E}(0) + \mathcal{F}_p(0)+ \exp\Big(\frac{6}{\kappa^{\mathfrak{P}^\prime}}\Big)\Big) 
  \Big\{
 1+ \frac{\mathfrak{C}_2 t }{\kappa^{\mathfrak{P}}}
 \exp\Big(  \frac{\mathfrak{C}_2  t }{\kappa^{\mathfrak{P}}} \Big)\Big\} 
 .
 \end{split}
 \Ee 
 Applying this estimate to the last term of (\ref{gronwall}) and using the fact $\mathfrak{P}^\prime< \mathfrak{P}$ we derive that, after redefining $\mathfrak{C}_1$ if necessary,  
\Be 
\sup_{0 \leq s \leq t}\mathcal{E} (s) +   \mathcal{D}(t) + \mathcal{F}_p (t)  \leq      \mathfrak{C}_1  \big(
\mathcal{E}(0) + \mathcal{F}_p(0)+ 1 \big) 
  \exp\Big(  \frac{  2\mathfrak{C}_2 t}{ \kappa^{ \mathfrak{P}} }\Big)
    \ \ \text{for all } \ t \leq T_*,
  \label{est:E_1}
\Ee
under the assumptions of (\ref{condition:theorem}), (\ref{condition:close}), and (\ref{assump_Gronwall}).


\

\textbf{Step 3. }   Now we find out the ranges of $\delta, \kappa, \e$ satisfying the assumptions of (\ref{condition:close}) and (\ref{assump_Gronwall}). 
From (\ref{initial_EF}) and (\ref{est:E_1}), if we choose $\delta$ as
\Be\label{choice:delta1}
\delta  \leq\left[  \frac{
\kappa^{  \frac{3}{2} + \mathfrak{P} + \frac{3}{  p} (1- \frac{p}{3}) }
}{\mathfrak{C}_1  \big(
\mathcal{E}(0) + \mathcal{F}_p(0)+ 1  \big) }  
\exp\Big(  \frac{  -2\mathfrak{C}_2 T}{ \kappa^{ \mathfrak{P}} }\Big) 
    \right]^{\frac{1}{1- \frac{6}{p} (1- \frac{p}{3})}}. 
\Ee
then we can achieve (\ref{assump_Gronwall_S}) and hence all conditions of (\ref{assump_Gronwall}). Clearly (\ref{choice:delta}) and (\ref{initial_EF}) ensure (\ref{choice:delta1}).

Now from (\ref{est:E_1}) and (\ref{choice:delta}) we derive (\ref{est:E}), 
which implies 
\Be
\begin{split}
&\sup_{0 \leq s \leq t} \Big\{ 
   \| \kappa^{1/2} Pf_R(s)\|_{L^6_x}  
 + \| \e^{1/2}\kappa \mathfrak{w}_{\varrho,\ss} f_R(s) \|_{L_{x,v}^{\infty} }
+ \|   (\e  \kappa)^{3/p}  \kappa^{ \frac{1}{2}+\mathfrak{P}} \mathfrak{w}_{\varrho^\prime,\ss} f_R(s) \|_{L^2((0,s);L_{x,v}^{\infty}) } 
\Big\}\\
&\lesssim \delta^{-\frac{1}{2}+ \frac{3}{p} (1- \frac{p}{3})}.\notag
\end{split}\Ee
These imply $\min \{d_2, d_{2,t}, d_6, d_3, d_{3,t}, d_\infty, d_{\infty, t}\}
\geq \frac{1}{4}$ and $\frac{\delta \e^{1/2}}{\kappa} \sqrt{\mathcal{D}(t)}\ll 1$ from (\ref{d2}), (\ref{d2t}), (\ref{d6}), (\ref{d3}), (\ref{d3t}), (\ref{dinfty}) and (\ref{dinftyt}).

\hide

From (\ref{est:E}) we check the last condition of (\ref{assump_Gronwall}) holds, which is most singular among other conditions of (\ref{assump_Gronwall}), if 
 
or

From (\ref{choice:delta}) for large enough $\mathfrak{C}>0$
\Be\notag
\begin{split}
 &[\delta^2 \e^{- \frac{6}{p} (1- \frac{p}{3})  } \kappa^{-1 - \frac{6}{p}- 2 \mathfrak{P} }]^{1/4}
 (\ref{L6Linfty_ED})_*\\
 &\lesssim  [\delta^2 \e^{- \frac{6}{p} (1- \frac{p}{3})  } \kappa^{-1 - \frac{6}{p}- 2 \mathfrak{P} }]^{1/4}\bigg\{\exp \Big(\frac{3}{\kappa^{\mathfrak{P}^\prime}}\Big)
 +  \sqrt{\mathcal{F}  (0)}   +  
  \sqrt{\mathfrak{C}_1 }   \sqrt{\mathcal{E}(0)}  
    \exp\big(   \mathfrak{C}_2\kappa^{-\mathfrak{P}} t\big)
\bigg\}\\  
&\ll 1.
\end{split}
\Ee
We can also verify all assumptions of (\ref{assump_Gronwall}).

At the same time from (\ref{est:E}) we can show that  
\Be\label{dED}
\delta^{\mathfrak{q}} \big\{ \sup_{0 \leq s \leq t} \sqrt{\mathcal{E}(s)}
+  \sqrt{\mathcal{D}(t)}\big\} \lesssim 1. 
\Ee 
From (\ref{choice:delta}), (\ref{dED}), for $|p-3|\ll 1$ we can verify (\ref{condition:close}). \unhide Then by the standard continuation argument we can verify all assumptions (\ref{condition:close}) up to $t\leq T$ and $T=T_*$. The estimate (\ref{est:E}) follows easily.     %
%
 
\hide
\Be
\begin{split}
& \delta \times \bigg\{\exp \Big(\frac{3}{\kappa^{\mathfrak{P}^\prime}}\Big)
 +  \sqrt{\mathcal{F}  (0)}   +  
\sup_{0 \leq s \leq t} \big\{\sqrt{\mathcal{E}(s)}
 +  \sqrt{\mathcal{D}(s)}
\big\}\bigg\}\\
&\leq \delta \times(\ref{L6Linfty_ED})\\
&\lesssim \delta \times [\delta^2 \e^{- \frac{6}{p} (1- \frac{p}{3})  } \kappa^{-1 - \frac{6}{p}- 2 \mathfrak{P} }]^{-1/4}\\
&= \delta^{1/2} \e^{\frac{3}{2p} (1- \frac{p}{3}) } \kappa ^{\frac{1}{4} (1+ \frac{6}{p}+ 2 \mathfrak{P})}.
\end{split}\notag
\Ee

From the last condition of (\ref{assump_Gronwall}) we choose 
\Be
\delta^{\mathfrak{q}} \lesssim   \e^{  \frac{3}{p} (1- \frac{p}{3})  } \kappa^{\frac{1}{2} + \frac{3}{p}+  \mathfrak{P} } 
e^{- 2\mathfrak{C}^\prime \kappa^{-\mathfrak{P}}T}.
\Ee

 Then by the standard continuation argument we can verify all assumptions (\ref{condition1:close}), (\ref{condition2:close}), (\ref{average_3D_ED:assump}), (\ref{average_3Dt_ED:assump}). 
 
 \Be
\delta\Big\{  \sup_{0 \leq s \leq t} \sqrt{\mathcal{E}(s)}
+  \sqrt{\mathcal{D}(t)} \Big\} \lesssim \kappa . 
\Ee
\unhide

\hide

\subsection{Proof of Theorem \ref{main_theorem:conditional}: a conditional statement of Theorem \ref{main_theorem}}

We define an energy and dissipation as 
\Be\label{ED}
\begin{split}
\mathcal{E} (t):= & \ \| f_R (t) \|_{L^2_{x,v}}^2 + \| \p_t  f_R (t) \|_{L^2_{x,v}}^2 ,\\
\mathcal{D} (t):=&  \ \int^t_0\Big( \| \kappa^{-\frac{1}{2}} \e^{-1}
\sqrt{\nu}
 (\mathbf{I} - \mathbf{P}) f_R (s) \|_{L^2_{x,v}}^2  +  \|
 \kappa^{-\frac{1}{2}} \e^{-1} \sqrt{\nu}  (\mathbf{I} - \mathbf{P})\p_t  f_R (s) \|_{L^2_{x,v}}^2
 \Big) \dd s \\
 &
 + \int^t_0 \Big(   | \e^{-\frac{1}{2}}
f_R(s)
|_{L^2_\gamma  }^2 +  | \e^{-\frac{1}{2}}
\p_t f_R(s)
|_{   L^2_\gamma}^2\Big) \dd s
 .
\end{split}
\Ee
We also define an auxiliary norm for $p<3$ and $t>0$
\Be
\begin{split}
\mathcal{F}_p(t):= \sup_{0 \leq s \leq t} \Big\{&  \| \kappa^{1/2} Pf_R(s)\|_{L^6_x} 
+ \| \kappa^{1/2} P   f_R\|_{L^2((0,s); L^p_x)}\\
&+ \|  \kappa^{ \ss+1/2}   P \p_t f_R\|_{L^2((0,s); L^p_x)}
+ \| \e^{1/2}\kappa \mathfrak{w}_{\varrho,\ss} f_R(s) \|_{L_{x,v}^{\infty} }
+ \|  \e^{1/2}\kappa \mathfrak{w}_{\varrho^\prime,\ss} f_R(s) \|_{L^2((0,s);L_{x,v}^{\infty}) }
\Big\},
\end{split}
\Ee
where $\mathfrak{w}_{\varrho,\ss}$ and $\mathfrak{w}_{\varrho^\prime,\ss}$ are defined in (\ref{weight}) and (\ref{w_prime}). For $t=0$ we define 
\Be\label{initial_F}
 \mathcal{F}_p(0):= \kappa^{1/2}\|  f_R(0) \|_{L^2_\gamma}
+ \kappa^{\ss+1}   \| \p_t f_R(0) \|_{L^2_\gamma}+
\e^{1/2} \kappa \| \mathfrak{w} f_R(0)\|_{L^\infty_{x,v}}
+ \e  \kappa^{3/2} \| \mathfrak{w}^\prime \p_t f_R(0)\|_{L^\infty_{x,v}}.
\Ee

\begin{theorem}\label{main_theorem:conditional}
Suppose $\mathfrak{P}>0$
\Be
\|  \p_t u\|_{L^\infty ([0,T]\times \bar{\O})}+\kappa^{ 1/2}  \|\nabla_x  u\|_{L^\infty ([0,T]\times \bar{\O})} <
\mathfrak{C}^\prime
\kappa^{1/2}\kappa^{ -\mathfrak{P}}.\label{mathfrak_C'}
\Ee 
Further we assume that 
\Be\label{condition:theorem}
\begin{split}
\end{split}
\Ee
If 
\Be
\mathcal{E}(0) + \mathcal{F}_p (0) <\infty
\Ee
then for all $\e>0$ there exists a unique solution $f_R(t,x,v)$ of (\ref{F_e}) to (\ref{Boltzmann}) and (\ref{diffuse_BC}) with (\ref{scale}) in $t \in [0,T]$ such that 
\Be\label{est:E}
 \sup_{0 \leq t \leq T} \big\{\mathcal{E}(t) + \mathcal{D}(t)+ \mathcal{F}_p (t)\big\} 
 \lesssim 
\Ee
\end{theorem}

\begin{proof}
An existence of a unique global solution $F$ for each $\e>0$ can be found in \cite{EGKM2}. Thereby we only focus to show the estimate. 

\textit{Step 1.} Let $T_*>0$ such that 
\Be\label{condition:close}
\sup_{t \in [0,T_*]}
\big\{
\delta \e \| \mathfrak{w} f_R  (t) \|_{L^\infty_{x,v}}
+ \frac{\delta }{\kappa } \| Pf_R (t) \|_{L^6_x}^{\frac{1}{2}} \| Pf_R (t) \|_{L^2_x}^{\frac{1}{2}}
+  
\frac{\delta}{\kappa} \| Pf_R(t) \|_{ L_x^{6} }^{\frac{3(p-2)}{p}} 
\|  \mathfrak{w}_{\varrho,\ss}  f_R(t) \|_{L_{ x,v}^{\infty} }^{\frac{6-2p}{p}}
\big\}\ll 1 \ \ 
\text{for some }    T_*  \in (0, T].
\Ee
From (\ref{condition:theorem}) and (\ref{condition:close}) we read all the estimates of Proposition \ref{prop:energy}, Proposition \ref{prop:L6}, Proposition \ref{est:Linfty}, Proposition \ref{prop:average}, and Proposition \ref{prop:Linfty_t} in terms of $\mathcal{E} (t)$ and $\mathcal{D} (t)$ as follow. 

From (\ref{Linfty_3D}), for all $0 \leq t \leq T_*$,
\Be\notag \label{Linfty_3D:final}
\sup_{0 \leq s \leq t} \| \mathfrak{w}_{\varrho, \ss} f _R(s) \|_{L^\infty_{x,v}} 
 \lesssim   \ 
 \frac{1}{\e^{1/2} \kappa^{1/2}
 } \sup_{0 \leq s  \leq t} \| \mathbf{P} f_R(s) \|_{L^6_
{x,v}}
+
 \frac{1}{\e^{1/2} \kappa } \sqrt{ \mathcal{D}(t) } + \|  \mathfrak{w}_{\varrho, \ss} f (0)\|_\infty
 +\frac{\e}{\delta}
 \| (\ref{est:f2}) \|_{L^\infty_{t,x}} 
 + 
\e \kappa^2 \| (\ref{est:R2})\|_{L^\infty_{t,x}} .
 \Ee
From (\ref{L6}) and (\ref{Linfty_3D:final}) we derive that 
\Be\notag\label{L6_ED}
\begin{split} 
\|   {P} f_R(t) \|_{L^6_x}  
& \lesssim   
 ( \e   \kappa)^{\frac{1}{2}}    \|  \mathfrak{w}_{\varrho, \ss} f (0)\|_{L^\infty_{x,v}}+  \frac{1}{\kappa^{1/2}} \Big[\sqrt{\mathcal{E}(t)}
 +  \sqrt{\mathcal{D}(t)}\Big]\\
 & \ \ 
 +   
 \frac{\e^{\frac{3}{2}}\kappa^{\frac{1}{2}}  }{\delta}
 \| (\ref{est:f2}) \|_{L^\infty_{t,x}} 
 +  \e^{\frac{3}{2}}   \kappa^{\frac{5}{2}} 
  \| (\ref{est:R2})\|_{L^\infty_{t,x}} 
 +  \frac{\e}{\delta}| (\ref{est:f2})|_{L^4(\p\O)}   + \e\{ \| 
 (\ref{est:R1})  \|_{L^2_{x,v}}+ \|(\ref{est:R2})
 \|_{L^2_{x,v}}\},
\end{split}
\Ee
where we have used $
 \|  (\mathbf{I} - \mathbf{P}) f_R (t) \|_{ {L^2  _{x,v}}}
 \lesssim  
   \e \kappa^{ 1/2} \sqrt{\mathcal{D}(t)}$, $|(1- P_{\gamma_+}) f_R(t)|_{L^2_{\gamma_+}}
  \lesssim    \e^{1/2} \sqrt{\mathcal{D}(t)}$ from (\ref{Sob_1D}). From these two estimates above in $L^\infty$ and $L^6$ we conclude that
 \Be\label{L6Linfty_ED}
\begin{split}
 \sup_{0 \leq s \leq t} \big\{\|   {P} f_R(s) \|_{L^6_x}  + 
(\e \kappa )^{\frac{1}{2}} \| \mathfrak{w}_{\varrho, \ss} f _R(s) \|_{L^\infty_{x,v}}\big\}  
 \lesssim   \e    ^{\frac{1}{2}}  \kappa  \|  \mathfrak{w}_{\varrho, \ss} f (0)\|_{L^\infty_{x,v}} + \kappa^{-\frac{1}{2}}
 \sup_{0 \leq s \leq t}\Big[\sqrt{\mathcal{E}(s)}
 +  \sqrt{\mathcal{D}(s)}
 +   (\ref{L6Linfty:force})\Big],
\end{split}
\Ee
with 
\Be\label{L6Linfty:force}
\begin{split}
& \frac{\e^{\frac{3}{2}}\kappa  }{\delta}
 \| (\ref{est:f2}) \|_{L^\infty_{t,x}} 
 +  \e^{\frac{3}{2}}   \kappa^{3} 
  \| (\ref{est:R2})\|_{L^\infty_{t,x}} 
 +  \frac{\e \kappa^{\frac{1}{2}}}{\delta}| (\ref{est:f2})|_{L^4(\p\O)}   + \e
  \kappa^{\frac{1}{2}}
 \{ \| 
 (\ref{est:R1}) \|_{L^2_{x,v}} + \|(\ref{est:R2})
 \|_{L^2_{x,v}}\}.
\end{split}
\Ee

From (\ref{average_3D}), (\ref{L6Linfty_ED}), and (\ref{condition:close}) and (\ref{condition:theorem})
 \Be\begin{split}\label{average_3D_ED}
 \big\|    {P} f_R
 \big\|_{L^2_t((0,T);L^p_x(\tilde{\O}))}
\lesssim & \ 
 \| f_R (0) \|_{L^2_\gamma} 
+
\Big[ 
\kappa^{-\frac{1}{2}-\ss} \sqrt{\mathcal{E}(t)} + \kappa^{-\frac{1}{2}} \sqrt{\mathcal{D}(T)}
 + (\ref{average_3D_ED:force})\Big],
\end{split}\Ee
with 
\Be\label{average_3D_ED:force}
 \e^{\frac{1}{2}} \kappa^{-\frac{1}{2}} \delta (\ref{L6Linfty:force})
+\kappa^{ \frac{1}{2}}  \| f_R (0) \|_{L^2_\gamma} +
 \e \kappa^{ \frac{1}{2}}  \|
 (\ref{est:R3}) + (\ref{est:R4})
  \|_{L^2_t ((0,T); L_x^{2}(\O))},
\Ee
where we have assumed 
\Be\label{average_3D_ED:assump}
\e^{1/2} \kappa^{-1/2} \delta
+\e^{\frac{1}{2}} \kappa^{-1} \delta +   \e \kappa  \delta^{- \frac{p-2}{2}}  \lesssim 1.
\Ee

We read (\ref{Linfty_3D}), (\ref{Linfty_3D_t}) as follows. Suppose there exists $T>0$ such that 
\Be\begin{split}\label{condition1:close}
  &\frac{\delta}{\kappa} \| Pf_R(t)\|_{L^6(\O)}^{1/2} \| Pf_R (t) \|_{L^2(\O)}^{1/2} 
   + \frac{\delta}{\kappa} \| Pf_R \|_{L_x^{6}(\O)}^{\frac{3(p-2)}{p}} 
\|  \mathfrak{w}_{\varrho,\ss} f_R \|_{L_x^{\infty}(\O)}^{\frac{6-2p}{p}}
+
 ( \e\delta 
+ \e^2
\| (\ref{est:f2}) \|_{L^\infty_{t,x}}
 ) \sup_{0 \leq s \leq t}\| \mathfrak{w}_{\varrho, \ss} f_R (s)  \|_\infty  
  \\
  &  + \e \| u(t) \|_{L^\infty (\O)}  
+
 \frac{1}{N}  +
  \e
\| (\ref{transp:mu})\|_{L^\infty_{t } ([0,T ] ; L^{\frac{2p}{p-2}}_x( \O))}
 +
\e^2 \| (\ref{est:f2})\|_{L^\infty_{t,x}} 
+
 \frac{\e}{\kappa} \| 
 (\ref{est:f2})
  \|_{L^\infty_t ([0,T]; L_x^{\frac{2p}{p-2}}(\O)) } \ll 1,
\end{split}\Ee
and 
\Be\label{condition2:close}
\frac{\e^{1/2} \delta}{\kappa } \mathcal{D}^{1/2} \lesssim 1
\Ee
 and
 \Be
 \e\| (\ref{transp:mu})\|_{\infty}
+  \e  \kappa^{-1} 
 \|(\ref{est:f2})\|_\infty\lesssim \frac{1}{\kappa^{1/2}}
 \Ee

From (\ref{Linfty_3D})
\Be 
\notag
\sup_{0 \leq s \leq t} \| \mathfrak{w}_{\varrho, \ss} f _R(s) \|_\infty   
 \lesssim   \ 
 \frac{1}{\e^{1/2} \kappa^{1/2}
 } \sup_{0 \leq s  \leq t} \| \mathbf{P} f_R(s) \|_{L^6_
{x,v}}
+
 \frac{1}{\e^{1/2} \kappa } \sqrt{ \mathcal{D}(t) } + \|  \mathfrak{w}_{\varrho, \ss} f (0)\|_\infty
 +\frac{\e}{\delta}
 \| (\ref{est:f2}) \|_{L^\infty_{t,x}} 
 + 
\e \kappa^2 \| (\ref{est:R2})\|_{L^\infty_{t,x}} .
 \Ee
Together with (\ref{L6}) we derive that 
 \Be\notag\label{L6_ED}
\begin{split}
& 
\|   {P} f_R(t) \|_{L^6(\O)} \\
\lesssim & \  
 (\e+\e\| (\ref{transp:mu})\|_{\infty}
+  \e  \kappa^{-1} 
 \|(\ref{est:f2})\|_\infty
)
\sqrt{\mathcal{E}(t)}
%
%
  + 
\frac{1}{ \kappa^{1/2} }\big(1
+
 \delta \e   \| \mathfrak{w}_{\varrho,\ss} f_R(t) \|_{L^\infty (\O \times \R^3)}
 \big)
 \big\{
   \|\e^{-1}\kappa^{-1/2}  (\mathbf{I} - \mathbf{P}) f_R (0) \|_{ {L^2 (\O \times \R^3)}}
 +
 \sqrt{\mathcal{D}(t)}\big\}
\\
& + 
 \e^{-1/2} \kappa^{-1/2}
 |(1- P_{\gamma_+}) f_R(t)|_{L^2({\gamma_+})} + 
 o(1) \e^{1/2} \kappa^{1/2} \|   \mathfrak{w}_{\varrho, \ss} f_R(t) \|_{L^\infty (\O \times \R^3)} 
+  \frac{\e}{\delta}| (\ref{est:f2})|_{L^4(\p\O)} 
  + \e \| 
 (\ref{est:R1}) + (\ref{est:R2})
 \|_{L^2_{x,v}}\\
 \lesssim & \ 
 o(1) \|   {P} f_R(t) \|_{L^6(\O)}
 + \frac{1}{\kappa^{1/2}} \Big[\sqrt{\mathcal{E}(t)}
 +  \sqrt{\mathcal{D}(t)}\Big]
 +   
 \frac{\e^{\frac{3}{2}}\kappa^{\frac{1}{2}}  }{\delta}
 \| (\ref{est:f2}) \|_{L^\infty_{t,x}} 
 +  \e^{\frac{3}{2}}   \kappa^{\frac{5}{2}} 
  \| (\ref{est:R2})\|_{L^\infty_{t,x}} 
 +  \frac{\e}{\delta}| (\ref{est:f2})|_{L^4(\p\O)}   + \e \| 
 (\ref{est:R1}) + (\ref{est:R2})
 \|_{L^2_{x,v}}\\
 &
 +  \kappa^{-\frac{1}{2}}   
   \|\e^{-1}\kappa^{-\frac{1}{2}}  (\mathbf{I} - \mathbf{P}) f_R (0) \|_{ {L^2 (\O \times \R^3)}} 
   +( \e   \kappa)^{\frac{1}{2}}    \|  \mathfrak{w}_{\varrho, \ss} f (0)\|_\infty
   + 
 (\e \kappa )^{-\frac{1}{2}}   |(1- P_{\gamma_+}) f_R(0)|_{L^2({\gamma_+})} ,
\end{split}
\Ee
 where we have used, from $a(t)^2=a(0)^2+ \int_0^t 2 a^\prime (s) a(s) \dd s 
 \leq  a(0)^2+\sqrt{ \int^t_0 | a (s)| \dd s} \sqrt{ \int^t_0 | a^\prime (s)| \dd s}$,
 \begin{align*}
 \|  (\mathbf{I} - \mathbf{P}) f_R (t) \|_{ {L^2 (\O \times \R^3)}}
 \lesssim  
  \|  (\mathbf{I} - \mathbf{P}) f_R (0) \|_{ {L^2 (\O \times \R^3)}}
 + \e \kappa^{ 1/2} \sqrt{\mathcal{D}(t)}, \ \
  |(1- P_{\gamma_+}) f_R(t)|_{L^2({\gamma_+})}
  \lesssim     |(1- P_{\gamma_+}) f_R(0)|_{L^2({\gamma_+})} 
  + \e^{1/2} \sqrt{\mathcal{D}(t)}.
 \end{align*}
 
 From these two estimates above we conclude that
 \Be\label{L6Linfty_ED}
\begin{split}
 \sup_{0 \leq s \leq t} \big\{\|   {P} f_R(s) \|_{L^6(\O)}  + 
(\e \kappa )^{\frac{1}{2}} \| \mathfrak{w}_{\varrho, \ss} f _R(s) \|_\infty\big\}  
 \lesssim   \kappa^{-\frac{1}{2}}
 \Big[\sqrt{\mathcal{E}(t)}
 +  \sqrt{\mathcal{D}(t)}
 +   (\ref{L6Linfty:force})\Big],
\end{split}
\Ee
with 
\Be\label{L6Linfty:force}
\begin{split}
& \frac{\e^{\frac{3}{2}}\kappa  }{\delta}
 \| (\ref{est:f2}) \|_{L^\infty_{t,x}} 
 +  \e^{\frac{3}{2}}   \kappa^{3} 
  \| (\ref{est:R2})\|_{L^\infty_{t,x}} 
 +  \frac{\e \kappa^{\frac{1}{2}}}{\delta}| (\ref{est:f2})|_{L^4(\p\O)}   + \e
  \kappa^{\frac{1}{2}}
  \| 
 (\ref{est:R1}) + (\ref{est:R2})
 \|_{L^2_{x,v}}\\
 &
 +    
   \|\e^{-1}\kappa^{-\frac{1}{2}}  (\mathbf{I} - \mathbf{P}) f_R (0) \|_{ {L^2 (\O \times \R^3)}} 
   +  \e    ^{\frac{1}{2}}  \kappa  \|  \mathfrak{w}_{\varrho, \ss} f (0)\|_\infty
   + 
 \e  ^{-\frac{1}{2}}   |(1- P_{\gamma_+}) f_R(0)|_{L^2({\gamma_+})} .
\end{split}
\Ee

\hide
 \Be 
 \begin{split}
  \label{Linfty_3D_ED}
\sup_{0 \leq s \leq t} \| \mathfrak{w}_{\varrho, \ss} f _R(s) \|_\infty   
 \lesssim  & \ 
 \frac{1}{\e^{1/2} \kappa }
 \Big[\sqrt{ \mathcal{E}(t) }+ \sqrt{ \mathcal{D}(t) } \Big]  +\frac{\e}{\delta}
 \| (\ref{est:f2}) \|_{L^\infty_{t,x}} 
 +  \frac{\e^{\frac{1}{2}}}{ \kappa^{\frac{1}{2}}\delta}| (\ref{est:f2})|_{L^4(\p\O)} 
 + 
\e \kappa^2 \| (\ref{est:R2})\|_{L^\infty_{t,x}} 
 + \e^{\frac{1}{2}} \kappa^{-\frac{1}{2}} \| 
 (\ref{est:R1}) + (\ref{est:R2})
 \|_{L^2_{x,v}}
\\
&
 + \e^{-\frac{1}{2}} \kappa^{-1}   
   \|\e^{-1}\kappa^{-\frac{1}{2}}  (\mathbf{I} - \mathbf{P}) f_R (0) \|_{ {L^2 (\O \times \R^3)}} 
   +     \|  \mathfrak{w}_{\varrho, \ss} f (0)\|_\infty
   + 
 (\e \kappa )^{-1}   |(1- P_{\gamma_+}) f_R(0)|_{L^2({\gamma_+})}.
 \end{split}
 \Ee
 \unhide
 
 From (\ref{average_3D}), (\ref{L6Linfty_ED}), and (\ref{condition1:close})
 \Be\begin{split}\label{average_3D_ED}
 \big\|    {P} f_R
 \big\|_{L^2_t((0,T);L^p_x(\tilde{\O}))}
\lesssim & \ 
\kappa^{-\frac{1}{2}}
\Big[
 \sqrt{\mathcal{E}(t)} +  \sqrt{\mathcal{D}(T)}
 + (\ref{average_3D_ED:force})\Big],
\end{split}\Ee
with 
\Be\label{average_3D_ED:force}
 \e^{\frac{1}{2}} \kappa^{-\frac{1}{2}} \delta (\ref{L6Linfty:force})
+\kappa^{ \frac{1}{2}}  \| f_R (0) \|_{L^2_\gamma} +
 \e \kappa^{ \frac{1}{2}}  \|
 (\ref{est:R3}) + (\ref{est:R4})
  \|_{L^2_t ((0,T); L_x^{2}(\O))},
\Ee
where we have assumed 
\Be\label{average_3D_ED:assump}
\e^{1/2} \kappa^{-1/2} \delta
+\e^{\frac{1}{2}} \kappa^{-1} \delta +   \e \kappa  \delta^{- \frac{p-2}{2}}  \lesssim 1.
\Ee

Using (\ref{L6Linfty_ED}) from (\ref{average_3Dt}) and (\ref{Linfty_3D_t}) we deduce that, for $p<3$ but $|p-3|\ll 1$,
\Be\begin{split}\label{average_3Dt_ED}
  \big\|    {P}  \p_t f_R
 \big\|_{L^2_t((0,T);L^p_x(\tilde{\O}))}
 + (\e  \kappa)^{3/p} \| \mathfrak{w} ^\prime  \p_t  f_R  \|_{L^2 ((0,T);L^{\infty} (\O \times \R^3))} 
\lesssim    
\kappa^{-\frac{1}{2}} \Big[
 \sqrt{\mathcal{E}(t)}+ \sqrt{\mathcal{D}(T)}
+ (\ref{average_3Dt_ED:force}) \Big],
\end{split}\Ee
with 
\Be\label{average_3Dt_ED:force}
\begin{split}
& \kappa^{\frac{1}{2}}  \|\p_t  f_R (0) \|_{L^2_\gamma} 
+ 
(\e  \kappa)^{\frac{3}{p}+1} 
\| \mathfrak{w} ^\prime(x,v) \p_t  f_R(0,x,v) \|_{{ L^{\infty}_{x,v}}}
+ \e \kappa^{\frac{1}{2}}\Big[
 (
 \| (\ref{est:R3})\|_{L^2_{t,x}}
+    \| (\ref{est:R4})\|_{L^2_{t,x}})
  + (\e  \kappa)^{\frac{3}{p}+1} ( \|(\ref{est:R3}) \|_\infty+   \|(\ref{est:R4}) \|_\infty)
  \Big]
\\
&
+(\e  \kappa)^{3/p} \frac{\e \kappa^{\frac{1}{2}}  }{\delta}
\Big[
 \|(\ref{est:f2_t})\|_\infty 
 + \e \|\p_t u\|_\infty  \|(\ref{est:f2}) \|_\infty
 \Big] +(\e  \kappa)^{3/p} \e^{\frac{1}{2}}  \big[  \|\p_t u \|_\infty+ \e 
\|(\ref{est:f2_t})\|_\infty+
  \e  \kappa \|(\ref{transp:mu_t}) \|_\infty
 \big]  (\ref{L6Linfty:force}),
\end{split}
\Ee
where we have assumed that 
\Be\label{average_3Dt_ED:assump}
\begin{split}
 \kappa^{ \frac{1}{2}}  + 
 \kappa^{ \frac{1}{2}} \e  (
1
 + \e^{\frac{3}{p}+ \frac{1}{2}} 
 )
  \| (\ref{transp:mu_t})\|_\infty
+ \e(\e  \kappa)^{\frac{3}{p}- \frac{1}{2}}  
   \big(  \|\p_t u \|_\infty+ \e 
\|(\ref{est:f2_t})\|_\infty
 \big)   
  + \e^2 \|   (\ref{est:f2})\|_{L^\infty_{x,v}}  
+ \e^2 \kappa  \| 
(\ref{transp:mu})
 \|_{L^\infty_{t,x} ([0,T ] \times \O)}
  \lesssim 1.
\end{split}
\Ee

\smallskip

\textit{Step 2. } Using all of estimates of the previous step we will close the estimate ultimately in the basic energy estimates (\ref{est:Energy}) and (\ref{est:Energy_t}) via the Gronwall's inequality. 

Applying (\ref{L6Linfty_ED}), (\ref{average_3D_ED}), and (\ref{average_3Dt_ED}) to (\ref{est:Energy}) and (\ref{est:Energy_t}) and collecting the terms, we derive that    
\begin{align}
&\sup_{0 \leq t \leq T}\mathcal{E} (t) + \mathcal{D}(T)\notag\\
\lesssim &\   \big( \|(\ref{transp:mu})\|_{\infty}   + \frac{\e^2}{\delta^2 \kappa^2}
\big)
\int^T_0 \mathcal{E} (s) \dd s
\notag
 \\
&+ \Big(
\frac{\e}{\kappa^2}
+\frac{\delta}{\kappa^{\frac{5}{2}} }\| (\ref{est:f2})\|_{L^\infty_t L^6_x}^2
+ \frac{\e^{2- \frac{6}{p}}e^{- \frac{\varrho}{3 \e^2}} }{\kappa^{\frac{6}{p}}}
\| (\ref{transp:mu}) \|_{L^\infty_{t,x}}^2
+ \|(\ref{transp:mu_t})\|_\infty
+
 \frac{\delta^2}{\kappa^3}\sup_{0 \leq t \leq T}\mathcal{E} (t) 
\Big)\Big[\sup_{0 \leq t \leq T}\mathcal{E} (t)+ \mathcal{D}(T) \Big]
\label{final_E0}
\\
&+ \frac{\delta^2}{\kappa^3} \big\{ (\ref{L6Linfty:force})^4 + (\ref{average_3D_ED:force})^4
+ (\ref{average_3Dt_ED:force})^4
\big\}
+ \frac{\e}{\kappa^2} (\ref{average_3D_ED:force})^2
+
\Big\{\frac{\delta}{\kappa^{\frac{5}{2}} }\| (\ref{est:f2})\|_{L^\infty_t L^6_x}^2
+
 \frac{\e^{2- \frac{6}{p}}e^{- \frac{\varrho}{3 \e^2}} }{\kappa^{\frac{6}{p}}}
\| (\ref{transp:mu}) \|_{L^\infty_{t,x}}^2
\Big\}
 (\ref{average_3Dt_ED:force})^2\label{final_E1}
 \\
 &
+ \kappa  \e^2  
  \|(\ref{est:R1})+ (\ref{est:R2})+ (\ref{est:R3}) + (\ref{est:R4})
  \|^2_{L^2_{t,x}}   
  + \frac{\e}{2\delta^2}  
 | (\ref{est:f2}) |^2_{L^2_t L^2_\gamma}
   +( \frac{\e}{\delta^2} + \frac{\e^3}{\delta}\| \p_t u \|_{L^\infty} )| (\ref{est:f2_t}) |_{L^2_tL^2_{\gamma }}^2+ \mathcal{E}(0).\label{final_E2}
\end{align} 
From (\ref{relat:kappa_delta}) and (??) we bound (\ref{final_E0}) by $o(1) \times \big[\sup_{0 \leq t \leq T}\mathcal{E} (t)+ \mathcal{D}(T) \big]$ which can be absorbed in the left hand side. From (??) there exists a constant $\mathfrak{C}>0$ independent of $\e,\kappa, \delta$, satisfying 
\Be
(\ref{final_E1})+(\ref{final_E2}) \ll \mathfrak{C}.\label{mathfrak_C}
\Ee

From (\ref{mathfrak_C'}) Then we derive that  (we may redefine $\mathcal{E}(T)$ as $\sup_{0 \leq t \leq T}\mathcal{E}(t)$ if necessary)
\Be\label{gronwall}
\mathcal{E}(T)  + \mathcal{D}(T) 
\leq \mathfrak{C} \Big\{ 1 +   \mathfrak{C}^\prime
\kappa^{-\mathfrak{P}} \int^T_0\mathcal{E}(s)  \dd s  \Big\}.
\Ee 
Applying the Gronwall's inequality to (\ref{gronwall}), we derive that $\sup_{0 \leq t \leq T}\mathcal{E} (t)
 \leq 
  \mathfrak{C}  \exp\big(  \mathfrak{C}^\prime \kappa^{-\mathfrak{P}} T\big).$ Using this estimate to the last term of (\ref{gronwall}) we derive that 
\Be 
\sup_{0 \leq t \leq T}\mathcal{E} (t)+   \mathcal{D}(T) \leq 
   \mathfrak{C} 
   \big\{1+  \mathfrak{C} \mathfrak{C}^\prime \kappa^{-\mathfrak{P}} T
    \exp\big(  \mathfrak{C}^\prime \kappa^{-\mathfrak{P}} T\big)
    \big\}\leq   \mathfrak{C} 
   \big\{1+  \mathfrak{C} 
    \exp\big( 2 \mathfrak{C}^\prime \kappa^{-\mathfrak{P}} T\big)
    \big\}.
  \label{est:E}
\Ee
If we choose 
 \Be\label{relat:kappa_delta}
\kappa =
\left(
\frac{ \mathfrak{C}^\prime T}{ \ln \left(
\frac{1}{2\mathfrak{C}} (1+ \frac{1}{\delta^{1/10} })
\right)}
\right)^{1/\mathfrak{P}}
\Ee
then we derive that 
\Be\label{gronwall_final}
\sup_{0 \leq t \leq T} \delta^{\frac{1}{10}}\mathcal{E} (t) + \delta^{\frac{1}{10}}\mathcal{D} (T) \lesssim 1 .
\Ee
Then by the standard continuation argument we can verify all assumptions (\ref{condition1:close}), (\ref{condition2:close}), (\ref{average_3D_ED:assump}), (\ref{average_3Dt_ED:assump}). 
 \end{proof}
 
 \unhide


\hide

\hide
 \Be\notag\label{Linfty_3D_t}
\begin{split}
& 
\\
\lesssim & \ 
  \| P \p_t f  \|_{L^2_t ((0,T);L^p_{x}(\O))}
 + 
  \frac{1}{\kappa^{1/2}}
  \e(\e  \kappa)^{\frac{3}{p}- \frac{1}{2}}  
   \big(  \|\p_t u \|_\infty+ \e 
\|(\ref{est:f2_t})\|_\infty+
  \e  \kappa \|(\ref{transp:mu_t}) \|_\infty
 \big)  
\{
\Big[\sqrt{\mathcal{E}(t)}
 +  \sqrt{\mathcal{D}(t)}\Big]
\}
\\
.
\end{split}
\Ee\unhide

 \Be\label{Linfty_3D_t_ED}
\begin{split}
&
\| \mathfrak{w} ^\prime  \p_t  f_R  \|_{L^2 ((0,T);L^{\infty} (\O \times \R^3))}\\
\lesssim & \ 
\frac{1}{\e^{3/p} \kappa^{3/p}}
\| P \p_t f  \|_{L^2_t ((0,T);L^p_{x}(\O))}
 +  \frac{1}{\e^{1/2} \kappa }
\sqrt{ \mathcal{D}(t)}+
\| \mathfrak{w} ^\prime(x,v) \p_t  f_R(0,x,v) \|_{{L^2_tL^{\infty}_{x,v}}}\\
& + \frac{\e}{\delta}\{ |\p_t p| +|\p_t \tilde{u}| + \kappa |\nabla_x \p_t u |\}
+\e  |\p_t u | \{ \| \mathfrak{w} f_R\|_\infty+ 
 |  p| +|  \tilde{u}| + \kappa |\nabla_x   u |
\}\\  
 & +  \e^2
  \big\{
|\p_t p|  + |\p_t \tilde{u}| + \e |\p_t u| (|p| + |\tilde{u}|) + \kappa (|\nabla_x \p_t u| + \e |\p_t u| |\nabla_x u|)
\big\}\| \mathfrak{w} f_R\|_\infty
 \\
&+  \e^3 \kappa  (|\p_t^2 u| + |u||\nabla_x \p_t u| + |\p_t u| |\nabla_x u| )+ \e^4 \kappa  |\p_t u| (|\p_t u| + |u||\nabla_x u|) \| \mathfrak{w} f_R\|_\infty +  \frac{\e^2 \kappa}{\delta}
\mathfrak{q}(
|\nabla_{x} \p_t \tilde{u}|, 
|\nabla_x^2 \p_t u | 
)
 \\
 &+
  \frac{\e^3}{\delta }\{1+ 
|\p_t^2 p| + |\nabla_x \p_t^2 u| 
\}  
\\
& \ \ \   \ \ \  \times 
 \mathfrak{q} (
|p|, |\nabla_x p|, |\p_t p|,|\nabla_x \p_t p | , |u| , |\nabla_x u|, |\p_t u|, |\nabla_x \p_t u|, |\nabla_x^2 u|, |\nabla_x^2 \p_t u|, 
|\tilde{u}|, |\nabla_x \tilde{u}|,|\p_t \tilde{u}| ,|\nabla_x \p_t \tilde{u}|, |\p_t ^2 \tilde{u}|
).
\end{split}
\Ee

with 
\Be
\begin{split}\label{Linfty:force}
 & \|  \mathfrak{w}_{\varrho, \ss} f (0)\|_\infty
 +\frac{\e}{\delta} \sup_{0 \leq s \leq t} \{\|p(s)\|_\infty + \|\tilde{u}(s)\|_\infty + \kappa \|\nabla_x u(s)\|_\infty   \} 
 \\
 & + 
\frac{\e^2 \kappa }{\delta} 
\sup_{0 \leq s \leq t} \mathfrak{q}( \|u \|_\infty , \|\nabla_{x } u \|_\infty, \|\p_{t}u\|_\infty, \|\nabla_{x}^{2} u\|_\infty, \|\nabla_{x} \p_{t} u\|_\infty, \|p\|_\infty, \|\nabla_{x } p \|_\infty, \|\p_{t} p\|_\infty, \|\tilde{u}\|_\infty, \|\nabla_{x } \tilde{u}\|_\infty, \|\p_{t} \tilde{u}\|_\infty).
\end{split}
\Ee

\Be\label{cond:Linfty_3D}
\Big( \e\delta 
+ \e^2
 \sup_{0 \leq s \leq T}\{
 |p(s)| + |\tilde{u}(s)| + \kappa |\nabla_x u(s)|  
 \}
\Big) \sup_{0 \leq s \leq T}\| \mathfrak{w}_{\varrho, \ss} f_R (s)  \|_\infty  \ll1.
\Ee

\hide
Then 
 \Be\label{Linfty_3D_ED}
 \begin{split}
&\sup_{0 \leq s \leq t} \| \mathfrak{w}_{\varrho, \ss} f _R(s) \|_\infty  \\
 \lesssim &  \ 
 \frac{1}{\e^{1/2} \kappa^{1/2}
 } \sup_{0 \leq s  \leq t} \| \mathbf{P} f_R(s) \|_{L^6_
{x,v}} + \frac{1}{\e^{1/2} \kappa } \sqrt{ \mathcal{D}(t)}
+  \|  \mathfrak{w}_{\varrho, \ss} f (0)\|_\infty
 +\frac{\e}{\delta} \sup_{0 \leq s \leq t} \{\|p(s)\|_\infty + \|\tilde{u}(s)\|_\infty + \kappa \|\nabla_x u(s)\|_\infty   \} 
 \\
 & + 
\frac{\e^2 \kappa }{\delta} 
\sup_{0 \leq s \leq t} \mathfrak{q}( \|u \|_\infty , \|\nabla_{x } u \|_\infty, \|\p_{t}u\|_\infty, \|\nabla_{x}^{2} u\|_\infty, \|\nabla_{x} \p_{t} u\|_\infty, \|p\|_\infty, \|\nabla_{x } p \|_\infty, \|\p_{t} p\|_\infty, \|\tilde{u}\|_\infty, \|\nabla_{x } \tilde{u}\|_\infty, \|\p_{t} \tilde{u}\|_\infty).
\end{split}\Ee
and 
\unhide

Now we read (\ref{L6}), (\ref{average_3D}), and (\ref{average_3Dt}) as follows: Suppose
\Be\label{condition:energy_est}
\begin{split}
\sup_{0 \leq t \leq T} 
\Big\{&
\frac{\delta}{\kappa} \| Pf_R(t)\|_{L^6(\O)}^{1/2} \| Pf_R (t) \|_{L^2(\O)}^{1/2}
+ \frac{\delta}{\kappa} \| Pf_R \|_{L_x^{6}(\O)}^{\frac{3(p-2)}{p}} 
\|  \mathfrak{w}_{\varrho,\ss} f_R \|_{L_x^{\infty}(\O)}^{\frac{6-2p}{p}}
\\
& + 
 \e \| u \|_{L^\infty (\O)} +  \frac{\e}{\kappa} \| \mathfrak{w}_{\varrho, \ss} f_2 \|_{L^\infty_t ((0,T); L_x^{\frac{2p}{p-2}}(\O)) }
+ \e \| 
 |\nabla_x u| + \e  |\p_t u| + \e |u||\nabla_x u|  
 \|_{L^\infty_{t } ([0,T ] ; L^{\frac{2p}{p-2}}_x( \O))}\Big\}
 \ll1.
\end{split}
\Ee   
Then from (\ref{L6}) we have
\Be\label{L6_ED}
\begin{split}
& 
\|   {P} f_R(t) \|_{L^6(\O)} \\
\lesssim & \   \|  (\mathbf{I} - \mathbf{P}) f_R(t)\|_{L^6 (\O \times \R^3)} 
+ 
\Big(
\frac{1}{\e \kappa }+
 \frac{\delta}{\kappa} \| \mathfrak{w}_{\varrho,\ss} f_R(t) \|_{L^\infty (\O \times \R^3)}
 \Big)
\| (\mathbf{I} - \mathbf{P}) f_R (t) \|_{ {L^2 (\O \times \R^3)}}\\
&+\big(
\e \kappa^{-1} (\| p \|_\infty + \| \tilde{u} \|_\infty + \kappa \| \nabla_x u \|_\infty)
+\e^2 \| \p_t u \|_\infty + \e \| \nabla_x u \|_\infty
\big)
\|   f_R(t) \|_{ {L^2 (\O \times \R^3)}}
+ \e \| \p_t f_R(t) \|_{ {L^2 (\O \times \R^3)}} \\
& +  |(1- P_{\gamma_+}) f_R(t)|_{L^2({\gamma_+})}^{1/2} \| \mathfrak{w}_{\varrho, \ss} f_R(t) \|_{L^\infty (\O \times \R^3)}^{1/2}
\\
&+  \frac{\e}{\delta}
\{
|p(t)|_{L^4(\p\O)} + |\tilde{u}(t)|_{L^4(\p\O)} + \kappa  |\nabla_x u(t)|_{L^4(\p\O)} 
\}
  + \frac{\e}{\delta \kappa  } \mathfrak{q}(|u|, |\nabla_{x } u |, |\p_{t}u|, |\nabla_{x}^{2} u|, |\nabla_{x} \p_{t} u|, |p|, |\nabla_{x } p |, |\p_{t} p|, |\tilde{u}|, |\nabla_{x } \tilde{u}|, |\p_{t} \tilde{u}|)\\
 \lesssim  & \ 
\| \mathfrak{w} f_R \|_\infty^{2/3} (\e \kappa^{1/2})^{1/3} \mathcal{D}^{1/6}
+ (\kappa^{-1/2} +\frac{\e \delta}{\kappa^{1/2}} \| w f \|_\infty) \mathcal{D}^{1/2}
+ \big\{\e + \e \kappa^{-1} (\| p \|_\infty + \| \tilde{u} \|_\infty + \kappa \| \nabla_x u \|_\infty)
+\e^2 \| \p_t u \|_\infty + \e \| \nabla_x u \|_\infty\big\} \mathcal{E}^{1/2}\\
&+ \e^{1/4} \mathcal{D}^{1/4} \| wf \|_\infty^{1/2}+  \frac{\e}{\delta}
\{
|p(t)|_{L^4(\p\O)} + |\tilde{u}(t)|_{L^4(\p\O)} + \kappa  |\nabla_x u(t)|_{L^4(\p\O)} 
\}
  + \frac{\e}{\delta \kappa  } \mathfrak{q}(|u|, |\nabla_{x } u |, |\p_{t}u|, |\nabla_{x}^{2} u|, |\nabla_{x} \p_{t} u|, |p|, |\nabla_{x } p |, |\p_{t} p|, |\tilde{u}|, |\nabla_{x } \tilde{u}|, |\p_{t} \tilde{u}|)
\end{split}
\Ee

From (\ref{average_3D}), and (\ref{average_3Dt}) 
\Be\begin{split}\label{average_3D}
&   \big\|    {P} f_R
 \big\|_{L^2_t((0,T);L^p_x(\tilde{\O}))}
 \\
\lesssim & \ 
(1+ \e \big\| |\nabla_x u| + \e |\p_t u| + \e |u| |\nabla_x u| \big\|_{L^2_t ((0,T); L^\infty_x(\O))})
 \| f_R \|_{L^\infty ([0,T];L^2( {\O} \times \R^3))}
\\
&+\Big\{
\frac{1}{\e \kappa }
+ \frac{\delta}{\kappa} \| \mathfrak{w}_{\varrho, \ss} f_R \|_{L^\infty_t ((0,T) \times \O \times \R^3) } +   \| \mathfrak{w}_{\varrho, \ss } f_R (t) \|_{ L^\infty((0,T) \times \tilde{\O} \times \R^3)}^{\frac{p-2}{p}}
\Big\}
 \| (\mathbf{I} - \mathbf{P})f_R \|_{L^2_t ((0,T) \times \O \times \R^3) }
 \\
 &+
 (\ref{average_3D:force}) ,
\end{split}\Ee
with 
\Be\label{average_3D:force}
 \| f_R (0) \|_{L^2_\gamma} + \frac{  \e  }{\delta } \|\mathfrak{q}(|\nabla_x \tilde{u}|, |\nabla_x^2 u|)\| _{L^2_t ((0,T); L_x^{2}(\O))}
  +  \frac{\e^2 }{\delta \kappa  } \| \mathfrak{q}(|u|, |\nabla_{x } u |, |\p_{t}u|, |\nabla_{x}^{2} u|, |\nabla_{x} \p_{t} u|, |p|, |\nabla_{x } p |, |\p_{t} p|, |\tilde{u}|, |\nabla_{x } \tilde{u}|, |\p_{t} \tilde{u}|) \|_{L^2_t ((0,T); L_x^{2}(\O))}
\Ee

and
\Be\begin{split}\label{average_3Dt}
&  \big\|    {P}  \p_t f_R
 \big\|_{L^2_t((0,T);L^p_x(\tilde{\O}))}
 \\
\lesssim & \  (\kappa \e ) ^{\frac{2}{p-2}} \|  \mathfrak{w} \p_t  f_R \|_{L^2_t ((0,T);L^\infty_{x,v} (  \O \times \R^3))}+
\| \p_t f_R \|_{L^\infty ([0,T]; L^2(\O\times \R^3))}
+ \|\p_t  f_R (0) \|_{L^2_\gamma}\\
&+
\Big\{\frac{1}{\kappa \e} + \frac{\delta}{\kappa} \|   wf_R\|_{L^\infty_{x,v}} +  \frac{\e}{\kappa} \|   wf_2\|_{L^\infty_{x,v}}  + \e \| 
 |\nabla_x u| + \e  |\p_t u| + \e |u||\nabla_x u|  
 \|_{L^\infty_{t,x} ([0,T ] \times \O)}
\Big\}
\| \sqrt{\nu} (\mathbf{I} - \mathbf{P}) \p_t f_R \|_{L^2_{t,x,v}}\\
&+ \e ^2 \{(\| \p_t ^2 u \|_\infty + \| u \|_\infty \| \nabla_x \p_t u \|_\infty + \| \p_t u \|_\infty \| \nabla_x u \|_\infty) + \e  \| \p_t u \|_\infty (\| \p_t u \|_\infty + \| u \|_\infty \| \nabla_x u \|_\infty) 
\}\| f_R\|_{L^\infty_t L^2_{x,v}} 
\\
 &+ \frac{ \e }{\delta} \|\mathfrak{q}(|\nabla_x  \p_t \tilde{u}|, |\nabla_x^2 \p_t u|)\|_{L^2_{t,x}}  \\
& +
\frac{\e^2}{\delta \kappa}
\big\|
(1+ 
|\p_t^2 p| + |\nabla_x \p_t^2 u| 
) 
\\
 & \ \ \   \times 
 \mathfrak{q} (
|p|, |\nabla_x p|, |\p_t p|,|\nabla_x \p_t p | , |u| , |\nabla_x u|, |\p_t u|, |\nabla_x \p_t u|, |\nabla_x^2 u|, |\nabla_x^2 \p_t u|, 
|\tilde{u}|, |\nabla_x \tilde{u}|,|\p_t \tilde{u}| ,|\nabla_x \p_t \tilde{u}|, |\p_t ^2 \tilde{u}|
) 
 \big\|_{L^2_{t,x,v}} .
\end{split}\Ee

We conclude that if $\frac{\e^{1/2} \delta}{\kappa } \mathcal{D}^{1/2}\ll 1$

and then 
\Be\label{L6_ED:last}
\begin{split} 
\|   {P} f_R(t) \|_{L^6(\O)}  
 \lesssim  &\   
 \kappa^{-1/2} (1+ \kappa^{1/2} + \e \delta + \frac{\e^{1/2} \delta}{\kappa^{1/2}} \mathcal{D}^{1/2}) \mathcal{D}^{1/2} \\ 
&
+ \big\{\e + \e \kappa^{-1} (\| p \|_\infty + \| \tilde{u} \|_\infty + \kappa \| \nabla_x u \|_\infty)
+\e^2 \| \p_t u \|_\infty + \e \| \nabla_x u \|_\infty\big\} \mathcal{E}^{1/2}+ \e^{1/2} \kappa^{1/2} (\ref{Linfty:force})+ (\ref{L6:force}),
\end{split}
\Ee
with 
\Be
\begin{split}\label{L6:force}
 \frac{\e}{\delta}
\{
|p(t)|_{L^4(\p\O)} + |\tilde{u}(t)|_{L^4(\p\O)} + \kappa  |\nabla_x u(t)|_{L^4(\p\O)} 
\}
  + \frac{\e}{\delta \kappa  } \mathfrak{q}(|u|, |\nabla_{x } u |, |\p_{t}u|, |\nabla_{x}^{2} u|, |\nabla_{x} \p_{t} u|, |p|, |\nabla_{x } p |, |\p_{t} p|, |\tilde{u}|, |\nabla_{x } \tilde{u}|, |\p_{t} \tilde{u}|).
\end{split}
\Ee
We further derive 
\Be\begin{split}\label{average_3D_ED}
&   \big\|    {P} f_R
 \big\|_{L^2_t((0,T);L^p_x(\tilde{\O}))}
 \\
\lesssim & \ 
(1+ \e \big\| |\nabla_x u| + \e |\p_t u| + \e |u| |\nabla_x u| \big\|_{L^2_t ((0,T); L^\infty_x(\O))})
\sup_{t \in [0,T]}  \sqrt{\mathcal{E}(t)}
\\
&+\Big\{
\frac{1}{  \kappa^{1/2} }
+  \frac{\e\delta}{\kappa^{1/2}} \| \mathfrak{w}_{\varrho, \ss} f_R \|_{L^\infty_t ((0,T) \times \O \times \R^3) } + \e \kappa^{1/2}  \| \mathfrak{w}_{\varrho, \ss } f_R (t) \|_{ L^\infty((0,T) \times \tilde{\O} \times \R^3)}^{\frac{p-2}{p}}
\Big\} 
\sqrt{\mathcal{D}(T)}
 \\
 &+
 (\ref{average_3D:force}) \\
 \lesssim & \ 
\sup_{t \in [0,T]}  \sqrt{\mathcal{E}(t)}
 +
\frac{1}{  \kappa^{1/2} } \sqrt{\mathcal{D}(T)}
 +
 (\ref{average_3D:force})
\end{split}\Ee

From 
\Be\label{Linfty_3D_t_ED1}
\begin{split}
&
\| \mathfrak{w} ^\prime  \p_t  f_R  \|_{L^2 ((0,T);L^{\infty} (\O \times \R^3))}\\
\lesssim & \ 
\frac{1}{\e^{3/p} \kappa^{3/p}}
\| P \p_t f  \|_{L^2_t ((0,T);L^p_{x}(\O))}
 +  \frac{1}{\e^{1/2} \kappa }
\sqrt{ \mathcal{D}(t)}\\
& + 
\Big\{
\e  |\p_t u | 
+ \e^2
  \big[
|\p_t p|  + |\p_t \tilde{u}| + \e |\p_t u| (|p| + |\tilde{u}|) + \kappa (|\nabla_x \p_t u| + \e |\p_t u| |\nabla_x u|)
\big]
+ \e^4 \kappa  |\p_t u| (|\p_t u| + |u||\nabla_x u|) 
\Big\}
\| \mathfrak{w} f_R\|_\infty  \\
&+ (\ref{Linfty:force}) 
\end{split}
\Ee
with 
\Be
\begin{split}\label{Linfty:force}
 &
\| \mathfrak{w} ^\prime(x,v) \p_t  f_R(0,x,v) \|_{{L^2_tL^{\infty}_{x,v}}}
+ \frac{\e}{\delta}\{ |\p_t p| +|\p_t \tilde{u}| + \kappa |\nabla_x \p_t u |\}
+\e  |\p_t u |   \{  
 |  p| +|  \tilde{u}| + \kappa |\nabla_x   u |
\}
 \\
 &+  \e^3 \kappa  (|\p_t^2 u| + |u||\nabla_x \p_t u| + |\p_t u| |\nabla_x u| )
 +  \frac{\e^2 \kappa}{\delta}
\mathfrak{q}(
|\nabla_{x} \p_t \tilde{u}|, 
|\nabla_x^2 \p_t u | 
)
  \\
 &+
  \frac{\e^3}{\delta }\{1+ 
|\p_t^2 p| + |\nabla_x \p_t^2 u| 
\}  
\\
& \ \ \   \ \ \  \times 
 \mathfrak{q} (
|p|, |\nabla_x p|, |\p_t p|,|\nabla_x \p_t p | , |u| , |\nabla_x u|, |\p_t u|, |\nabla_x \p_t u|, |\nabla_x^2 u|, |\nabla_x^2 \p_t u|, 
|\tilde{u}|, |\nabla_x \tilde{u}|,|\p_t \tilde{u}| ,|\nabla_x \p_t \tilde{u}|, |\p_t ^2 \tilde{u}|
).
\end{split}
\Ee

From
\Be\begin{split}\label{average_3Dt_ED}
&  \big\|    {P}  \p_t f_R
 \big\|_{L^2_t((0,T);L^p_x(\tilde{\O}))}
 \\
\lesssim & \  (\kappa \e ) ^{\frac{2}{p-2}} \|  \mathfrak{w} \p_t  f_R \|_{L^2_t ((0,T);L^\infty_{x,v} (  \O \times \R^3))}\\
&
+ \Big(
1  + \e ^2 \{(\| \p_t ^2 u \|_\infty + \| u \|_\infty \| \nabla_x \p_t u \|_\infty + \| \p_t u \|_\infty \| \nabla_x u \|_\infty) + \e  \| \p_t u \|_\infty (\| \p_t u \|_\infty + \| u \|_\infty \| \nabla_x u \|_\infty) 
\}
\Big)
\sup_{0 \leq t \leq T} \sqrt{\mathcal{E} (t)}
\\
&+
\Big\{\frac{1}{\kappa \e} + \frac{\delta}{\kappa} \|   wf_R\|_{L^\infty_{x,v}} +  \frac{\e}{\kappa} \|   wf_2\|_{L^\infty_{x,v}}  + \e \| 
 |\nabla_x u| + \e  |\p_t u| + \e |u||\nabla_x u|  
 \|_{L^\infty_{t,x} ([0,T ] \times \O)}
\Big\}
\sqrt{\mathcal{D}(t)}\\
& + (\ref{average_3Dt:force})
\end{split}\Ee
with 
\Be\label{average_3Dt:force}
\begin{split}
 &  \|\p_t  f_R (0) \|_{L^2_\gamma}+ \frac{ \e }{\delta} \|\mathfrak{q}(|\nabla_x  \p_t \tilde{u}|, |\nabla_x^2 \p_t u|)\|_{L^2_{t,x}}  \\
& +
\frac{\e^2}{\delta \kappa}
\big\|
(1+ 
|\p_t^2 p| + |\nabla_x \p_t^2 u| 
) 
\\
 & \ \ \   \times 
 \mathfrak{q} (
|p|, |\nabla_x p|, |\p_t p|,|\nabla_x \p_t p | , |u| , |\nabla_x u|, |\p_t u|, |\nabla_x \p_t u|, |\nabla_x^2 u|, |\nabla_x^2 \p_t u|, 
|\tilde{u}|, |\nabla_x \tilde{u}|,|\p_t \tilde{u}| ,|\nabla_x \p_t \tilde{u}|, |\p_t ^2 \tilde{u}|
) 
 \big\|_{L^2_{t,x,v}} .
\end{split}
\Ee
Then 
\Be\begin{split}\notag
&
\Big(
1-  (\kappa \e)^{\frac{2}{p-2}- \frac{3}{p}}
\Big)
  \big\|    {P}  \p_t f_R
 \big\|_{L^2_t((0,T);L^p_x(\tilde{\O}))}
 \\
\lesssim & \  (\kappa \e ) ^{\frac{2}{p-2}} 
 \frac{1}{\e^{1/2} \kappa }
\sqrt{ \mathcal{D}(t)}\\
&+ (\kappa \e ) ^{\frac{2}{p-2}} \Big\{
\e  |\p_t u | 
+ \e^2
  \big[
|\p_t p|  + |\p_t \tilde{u}| + \e |\p_t u| (|p| + |\tilde{u}|) + \kappa (|\nabla_x \p_t u| + \e |\p_t u| |\nabla_x u|)
\big]
+ \e^4 \kappa  |\p_t u| (|\p_t u| + |u||\nabla_x u|) 
\Big\}
\Big\{ \frac{1}{\e^{1/2} \kappa } \sqrt{ \mathcal{D}(t) } + (\ref{Linfty:force}) \big\}\\
&+(\kappa \e ) ^{\frac{2}{p-2}} 
(\ref{Linfty:force})
\\
&
+ \Big(
1  + \e ^2 \{(\| \p_t ^2 u \|_\infty + \| u \|_\infty \| \nabla_x \p_t u \|_\infty + \| \p_t u \|_\infty \| \nabla_x u \|_\infty) + \e  \| \p_t u \|_\infty (\| \p_t u \|_\infty + \| u \|_\infty \| \nabla_x u \|_\infty) 
\}
\Big)
\sup_{0 \leq t \leq T} \sqrt{\mathcal{E} (t)}
\\
&+
\Big\{\frac{1}{\kappa \e} + \frac{\delta}{\kappa} \|   wf_R\|_{L^\infty_{x,v}} +  \frac{\e}{\kappa} \|   wf_2\|_{L^\infty_{x,v}}  + \e \| 
 |\nabla_x u| + \e  |\p_t u| + \e |u||\nabla_x u|  
 \|_{L^\infty_{t,x} ([0,T ] \times \O)}
\Big\}
\sqrt{\mathcal{D}(t)}\\
& + (\ref{average_3Dt:force})\\
\lesssim&
 \ \frac{1}{\kappa \e } \sqrt{\mathcal{D}(t)} + \sup_{0 \leq t \leq T} \sqrt{\mathcal{E}(t)}
 +(\kappa \e ) ^{\frac{2}{p-2}} 
(\ref{Linfty:force}) + (\ref{average_3Dt:force})
\end{split}\Ee

\hide

\Be\begin{split}
\| f_R (t) \|_{2}^2 \lesssim e^{(1+ \| \nabla_x u \|_{L^\infty_{t,x}})t } \Big\{&
\| f_R (0)\|_2^2  
  +
   \Big(\frac{\e}{\kappa^2}+  \frac{ \delta^2}{\kappa^{3 }}
\|\kappa^{1/2}  w \mathbf{P} f_R \|_{L^\infty_tL^6_{x,v}}^2  
\Big)
\| \kappa^{1/2} w \mathbf{P} f_R \|_{L^2_tL^3_{x,v}}^2+ \frac{\kappa  \e^2 }{\delta^2} \|\mathfrak{q}(|\nabla_x \tilde{u}|, |\nabla_x^2 u|)\|^2_{L^2_{t,x}} \\
& +  \frac{\e^4}{\delta^2 \kappa } \| \mathfrak{q}(|u|, |\nabla_{x } u |, |\p_{t}u|, |\nabla_{x}^{2} u|, |\nabla_{x} \p_{t} u|, |p|, |\nabla_{x } p |, |\p_{t} p|, |\tilde{u}|, |\nabla_{x } \tilde{u}|, |\p_{t} \tilde{u}|) \|_{L^2_{t,x}}^2
\Big\}.
\end{split}\Ee
Recall the assumption $\| \nabla_x u \|_{L^\infty_{t,x}}\lesssim \kappa^{-\mathfrak{p}}$ in (). For any fixed $T>0$ then 
\Be
\sup_{0 \leq t \leq T}\| f_R (t) \|_{2} \lesssim  e^{T/\kappa^p}\| f_R(0) \|_2 
\Ee
We choose 
\Be
\kappa = \min \Big\{\frac{1}{T}, \frac{1}{T\ln (\ln \frac{1}{\e})}\Big\}^{1/\mathfrak{p}}.
\Ee
Then $e^{T \kappa^{-\mathfrak{p}}} \leq e^{ \min \{1,  \ln (\ln \frac{1}{\e})\}} \leq
\min\{ e, 
 \ln \frac{1}{\e}\}$. It follows that 
 \Be
\sup_{0 \leq t \leq T}\|  \frac{1}{\ln \frac{1}{\e}}  f_R (t) \|_{2} \lesssim1. 
 \Ee
 From (\ref{est:Energy})
 \Be
\| \kappa^{- \frac{1}{2}} \e^{-1} \sqrt{\nu} (\mathbf{I} - \mathbf{P}) f_R \|_2^2
\lesssim   T^2 \ln (\ln \frac{1}{\e})  (\ln \frac{1}{\e})^2  \lesssim _T ( \ln \frac{1}{\e}  )^3
 \Ee
 and hence 
 \Be
 \left\| \kappa^{- \frac{1}{2}} \frac{1}{(\ln \frac{1}{\e})^{3/2}} \e^{-1} \sqrt{\nu} (\mathbf{I} - \mathbf{P}) f_R \right\|_2^2 \lesssim 1
 \Ee

 \unhide

  \unhide

 \section{Navier-Stokes approximations of the Euler equations}\label{sec:5}
 In this section we prove Theorem \ref{thm_bound}. The proof of the theorem relies on the integral representation of the solution to the Navier-Stokes equations using the Green's function for the Stokes problem in the same spirit of \cite{NN2018}. 

 \subsection{Elliptic estimates and Nonlinear estimates
}
 
 In this section, we prove the estimates of the solutions of incompressible Navier-Stokes equations in large Reynolds numbers with the no slip boundary condition satisfying \eqref{NS_k}-\eqref{noslip} based on recent Green's function approach using the vorticity formulation of \eqref{NS}-\eqref{NSB} applied to the inviscid limit problem \cite{Mae14, NN2018, KVW, FW}. An advantage of working with analytic function spaces is the Cauchy estimates useful for recovery of the loss of derivatives. We recall the spaces, norms, and terminology we have defined in Section \ref{sec:MR}. 

\begin{lemma}[\cite{NN2018,FW}, Embeddings and Cauchy estimates] The following holds\label{lem_embedding}
\begin{enumerate}
\item $\mathfrak B^{\lambda, \kappa t} \subset \mathfrak{L}^{1,\lambda}$ and 
$\mathfrak B^{\lambda, \kappa } \subset \mathfrak{L}^{1,\lambda}$. 
\item $\| g_1 g_2\|_{\ast,\lambda}\lesssim \|g_1\|_{\infty,\lambda} \|g_2\|_{\ast,\lambda}$. 
\item $ \sum_{|\beta|=1} \|D^\beta g\|_{\ast,\lambda}   
\lesssim \frac{\|g\|_{\ast,\tilde{\lambda}}}{\tilde\lambda -\lambda}$, for any $0< \lambda <\tilde\lambda$. 
\end{enumerate}
For (2) and (3), $\|\cdot \|_{\ast,\lambda}$ can be either $\|\cdot\|_{\infty,\lambda,\kappa}$ or $  \|\cdot \|_{\infty,\lambda,\kappa t}$ or $\| \cdot \|_{\infty, \lambda , 0}$ or $ \|\cdot\|_{1,\lambda}$. 
\end{lemma}



\begin{lemma}[\cite{NN2018,FW}, Elliptic estimates] \label{lem_elliptic} Let $\phi$ be the solution of $-\Delta\phi = \o$ with the zero Dirichlet boundary condition, and let $u=\nabla \times \phi$. Then 
\Be
\begin{split}\label{est:elliptic}
\| u  \|_{\infty,\lambda}
+ \|\nabla u \|_{1,\lambda} 
 &\lesssim \|\o\|_{1,\lambda} , \\ 
\| \nabla_h u  \|_{\infty,\lambda} + \| \nabla u_3\|_{\infty,\lambda} 
  &\lesssim \sum_{0 \leq |\beta| \leq1}
\| \nabla_h^\beta\o\|_{1,\lambda} , \\
\|   \p_3 u _h\|_{\infty, \lambda }  &\lesssim  \sum_{0 \leq |\beta|\leq 1}
\| \nabla_h^\beta\o\|_{1,\lambda}
+ \|   \o_h \|_{\infty, \lambda },\\
\| \zeta^{-1} \nabla_h^{\beta^\prime} u_3 \|_{\infty, \lambda} &\lesssim \sum_{0 \leq |\beta|\leq 1} \|   \nabla_h ^{\beta+ \beta^\prime} \o_h \|_{1, \lambda}.
\end{split}
\Ee 
\end{lemma}

\begin{proof}
Here we only sketch the proofs. For full justification we refer to Proposition 2.3 in \cite{NN2018} for 2D and Section 4 of \cite{FW} for 3D and the proofs therein.  From $(|\xi|^2   - \p_z^2 )\phi_\xi = \o_\xi$ and $\phi_\xi (0)=0$ we write 
\Be\label{phi_xi}
\begin{split}
\phi_\xi(z)=& \int^z_0 G_- (y,z) \o_\xi (y) \dd y + \int^\infty_z G_+ (y,z) \o_\xi (y) \dd y, \\
&\text{with} \  \ G_{\pm} (y,z) := \frac{-1}{2 |\xi|} \Big( e^{\pm |\xi| (z-y)} -e^{- |\xi | (y+z)}\Big).
\end{split}\Ee
The first two estimates of (\ref{est:elliptic}) can be easily derived from this explicit form. For the third estimate of (\ref{est:elliptic}), we write $u_1 = \p_2 (-\Delta)^{-1}\o_3 - \p_3 (-\Delta)^{-1} \o_2$ and $\p_3u_1 =\p_3 \p_2 (-\Delta)^{-1}\o_3 -\p_3 \p_3 (-\Delta)^{-1} \o_2$.  Then the third estimate of (\ref{est:elliptic}) follows from the identity  
\Be
\begin{split}
 \p_z (\p_3 (-\Delta)^{-1} \o_2)_\xi = \frac{1}{2} \bigg(& \int^z_0 |\xi| e^{-|\xi| (z-y)} (1- e^{-2 |\xi| y}) \o_{ \xi,2} (s, y) \dd y \\
 &+ \int^\infty_z |\xi| e^{-|\xi| (y-z)} (1+ e^{-2 |\xi| z}) \o_{ \xi,2 } (x, y) \dd y \\&
+ \int^\infty_z (-2 |\xi|) e^{- |\xi| (y-x)} e^{-2 |\xi|z } \o_{ \xi,2 } (s, y) \dd y 
\bigg)-  \o_{\xi,2}(z).\notag
\end{split}
\Ee
Next we prove the last estimate. Note that 
\Be
\begin{split}\notag
\frac{1+z}{z} \nabla_ h u_3 (z)  =& \frac{1 }{z} \int^z_0 \p_y \nabla_h u_3 (x_h, y) \dd y  + \nabla_h u_3 (z)\\
&
 = \frac{1 }{z} \int^z_0   \nabla_h 
\big(\p_1\p_3 (-\Delta)^{-1}   \o_2
- \p_2 \p_3 (-\Delta)^{-1}\o_1\big)
 (x_h, y) \dd y + \nabla_h u_3 (z).
\end{split}
\Ee
From (\ref{phi_xi}) we read that for $i=1,2$
\Be
\begin{split}\notag
&\Big||\xi|^{|\beta|}(\p_3 (-\Delta)^{-1} \o_i)_\xi (s,z)\Big|\\
&\leq  \frac{1}{2}
 \Big(
 \int^z_0 e^{-|\xi| (z-y)} (1- e^{- 2|\xi| y})|\xi|^{|\beta|}| \o_{\xi, i} (s,y)| \dd y
 + \int^\infty_z e^{-|\xi| (y-z)} (1+ e^{- 2 |\xi| z})  |\xi|^{|\beta|}| \o_{\xi, i} (s,y) |\dd y
 \Big) \\
 &\lesssim 
\sup_{0 \leq \sigma< \lambda} \big\| |\xi|^{|\beta|} \o_{\xi ,h} \big\|_{L^1 (\p\mathcal{H}_\sigma)}.
 \end{split}
\Ee
From the identity and estimate above we conclude the last bound of (\ref{est:elliptic}).
\end{proof}

\bigskip


As a consequence of Lemma \ref{lem_elliptic}, we have the following nonlinear estimates. 


\begin{lemma}[\cite{NN2018,FW}]\label{lem_bilinear} Let $u$ and $\tilde u$ be the velocity field associated with $\o= \nabla_x \times u$ and $\tilde\o= \nabla_x \times \tilde u$ respectively. Then 
\Be
\begin{split}\label{nonlinear1}
\| u \cdot \nabla \tilde\o \|_{1,\lambda }&\lesssim \|\o\|_{1,\lambda}  \|\nabla_h \tilde\o \|_{1,\lambda} + 
\| (1+ |\nabla_h|) \o \|_{1, \lambda}
\| \zeta \p_z \tilde \o \|_{1,\lambda}, \\
\|\o \cdot \nabla \tilde u_3\|_{1, \lambda} &\lesssim \| \o_h \|_{1, \lambda} \| \nabla_h \tilde{u}_3 \|_{\infty, \lambda} + \| \o_3 \|_{1,\lambda} \| \p_3 \tilde{u}_3 \|_{\infty, \lambda}
\lesssim \| \o \|_{1,\lambda} 
\| (1+ |\nabla_h|) \tilde \o \|_{1, \lambda}
 ,  \\
\|\o \cdot \nabla \tilde u_h\|_{1, \lambda} &\lesssim \| \o_h \|_{1, \lambda} \| \nabla_h \tilde{u}_h \|_{\infty, \lambda} + \| \o_3 \|_{\infty,\lambda} \| \p_3 \tilde{u}_h \|_{1, \lambda}
\lesssim \| \o \|_{1,\lambda} \big(
\| \tilde \o_3 \|_{\infty, \lambda} + 
\| (1+ |\nabla_h|) \o \|_{1, \lambda}
\big)
 .
\end{split}
\Ee

 Moreover
\Be\label{nonlinear2}
\begin{split}
\| u \cdot \nabla \tilde\o_h \|_{*,\lambda }  &\lesssim \|\o\|_{1,\lambda} \|\nabla_h \tilde \o_h\|_{*,\lambda} + 
\big(
\| (1+ |\nabla_h|) \o \|_{1, \lambda}
+ \| \zeta \p_z  \o_3 \|_{\infty, \lambda}
\big)   \| \zeta \p_z \tilde \o_h\|_{*,\lambda}, \\
\| \o \cdot \nabla \tilde u_h \|_{*,\lambda }
&\lesssim 
\| \o_3 \|_{\infty, \lambda, 0}
\big( 
\| (1+ |\nabla_h|)\tilde \o \|_{1, \lambda}
+ \|  \tilde  \o_h \|_{*, \lambda } \big)
+ \|\o_h \|_{*,\lambda }
 \sum_{0\leq |\beta|\leq 1}
\| \nabla_h^\beta\tilde \o\|_{1,\lambda} 
,
\end{split}
\Ee
where $\| \cdot \|_{*, \lambda}$ can be either $\| \cdot \|_{\infty, \lambda , \kappa}$ or $\| \cdot \|_{\infty, \lambda , \kappa t}$.

Furthermore 
\Be
\begin{split}\label{nonlinear2_3}
\| u \cdot \nabla \tilde\o_3 \|_{\infty,\lambda,0} 
 &\lesssim \|\o\|_{1,\lambda} \|\nabla_h\tilde\o_3\|_{\infty,\lambda,0} + 
\| (1+ |\nabla_h|) \o \|_{1, \lambda}
    \| \zeta \p_3 \tilde \o_3\|_{\infty,\lambda,0 },\\
\| \o  \cdot \nabla  \tilde u_3 \|_{\infty,\lambda,0} &\lesssim 
   \| \o_h \|_{*, \lambda } 
   \| (1+ |\nabla_h|^2) \tilde \o_h \|_{1, \lambda}
+ \| \o_3 \|_{\infty, \lambda, 0}  
\| (1+ |\nabla_h|) \tilde \o_h \|_{1, \lambda},
\end{split}
\Ee
where $\| (1+ |\nabla_h|^k) g \|_* = \sum_{\ell=0}^k \| \nabla_h^\ell g \|_*$. 
\end{lemma}

\begin{proof}
 Again we refer to Proposition 2.3 in \cite{NN2018} for 2D and Section 4 of \cite{FW} for the full justification. The bounds (\ref{nonlinear1}) and (\ref{nonlinear2}) directly follow from Lemma \ref{lem_elliptic}. The proof of the first estimate of (\ref{nonlinear2_3}) is an outcome of applying (\ref{est:elliptic}) to an easy bound 
\Be
\| u \cdot \nabla \tilde\o_3 \|_{\infty,\lambda,0} 
 \lesssim \| u_h \|_{\infty, \lambda} \| \nabla_h \tilde{\o}_3 \|_{\infty, \lambda,0}
+  \| \zeta(z) ^{-1} u_3 
 \|_{\infty, \lambda}
\| \zeta(z)  \p_3 \tilde{\o}_3 \|_{\infty, \lambda,0}.\notag
\Ee 
For the second estimate of (\ref{nonlinear2_3}) it suffices to prove the bound for $\o_h \cdot \nabla_h \tilde u_3$. From $|\zeta(z)(1+ \phi_\kappa (z)
 )|\lesssim 1$ or $|\zeta(z)(1+ \phi_\kappa (z)
+ \phi_{\kappa t} (z))|\lesssim 1$,
\begin{align*}
\| \o_h \cdot  \nabla_h \tilde u_3 \|_{\infty, \lambda, 0} & \lesssim \|  \o_h   \| _{*, \lambda } \Big\|  \zeta(z)(1+ \phi_\kappa (z)
+ \phi_{\kappa t} (z)
)  \frac{\nabla_h \tilde u_3}{\zeta(z)} \Big\|_{\infty, \lambda}
\lesssim \|  \o_h   \| _{*, \lambda }
 \|
  \zeta^{-1}{\nabla_h \tilde u_3} 
 \|_{\infty, \lambda}.
\end{align*} 
Then we use the last bound of (\ref{est:elliptic}) to finish the proof. 
\end{proof}

 We finally record the crucial estimate of nonlinear forcing terms $N=-u\cdot\nabla \o+ \o \cdot \nabla u$, as an outcome of {Lemma} \ref{lem_bilinear}, that will be also crucially used to control  $B= [\p_{x_3} (-\Delta)^{-1} (-u \cdot \nabla \o + \o \cdot \nabla u ) ] \, |_{x_3=0}$ in the vorticity formulation  \eqref{NS} and \eqref{NSB}. 

 \begin{lemma}[\cite{NN2018,FW}, Nonlinear estimate]\label{lem_est:N} {Let $\lambda\in(0,\lambda_0-\gamma s)$ be given. We have the following: }
  \Be
 \begin{split}\label{est:N_1}
\| (1+ |\nabla_h |) N \|_{1, \lambda }&\lesssim\big( \| (1+ |\nabla_h| ) \o \|_{1, \lambda} 
+ \|(1+ |\nabla_h| ) \o_3\|_{\infty, \lambda, 0}\big)
 \| (1+ |\nabla_h|^2 ) \o \|_{1, \lambda} \\
& \ \  + \sum_{|\beta|=1}
 \| (1+ |\nabla_h|) D^\beta \o\|_{1, \lambda}  \| (1+ |\nabla_h|^2) \o \|_{1, \lambda} ,
 \end{split}
\Ee
\Be
\begin{split}\label{est:DN_1}
&\sum_{|\beta|=1}\|D^\beta (1+ |\nabla_h |) N \|_{1, \lambda }\\
&\lesssim
 \sum_{|\beta  | \leq 1}  \| D^\beta (1+ |\nabla_h| ) \o \|_{1, \lambda} 
 \bigg(
  \sum_{|\beta| \leq 2}\| D^\beta(1+ |\nabla_h| ) \o\|_{1,\lambda})
  + \| (1+ |\nabla_h|) \o \|_{\infty, \lambda, 0}
  \bigg)
  \\
  & \ \ + \sum_{|\beta | \leq 1} \| D^\beta (1+ |\nabla_h| ) \o_3 \|_{\infty, \lambda, 0}
  \| (1+ |\nabla_h|)^2 \o \|_{1,\lambda}. 
\end{split}
\Ee
 
 For $[[ \ \cdot \ ]]_{*,\lambda}$ to be either $[[ \ \cdot \ ]]_{\infty, \lambda, \kappa }$ or $[[ \ \cdot \ ]]_{\infty, \lambda, \kappa t}$,
\Be\label{est:N_infty}
[[N]]_{*,\lambda}
\lesssim \| (1+ |\nabla_h|^2) \o \|_{1, \lambda}
[[\o]]_{*,\lambda}
+
 \| (1+ |\nabla_h| )\o \|_{1, \lambda}
 [[
 D \o 
 ]]_{*,\lambda},
\Ee
\Be\label{est:DN_infty}
\begin{split}
 \sum_{|\beta|=1}
[[
D^\beta N
]]_{*,\lambda}
&\lesssim \sum_{|\beta|=1}  \| (1+ |\nabla_h|^{|\beta_h|+2}) \o \|_{1, \lambda}
[[ \o ]]_{*,\lambda}\\
& \ \ + \sum_{|\beta|=1} [[D^\beta \o ]]_{*,\lambda} (\| (1+ |\nabla_h|^2) \o \|_{1, \lambda} + \beta_3[[D_3^{\beta_3} \o ]]_{*,\lambda}) \\
& \ \  
 + \sum_{|\beta|=2} [[ D^\beta \o ]]_{*,\lambda} \| (1+ |\nabla_h| )\o\|_{1,\lambda} .
\end{split}\Ee
\end{lemma}

The proof relies on Lemma \ref{lem_bilinear}. We refer to Lemma 4.2 and Lemma 4.5 in \cite{FW} for the detailed proof.

\subsection{Green's function and integral representation for the Vorticity formulation}

By taking the Fourier transform of \eqref{NS}-\eqref{NSB} in $x_h \in \mathbb{T}^2$, we obtain 
 \begin{align}
\p_t \o_\xi - \kappa \eta_0 \Delta_\xi \o_\xi &= N_\xi  \quad \text{in }\mathbb R_+, \label{NS_f} \\
\kappa \eta_0 (\p_{x_3} +|\xi|)\o_{\xi ,h}  &= B_\xi, \ \ \o_{\xi,3} =0 \quad \text{on } x_3=0 ,\label{NSB_f}
\end{align} 
with $\o_\xi|_{t=0} = {\o_0}_\xi$ for $\xi\in {\mathbb Z^2}$. Here 
\Be 
\Delta_\xi = - |\xi|^2 + \p_{x_3}^2,
\Ee
and
\Be
N_\xi = N_\xi(t,x_3):=  \left( -u \cdot \nabla \o + \o \cdot \nabla u  \right)_\xi (t,x_3), \quad B_\xi=B_\xi(t):=  (\p_{x_3} (-\Delta_\xi)^{-1} N_{\xi,h} (t))|_{x_3=0}.
\Ee
Here $(-\Delta_\xi)^{-1}$ denotes the inverse of $-\Delta_\xi$ with the zero Dirichlet boundary condition at $x_3=0$.

We give the integral representation and present key estimates on Green's function for the Stokes problem. As shown in \cite{NN2018,FW}, letting $G_{\xi}(t,x_3,y)$ be the Green's function for \eqref{NS_f}-\eqref{NSB_f}, the solution can be represented by the integral formula via Duhamel's principle: 
\Be
 \begin{split}\label{o_xi}
 \o_\xi (t,x_3) =& {\int^\infty_0 G_{\xi } (t,x_3, y) \o_{0\xi} (y) \dd y} 
 + {\int^t_0\int^\infty_0 G_{\xi } (t-s, x_3, y) N_\xi (s,y ) \dd y \dd s}\\
 &
- {\int^t_0  G_{\xi } (t-s, x_3, 0) (B_\xi (s) ,0)\dd s} ,
 \end{split}
 \Ee
 where \Be\label{G}
 G_{\xi }  = \begin{bmatrix}
 G_{\xi h} & 0 & 0\\
 0 &  G_{\xi h} & 0 \\
 0& 0&  G_{\xi 3}
 \end{bmatrix},
 \Ee
 with $G_{\xi h}$ of (\ref{G_xi}) and $G_{\xi  3}$ of (\ref{G_xi2}): for $G_{\xi *}$ can be either $G_{\xi h}$ or $G_{\xi  3}$
\begin{align}
\p_ t G_{\xi *} (t, x_3, y) - \kappa\eta_0 \Delta_\xi G_{\xi *} (t, x_3, y) =0, \ \ x_3>0,\label{eqtn:G_xi2}\\
\kappa\eta_0 (\p_{x_3} + |\xi|) G_{\xi h} (t, x_3, y)=0, \ \ x_3=0, \label{bdry:G_xi2}\\
 G_{\xi 3} (t, x_3, y)=0, \ \ x_3=0. \label{bdry:G_xi3}
\end{align}


The following estimates and properties for $G_{\xi}$ will be useful to show the propagation of analytic norms of $\o$, $\p_t\o$ and $\p_t^2\o$.  

 \begin{lemma}[\cite{NN2018,FW}] \label{lem_G}
 \begin{enumerate}
 
 \item (Bounds on $G_{\xi h}$)
The Green's function $G_{\xi h}$ for the Stokes problem (\ref{eqtn:G_xi2}) and (\ref{bdry:G_xi2}) is given by 
  \Be\label{G_xi}
 G_{\xi h} = \tilde{H}_\xi + R_\xi,
 \Ee
 where $\tilde{H}_\xi$ is the one dimensional Heat kernel in the half-space with  the homogeneous Neumann boundary condition which takes the form of 
\Be
 \tilde{H}_\xi (t,x_3,y)= H_\xi(t,x_3-y) + H_\xi(t,x_3+y) = \frac{1}{\sqrt{\kappa\eta_0 t}} \bigg(
 e^{- \frac{|x_3-y|^2}{4 \kappa \eta_0 t}} +  e^{- \frac{|x_3+y|^2}{4 \kappa\eta_0 t}}
 \bigg) e^{- \kappa\eta_0 |\xi|^2 t},
 \label{H_xi}
 \Ee
and the residual kernel $R_\xi$ due to the boundary condition satisfies 
 \Be
 |\p_{x_3}^k R_\xi(t,x_3, y)| \lesssim b^{k+1} e^{- \theta_0 b (x_3 + y)} + 
 \frac{1}{(\kappa \eta_0 t)^{(k+1)/2}} e^{- \theta_0 \frac{|x_3+y|^2}{\kappa\eta_0 t}} e^{- \frac{\kappa\eta_0 |\xi|^2 t}{8}} 
 ,\label{R_xi}
\Ee
 with $b= |\xi| + \frac{1}{\sqrt{\kappa\eta_0}}$  and $R_\xi (t,x_3,y) = R_\xi (t, x_3 + y)$. 
 
  \item  (Formula of $G_{\xi 3}$) The Green's function $G_{\xi 3}$ for the Stokes problem (\ref{eqtn:G_xi2}) and (\ref{bdry:G_xi3}) is given by one dimensional Heat kernel in the half-space with the homogeneous Dirichlet boundary condition as   \Be
  \label{G_xi2}
G_{\xi 3} (t,x_3,y) = H_\xi(t,x_3-y) -H_\xi(t,x_3+y) = \frac{1}{\sqrt{\kappa\eta_0 t}} \bigg(
 e^{- \frac{|x_3-y|^2}{4 \kappa \eta_0 t}} -  e^{- \frac{|x_3+y|^2}{4 \kappa\eta_0 t}}
 \bigg) e^{- \kappa\eta_0 |\xi|^2 t}.
  \Ee

 \item (Complex extension) The Green's function $G_{\xi}$ has a natural extension to the complex domain $\mathcal{H}_\lambda$ for small $\lambda>0$ with similar bounds in terms of $\text{Re}\, y$ and $\text{Re}\, z$ (cf. (3.16) in \cite{NN2018}). 
 The solution $\o_\xi$ to \eqref{NS_f}-\eqref{NSB_f} in $\mathcal{H}_\lambda$ has a similar representation: for any $z\in\mathcal{H}_\lambda$, let $\sigma$ be the positive constant so that $z\in \p \mathcal{H}_\lambda$, then $\o_\xi$ satisfies
 \Be
 \begin{split}\notag
 \o_\xi (t,z) =& {\int_{\p \mathcal{H}_\sigma} G_{\xi } (t,z, y) \o_{0\xi} (y) \dd y} 
 + {\int^t_0\int_{\p \mathcal{H}_\sigma} G_{\xi } (t-s, z, y) N_\xi (s,y ) \dd y \dd s}\\
 &
- {\int^t_0  G_{\xi } (t-s, z, 0) (B_\xi (s), 0) \dd s} .
 \end{split}
 \Ee
 \end{enumerate}
 \end{lemma}
 
 The proof of Lemma \ref{lem_G} can be found in Proposition 3.3 and Section 3.3 of \cite{NN2018}.  The next lemma concerns the convolution estimates.

\begin{lemma} \label{lem_Gc} Let $T>0$ be given. Recall the norms defined in Section \ref{sec:MR}. For any $0\leq s < t\leq T$ and $k\geq 0$, there exists a constant $C_T>0$ so that the following estimates hold: for $G_{\xi *}$ can be either $G_{\xi h}$ or $G_{\xi 3}$
\begin{enumerate}
\item ($\mathcal{L}^1_\lambda$ estimates)  
\begin{align}
\sum_{j=0}^k \left\| (\zeta(z)\p_{z})^j \int_0^\infty G_{\xi*}(t, z, y ) g_\xi (y) \dd y \right\|_{\mathcal{L}^1_\lambda} \leq C_T \sum_{j=0}^k \left\|   (\zeta(z)\p_{z})^j g_\xi \right\|_{\mathcal{L}^1_\lambda},\label{est:L1G1} \\
\sum_{j=0}^k \left\|  (\zeta(z)\p_{z})^j \int_0^\infty G_{\xi *}(t-s,z, y ) g_\xi (y) \dd y \right\|_{\mathcal{L}^1_\lambda} \leq C_T \sum_{j=0}^k \left\|  (\zeta(z)\p_{z})^j g_\xi \right\|_{\mathcal{L}^1_\lambda}.\label{est:L1G2} 
\end{align}

\item ($\mathcal{L}^\infty_{\lambda,\kappa t}$ estimates)
\begin{align}
\sum_{j=0}^k \left\| (\zeta(z)\p_{z})^j \int_0^\infty G_{\xi *} (t, z, y ) g_\xi (y) \dd y \right\|_{\mathcal{L}^\infty_{\lambda,\kappa t}} \leq C_T \sum_{j=0}^k \left\|   (\zeta(z)\p_{z})^j g_\xi \right\|_{\mathcal{L}^\infty_{\lambda,\kappa }}, \label{est:Linfty_tG1}  \\
\sum_{j=0}^k \left\|  (\zeta(z)\p_{z})^j \int_0^\infty G_{\xi *}(t-s, z, y ) g_\xi (y) \dd y \right\|_{\mathcal{L}^\infty_{\lambda,\kappa t}} \leq C_T \sum_{j=0}^k \sqrt{\frac{t}{s}}  \left\|  (\zeta(z)\p_{z})^j g_\xi \right\|_{\mathcal{L}^\infty_{\lambda,\kappa s}}.  \label{est:Linfty_tG2}
\end{align}

\item ($\mathcal{L}^\infty_{\lambda,\kappa}$ estimates) For either $\kappa=0$ or $\kappa>0$
\begin{align}
\sum_{j=0}^k \left\| (\zeta(z)\p_{z})^j \int_0^\infty G_{\xi *}(t,z, y ) g_\xi (y) \dd y \right\|_{\mathcal{L}^\infty_{\lambda,\kappa}} \leq C_T \sum_{j=0}^k \left\|   (\zeta(z)\p_{z})^j g_\xi \right\|_{\mathcal{L}^\infty_{\lambda,\kappa}},  \label{est:LinftyG1} \\
\sum_{j=0}^k \left\|  (\zeta(z)\p_{z})^j \int_0^\infty G_{\xi *}(t-s, z, y ) g_\xi (y) \dd y \right\|_{\mathcal{L}^\infty_{\lambda,\kappa}} \leq C_T \sum_{j=0}^k \left\|  (\zeta(z)\p_{z})^j g_\xi \right\|_{\mathcal{L}^\infty_{\lambda,\kappa}}.  \label{est:LinftyG2}
\end{align}

\end{enumerate}
\end{lemma}

\begin{proof} We only give a proof for $G_{\xi h}$ since $G_{\xi 3}$ can be handled easier than the other. The proof of (1) and (2) can be found in Propositions 3.7 and 3.8 of \cite{NN2018}. Here we present the detail for (3), the second inequality. We consider real values $y,z\in \mathbb R_+$ only as the complex extension follows similarly (cf. {(3)} in Lemma \ref{lem_G}). Note that in view of \eqref{G_xi}, \eqref{H_xi} and \eqref{R_xi}, it suffices to show
\begin{align}
\sum_{j=0}^k \left\| (\zeta(z)\p_{z})^j \int_0^\infty R (t-s, z, y ) g_\xi (y) \dd y \right\|_{\mathcal{L}^\infty_{\lambda,\kappa}} \leq C_T \sum_{j=0}^k \left\|   (\zeta(z)\p_{z})^j g_\xi \right\|_{\mathcal{L}^\infty_{\lambda,\kappa}} ,\label{R_c}\\
\sum_{j=0}^k \left\| (\zeta(z)\p_{z})^j \int_0^\infty H (t-s, z, y ) g_\xi (y) \dd y \right\|_{\mathcal{L}^\infty_{\lambda,\kappa}} \leq C_T \sum_{j=0}^k \left\|   (\zeta(z)\p_{z})^j g_\xi \right\|_{\mathcal{L}^\infty_{\lambda,\kappa}},\label{H_c}
\end{align}
where $R(t,z,y)= b e^{- b (y+z)}$ and $H(t,z,y) = \frac{1}{\sqrt{\kappa t}} e^{-\frac{|y-z|^2}{M\kappa t}}$ for some $M>0$. 
We start with \eqref{R_c}.  Let $k=0$ first. First note that 
\begin{align*}
\left| \int_0^\infty R (t-s, z, y ) g_\xi (y) \dd y \right|&= \left| e^{-bz}\int_0^\infty b e^{-(\bar\alpha+b)y} (1+\phi_\kappa(y)) \frac{e^{\bar\alpha y}}{1+\phi_\kappa(y)}  g_\xi (y) \dd y\right| \\
&\leq  e^{-bz} (1+\phi_\kappa(0)) \left\|   g_\xi \right\|_{\mathcal{L}^\infty_{\lambda,\kappa}}   \int_0^\infty b e^{-(\bar\alpha+b)y}  \dd y ,
\end{align*}
since $\phi_\kappa$ is a decreasing function. 
The last integral is uniformly finite for all $|\xi|$ and $\kappa$.  Hence,
\begin{align*}
\left\| \int_0^\infty R (t-s, z, y ) g_\xi (y) \dd y \right\|_{\mathcal{L}^\infty_{\lambda,\kappa}} 
&\lesssim \sup_{z}\left( \frac{1+\phi_\kappa(0)}{1+\phi_\kappa(z)}e^{(\bar\alpha-b) z}  \right) \left\|   g_\xi \right\|_{\mathcal{L}^\infty_{\lambda,\kappa}} .
\end{align*}
For $\bar\alpha>0$, if $\kappa< \frac{1}{4 \eta_0 \bar\alpha^2}$ then 
\begin{align}\label{bound}
\sup_{z}\left( \frac{1+\phi_\kappa(0)}{1+\phi_\kappa(z)}e^{(\bar\alpha-b) z}  \right) \lesssim \frac{ \sqrt{\kappa}+1 }{ \inf_{z}\left[ (\sqrt{\kappa} + \frac{1}{1+|\frac{z}{\sqrt\kappa}|^p} )e^{\frac{z}{2\sqrt{ \eta_0\kappa}}} \right]} <  \min \Big\{1,\frac{\sqrt{\kappa} +1}{\sqrt{\kappa} +
 \frac{1}{ (2 \sqrt{\eta_0})^pp!}
 }\Big\},
\end{align}
where the last bound follows from the fact that $$
 \frac{1}{1+|\frac{z}{\sqrt{\kappa}}|^p} 
 e^{\frac{z}{2\sqrt{\eta_0\kappa}}} \geq  \frac{1}{1+|\frac{z}{\sqrt{\kappa}}|^p} 
\{1+ \frac{1}{p!} | \frac{z}{2\sqrt{\eta_0\kappa}}|^p\} 
\geq \min \big\{1, \frac{1}{ (2 \sqrt{\eta_0})^pp!}\big\}
.$$ 
For $k\geq 1$, since $|\zeta(z)\p_z R |\lesssim be^{-\frac{b z}{2}}$, the derivative estimates follow analogously. Therefore, \eqref{R_c} holds true. 

We move onto \eqref{H_c}. Let $k=0$ first. Note that, for $0 \leq s < t \leq T$ and $\kappa \lesssim 1$ 
\Be
e^{-\frac{|y-z|^2}{2M\kappa(t-s)}} e^{-\bar\alpha y} = e^{-\frac{1}{2}|\frac{y-z}{\sqrt{M\kappa(t-s)}} + \bar\alpha \sqrt{M\kappa(t-s)} |^2}  e^{\frac{M}{2}\bar\alpha^2\kappa(t-s)} e^{-\bar\alpha z} \leq e^{\frac{M}{2}\bar\alpha^2\kappa(t-s)} e^{-\bar\alpha z} \lesssim e^{-\bar\alpha z},
\Ee
and thus
\begin{align*}
\left| \int_0^\infty H (t-s, z, y ) g_\xi (y) \dd y \right|&= \left| \int_0^\infty   \frac{1}{\sqrt{\kappa (t-s)}} e^{-\frac{|y-z|^2}{M\kappa (t-s)}} e^{-\bar\alpha y}(1+\phi_\kappa(y)) \frac{e^{\bar\alpha y}}{1+\phi_\kappa(y)}  g_\xi (y) \dd y\right| \\
&\lesssim  e^{-\bar\alpha z} \left\|   g_\xi \right\|_{\mathcal{L}^\infty_{\lambda,\kappa}}   \int_0^\infty  \frac{1}{\sqrt{\kappa (t-s)}} e^{-\frac{|y-z|^2}{2M\kappa (t-s)}} (1+\phi_\kappa(y))  \dd y. 
\end{align*}
For the last integral, we divide the integral into two: $\int_0^\infty = \int_0^{\frac{z}{2}} + \int_{\frac{z}{2}}^\infty $. For the latter, since $\phi_\kappa$ is decreasing and the kernel is in $L^1_y$, we deduce 
\begin{align*}
\int_{\frac{z}{2}}^\infty \frac{1}{\sqrt{\kappa (t-s)}} e^{-\frac{|y-z|^2}{2M\kappa (t-s)}} (1+\phi_\kappa(y))  \dd y \lesssim 
 1+\phi_\kappa({z}) .
 \end{align*}
For $ \int_0^{\frac{z}{2}} \dd y$, note $|y-z|\geq \frac{z}{2}$ and $1+\phi_\kappa(y) \leq 1+\phi_\kappa(0)$ for $y\in(0,\frac{z}{2})$. Hence
\begin{align*}
\int^{\frac{z}{2}}_0\frac{1}{\sqrt{\kappa (t-s)}} e^{-\frac{|y-z|^2}{2M\kappa (t-s)}} (1+\phi_\kappa(y))  \dd y \lesssim
 e^{-\frac{|z|^2}{16M\kappa (t-s)}} (1+\phi_\kappa(0)) .
 \end{align*}
Then 
\begin{align*}
\left\| \int_0^\infty H (t-s, z, y ) g_\xi (y) \dd y \right\|_{\mathcal{L}^\infty_{\lambda,\kappa}} 
&\lesssim  \sup_{z}\left( \frac{1+\phi_\kappa(0)}{1+\phi_\kappa(z)}e^{-\frac{|z|^2}{16M\kappa (t-s)}}+1 \right) \left\|   g_\xi \right\|_{\mathcal{L}^\infty_{\lambda,\kappa}} .
\end{align*}
A similar argument as in \eqref{bound} shows that $\frac{1+\phi_\kappa(0)}{1+\phi_\kappa(z)}e^{-\frac{|z|^2}{16M\kappa (t-s)}}$ is uniformly finite in $\kappa$. This shows \eqref{H_c}  for $k=0$. 

For the derivative estimate, by splitting the integral into two parts and using $\p_z H (t,z,y)= - \p_y H(t,z,y)$, we rewrite 
\begin{align*}
 \int_0^\infty  \zeta(z)\p_z H (t-s, z, y ) g_\xi (y) \dd y& = \int_0^\frac{z}{2} \zeta(z)\p_z H (t-s, z, y ) g_\xi (y) \dd y - \zeta(z) H(t-s, z, \frac{z}{2}) g_\xi (\frac{z}{2}) \\
&+  \int_\frac{z}{2}^\infty \zeta(z) H (t-s, z, y ) \p_y g_\xi (y) \dd y .
\end{align*}
For the first integral, since  $ |y-z|\geq \frac{z}{2}$  for $y\in(0,\frac{z}{2})$, 
\begin{align*}
| \zeta(z)\p_z H (t-s, z, y ) | &\lesssim \frac{z}{1+z}\frac{1}{\kappa(t-s)} e^{-\frac{|y-z|^2}{2M\kappa (t-s)}}\lesssim |y-z| \frac{1}{\kappa(t-s)} e^{-\frac{|y-z|^2}{2M\kappa (t-s)}}\\
& \lesssim \frac{1}{\sqrt{\kappa(t-s)}} e^{-\frac{|y-z|^2}{4M\kappa (t-s)}}.
\end{align*}
Hence, by the same argument as in $k=0$ leads to the desired bound. For the second term, 
\begin{align*}
|\zeta(z) H(t-s, z, \frac{z}{2}) g_\xi (\frac{z}{2}) |& \lesssim  \frac{z}{\sqrt{\kappa(t-s)}} e^{-\frac{|z|^2}{4M\kappa (t-s)}} e^{-\bar\alpha\frac{z}{2}}(1+\phi_\kappa(\frac{z}{2}))  \left\|   g_\xi \right\|_{\mathcal{L}^\infty_{\lambda,\kappa}} \\
&\lesssim e^{-\frac{|z|^2}{8M\kappa (t-s)}} e^{-\bar\alpha\frac{z}{2}}(1+\phi_\kappa({z}))   \left\|   g_\xi \right\|_{\mathcal{L}^\infty_{\lambda,\kappa}} \\
&= e^{-\frac12| \frac{z}{2\sqrt{M\kappa(t-s)}} -\bar\alpha \sqrt{M\kappa(t-s)} |^2} e^{\frac{M}{2}\bar\alpha^2\kappa(t-s)} e^{-\bar\alpha z}(1+\phi_\kappa({z}))   \left\|   g_\xi \right\|_{\mathcal{L}^\infty_{\lambda,\kappa}} \\
&\lesssim e^{-\bar\alpha z}(1+\phi_\kappa({z}))   \left\|   g_\xi \right\|_{\mathcal{L}^\infty_{\lambda,\kappa}},
\end{align*}
which leads to the desired bound. For the last integral, note that $\zeta(z) \leq 2 \zeta(y)$ for $y\geq \frac{z}{2}$. Therefore the corresponding integral can be treated in the same way as in $k=0$ with $g_\xi(y)$ replaced by $\zeta(y)\p_yg_\xi(y)$. This shows \eqref{H_c} for $k=1$. Other $k\geq 2$ can be estimated analogously. 
\end{proof}

The next result concerns the estimates for the trace kernel. 

\begin{lemma}\label{lem_Gt} Let $a_\xi(s)= [\p_{z}(- \Delta_\xi)^{-1} g_\xi ] \, |_{z=0}$. Then for any $0\leq s<t\leq T$ and $k\geq 0$, we have the following
\begin{align}
\sum_{j=0}^k \left\|  (\zeta(z)\p_{z})^j  G_{\xi h }(t-s,z, 0 ) a_\xi (s) \right\|_{\mathcal{L}^1_\lambda} &\lesssim  \left\|  g_\xi \right\|_{\mathcal{L}^1_\lambda}, \label{trace1} \\
\sum_{j=0}^k \left\|  (\zeta(z)\p_{z})^j  G_{\xi  h}(t-s,z, 0 ) a_\xi (s)  \right\|_{\mathcal{L}^\infty_{\lambda,\kappa}}  &\lesssim \frac{1}{\sqrt{t-s}} \left\| g_\xi \right\|_{\mathcal{L}^1_\lambda}. 
\label{trace2}
\end{align}
\end{lemma}

\begin{proof} Note that from \eqref{G_xi}, \eqref{H_xi} and \eqref{R_xi}, the conormal derivatives $(\zeta(z)\p_z)^j$ of  $G_{\xi h}(t-s,z,0)$ enjoy the same bounds as $G_{\xi  h}(t-s,z,0)$: for some small constant $c_0$,
\Be\label{est:Con_G}
| (\zeta(z)\p_{z})^j  G_{\xi  }(t-s,z, 0 ) | \lesssim b e^{-c_0 b z} + \frac{1}{\sqrt{\kappa(t-s)}} e^{-c_0 \frac{|z|^2}{\kappa (t-s)}}.
\Ee
Therefore, it suffices to show the bounds for $k=0$. 
We first recall the representation formula for $a_\xi$ 
(cf. (4.29) of \cite{KVW} or (4.2) of \cite{FW}): 
\Be\notag
a_\xi(s) = \int_0^\infty e^{-|\xi| y} g_\xi (y) \dd y ,
\Ee
from which we have $\|a_\xi\|_{\mathcal{L}^\infty_\lambda} \lesssim  \left\|  g_\xi \right\|_{\mathcal{L}^1_\lambda}$. Since the above upper bound of $G(t-s,z,0)$ is integrable in $z$,  \eqref{trace1} follows. To show \eqref{trace2}, we compute $\| G_{\xi h}(t-s,z, 0 ) \|_{\mathcal{L}^\infty_{\lambda,\kappa}}$: 
\begin{align*}
\| G_{\xi h}(t-s,z, 0 ) \|_{\mathcal{L}^\infty_{\lambda,\kappa}} \lesssim  \sup_z \left[\frac{ be^{(\bar\alpha-c_0 b) z}}{1+\phi_\kappa(z)}\right]   +  \frac{1}{\sqrt{t-s}} \sup_z\left[ \frac{e^{ \bar\alpha z-c_0 \frac{|z|^2}{\kappa (t-s)}}}{\sqrt\kappa+\sqrt\kappa\phi_\kappa(z)}  \right].
\end{align*}
It is a routine to check that both supremum norms are uniformly bounded in $\kappa$ and $|\xi|$. Therefore \eqref{trace2} is obtained. 
\end{proof}


 \subsection{Proof of Theorem \ref{thm_bound}} 
Our goal is to show that $\o(t)$ indeed belongs to $C^1([0,T];\mathfrak B^{\lambda, \kappa})$ without the initial layer under the compatibility condition \eqref{CC}, and that $\p_t^2\o$ in $\mathfrak B^{\lambda, \kappa t}$ with the initial layer. The existence of $\o(t)$ in $C^1([0,T];\mathfrak B^{\lambda, \kappa t})$ under the assumption of Theorem \ref{thm_bound} can be proved by following the argument of \cite{NN2018} and \cite{FW}. For the 2D case, Theorem 1.1 of \cite{NN2018} indeed ensures the existence of $\o(t)$ in $C^1([0,T];\mathfrak B^{\lambda, \kappa t})$ under the assumption of Theorem \ref{thm_bound}.  Such a result 
follows from Lemma \ref{lem_embedding}, Lemma \ref{lem_G},  Lemma \ref{lem_Gc},  Lemma \ref{lem_elliptic},  Lemma \ref{lem_bilinear}.  A 3D result can be obtained analogously. Hence, it suffices to show the propagation of the analytic norms in \eqref{norm_bound}.  

\

{\bf Step 1: Propagation of analytic norms for $\o$.} It is convenient to define  
\Be\label{|||}
 \vertiii{\o(t)}_t:=  \vertiii{\o(t)}_{\infty, \kappa } + \vertiii{\o(t)}_{1}.
\Ee 
The estimation of $\o$ follows from the nonlinear iteration using the representation formula \eqref{o_xi}.
 
 The estimates for the $L^1$-based norm  $\vertiii{ \o(t)}_1$ are already available in Section 5 of \cite{FW} (for 2D see Section 4.1 of \cite{NN2018}): From (\ref{est:L1G1}), (\ref{est:L1G2}) and (\ref{trace1}), we have that for $k=0,1,2$ 
 \begin{align*}
& \sum_{j=0}^k \| (\zeta(x_3)\p_{x_3})^j \o_{\xi } \|_{\mathcal{L}^1_{\lambda}} \\
&\leq \sum_{j=0}^k  \left\|(\zeta(x_3)\p_{x_3})^j{\int^\infty_0 G_{\xi } (t,x_3, y) \o_{0\xi } (y) \dd y} \right\|_{\mathcal{L}^1_{\lambda}} \\
& \ \ + \sum_{j=0}^k {\int^t_0 \left\| (\zeta(x_3)\p_{x_3})^j\int^\infty_0 G_{\xi  } (t-s, x_3, y) N_{\xi } (s,y ) \dd y \right\|_{\mathcal{L}^1_{\lambda }} \dd s}\\
& \ \ +\sum_{j=0}^k  {\int^t_0 \left\|(\zeta(x_3)\p_{x_3})^j G_{\xi } (t-s, x_3, 0) (B_{\xi } (s),0) \right\|_{\mathcal{L}^1_{\lambda }}\dd s}  \\
&\lesssim \sum_{j=0}^k \left\|  (\zeta(x_3)\p_{x_3})^j \o_{0\xi  } \right\|_{\mathcal{L}^1_{\lambda }} + \sum_{j=0}^k \int_0^t \left\|  (\zeta(x_3)\p_{x_3})^j N_{\xi } (s) \right\|_{\mathcal{L}^1_{\lambda }} \dd  s +  \int_0^t  \| N_{\xi }(s) \|_{\mathcal{L}^1_\lambda} \dd s.
 \end{align*}
 For $k=1$, after summing up over $\xi\in \mathbb Z^2$, we deduce that 
 \Be\label{1_1}
 \begin{split}
 \sum_{0 \leq |\beta| \leq 1} \| D^\beta (1+ |\nabla_h|) \o (s)\|_{1,\lambda}
 \lesssim &
\sum_{0 \leq |\beta| \leq 1} \| D^\beta (1+ |\nabla_h|) \o_0\|_{1,\lambda}\\
&
+ \int^t_0 \sum_{0 \leq |\beta| \leq 1} \| D^\beta (1+ |\nabla_h|) N(s) \|_{1,\lambda} \dd s. 
\end{split}
 \Ee
 Using (\ref{est:N_1}), (\ref{est:DN_1}), and the definition of $\vertiii{ \  \cdot \  }_s$ in (\ref{|||}) we derive that 
 \Be\label{1_1N}
 \begin{split}
 \int^t_0 \sum_{0 \leq |\beta| \leq 1} \| D^\beta (1+ |\nabla_h|) N(s) \|_{1,\lambda} \dd s
& \lesssim  
  \int^t_0
  \vertiii{\o(s)}_s^2  \big[1+(\lambda_0-\lambda -\gamma_0 s)^{- \alpha  }\big] 
  \dd s\\
  & \lesssim \Big(t+ \frac{1}{\gamma_0}\Big) \sup_{0  \leq s \leq t}\vertiii{\o (s) }_s . 
 \end{split}\Ee
 
The second order derivatives can be treated similarly except for the contributions of $N$ for which we apply the analyticity recovery estimate using $(3)$ of Lemma \ref{lem_embedding} while other terms are estimated in the same way. More precisely, we have \Be
\sum_{|\beta|=2}\| D^\beta (1+ |\nabla_h|) N (s) \|_{1,\lambda} \lesssim \frac{1}{\tilde\lambda-\lambda}
\sum_{0 \leq |\beta|\leq1}\| D^\beta (1+ |\nabla_h|)  N (s) \|_{1,\tilde\lambda} \ \text{for any}   \ \tilde\lambda >\lambda, \label{recovery_1}
\Ee
while we choose $\tilde\lambda=\frac{\lambda+\lambda_0-\gamma_0 s}{2}$ in particular. We note that still $\tilde{\lambda}< \lambda_0 - \gamma_0 s$ if $\lambda< \lambda_0 -\gamma_0 s$ and hence from (\ref{est:N_1}) and (\ref{est:DN_1})
\Be
\begin{split}\notag
 &\sum_{0 \leq |\beta| \leq 1}
\|D^\beta (1+ |\nabla_h|) N  (s) \|_{1, \tilde{\lambda}(s) }\\
&\lesssim 
\Big(
\sum_{0 \leq |\beta|\leq 1} \|  D^\beta(1+ |\nabla_h|) \o(s) \|_{1,\tilde\lambda(s) }\Big)
\Big(\sum_{0 \leq |\beta|\leq 2}[[ D^\beta \o(s) ]]_{\infty, \tilde{\lambda}(s), \kappa }
\Big)\\
& \ \ \ 
+\Big(
\sum_{0 \leq |\beta|\leq2} \| D^\beta(1+ |\nabla_h|) \o(s) \|_{1,\tilde\lambda(s) }\Big)
\Big(\sum_{0 \leq |\beta|\leq 1}\| D^\beta(1+ |\nabla_h|) \o (s)\|_{1,\tilde\lambda(s) }\Big)\\
&\lesssim \big[1+
(\lambda_0 -\lambda -\gamma_0 s)^{-\alpha } \big]
\vertiii{\o(s)}_s ^2. 
\end{split}\Ee
Therefore we derive that for $t< \frac{\lambda_0}{2\gamma_0}$ and $\lambda< \lambda_0 - \gamma_0 t$
\Be\label{1_2}
 \begin{split}
   &\sum_{|\beta|=2} \|D^\beta (1+ |\nabla_h|) 
\o(t) \|_{1, \lambda} \\
& \lesssim  \sum_{|\beta|=2} \| D^\beta (1+ |\nabla_h|)
  \o_{0} \|_{1, \lambda_0 }
+ \int_0^t 
 \big[1+ (\lambda_0-\lambda -\gamma_0 s)^{ -(\alpha + 1)} \big]
   \vertiii{ \o(s)}_{s
 }^2 
  \dd s 
 \\
 &  \lesssim    \sum_{|\beta|=2} \| D^\beta (1+ |\nabla_h|)
  \o_{0} \|_{1, \lambda_0 } +
  \Big(  (\lambda_0 - \lambda - \gamma_0 t)^{-\alpha} \frac{1}{\gamma_0} +{t} \Big)
   \sup_{0 \leq s \leq t}\vertiii{ \o(s)}_s^2.
\end{split}
\Ee

Therefore, we conclude that, from (\ref{1_1}) with (\ref{1_1N}), and (\ref{1_2})
\Be\label{est:1}
 \vertiii{\o (t)}_1\lesssim \sum_{0 \leq |\beta|\leq 2} \| D^\beta (1+ |\nabla_h|) 
  \o_{0} \|_{1, \lambda_0 } + (t+ \frac{1}{\gamma_0})  \sup_{0 \leq s \leq t}\vertiii{ \o(s)}_s^2 \ \ \text{for} \ t< \frac{\lambda_0}{2 \gamma_0}. 
\Ee

The propagation of the boundary layer norm $\vertiii{ \o(t)}_{\infty,\kappa}$ can be shown analogously using $\mathcal{L}^\infty_{\lambda, \kappa}$ estimates of Lemma \ref{lem_Gc} and Lemma \ref{lem_Gt}: For $k=0,1,2$ and $\kappa>0$ for $i=1,2$ and $\kappa=0$ for $i=3$ we have 
 \begin{align*}
 &\sum_{j=0}^k \| (\zeta(x_3)\p_{x_3})^j \o_{\xi, i} \|_{\mathcal{L}^\infty_{\lambda,\kappa}} \\
 &\leq \sum_{j=0}^k  \left\|(\zeta(x_3)\p_{x_3})^j{\int^\infty_0 G_{\xi i} (t,x_3, y) \o_{0\xi ,i} (y) \dd y} \right\|_{\mathcal{L}^\infty_{\lambda,\kappa}} \\
& \ \ + \sum_{j=0}^k {\int^t_0 \left\| (\zeta(x_3)\p_{x_3})^j\int^\infty_0 G_{\xi ,i} (t-s, x_3, y) N_{\xi, i} (s,y ) \dd y \right\|_{\mathcal{L}^\infty_{\lambda,\kappa}} \dd s}\\
& \ \ +\sum_{j=0}^k  {\int^t_0 \left\|(\zeta(x_3)\p_{x_3})^j G_{\xi, i} (t-s, x_3, 0) B_{\xi ,i} (s) \right\|_{\mathcal{L}^\infty_{\lambda,\kappa}}\dd s}  \\
&\lesssim \sum_{j=0}^k \left\|  (\zeta(x_3)\p_{x_3})^j \o_{0\xi,  i } \right\|_{\mathcal{L}^\infty_{\lambda,\kappa}} + \sum_{j=0}^k \int_0^t \left\|  (\zeta(x_3)\p_{x_3})^j N_{\xi,i} (s) \right\|_{\mathcal{L}^\infty_{\lambda,\kappa}} \dd  s\\
& \ \ \ \ \  +(1- \delta_{i3}) \int_0^t \frac{1}{\sqrt{t-s}} \| N_{\xi, i} \|_{\mathcal{L}^1_\lambda}.
 \end{align*}

Let $k=1$. After summing up over $\xi\in \mathbb Z$ and $i=1,2$ (with $\kappa>0$) and $i=3$ (with $\kappa=0$)
, we deduce that 
\Be\notag
  \sum_{0\leq |\beta|\leq 1} [[D^\beta 
\o (t) ]]_{\infty, \lambda, \kappa} 
 \lesssim 
  \sum_{0\leq|\beta|\leq 1} [[ D^\beta
  \o_{0} ]]_{\infty, \lambda_0, \kappa}
  + \int^t_0  \sum_{0\leq|\beta|\leq 1}  [[ D^\beta N(s) ]]_{\infty, \lambda,\kappa} \dd s
  +  \int_0^t   \frac{1}{\sqrt{t-s}} \vertiii{ \o(s)}_1^2 \dd s.
  \Ee
Using the definition of $\vertiii{ \ \cdot \  }_s$ in (\ref{|||}), and applying Lemma \ref{lem_est:N} with (\ref{est:N_1}), (\ref{est:N_infty}), and (\ref{est:DN_infty}), we derive 
\Be \notag
\begin{split}
\sum_{0\leq|\beta|\leq 1}  [[ D^\beta N(s) ]]_{\infty, \lambda,\kappa}
&\lesssim \Big( \sum_{0 \leq |\beta | \leq 2} \| D^\beta(1+ |\nabla_h| )\o(s) \|_{1,\lambda} \Big)\Big(\sum_{0 \leq |\beta| \leq 1} [[ D^\beta \o (s) \|_{\infty, \lambda, \kappa}\Big) \\
& \ \ \ +  \| (1+ |\nabla_h| )\o(s) \|_{1,\lambda}
\sum_{|\beta|=2}[[ D^\beta \o (s) ]]_{\infty,\lambda, \kappa}\\
&\lesssim \big[1+(\lambda_0 -\lambda -\gamma_0 s)^{- \alpha } \big]
\vertiii{\o(s)}_s ^2.
\end{split}
\Ee
Therefore we derive that 
\Be
 \begin{split}
 \label{infty_1}
 &  \sum_{0\leq |\beta|\leq 1} [[D^\beta 
\o (t) ]]_{\infty, \lambda, \kappa} \\
& \lesssim  \sum_{0\leq|\beta|\leq 1} [[ D^\beta
  \o_{0} ]]_{\infty, \lambda_0, \kappa}
+ \int_0^t 
   \vertiii{ \o(s)}_{s
 }^2 
 \big[1+
 (\lambda_0-\lambda -\gamma_0 s)^{ - \alpha  } 
 \big]
  \dd s \\
  & \ \ \ \ \ \ \ 
  + \int_0^t   \frac{1}{\sqrt{t-s}} \vertiii{ \o(s)}_s^2 ds \\
 &  \lesssim   \sum_{0\leq|\beta|\leq 1} [[ D^\beta
  \o_{0} ]]_{\infty, \lambda_0, \kappa} +
  \Big(  \sqrt{t} + \frac{1}{\gamma_0}\Big)
   \sup_{0 \leq s \leq t}\vertiii{ \o(s)}_s^2.
\end{split}\Ee

Now we control the second order derivatives similarly except for the $N$. As in (\ref{recovery_1})
we use the analyticity recovery estimate using Lemma \ref{lem_embedding} \Be
\sum_{|\beta|=2}[[  D^\beta N (s) ]]_{\infty,\lambda,\kappa} \lesssim \frac{1}{\tilde\lambda-\lambda}
\sum_{0 \leq |\beta|\leq1}[[ D^\beta N (s) ]]_{\infty,\tilde\lambda,\kappa} \ \text{for any}   \ \tilde\lambda >\lambda, \label{recovery_infty}
\Ee
while again we choose $\tilde\lambda=\frac{\lambda+\lambda_0-\gamma_0 s}{2}$ in particular. We note that still $\tilde{\lambda}< \lambda_0 - \gamma_0 s$ if $\lambda< \lambda_0 -\gamma_0 s$ and hence from (\ref{est:N_infty}) and (\ref{est:DN_infty})
\Be\notag
 \sum_{0 \leq |\beta| \leq 1}
  [[ D^\beta N  (s) ]]_{\infty, \tilde{\lambda}(s), \kappa}\lesssim (\lambda_0 -\lambda -\gamma_0 s)^{- \alpha } 
\vertiii{\o(s)}_s ^2. 
\Ee
Therefore we derive that for $t< \frac{\lambda_0}{2\gamma_0}$ and $\lambda< \lambda_0 - \gamma_0 t$
\Be\begin{split}\label{infty_2}
& \sum_{ |\beta|=2}[[ D^\beta \o(t) ]]_{\infty, \lambda, \kappa}\\ &\lesssim 
  \sum_{ |\beta|=2} \| D^\beta\o_{0} \|_{\infty, \lambda_0, \kappa}
  + \int^t_0
( \lambda_0 - \lambda -\gamma_0 s)^{- (\alpha+ \frac{3}{2}) }
\vertiii{\o(s)}_s^2
\dd s  +  \int_0^t   \frac{1}{\sqrt{t-s}} \vertiii{ \o(s)}_1^2 \dd s 
 \\
&\lesssim  \sum_{ |\beta|=2} \| D^\beta\o_{0} \|_{\infty, \lambda_0, \kappa}
+(\lambda_0- \lambda - \gamma_0 t)^{-\alpha}\Big(\int^t_0 (\lambda_0 - \lambda - \gamma_0 s)^{- \frac{3}{2}} \dd s \Big) \sup_{0 \leq s\leq t} \vertiii{\o(s)}_s^2\\
& \ \ \ \ \ \ \  
+ \sqrt{t} \sup_{0 \leq s \leq t} \vertiii{\o(s)}_s^2
   \\
 &
 \lesssim   \sum_{ |\beta|=2} \| D^\beta\o_{0} \|_{\infty, \lambda_0, \kappa}  
 +\Big( (\lambda_0-\lambda-\gamma_0 t)^{-\alpha} \frac{1}{\gamma_0}  + \sqrt{t}\Big)
 \sup_{0 \leq s \leq t}
 \vertiii{ \o(s)}_{s}^2.
\end{split}\Ee

Therefore we conclude that, from (\ref{infty_1}) and (\ref{infty_2}),
\Be\label{est:infty}
\vertiii{ \o(t)}_{\infty,\kappa} \lesssim  \sum_{0 \leq  |\beta|  \leq 2} \| D^\beta\o_{0} \|_{\infty, \lambda_0, \kappa}
 +\Big( \sqrt{t}+ \frac{1}{\gamma_0}\Big) \sup_{0 \leq s \leq t}\vertiii{ \o(s)}_{s}^2 \ \ \text{for} \ t < \frac{\lambda_0}{2\gamma_0}. 
\Ee
 
In conclusion, from (\ref{est:1}), (\ref{est:infty}), and by a standard continuity argument we obtain for sufficiently large $\gamma_0$
\Be\label{est:|||_t}
\sup_{0 \leq t <\frac{\lambda_0}{2\gamma_0}}\vertiii{ \o(t)}_{t}\lesssim  \sum_{0 \leq  |\beta|  \leq 2} \| D^\beta\o_{0} \|_{\infty, \lambda_0, \kappa}+  \sum_{0 \leq |\beta|\leq 2} \| D^\beta (1+ |\nabla_h|) 
  \o_{0} \|_{1, \lambda_0 }.
\Ee

 \

{\bf Step 2: Propagation of analytic norms for $\p_t\o$.}   The continuity of $\o(t)$ in $t$ follows from the mild solution form \eqref{o_xi} of $\o_\xi(t)$.  We claim that $\o(t) \in C^1([0,T]; \mathfrak B^{\lambda,\kappa})$ and moreover 
$\vertiii{ \p_t\o(t)}_{t}
$ is bounded. To this end, we first derive the mild form of $\p_t\o_\xi$ from \eqref{o_xi}: 
\Be
 \begin{split}\label{o_xi2t}
 \p_t\o_\xi (t,x_3) =& {\int^\infty_0 G_\xi (t,x_3, y) \p_t \o_{0 \xi} (y) \dd y} 
 + {\int^t_0\int^\infty_0 G_\xi (t-s, x_3, y) \p_s N_\xi (s,y ) \dd y \dd s}\\
 &
- {\int^t_0  G_\xi (t-s, x_3, 0) (\p_s B_\xi (s), 0) \dd s} ,
 \end{split}
 \Ee
where we recall $\p_t \o_0$ in \eqref{idata}. To justify this formula, we first recall (\ref{eqtn:G_xi2})-(\ref{bdry:G_xi3}). We start with the horizontal part of the formula (\ref{o_xi2t}) for 
$\p_t\o_{\xi,h}$.  
From Lemma \ref{lem_G}, $G_{\xi h} (t,x_3,y) = H_\xi (t, x_3-y) +H_\xi (t,x_3+y)+ R_\xi (t,x_3+y)$. Then by using the fact that $H^\prime_\xi (t,  \cdot)$ is an odd function, we see 
$
\p_{x_3} G_{\xi h} (t,x_3,y)|_{x_3=0} 
= R^\prime_\xi (t, y)$. 
Now we read (\ref{bdry:G_xi2}) as 
\Be\begin{split}
\label{bdry:G_xi12}
\kappa \eta_0 R^\prime_\xi (t,y ) + \kappa\eta_0 |\xi| G_{\xi h} (t,0,y)=0, \quad \kappa\eta_0 R^\prime_\xi (t,x_3) + \kappa\eta_0 |\xi| G_{\xi h} (t,x_3,0)=0 ,
\end{split}\Ee
where we have used that $H_\xi (t, \cdot)$ is an even function for the second relation. On the other hand, since we also have 
$
\p_{y_3} G_{\xi h} (t,x_3,y)|_{y=0} 
= R^\prime_\xi (t, x_3)$, we deduce that 
\Be\label{bdry:G_xiy2}
\kappa\eta_0 (\p_{y_3}  +|\xi|) G_{\xi h} (t,x_3,y_3) =0, \ \ y_3=0.
\Ee
It is straightforward to see $\Delta_\xi G_{\xi h} = \p_{x_3}^2G_{\xi h} - |\xi|^2G_{\xi h} = \p_y^2 G_{\xi h} - |\xi|^2 G_{\xi h}.$  

We now take $\p_t$ of \eqref{o_xi}: 
\begin{align*}
& \p_t \int^\infty_0 G_{\xi h} (t,x_3,y)\o_{0\xi,h}(y) \dd y = \int^\infty_0 \p_t G_{\xi, h} (t,x_3,y)\o_{0\xi, h}(y) \dd y \\
 &\quad= \int^\infty_0 \kappa\eta_0 (\p_{y}^2 - |\xi|^2) G_{\xi h} (t,x_3,y)\o_{0\xi,h}(y) \dd y \\
 &\quad= \big[ \kappa\eta_0 \p_y G_{\xi h} (t,x_3,y)\o_{0\xi, h}(y) \big]_{y=0}^{y=\infty}
 -  \int^\infty_0 \kappa\eta_0  |\xi|^2  G_{\xi h} (t,x_3,y)\o_{0\xi, h }(y) \dd y\\
 &\quad \ \ \ 
 -  \int^\infty_0 \kappa \eta_0  \p_y G_{\xi h} (t,x_3,y)\p_y\o_{0\xi, h}(y) \dd y
 \\
  &\quad= \big[ \kappa\eta_0 \p_y G_{\xi h} (t,x_3,y)\o_{0\xi, h }(y) \big]_{y=0}^{y=\infty}  -  
 \big[
  \kappa\eta_0    G_{\xi h} (t,x_3,y)\p_y\o_{0\xi, h}(y) 
 \big]_{y=0}^{y=\infty} \\
 &\quad \ \ \ 
+ \int^\infty_0   G_{\xi h} (t,x_3,y)\kappa \eta_0 \Delta_\xi \o_{0\xi, h}(y) \dd y\\
& \quad =- \kappa\eta_0 \p_y G_{\xi h} (t,x_3,0)\o_{0\xi,h}(0)  +  
  \kappa\eta_0    G_{\xi h} (t,x_3,0)\p_y\o_{0\xi,h}(0)  \\
  & \quad \ \ \ \ 
+ \int^\infty_0   G_{\xi h} (t,x_3,y)\kappa\eta_0  \Delta_{\xi h} \o_{0\xi,h}(y) \dd y,
\end{align*}
and 
\begin{align*}
\p_t  {\int^t_0\int^\infty_0 G_{\xi h} (t-s, x_3, y) N_{\xi, h} (s,y ) \dd y \dd s}
  &=\int_0^\infty G_{\xi h} (t,x_3,y)  N_{\xi,h} (0,y)\dd y\\
  & \ \ \ 
+ \int^t_0 \int^\infty_0
G_{\xi h} (s,x_3, y) \p_t N_{\xi, h} (t-s,y)
 \dd y \dd s,\\
 \p_t  {\int^t_0  G_{\xi h} (t-s, x_3, 0) B_\xi (s) \dd s} &=
 G_{\xi h} (t,x_3,0) B_\xi (0)
+\int^t_0 G_{\xi h} (t-s,x_3,0) \p_s B_\xi( s) \dd s 
\end{align*}
Therefore we obtain 
\Be\begin{split}
\p_t \o_{\xi, h} (t,x_3) =&{- \kappa\eta_0 \p_y G_{\xi h} (t,x_3,0)\o_{0\xi,h}(0)  +  
  \kappa\eta_0    G_{\xi h} (t,x_3,0)\p_y\o_{0\xi, h}(0)  - G_{\xi h} (t,x_3,0) B_\xi (0)} \\
  &+ \int^\infty_0   G_{\xi h} (t,x_3,y) \{ \kappa \eta_0\Delta_\xi \o_{0\xi,h}(y)
+ N_{\xi, h} (0,y) \}
 \dd y\\
 &+ \int^t_0 \int^\infty_0
G_{\xi h} (t-s,x_3, y) \p_s N_{\xi, h} (s,y) 
 \dd y \dd s - \int^t_0 G_{\xi h} (t-s,x_3,0) \p_s B_\xi(s) \dd s\label{pt_o2}.
 \end{split}
\Ee
Next we show that the first line in the right-hand side is 0. From (\ref{bdry:G_xiy2})
\begin{align*} 
- \kappa\eta_0 \p_y G_{\xi h} (t,x_3,0)\o_{0\xi, h}(0)  +  
  \kappa\eta_0    G_{\xi h} (t,x_3,0)\p_y\o_{0\xi,h}(0)=G_{\xi h} (t,x_3,0) \kappa\eta_0(|\xi | +  \p_y) \o_{0\xi, h}(0),
\end{align*}
and hence the first line of \eqref{pt_o2} reads 
\Be\label{initial_layer0}
G_{\xi h} (t,x_3,0) \left[ \kappa\eta_0(|\xi | +  \p_{x_3}) \o_{0\xi, h}(0) - B_\xi (0) \right]  , 
\Ee
which is zero due to the first compatibility condition of \eqref{CC}. Recalling $\p_t \o_0$ in \eqref{idata}, the formula \eqref{o_xi2t} for $\p_t\o_{\xi,h}$ has been established. We may follow the same procedure to verify the vertical part of the formula \eqref{o_xi2t} for $\p_t\o_{\xi,3}$ by noting  that the second compatibility condition of \eqref{CC} removes the term $- \kappa\eta_0 \p_y G_{\xi 3} (t,x_3,0)\o_{0\xi,3}(0)$ which would create the initial layer otherwise because $ \p_y G_{\xi 3} (t,x_3,0)$ does not vanish.  

We may now repeat Step 1 for $\p_t\o$ using the representation formula \eqref{o_xi2t}. The estimates are obtained in the same fashion. For the nonlinear terms, since $\p_t N = - u \cdot \nabla \p_t \o - \p_t u\cdot \nabla \o
+ \o \cdot \nabla \p_t  u+\p_t \o  \cdot \nabla  u
$, the structure of $\p_t N$ with respect to $\p_t\o$ is consistent  with the one of $N$ with respect to $\o$ and we can use the bilinear estimates \eqref{nonlinear1} and \eqref{nonlinear2}. In summary, one can derive that for $t < \frac{\lambda_0}{2 \gamma_0}$
\begin{align}
\vertiii{ \p_t\o(t)}_{1} &\lesssim \sum_{0\leq |\beta|\leq 2} \| D^\beta (1+|\nabla_h|) \p_t \o_0  \|_{1, \lambda_0}  
 +( t+  \frac{1}{\gamma_0}) \sup_{0 \leq s \leq t}\vertiii{ \o(s)}_{s} \sup_{0 \leq s \leq t} \vertiii{\p_t \o(s)}_s,
 \label{est:t_1}
 \\
\vertiii{ \p_t\o(t)}_{\infty,\kappa} &\lesssim \sum_{0\leq|\beta|\leq 2} \| D^\beta \p_t \o_{0} \|_{\infty, \lambda_0, \kappa} +( \sqrt{t}+  \frac{1}{\gamma_0}) \sup_{0 \leq s \leq t}\vertiii{ \o(s)}_{s} \sup_{0 \leq s \leq t} \vertiii{\p_t \o(s)}_s, \label{est:t_infty}
\end{align}
which lead to the desired bounds for $\p_t\o(t)$ by choosing sufficiently large $\gamma_0$.  

\hide
\begin{align}
\vertiii{ \p_t\o(t)}_{1} &\lesssim \sum_{0\leq i+j\leq 2} \| \p_{x_1}^i (\zeta(z)\p_{z})^j \o_{1\xi} \|_{1, \lambda_0}  
 + \frac{1}{\gamma_0}\vertiii{ \o(t)}_{1} \vertiii{\p_t \o(t)}_1\\
\vertiii{ \p_t\o(t)}_{\infty,\kappa} &\lesssim \sum_{0\leq i+j\leq 2} \| \p_{x_1}^i (\zeta(z)\p_{z})^j \o_{1\xi} \|_{\infty, \lambda_0, \kappa}  + \vertiii{ \o(t)}_{1} \vertiii{\p_t \o(t)}_1 
 + \frac{1}{\gamma_0}\vertiii{ \o(t)}_{\infty,\kappa}\vertiii{ \p_t\o(t)}_{\infty,\kappa}
\end{align}\unhide
 
 \

 {\bf Step 3: Propagation of analytic norms for $\p_t^2\o_\xi$.}  As a consequence of Step 2, $\p_t\o_\xi(t,x_3)$ solves the following system
  \begin{align}
\p_t^2 \o_\xi - \kappa \eta_0 \Delta_\xi \p_t \o_\xi =\p_t N_\xi  \quad &\text{in }\mathbb R_+, \label{NS_f1} \\
\kappa \eta_0 (\p_{x_3} + |\xi| )\p_t\o_{\xi,h}   = \p_tB_\xi \quad &\text{on } x_3=0, \label{NSB_f1}\\
\p_t \o_{\xi, 3}  = 0 \quad &\text{on } x_3=0, \label{NSB_f3}
\end{align} 
with $\p_t\o_\xi|_{t=0} = \p_t \o_{0 \xi}$ for $\xi\in \mathbb Z^2$ where $\p_t \o_0$ is defined in (\ref{idata}). Then as done in Step 2, by using the properties of $G_\xi$ and  integration by parts and by the last compatibility condition of \eqref{CC}, we can derive the representation formula for $\p_t^2\o$: 
 \Be
 \begin{split}\label{o_xi2tt}
& \p_t^2\o_\xi (t,x_3) =( G_{\xi h} (t,x_3,0) \left[ \kappa\eta_0(|\xi | +  \p_{x_3})  \p_t \o_{ 0\xi, h}(0) - \p_tB_\xi (0) \right] , 0 )  \\
&\quad+  {\int^\infty_0 G_\xi (t,x_3, y) \p_t^2 \o_{0\xi} (y) \dd y} 
 + {\int^t_0\int^\infty_0 G_\xi (t-s, x_3, y) \p_s^2 N_\xi (s,y ) \dd y \dd s}\\
 &\quad 
- {\int^t_0  G_\xi (t-s, x_3, 0)( \p_s^2  B_\xi (s),0) \dd s} ,
 \end{split}
 \Ee
where we recall $\p_t^2\o_0$ in \eqref{idata}. As we do not require higher order compatibility condition for the horizontal vorticity, a new term representing the initial-boundary layer emerges. We first examine $G_\xi (t,z,0)$. Recall (\ref{est:Con_G}).
%

Similar to Lemma \ref{lem_Gt}, we have for $C_0<\infty$
\begin{align}
\sum_{j=0}^k \left\|  (\zeta(z)\p_{z})^j  G_\xi(t,z, 0 ) \right\|_{\mathcal{L}^1_\lambda} \lesssim  C_0 , \quad \sum_{j=0}^k \left\|  (\zeta(z)\p_{z})^j  G_\xi(t,z, 0 )  \right\|_{\mathcal{L}^\infty_{\lambda,\kappa t}}  \lesssim C_0.
\label{trace4}
\end{align}
From (\ref{trace4}), (\ref{trace1}) and (\ref{nonlinear1}) 
\begin{align*}
&
\sum_{0 \leq |\beta| \leq 2}
\sum_{\xi \in \mathbb{Z}^2} e^{\lambda |\xi|}
\left\|
D^\beta_\xi
(1+ |\xi |)
\Big[
( G_{\xi h} (t,x_3,0) \left( \kappa\eta_0(|\xi | +  \p_{x_3})  \p_t \o_{ 0\xi, h}(0) - \p_tB_\xi (0) \right) , 0 )\Big]\right\|_{\mathcal{L}^1_\lambda}\\
&\lesssim \kappa \eta_0 \sum_{0 \leq |\beta| \leq2} \| \nabla_h^\beta(1+ |\nabla_h|) \nabla  \p_t \o_{0, h}\|_{1,\lambda}
+ \sum_{0 \leq |\beta| \leq 2} \| \nabla^\beta_h (1+ |\nabla_h|) \p_t N(0) \|_{1,\lambda}\\
&\lesssim  \kappa \eta_0
\| (1+ |\nabla_h|^3) \nabla \p_t \o_{0 } \|_{1,\lambda}+ 
\| (1+|\nabla_h|^4) \p_t \o_{0} \|_{1,\lambda}
\sum_{0 \leq |\beta| \leq1} \| D^\beta (1+ |\nabla_h|^3) \p_t \o_{0}\|_{1,\lambda}. 
\end{align*}
Hence an ${L}^1$-based analytic norm is easily obtained as 
\Be\label{est:tt_1}
\begin{split}
&\vertiii{\p_t^2 \o (t)} _1\\&\lesssim
 \kappa \eta_0
\| (1+ |\nabla_h|^3) \nabla \p_t \o_{0 } \|_{1, \lambda} + 
\| (1+|\nabla_h|^4) \p_t \o_{0} \|_{1, \lambda}
\sum_{0 \leq |\beta| \leq1} \| D^\beta (1+ |\nabla_h|^3) \p_t \o_{0}\|_{1, \lambda}
 \\
& \ \ 
 +  \sum_{0\leq |\beta|\leq 2} \| D^\beta (1+|\nabla_h|) \p^2_t \o_0  \|_{1, \lambda_0}  
 +( t+  \frac{1}{\gamma_0})
  \sup_{0 \leq s \leq t}\vertiii{ \o(s)}_{s} \sup_{0 \leq s \leq t} \vertiii{\p_t^2 \o(s)}_s
 \\
 &\ \  +( t+  \frac{1}{\gamma_0})
  \sup_{0 \leq s \leq t}\vertiii{ \p_t \o(s)}_{s}  ^2.
\end{split}
\Ee

Now we move to the $L^\infty$-based analytic norm bound. We compute $\| G_\xi(t,z, 0 ) \|_{\mathcal{L}^\infty_{\lambda,\kappa t}}$: 
\begin{align*}
\|  (\zeta(z) \p_z)^jG_\xi(t,z, 0 ) \|_{\mathcal{L}^\infty_{\lambda,\kappa t}} \lesssim  \sup_z \left[\frac{ be^{(\bar\alpha-c_0 b) z}}{1+\phi_\kappa(z) +\phi_{\kappa t} (z)}\right]   +  \sup_z\left[ \frac{e^{ \bar\alpha z-c_0 \frac{|z|^2}{\kappa t}}}{\sqrt{\kappa t}+\sqrt{\kappa t}\phi_\kappa(z) +\sqrt{\kappa t}\phi_{\kappa t}(z)  }  \right]. 
\end{align*}
It is a routine to check that both supremum norms are uniformly bounded in $\kappa$ and $|\xi|$. Hence \eqref{trace4} shows that the kernel $G_\xi(t,z,0)$ is well-behaved in $\mathcal{L}^1_\lambda$ and the initial-boundary layer analytic space $\mathcal{L}^\infty_{\lambda,\kappa t}$. Then we run the same argument as in Step 2 but with $\mathcal{L}^\infty_{\lambda,\kappa t}$ in place of $\mathcal{L}^\infty_{\lambda,\kappa }$. Thanks to \eqref{trace4}, the estimates of the first term in \eqref{o_xi2tt} are bounded by the initial norm \eqref{initial_norm}:
\begin{align*}
&\sum_{0 \leq |\beta| \leq 2} \sum_{\zeta \in \mathbb{Z}^2}  e^{\lambda |\xi|}
 \left\|
D^{\beta}_\xi \Big[( G_{\xi h} (t,x_3,0) \left( \kappa\eta_0(|\xi | +  \p_{x_3})  \p_t \o_{ 0\xi, h}(0) - \p_t B_\xi (0) \right) , 0 ) \Big]\right\|_{\mathcal{L}^\infty_{\lambda, \kappa t}}
\\
&\lesssim \kappa \eta_0 \sum_{0 \leq |\beta| \leq 2} \| \nabla_h^\beta \nabla \p_t \o_{0} \|_{\infty, \lambda } + \sum_{0 \leq |\beta| \leq 2} \| \nabla^\beta_h   \p_t N(0) \|_{1,\lambda}\\
&\lesssim \kappa \eta_0 \sum_{0 \leq |\beta| \leq 2} \| \nabla_h^\beta \nabla \p_t \o_{0} \|_{\infty, \lambda } + \| (1+|\nabla_h|^3) \p_t \o_{0} \|_{1,\lambda}
\sum_{0 \leq |\beta| \leq1} \| D^\beta (1+ |\nabla_h|^2) \p_t \o_{0}\|_{1,\lambda}.
\end{align*}
Other three terms are estimated in the same way as in \cite{NN2018} or \cite{FW} and we arrive that 
\Be\label{est:tt_infty}
\begin{split}
&\vertiii{\p_t^2 \o (t)} _{\infty, \kappa t}\\
&\lesssim
\kappa \eta_0 \sum_{0 \leq |\beta| \leq 2} \| \nabla_h^\beta \nabla \p_t \o_{0} \|_{\infty, \lambda } + \| (1+|\nabla_h|^3) \p_t \o_{0} \|_{1,\lambda}
\sum_{0 \leq |\beta| \leq1} \| D^\beta (1+ |\nabla_h|^2) \p_t \o_{0}\|_{1,\lambda}
 \\
& \ \ 
 +  \sum_{0\leq |\beta|\leq 2} \| D^\beta  \p^2_t \o_0  \|_{\infty, \lambda_0, \kappa t}  
 +( \sqrt{t}+  \frac{1}{\gamma_0})
  \sup_{0 \leq s \leq t}\vertiii{ \o(s)}_{s} \sup_{0 \leq s \leq t} \vertiii{\p_t^2 \o(s)}_s
 \\
 & \ \  +(\sqrt{ t}+  \frac{1}{\gamma_0})
  \sup_{0 \leq s \leq t}\vertiii{ \p_t \o(s)}_{s}  ^2.
\end{split}
\Ee

Finally combining (\ref{est:tt_1}) and (\ref{est:tt_infty}) and then choosing sufficiently large $\gamma_0$ we derive a desired estimate for $\vertiii{\p_t^2 \o(t)}_t$ for $t \in (0, \frac{\lambda_0}{2\gamma_0})$.

Altogether from (\ref{est:|||_t}), (\ref{est:t_1}), (\ref{est:t_infty}), (\ref{est:tt_1}), and (\ref{est:tt_infty}), we finish the proof of the estimate (\ref{norm_bound}).

  \

 {\bf Step 4: Estimate (1), vorticity estimates. }Both \eqref{b1} and \eqref{b2} are direct consequences of (\ref{norm_bound}). To show \eqref{b3}, we first note that the boundedness of $\o(t)$ norms implies 
 $|\p_{x_3}\o_\xi(t,x_3)|\lesssim e^{-\bar{\alpha} x_3}e^{-\lambda |\xi|}$ for all $|\xi|$ and $x_3 \geq 1$ (away from the boundary). When $x_3\leq 1$, we draw on  the equation \eqref{NS_f} to rewrite $ \p_{x_3}^2\o_{\xi,h} = \frac{1}{\kappa\eta_0}\{ \p_t \o_{\xi,h}  + \kappa\eta_0 |\xi|^2 \o_{\xi,h} - N_{\xi,h}\}$ and the boundary condition \eqref{NSB_f}: 
 \Be\begin{split}
 \p_{x_3} \o_{\xi,h} (t, x_3) &= \p_{x_3} \o_{\xi,h} (t,0) + \int_0^{x_3}  \p_{x_3}^2\o_{\xi,h}(t,y) \dd y\\
 &=- |\xi| \o_{\xi,h} (t,0) + \frac{1}{\kappa\eta_0} B_\xi (t) +  \int_0^{x_3}  \frac{1}{\kappa\eta_0}[ \p_t \o_{\xi,h}  + \kappa\eta_0 |\xi|^2 \o_{\xi,h} - N_{\xi,h}](t,y)\dd y. 
\end{split}
\Ee
We now appeal to  $|B_\xi(t)| \leq \|N_\xi(t)\|_{\mathcal{L}^1_\lambda}$ and $\sum_{0\leq \ell\leq 1} ( \vertiii{ \p_t^\ell\o(t)}_{\infty,\kappa}+ \vertiii{ \p_t^\ell\o(t)}_{1}) <\infty$ to obtain that for all $x_3\in \mathbb R_+$
\Be 
| \p_{x_3} \o_{\xi, h} (t, x_3)|\lesssim  \frac{1}{\kappa} e^{-\bar{\alpha} x_3}e^{-\lambda |\xi|} \text{ for } 0<\lambda<\lambda_0,
\Ee
which proves \eqref{b3} for $\o_h$ and $\ell=0$. The remaining case can be estimated similarly. Near $O(1)$ boundary, from \eqref{NS_f1} and \eqref{NSB_f1}, we derive 
 \Be\begin{split}
 &\p_{x_3} \p_t\o_{\xi,h} (t, x_3)\\
  &=- |\xi| \p_t\o_{\xi,h} (t,0) + \frac{1}{\kappa\eta_0} \p_tB_\xi (t) +  \int_0^{x_3}  \frac{1}{\kappa\eta_0}[\p_t^2 \o_{\xi,h}  + \kappa\eta_0 |\xi|^2 \p_t\o_{\xi,h} - \p_tN_{\xi,h}](t,y)\dd y. 
\end{split}
\Ee
Together with  $\sum_{0\leq \ell\leq 1} \vertiii{ \p_t^\ell\o(t)}_{\infty,\kappa}+ \sum_{0\leq \ell\leq 2}  \vertiii{ \p_t^\ell\o(t)}_{1}<\infty$ we deduce \eqref{b3} for $\o_h$ and $\ell=1$.
For $\o_3$ we use $\nabla\cdot \o=0$ to write $\p_3 \o_3 = - \p_1\o_1- \p_2\o_2$. Now \eqref{b3} for $\o_3$ follows from \eqref{b1}. 


  \

 {\bf Step 5: Estimate (2), velocity estimates, except (\ref{ut}). }From (\ref{phi_xi}) 
\Be\label{est:phi}
|\xi|^{\beta_h} | \p_{z}^{\beta_3}\p_t^\ell\phi_\xi (t,z)|
\lesssim \int_{\p \mathcal{H}_\lambda} |\xi|^{|\beta|-1} e^{-|\xi| |y-z|} | \p_t^\ell \o_\xi (t,y)| |\dd y|
 \ \ \text{for} \ \beta_3\leq 1. 
\Ee
For $|\beta|= |\beta_h| + \beta_3=1$ we bound (\ref{est:phi}) by $e^{-\lambda |\xi|} \| \p_t^\ell \o(t) \|_{1,\lambda}$. Then from (\ref{norm_bound}) we conclude (\ref{est:u_t}).

For $|\beta|\geq 2$ and $\beta_3 \leq 1$, we bound (\ref{est:phi}) by 
\Be\label{est:phi1}
\begin{split}
(\ref{est:phi}) &\lesssim
\int_{\p \mathcal{H}_\lambda} |\xi|^{|\beta|-2} |\xi|e^{-|\xi| |y-z|} 
 e^{- {\bar{\alpha}}  \text{Re}\,y}e^{-\lambda |\xi|} \big(1+ \phi_\kappa (y)+ \phi_{\kappa t} (y)\big) 
|\dd y|  \\
&\lesssim |\xi|^{|\beta|-2} e^{-\lambda |\xi|} e^{- \min (1, \frac{\bar{\alpha}}{2})x_3}
\int_{\p \mathcal{H}_\lambda} 
e^{- \frac{\bar{\alpha}}{2} \text{Re} \, y}\big ( 1+ \phi_\kappa (y)+ \phi_{\kappa t} (y)\big) 
|\dd y|\\
&\lesssim |\xi|^{|\beta|-2} e^{-\lambda |\xi|} e^{- \min (1, \frac{\bar{\alpha}}{2})x_3} \ \ \text{for} \ |\beta| \geq 2, \ \text{and} \ \beta_3 \leq 1, \ \text{and} \  \ell=0,1,2, \ \text{and} \  t \in [0,T],
\end{split}
\Ee
where we have used $|\xi||y-z| + \frac{\bar{\alpha}}{2} \text{Re}\, y\geq \min (1, \frac{\bar{\alpha}}{2} ) x_3$ for $|\xi|\geq 1$ and (\ref{norm_bound}).

For $\beta_3=2,3$ we use $\p_z^2 \p_t^\ell\phi_\xi = |\xi|^2  \p_t^\ell\phi_\xi +   \p_t^\ell\o_\xi$. Then following the same argument of (\ref{est:phi1}), we derive 
\Be\label{est:phi2}
\begin{split}
&|\xi|^{\beta_h} | \p_{z}^{\beta_3}\p_t^\ell\phi_\xi (t,z)|\\
&\lesssim |\xi|^{|\beta_h|+2} | \p_{z}^{\beta_3-2}\p_t^\ell\phi_\xi (t,z)|
+|\xi|^{\beta_h} | \p_{z}^{\beta_3-2}\p_t^\ell\o_\xi (t,z)|\\
&\lesssim 
\begin{cases}
(|\xi|^{|\beta|-2}+|\xi|^{\beta_h}) e^{- \lambda |\xi|}  e^{- \min (1, \frac{\bar{\alpha}}{2}) \text{Re}\, z} (1+ \phi_\kappa (z)) \ \ \text{for} \ \ell=0,1, \ \text{and} \ \beta_3=2,\\
 (|\xi|^{|\beta|-2}+|\xi|^{\beta_h}) e^{- \lambda |\xi|}  e^{- \min (1, \frac{\bar{\alpha}}{2})\text{Re}\, z}\kappa^{-1}  \ \ \text{for} \ \ell=0,1, \ \text{and} \ \beta_3=3, \\
(|\xi|^{|\beta|-2}+|\xi|^{\beta_h}) e^{- \lambda |\xi|}  e^{- \min (1, \frac{\bar{\alpha}}{2})\text{Re}\, z} (1+ \phi_\kappa (z)+ \phi_{\kappa t} (z)) \ \ \text{for} \ \ell=2, \ \text{and} \ \beta_3=2. 
\end{cases}
\end{split}\Ee
Finally from (\ref{est:phi1}) and (\ref{est:phi2}) we conclude (\ref{est:u1}) and (\ref{est:u2}).

\hide
 From $e^{-|\xi| |y-z|} \leq 1$ or $|\xi|  e^{-|\xi| |\cdot |} \in L^1(\R)$ 
\Be\begin{split}\label{est:phi_t}
|\xi|^k|\p_t^\ell\phi_\xi (t,x_3)|
&\lesssim |\xi|^{ k-1 } e^{-\lambda |\xi|}
 \| \p_t^\ell\o  (t )\|_{1,\lambda},\\
(1+ \phi_\kappa (x_3)  )^{-1} |\xi|^k |\p_t^\ell \phi_\xi(t,x_3)| &\lesssim |\xi|^{k-2}e^{-\lambda |\xi|}     \|\p_t^\ell \o  (t )\|_{\infty, \lambda,\kappa } 
  \ \ \text{for} \ \ell=0,1,
 \\
(1+ \phi_\kappa (x_3) + \phi_{\kappa t} (x_3))^{-1}  |\xi|^k |\p_t^2 \phi_\xi(t,x_3)| &\lesssim |\xi|^{k-2}   e^{-\lambda |\xi|}
\|\p_t^2\o (t   )\|_{ \infty,\lambda, \kappa t }.
\end{split}\Ee
Once again from (\ref{phi_xi}) we have $\p_z \phi_\xi (z) = \int^z_0 \p_z G_- (y,z) \o_\xi (y) \dd y + \int^\infty_z \p_z G_+ (y,z) \o_\xi (y) \dd y$. Since $|\p_z G_{\pm}(y,z)|  \lesssim e^{-|\xi||y-z|}$
\Be\begin{split}\label{est:phi_tz}
|\xi|^k| \p_{x_3}\p_t^\ell\phi_\xi (t,x_3)| &\lesssim |\xi|^{k} e^{-\lambda |\xi|}
 \| \p_t^\ell\o  (t )\|_{1,\lambda},\\
 (1+ \phi_\kappa (x_3)  )^{-1} |\xi|^k | \p_{x_3} \p_t^\ell \phi_\xi(t,x_3)| &\lesssim |\xi|^{k-1}  e^{-\lambda |\xi|} \| \p_t^\ell \o (t )\|_{ \infty,\lambda,\kappa } 
  \ \ \text{for} \ \ell=0,1,
 \\
(1+ \phi_\kappa (x_3) + \phi_{\kappa t} (x_3))^{-1}  |\xi|^k |\p_{x_3} \p_t^2 \phi_\xi(t,x_3)| &\lesssim |\xi|^{k-1}   e^{-\lambda |\xi|}
\|\p_t^2\o  (t )\|_{ \infty,\lambda, \kappa t }.
\end{split}\Ee
Therefore we derive that, from (\ref{est:phi_t}), (\ref{est:phi_tz}), and $|\xi|^{|\beta_h|}| 
\p_t^\ell u_{\xi} (t,x_3)| \lesssim 
|\xi|^{|\beta_h|+1}| 
\p_t^\ell \phi_{\xi} (t,x_3)| 
+|\xi|^{|\beta_h| }| \p_z
\p_t^\ell \phi_{\xi} (t,x_3)|$, 
\begin{align}
| 
\p_t^\ell u_{\xi} (t,z)|&\lesssim 
 e^{-\lambda |\xi|} \| \p_t^\ell \o_\xi (t ) \|_{1,\lambda}   ,
 \notag
 \\
|\xi|^{|\beta_h|}| 
\p_t^\ell u_{\xi} (t,z)|&\lesssim |\xi|^{|\beta_h| -1}
e^{-\lambda |\xi|}  \| \p_t^\ell \o (t ) \|_{1,\lambda} 
+
(1+ \phi_\kappa(x_3))  
 |\xi|^{|\beta_h|-1} e^{-\lambda |\xi|}  \| \p_t^\ell \o  (t  ) \|_{ \infty,\lambda,\kappa}
  \ \ \text{for} \ |\beta_h|\geq 1 \ \text{and} \ \ell=0,1,
  \label{est:u_t}
  \\
  |\xi|^{|\beta_h|}| 
\p_t^2 u_{\xi} (t,z)|&\lesssim |\xi|^{|\beta_h| -1}
e^{-\lambda |\xi|}  \| \p_t^2\o (t ) \|_{1,\lambda} 
+
(1+ \phi_\kappa(x_3) +\phi_{\kappa t}(x_3) )  
 |\xi|^{|\beta_h|-1} e^{-\lambda |\xi|}  \| \p_t^2\o  (t ) \|_{ \infty,\lambda,\kappa t}
  \ \ \text{for} \ |\beta_h|\geq 1.\notag
\end{align}

From $\p_z^2 \p_t^\ell\phi_\xi = |\xi|^2  \p_t^\ell\phi_\xi +   \p_t^\ell\o_\xi$, (\ref{est:phi_t}) and (\ref{est:phi_tz}) we further derive that 
\Be\begin{split}\label{est:phi_tzz}
|\xi|^k| \p_{x_3}^2\p_t^\ell\phi_\xi (t,x_3)| &\lesssim |\xi|^{k+1} e^{-\lambda |\xi|}
 \| \p_t^\ell\o  (t )\|_{1,\lambda}
 + e^{-\lambda|\xi|} (1+ \phi_\kappa (x_3)) \sum_{m=0,1} \vertiii{ \p_t^m\o(t)}_{\infty, \kappa}  \ \ \text{for} \ \ell=0,1
 ,\\
 |\xi|^k| \p_{x_3}^2\p_t^2\phi_\xi (t,x_3)| &\lesssim |\xi|^{k+1} e^{-\lambda |\xi|}
 \| \p_t^2\o  (t )\|_{1,\lambda}
 + e^{-\lambda|\xi|}
  (1+ \phi_\kappa (x_3)+ \phi_{\kappa t} (x_3))
  \vertiii{ \p_t^2\o(t)}_{\infty, \kappa} .
  \end{split}\Ee
 Therefore we derive that, from (\ref{est:phi_tzz}) and (\ref{est:phi_tz})
\Be\begin{split} \label{est:u_tz}
 |\xi|^{|\beta_h|} | 
 \p_{x_3} \p_t^\ell u_{\xi} (t,z)|&\lesssim 
|\xi| ^{|\beta_h|+1}e^{-\lambda |\xi| }  \| \p_t^\ell \o(t) \|_{1, \lambda }
 + |\xi| ^{|\beta_h| } e^{-\lambda|\xi|} (1+ \phi_\kappa (x_3)) \sum_{m=0,1} \vertiii{ \p_t^m\o(t)}_{\infty, \kappa}  \ \ \text{for} \ \ell=0,1,
 \\
  |\xi|^{|\beta_h|} | 
 \p_{x_3} \p_t^2 u_{\xi} (t,z)|&\lesssim 
|\xi| ^{|\beta_h|+1}e^{-\lambda |\xi| }  \| \p_t^2 \o(t) \|_{1, \lambda }
 + |\xi| ^{|\beta_h| } e^{-\lambda|\xi|} (1+ \phi_\kappa (x_3)+ \phi_{\kappa t} (x_3))  \vertiii{ \p_t^2\o(t)}_{\infty, \kappa t}   .
\end{split}\Ee
From (\ref{est:u_t}), (\ref{est:u_tz}), $ 
\nabla_{x_h}^{\beta_h} \p_{z}^{\beta_3} \p_t ^\ell u(t,x_h,z)
= \sum_{\xi  \in \mathbb Z^2}
(i \xi)^{\beta_h } \p_t^\ell \p_{z}^{\beta_3} u_\xi (t,z) e^{i x_h \cdot \xi }$ and (\ref{norm_bound}) we derive that 
\Be
\begin{split}\label{est:u1}
\sum_{0 \leq |\beta | \leq 1} \sum_{\ell=0,1}  |\nabla ^{\beta } \p_t^\ell u(t,x )|
 \lesssim  1+ \phi_\kappa (x_3)  ,\ \
\sum_{0 \leq |\beta | \leq 1}  |\nabla ^{\beta } \p_t^2 u(t,x )|
 \lesssim  1+ \phi_\kappa (x_3) + \phi_{\kappa t} (x_3) .
\end{split}
\Ee

Now we estimate the second order derivatives of the velocity field. Note that $\p_{x_3}^3 \p_t^\ell\phi_\xi = |\xi|^2 \p_{x_3} \p_t^\ell\phi_\xi +  \p_{x_3} \p_t^\ell\o_\xi$. From (\ref{b3}) and (\ref{est:phi_tz}), 
\Be
|\p_{x_3}^3 \p_t^\ell\phi_\xi| \lesssim
|\xi|^{2} e^{-\lambda |\xi| } \| \p_t^\ell \o(t) \|_{1,\lambda }
 +\kappa^{-1} e^{-\bar \alpha x_3} e^{-\lambda |\xi|}
 \ \ \text{for} \ \ell=0,1.\notag
\Ee
Together this bound with (\ref{est:phi_tzz}) and $
\p_{x_3}^2\p_t^\ell u_{\xi} (t,x_3)| \lesssim 
|\xi| | 
\p_{x_3}^2\p_t^\ell \phi_{\xi} (t,x_3)| 
+  |  \p_{x_3}^3
\p_t^\ell \phi_{\xi} (t,x_3)|$, we derive that 
\[
\sum_{|\beta|=2} \sum_{\ell=0,1} |\nabla^\beta \p_t^\ell u(t,x)|\lesssim 1+ \phi_\kappa (x_3). 
\]
\unhide

 \

 {\bf Step 6: Estimate (3), pressure estimates and (\ref{ut}). }We next turn to the pressure. Taking the divergence to (\ref{NS_k}) and using (\ref{incomp}), we deduce 
\Be\label{eqtn:p}
- \Delta p = \sum_{\ell,m=1}^3 \p_{ \ell} u_m \p_{ m} u_\ell.
\Ee
We obtain the boundary condition of $p$ by reading the third component of (\ref{NS_k}), and then using (\ref{incomp}) and (\ref{noslip}),  
\Be
\begin{split}\label{BC:p}
\p_{ 3} p &= \kappa \eta_0 \Delta u_3=  \kappa \eta_0 \p_{ 3}  \p_{ 3}  u_3 =  - \kappa \eta_0  \p_1 \p_{ 3}  u_1 - \kappa \eta_0 \p_2\p_{ 3}  u_2\\
&
 = - \kappa \eta_0 \p_1 (\o_2 + \p_1 u_3) - \kappa \eta_0 \p_2 (-\o_1 + \p_2 u_3)\\
& = - \kappa \eta_0 \p_1 \o_2  + \kappa \eta_0 \p_2  \o_1  \ \ \text{for} \ x_3=0,
\end{split}
\Ee
 where $\o_1= \p_2 u_3 - \p_3 u_2$ and $\o_2 = - \p_1 u_3 + \p_3 u_1$.

In the Fourier side we read the problem as 
\Be
\begin{split}\label{p_xi}
(|\xi|^2 - \p_3^2) p_\xi (t,x_3) = g_\xi(t,x_3):=
\sum_{\ell, m=1}^3 (\p_\ell u_m \p_m u_\ell)_\xi (t,x_3)
  \ \ \text{for} \ x_3 \in \R_+,\\
\p_3 p_\xi (t,0)  =-  i   \kappa \eta_0 \xi_1 \o_{\xi,2} (t,0)+ i \kappa \eta_0 \xi_2 \o_{\xi, 1}(t,0).
\end{split}
\Ee
A representation of $p_\xi(t,x_3)$ is given by 
\Be\label{form:p_xi}
\begin{split}
p_\xi(t,x_3)&= -\int^{x_3}_0 \frac{1}{2|\xi|} e^{-|\xi| (x_3-y)} g_\xi(y) \dd y  
- \int_{x_3}^\infty \frac{1}{2|\xi|} e^{-|\xi| (y-x_3)} g_\xi(y) \dd y\\
& \ \ \ 
- \int^\infty_0 \frac{1}{2|\xi|} e^{-|\xi| (y+x_3)} g_\xi (y) \dd y 
\\
&
- \frac{1}{|\xi|} e^{-|\xi| x_3} (-  i   \kappa \eta_0 \xi_1 \o_{\xi,2} (t,0)+ i \kappa \eta_0 \xi_2 \o_{\xi, 1}(t,0)),
\end{split}
\Ee
which is valid for all $\xi\neq 0$. 
When $\xi=0$, by integrating \eqref{p_xi} and by using the boundary conditions $\p_3 p_0(t,0)=0$, $u(t,x_h,0)=0$ and the divergence free condition $\nabla \cdot u=0$, we first obtain 
\Be\label{p03}
\begin{split}
\p_3 p_0(t,x_3)&= - \frac{1}{(2\pi)^2} \int_0^{x_3}\iint_{\mathbb T^2}   \sum_{\ell,m=1}^3 \p_{ \ell} u_m \p_{ m} u_\ell \dd x_h \dd y_3  \\
&= - \frac{1}{(2\pi)^2}\iint_{\mathbb T^2} (u\cdot\nabla u_3)(t,x_h,x_3)  \dd x_h \\
&=  - \frac{2}{(2\pi)^2}\iint_{\mathbb T^2} (u_3\p_3 u_3)(t,x_h,x_3)  \dd x_h ,
\end{split}
\Ee 
where we have used the integration by parts and $\nabla \cdot u=0$ at the last step. 

Observe that $\p_3p_0$ decays exponentially in $x_3$, and in particular $\int_0^\infty | \p_3 p_0(t,x_3) |   \dd x_3 <\infty$. The integration yields 
\[
p_0(t,x_3) =p_0(t,0) -  \int_0^{x_3} \frac{2}{(2\pi)^2}\iint_{\mathbb T^2} (u_3\p_3 u_3)(t,x_h,y_3)  \dd x_h \dd y_3 . 
\]
Since $p_0(t,0)$ is a free constant in $x_3$, we fix $p_0(t,x_3)$ by choosing $$p_0(t,0)= \frac{2}{(2\pi)^2}\int_0^{\infty} \iint_{\mathbb T^2} (u_3\p_3 u_3)(t,x_h,y_3) \dd x_h  \dd y_3 <\infty,$$ such that  
\Be\label{p0}
\begin{split}
 p_0(t,x_3) &=  \frac{2}{(2\pi)^2}\int_{x_3}^\infty\iint_{\mathbb T^2} (u_3\p_3 u_3)(t,x_h,y_3) \dd x_h \dd y_3 .
\end{split}
\Ee 
The pressure $p$ is then recovered by 
\Be\label{pressure}
p(t,x_h,x_3) = p_0 (t,x_3) + \sum_{|\xi|\geq 1, \xi\in \mathbb Z^2} p_\xi(t,x_3) e^{ix_h\cdot\xi} ,
\Ee
where $ p_0 (t,x_3) $ and $p_\xi(t,x_3)$ are given in \eqref{p0} and \eqref{form:p_xi}.

Now the pressure estimate follows readily from the velocity and vorticity estimates. To show \eqref{est:p}, we first note from \eqref{est:u_t} and \eqref{est:u1} 
$|p_0(t,x_3)|\lesssim |u_3(t,x)| \int_\Omega |\p_3 u_3 (t,x)| \dd x \lesssim 1$
and from Lemma \ref{lem_elliptic} 
\[
\begin{split}
| g_\xi |& \lesssim e^{-\lambda |\xi|}\sum_{i=1}^2  \left( \|\p_i u_h \|^2_{\infty,\lambda}+ \|\zeta^{-1}\p_i u_3\|_{\infty,\lambda}\|\zeta \p_3 u_i\|_{\infty,\lambda} \right) \\
& \lesssim  e^{-\lambda |\xi|} \left[ \sum_{0\leq |\beta|\leq 1} \|\nabla_h^\beta \o\|^2_{1,\lambda}+\sum_{1\leq |\beta|\leq 2} \|\nabla_h^\beta \o_h\|_{1,\lambda} \left( \sum_{0\leq |\beta|\leq 1} \|\nabla_h^\beta \o\|_{1,\lambda}+ \|\zeta \o_h\|_{\infty,\lambda}\right)  \right] ,
\end{split}
\]
from which we deduce $|p(t,x_h,x_3)|\lesssim 1$. The estimation of $\p_t p $ and $\p_t^2p$ follows analogously. 

For the decay estimates \eqref{est:pdecay}, we start with $\ell=0$ and $\beta=0$. Due to our choice of $p_0(t,x_3)$ in \eqref{p0}, using \eqref{est:u_t} and \eqref{est:u1}, we have the spatial decay for $p_0(t,x_3)$:
\[
|p_0(t,x_3)|\lesssim  \int_{x_3}^\infty \iint_{\mathbb T^2}  (1+\phi_\kappa(y_3)) e^{-\min(1,\frac{\bar\alpha}{2}) y_3 } \dd x_h  \dd y_3\lesssim \kappa^{-\frac12} e^{-\min(1,\frac{\bar\alpha}{2}) x_3 }.
\]
For $\xi\neq 0$, we use another estimate for $|g_\xi|$ and Lemma \ref{lem_elliptic}
\Be\label{g_xi}
\begin{split}
|g_\xi(y)| &\lesssim \sum_{\ell,m=1}^3\sum_{\eta\in \mathbb Z^2} e^{-\lambda|\xi-\eta|} e^{-\min(1,\frac{\bar\alpha}{2}) y}(1+\phi_\kappa(y)) |(\p_m u_\ell)_\eta(y)|  \\
&\lesssim \kappa^{-\frac12} e^{-\lambda |\xi|} e^{-\min(1,\frac{\bar\alpha}{2}) y} \sum_{\ell,m=1}^3 
\sum_{\eta\in \mathbb Z^2} e^{\lambda |\eta| } |(\p_m u_\ell)_\eta(y)| ,
\end{split}
\Ee
from which we deduce that $| p_\xi(t,x_3) |\lesssim  \kappa^{-\frac12} e^{-\min(1,\frac{\bar\alpha}{2}) x_3} $. Hence \eqref{est:pdecay} holds for $\ell=0$ and $\beta=0$. For the pressure gradient estimate when $|\beta|=1$,  
from \eqref{p03} and \eqref{est:u1} we first note  
\[
|\p_3 p_0(t,x_3) | \lesssim \sup_{x_h \in \mathbb T^2}( |u_3| |\p_3 u_3|)  \lesssim (1+\phi_\kappa(x_3)) e^{-\min(1,\frac{\bar\alpha}{2}) x_3}.
\]
For $\xi\neq 0$, by \eqref{g_xi} it is easy to see that $|\xi p_\xi(t,x_3) |\lesssim  \kappa^{-\frac12} e^{-\min(1,\frac{\bar\alpha}{2}) x_3} $. Note that $\p_3 p_\xi(t,x_3)$ has a similar integral form as $|\xi| p_\xi(t,x_3)$ and the estimate follows in the same way, which results in  $ |\p_3 p_\xi(t,x_3) |\lesssim  \kappa^{-\frac12} e^{-\min(1,\frac{\bar\alpha}{2}) x_3}$. This finishes \eqref{est:pdecay} for $\ell=0$ and $|\beta|=1$. The remaining cases for $\ell=1$ and $|\beta|=0,1$ can be treated in the same way. 

For the decay estimate of $\p_t^2p$, we take into account the initial layer which occurs at $\p_t^2 \o$ and $\nabla \p_t^2 u$. First using \eqref{est:u_t}, \eqref{est:u1} and \eqref{est:u2} we have 
\[
\begin{split}
|\p_t^2 p_0(t,x_3) |&\lesssim \left| \int_{x_3}^\infty\iint_{\mathbb T^2} (u_3\p_3\p_t^2 u_3+ \p_t^2u_3\p_3 u_3+2\p_tu_3\p_3\p_t u_3)(t,x_h,y_3) \dd x_h \dd y_3 \right| \\
&\lesssim  \big(1+ \phi_\kappa (x_3) + \phi_{\kappa t} (x_3)\big)e^{-\min (1, \frac{\bar{\alpha}}{2} )x_3} ,
\end{split}
\]
while for $|\xi|\neq 0$ we have 
\[
\begin{split}
|\p_t^2 g_\xi(y)| &\lesssim \sum_{\ell,m=1}^3\sum_{\eta\in \mathbb Z^2} e^{-\lambda|\xi-\eta|} e^{-\min(1,\frac{\bar\alpha}{2}) y}(1+\phi_\kappa(y)) |(\p_m u_\ell)_\eta(y)|  \\
&\lesssim \kappa^{-\frac12} e^{-\lambda |\xi|} e^{-\min(1,\frac{\bar\alpha}{2}) y} \sum_{i=1}^2 \sum_{\ell,m=1}^3 
\sum_{\eta\in \mathbb Z^2} e^{\lambda |\eta| } |(\p_m \p_t^i u_\ell)_\eta(y)| ,
\end{split}
\]
from which we deduce \eqref{est:pt2}. 

The last estimate for $\p_t^\ell u$ for $\ell=1,2$ follows from the equation: $\p_t  u = \kappa\eta_0 \Delta u - u \cdot\nabla u - \nabla  p $ and $\p_t^2 u = \kappa\eta_0 \Delta \p_t u - u \cdot\nabla  \p_t u - \p_t u \cdot\nabla u - \nabla \p_t  p $.

\section*{Acknowledgements} 

Part of this work was conducted while the authors were participating in the INdAM worksop ``Recent advances in kinetic equations and applications'' organized by Francesco Salvarani in Rome. We thank the institute and the organizer for its generous hospitality and support. JJ was supported in part by the NSF Grant DMS-1608494 and by the Simons Fellowship (grant \# 616364). CK was supported in part by National Science Foundation under Grant No. 1501031, Grant No. 1900923, and the Wisconsin Alumni Research Foundation.

\appendix \section{Sobolev embedding in 1D}

 Often we have used a standard 1D embedding: For $T>0$, 
\Be\label{Sob_1D}
|g(t)|^2 \lesssim_T \int_0^T |g(s)|^2 \dd s + \int_0^T |g^\prime(s)|^2 \dd s \ \ \text{for } t \in [0,T].
\Ee

A proof is based on an equality:
\Be
|g(t)|^2= \frac{1}{T/2 }\int^{t+T /2}_t \Big(
g(s) - \int^s_t g^\prime (\tau) \dd \tau 
\Big)^2 \dd s .\notag
\Ee
For $0< t \leq T/2$, 
\Be\notag
\begin{split}
|g(t)|^2 &\leq  \frac{1}{T/2}\int^{t+T/2}_t \Big(
2 |g(s)|^2 + 2\Big | \int^s_t g^\prime (\tau) \dd \tau 
\Big|^2 \Big) \dd s\\
&\leq  \frac{1}{T/2}\int^{t+T/2}_t \Big(
2 |g(s)|^2 + 2 |s-t|  \int^s_t |g^\prime (\tau)|^2 \dd \tau 
  \Big) \dd s\\
  &\leq  \frac{2}{T/2}\int^{t+T/2}_t  
  |g(s)|^2 \dd s + \frac{2}{T/2} \int^{t+ T/2}_t   |s-t|  \int^s_t |g^\prime (\tau)|^2 \dd \tau 
   \dd s\\
   &\leq 
     \frac{2}{T/2}\int^{t+T/2}_t  
  |g(s)|^2 \dd s+ T\int^{t+T/2}_t  
  |g^\prime(s)|^2 \dd s\\
  & \lesssim_T    \int_0^T |g(s)|^2 \dd s + \int_0^T |g^\prime(s)|^2 \dd s.
\end{split}
\Ee

For $ T/2<t \leq T$, using 
\Be
|g(t)|^2= \frac{1}{T/2 }\int_{t-T /2}^t \Big(
g(s) - \int^s_t g^\prime (\tau) \dd \tau 
\Big)^2 \dd s ,\notag
\Ee
we derive that 
\Be\notag
\begin{split}
|g(t)|^2 &\leq  \frac{1}{T/2}\int_{t-T/2}^t \Big(
2 |g(s)|^2 + 2\Big | \int^s_t g^\prime (\tau) \dd \tau 
\Big|^2 \Big) \dd s\\
   &\leq 
     \frac{2}{T/2}\int_{t-T/2}^t  
  |g(s)|^2 \dd s+ T\int_{t-T/2}^t  
  |g^\prime(s)|^2 \dd s\\
  & \lesssim_T    \int_0^T |g(s)|^2 \dd s + \int_0^T |g^\prime(s)|^2 \dd s.
\end{split}
\Ee

\end{document}